\documentclass[12pt,twoside,a4paper]{amsart}
\usepackage{amssymb,amscd,amsxtra,calc,mathrsfs}
\usepackage{amsmath}
\usepackage[all]{xy}


\allowdisplaybreaks[1]

\title[K-stability]{On K-stability 
for Fano threefolds of rank $3$ and degree $28$}
\author{Kento Fujita} 
\date{\today}
\address{Department of Mathematics, Graduate School of Science, Osaka University, 
Toyonaka, Osaka 560-0043, Japan}
\email{fujita@math.sci.osaka-u.ac.jp}

\newcommand{\pr}{\mathbb{P}}

\newcommand{\Z}{\mathbb{Z}}
\newcommand{\Q}{\mathbb{Q}}
\newcommand{\R}{\mathbb{R}}
\newcommand{\C}{\mathbb{C}}

\newcommand{\D}{\mathbb{D}}

\newcommand{\B}{{\bf B}}

\newcommand{\G}{\mathbb{G}}
\newcommand{\E}{\mathbb{E}}

\newcommand{\NE}{\operatorname{NE}}
\newcommand{\Nef}{\operatorname{Nef}}

\newcommand{\Eff}{\operatorname{Eff}}

\newcommand{\Bl}{\operatorname{Bl}}
\newcommand{\Supp}{\operatorname{Supp}}

\newcommand{\Sing}{\operatorname{Sing}}

\newcommand{\Pic}{\operatorname{Pic}}

\newcommand{\CaCl}{\operatorname{CaCl}}

\newcommand{\Aut}{\operatorname{Aut}}

\newcommand{\coker}{\operatorname{coker}}
\newcommand{\lct}{\operatorname{lct}}

\newcommand{\ord}{\operatorname{ord}}

\newcommand{\vol}{\operatorname{vol}}
\newcommand{\Image}{\operatorname{Image}}
\newcommand{\interior}{\operatorname{int}}
\newcommand{\Cone}{\operatorname{Cone}}
\newcommand{\rest}{\operatorname{rest}}

\newcommand{\red}{\operatorname{red}}

\newcommand{\gen}{\operatorname{gen}}
\newcommand{\length}{\operatorname{length}}
\newcommand{\Nklt}{\operatorname{Nklt}}

\newcommand{\Gr}{\operatorname{Gr}}

\newcommand{\sC}{\mathcal{C}}

\newcommand{\sO}{\mathcal{O}}
\newcommand{\sN}{\mathcal{N}}

\newcommand{\sH}{\mathcal{H}}

\newcommand{\sF}{\mathcal{F}}
\newcommand{\sG}{\mathcal{G}}

\newcommand{\sS}{\mathcal{S}}

\newtheorem{thm}{Theorem}[section]
\newtheorem{lemma}[thm]{Lemma}
\newtheorem{proposition}[thm]{Proposition}
\newtheorem{corollary}[thm]{Corollary}
\newtheorem{claim}[thm]{Claim}

\theoremstyle{definition}
\newtheorem{definition}[thm]{Definition}
\newtheorem{remark}[thm]{Remark}
\newtheorem{notation}[thm]{Notation}

\newtheorem{example}[thm]{Example}
\newtheorem*{ack}{Acknowledgments}

\begin{document}

\maketitle 

\begin{abstract}
We show that there exists a K-stable smooth Fano threefold of the 
Picard rank $3$, the anti-canonical degree $28$ and the third Betti number $2$. 
\end{abstract}

\setcounter{tocdepth}{1}
\tableofcontents

\section{Introduction}\label{intro_section}

Let $X$ be an $n$-dimensional \emph{Fano manifold}, i.e., $X$ is a smooth projective 
variety over the complex number field $\C$ with $-K_X$ ample. It is a 
classical problem whether $X$ admits a \emph{K\"ahler-Einstein metric} or not. 
It has been known that the existence of a K\"ahler-Einstein metric 
equivalent to \emph{K-polystability} of $X$ 
(see \cite{don, tian, B, CDS1, CDS2, CDS3, tian2} and references therein). 
The condition of K-polystability is purely algebraic. However, in general, it is 
difficult to determine K-polystability of Fano manifolds. 
If $n\leq 2$, then we already know the complete answer (see \cite{Tia87, OSS}). 
However, for $n=3$, we had only few answers. 

Recently, the authors in \cite{FANO} started to understand the case $n=3$. 
It has been known that smooth Fano threefolds are classified by \cite{isk, isk2, MoMu} 
and each family is parametrized by an irreducible variety 
(see \cite{MM84, mukai, KPS} and references therein). 
The authors in \cite{FANO} considered the problem whether there exists 
a K-polystable member or not in each family. The problem is crucial from 
the moduli-theoretic viewpoint (see \cite[Hypothesis 1.2]{OSS} for example). 
The main techniques in \cite{FANO} are, 
the evaluation of \emph{$\alpha$-invariants} 
(see \cite{Tia87}) and \emph{$\delta$-invariants} (see \cite{FO, BJ}), etc. 
Especially, in order to evaluate local $\delta$-invariants, the theory of 
Ahmadinezhad--Zhuang \cite{AZ} is crucial. In fact, in the article \cite{FANO}, 
the authors interpreted the result 
\cite[Corollary 2.22]{AZ} in terms of intersection numbers 
when $X$ is a $3$-dimensional \emph{Mori dream space} (see \cite{HK}) 
and $W_{\vec{\bullet}}$ is the refinement of the complete linear series 
(see \S \ref{AZR_section}) by prime Cartier divisors $Y$ on $X$, 
or the refinement of the $W_{\vec{\bullet}}$ by Cartier divisors on $Y$. 
The authors in \cite{FANO} completely determined the above problem 
excepts for one family (denoted by ``No.\ 3.11") by using the above 
Ahmadinezhad--Zhuang's formula and so on. The family No.\ 3.11, corresponds to 
No.\ 11 of Table 3 in \cite{MoMu}, is characterized by the blowups of 
$V=\pr_{\pr^2}\left(\sO\oplus\sO(1)\right)$ with the centers smooth complete 
intersections of two members in $|-\frac{1}{2}K_V|$ (see \S \ref{311_section}). 
In order to consider the members in No.\ 3.11, we need a slight generalization 
of the formula in \cite{FANO} (see \S \ref{AZ_section}), 
and very careful analysis of the local 
$\delta$-invariants. The main result of the paper is the following: 

\begin{thm}[{see Theorem \ref{mainthm} in detail}]\label{main_thm}
There exists a K-polystable member in No.\ 11 of Table 3 in Mori--Mukai's table 
\cite{MoMu}. 
\end{thm}

Note that, K-stability of $X$ is equivalent to K-polystability of $X$ and the condition 
$\Aut^0(X)=\{1\}$. It is known in \cite{PCS} that any member $X$ 
in No.\ 3.11 satisfies that 
$\Aut^0(X)=\{1\}$. Therefore, Theorem \ref{main_thm} especially asserts the existence 
of K-stable member in No.\ 3.11. 
In particular, by \cite{don_aut, odk_aut, BL}, general members in 
No.\ 3.11 are K-stable. Hence, together with the result in \cite{FANO}, 
we complete the main problem in \cite{FANO}. 

We organize the structure of the paper. In \S \ref{K_section}, we recall the definition 
for K-stability of Fano manifolds. Especially, we consider an equivariant version 
of a valuative criterion for K-stability of Fano varieties, established in \cite{zhuang}. 
In \S \ref{AZR_section}--\S \ref{AZ_section}, 
we review Ahmadinezhad--Zhuang's theory and give a 
slight generalization of the formulas given in \cite{FANO}. In \S \ref{311_section}, 
we see the structures of the members in No.\ 3.11. Especially, we see the important 
examples provided by Cheltsov and Shramov. 
In \S \ref{local_general_section}--\S \ref{special-II_section}, we evaluate 
local $\delta$-invariants for various points by using the formulas in \S \ref{AZ_section}. 
The sections, especially \S \ref{special-II_section}, are the hardest parts in the paper. 
In \S \ref{main_section}, we show that the Fano threefold in 
Example \ref{CS_example} \eqref{CS_example2} is K-stable by using the 
evaluations in \S \ref{local_general_section}--\S \ref{special-II_section} and by 
applying the standard techniques established in \cite{V22KE, FANO}. 
In \S \ref{appendix_section}, we see several basic properties of 
local $\delta$-invariants, as an appendix of \S \ref{AZR_section}.

\begin{ack}
The author would like to thank Carolina Araujo, Ana-Maria Castravet, Ivan Cheltsov, 
Anne-Sophie Kaloghiros, Jesus Martinez-Garcia, Constantin Shramov, 
Hendrick S\"uss and Nivedita Viswanathan for many discussions. 
In fact, many parts the paper was originally a part of the project \cite{FANO}. 
Especially, Ivan Cheltsov and Constantin Shramov provided the author 
Example \ref{CS_example}. The examples are crucial for the paper. 
The work started while the author enjoyed the AIM workshop ``K-stability and 
related topics" on January 2020. 
The author thanks the staffs of AIM for the stimulating environment. 
The author thanks Hamid Ahmadinezhad and Ziquan Zhuang 
for answering questions about the paper \cite{AZ}. 
This work was supported by JSPS KAKENHI Grant Number 18K13388.
\end{ack}

We work over the complex number field $\C$. For the minimal model program, 
we refer the readers to \cite{KoMo}. For the theory of graded linear series, 
we refer the readers to \cite{AZ}. 

\section{K-stability of log Fano pairs}\label{K_section}

The notion of \emph{K-stability} was originally introduced by \cite{tian, don}. 
In this paper, we only see its interpretations \cite{li, vst, BX}. 
See \cite{xu} for backgrounds. 

\begin{definition}\label{logfano_definition}
A pair $(X, \Delta)$ is said to be a \emph{log Fano pair} if $(X, \Delta)$ is a 
projective klt pair with $\Delta$ effective $\Q$-Weil divisor and 
$-(K_X+\Delta)$ ample. 
\end{definition}

\begin{definition}\label{volume_definition}
Let $X$ be an $n$-dimensional projective variety, let $L$ be an $\R$-Cartier 
$\R$-divisor on $X$, and let $E$ be a prime divisor over the normalization of 
$X$, i.e., there exists a resolution $\sigma\colon\tilde{X}\to X$ 
of singularities such that $E$ is a prime divisor on $\tilde{X}$. 
\begin{enumerate}
\renewcommand{\theenumi}{\arabic{enumi}}
\renewcommand{\labelenumi}{(\theenumi)}
\item\label{volume_definition1}
We set 
\[
\vol_X(L-uF):=\vol_{\tilde{X}}\left(\sigma^*L-uE\right)
\]
for any $u\in\R_{\geq 0}$, where $\vol_{\tilde{X}}$ is the volume function 
(see \cite{L1, L2}).The function is continuous over $u\in\R_{\geq 0}$, 
and identically equal to zero when $u \gg 0$. 
We set 
\[
\tau_L(E):=\sup\left\{\tau\in\R_{\geq 0}\,\,|\,\,\vol_X(L-\tau E)>0\right\}.
\]
The definitions do not depend on the choice of $\sigma$. 
(In fact, when $X$ is normal, for any line bundle $M$ on $X$, the sub-vector space 
\[
H^0\left(X, M-uE\right):=H^0\left(\tilde{X}, \sigma^*M-u E\right)\subset H^0(X, M)
\]
of $H^0(X, M)$ is not depend on the choice of $\sigma$ for any $u\in\R_{\geq 0}$. 
See also \cite[Proposition 2.2.43]{L1}.) 
If $(X, \Delta)$ is a log Fano pair, then we set 
\[
\tau_{X, \Delta}(E):=\tau_{-(K_X+\Delta)}(E). 
\]
If moreover $\Delta=0$, we simply denote $\tau_{X, \Delta}(E)$ by $\tau_X(E)$. 
\item\label{volume_definition2}
If $L$ is big, then we set 
\[
S_L(E):=\frac{1}{\vol_X(L)}\int_0^{\tau_L(E)}\vol_X(L-uE)du.
\]
If $(X, \Delta)$ is a log Fano pair, then we set 
\[
S_{X, \Delta}(E):=S_{-(K_X+\Delta)}(E). 
\]
If moreover $\Delta=0$, we simply denote $S_{X, \Delta}(E)$ by $S_X(E)$. 
We remark that if $L$ is a nef and big $\Q$-divisor, then we have
\[
\frac{1}{n+1}\tau_L(E)\leq S_L(E)\leq\frac{n}{n+1}\tau_L(E) 
\]
by \cite[Proposition 2,1 and Lemma 2.2]{pltK} and \cite[Proposition 3.11]{BJ}. 
See Corollary \ref{a-d_corollary} in detail.
\item\label{volume_definition3}
Let $\Delta$ be an $\Q$-Weil divisor on $X$, that is, $\Delta$ is a finite 
$\Q$-linear sum of subvarieties of codimension one. 
Assume that, at the generic point $\eta$ of $c_X(E)$, the variety $X$ is normal and 
$K_X+\Delta$ is $\Q$-Cartier, where 
$c_X(E)$ is the center of $E$ on $X$. 
Let $A_{X,\Delta}(E)$ be the \emph{log discrepancy} of $(X, \Delta)$ along $E$, 
that is, around a neighborhood of $\eta$, we can take the pullback 
$\sigma^*(K_X+\Delta)$ and define 
\[
A_{X, \Delta}(E):=\ord_E\left(K_{\tilde{X}}-\sigma^*(K_X+\Delta)\right)+1. 
\]
If $\Delta=0$, then we simply denote it by $A_X(E)$. 
\item\label{volume_definition4}
Assume that $(X, \Delta)$ is a log Fano pair. Take an $m\in\Z_{>0}$ with 
$-m(K_X+\Delta)$ Cartier. If the graded $\C$-algebra
\[
\bigoplus_{j,k\in\Z_{\geq 0}}H^0\left(X, -m j(K_X+\Delta)-k E\right)
\]
is finitely generated over $\C$, then the divisor $E$ is said to be a 
\emph{dreamy} prime divisor over $(X, \Delta)$. 
\end{enumerate}
\end{definition}

The following is an interpretation of the classical definition of K-stability \cite{tian, don}.

\begin{definition}[{\cite{li, vst}}]\label{K_definition}
Let $(X, \Delta)$ be a log Fano pair. The pair $(X, \Delta)$ is said to be 
\emph{K-stable} if 
\[
\frac{A_{X,\Delta}(E)}{S_{X,\Delta}(E)}>1
\]
for any dreamy prime divisor $E$ over $(X, \Delta)$. 
\end{definition}

\begin{remark}\label{K-dfn_remark}
\begin{enumerate}
\renewcommand{\theenumi}{\arabic{enumi}}
\renewcommand{\labelenumi}{(\theenumi)}
\item\label{K-dfn_remark1}
We can remove ``dreamy" in Definition \ref{K_definition}. See \cite[Corollary 4.2]{BX}. 
\item\label{K-dfn_remark2}
As we have seen in \S \ref{intro_section}, \emph{K-polystability} of log 
Fano pairs $(X, \Delta)$ is important. However, we do not define it in this paper 
since the definition is rather complicated. 
See \cite{tian, don, LX} for the original definition and see also 
\cite[Theorem 1.3]{LWX}, \cite[Theorem 3.11]{ha}. 
We remark that, K-stability of $(X, \Delta)$ 
is equivalent to K-polystability of $(X, \Delta)$ and the condition 
$\Aut^0(X, \Delta)=\{1\}$ (see \cite[Corollary 1.3]{BX} for example). 
\end{enumerate}
\end{remark}

Although we do not give the definition of K-polystability, we give an important 
sufficient condition for log Fano pairs $(X, \Delta)$ being K-polystable. 

\begin{thm}[{\cite[Corollary 4.14]{zhuang}}]\label{equiv-K_thm}
Let $(X, \Delta)$ be a log Fano pair and let $G\subset\Aut(X, \Delta)$ be a 
reductive sub-algebraic group. If 
\[
\frac{A_{X, \Delta}(E)}{S_{X, \Delta}(E)}>1
\]
for any $G$-invariant dreamy prime divisor $E$ over $(X, \Delta)$, 
then $(X, \Delta)$ is K-polystable. 
\end{thm}

\begin{remark}\label{DS_remark}
If $X$ is a Fano manifold, then Theorem \ref{equiv-K_thm} is a consequence of 
\cite{li, vst} and \cite{DS}. See the proof of \cite[Corollary 4.14]{zhuang}. 
\end{remark}

We recall the notion of \emph{$\delta$-invariants} introduced in \cite{FO} and 
systematically developed in \cite{BJ}. We remark that $\delta$-invariants are 
sometimes called by \emph{stability thresholds}. 

\begin{definition}[{\cite{FO, BJ, zhuang}}]\label{delta_definition}
Let $X$ be a projective variety and let $\Delta$ be an effective 
$\Q$-Weil divisor on $X$. 
\begin{enumerate}
\renewcommand{\theenumi}{\arabic{enumi}}
\renewcommand{\labelenumi}{(\theenumi)}
\item\label{delta_definition1}
Take a big $\Q$-Cartier $\Q$-divisor $L$ on $X$ and 
a scheme-theoretic point $\eta\in X$. If $(X,\Delta)$ is klt at $\eta$ 
(in particular, $X$ is normal at $\eta$), then we set 
\[
\delta_\eta(X, \Delta;L):=\inf_{\substack{E: \text{ prime divisor}\\\text{over $X$; } 
\eta\in c_X(E)}}\frac{A_{X, \Delta}(E)}{S_L(E)}.
\] 
When $(X, \Delta)$ is a log Fano pair, we set 
\[
\delta_\eta(X, \Delta):=\delta_\eta\left(X,\Delta; -(K_X+\Delta)\right), 
\]
and call it the \emph{local $\delta$-invariant} of $(X, \Delta)$ at $\eta\in X$. 
\item\label{delta_definition2}
Assume that $(X,\Delta)$ is a klt pair. 
For any big $\Q$-Cartier $\Q$-divisor $L$ on $X$, 
we set 
\[
\delta(X, \Delta; L):=\inf_{\eta\in X}\delta_\eta(X, \Delta;L). 
\]
When $(X, \Delta)$ is a log Fano pair, we set 
\[
\delta(X, \Delta):=\delta\left(X,\Delta; -(K_X+\Delta)\right), 
\]
and call it the \emph{$\delta$-invariant} of $(X, \Delta)$. 
\end{enumerate}
\end{definition}

Clearly, if $\delta(X,\Delta)>1$, then $(X, \Delta)$ is K-stable. Moreover, 
it has been known by \cite{vst, FO, BJ} that the condition $\delta(X, \Delta)>1$ is 
equivalent to the condition $(X, \Delta)$ is \emph{uniformly K-stable}. 
In this paper, we do not discuss uniform K-stability. Recently, it has been shown 
in \cite{LXZ} that uniform K-stability of $(X,\Delta)$ is equivalent to  K-stability 
of $(X,\Delta)$.

We end with this section by recalling the notion of \emph{equivariant local 
$\alpha$-invariants} of log Fano pairs. For detail, see \cite{V22KE} for example. 

\begin{definition}\label{alpha_definition}
Let $(X, \Delta)$ be a log Fano pair and let $G\subset \Aut(X, \Delta)$ be a 
finite sub-algebraic group. 
\begin{enumerate}
\renewcommand{\theenumi}{\arabic{enumi}}
\renewcommand{\labelenumi}{(\theenumi)}
\item\label{alpha_definition1}
For any scheme-theoretic point $\eta\in X$, let 
$\alpha_{G,\eta}(X, \Delta)$ be the supremum of $\alpha\in\Q_{>0}$ such that 
$(X, \Delta+\alpha D)$ is lc at $\eta$ for any $G$-invariant and effective $\Q$-divisor 
$D\sim_\Q-(K_X+\Delta)$. If $\Delta=0$, then we simply denote it by 
$\alpha_{G,\eta}(X)$. 
\item\label{alpha_definition2}
Let $\alpha_G(X, \Delta)$ be the supremum of $\alpha\in\Q_{>0}$ such that 
$(X, \Delta+\alpha D)$ is lc for any $G$-invariant and effective $\Q$-divisor 
$D\sim_\Q-(K_X+\Delta)$. We call it the \emph{$G$-invariant $\alpha$-invariant} 
of $(X, \Delta)$. If $G=\{1\}$, we denote it by $\alpha(X, \Delta)$. 
\end{enumerate}
\end{definition}

We recall several important properties of $\alpha$-invariants. 

\begin{proposition}\label{alpha_proposition}
\begin{enumerate}
\renewcommand{\theenumi}{\arabic{enumi}}
\renewcommand{\labelenumi}{(\theenumi)}
\item\label{alpha_proposition1}
\cite[Theorem A]{BJ}
For any $n$-dimensional log Fano pair $(X, \Delta)$, we have 
\[
\frac{1}{n+1}\delta(X, \Delta)\leq\alpha(X, \Delta)\leq\frac{n}{n+1}\delta(X,\Delta). 
\]
More generally, for any scheme-theoretic point $\eta\in X$, we have 
\[
\frac{1}{n+1}\delta_\eta(X, \Delta)\leq\alpha_\eta(X, \Delta)
\leq\frac{n}{n+1}\delta_\eta(X,\Delta). 
\]
\item\label{alpha_proposition2}
\cite{FANO}
Let $X$ be an $n$-dimensional Fano manifold with 
$X\not\simeq\pr^n$, let $G\subset\Aut(X)$ be a finite 
sub-algebraic group, and let $\eta\in X$ be a scheme-theoretic point. 
If 
\[
\alpha_{G,\eta}(X)\geq \frac{n}{n+1},
\]
then we have 
\[
\frac{A_X(E)}{S_X(E)}>1
\]
for any $G$-invariant dreamy prime divisor $E$ over $X$ with $\eta\in c_X(E)$. 
\end{enumerate}
\end{proposition}

\begin{proof}
We give the proof of \eqref{alpha_proposition2} for the readers' convenience. 
By \cite[Lemma 2.5]{V22KE} and the property 
\[
S_X(E)\leq\frac{n}{n+1}\tau_X(E), 
\] 
we have 
\[
\frac{A_X(E)}{S_X(E)}\geq \frac{n+1}{n}\cdot\alpha_{G, \eta}(X)\geq 1.
\]
If $A_X(E)=S_X(E)$, then $X$ and $E$ satisfy the conditions of 
\cite[Theorem 4.1]{alpha}. Thus $X$ must be isomorphic to $\pr^n$. This leads to 
a contradiction. 
\end{proof}

\section{A review of Ahmadinezhad--Zhuang's theory}\label{AZR_section}

Recently, Ahmadinezhad and Zhuang introduced the important paper \cite{AZ}. 
We review their results. See also \S \ref{appendix_section}. 
In \S \ref{AZR_section}, we fix an $n$-dimensional projective variety $X$ unless 
otherwise stated. 

\subsection{Veronese equivalences and Okounkov bodies}\label{veronese_subsection}

Thanks to \cite[Lemma 2.24]{AZ}, it is natural to consider the following 
Veronese equivalences for graded linear series. 

\begin{definition}\label{veronese_definition}
Let us take $L_1,\dots,L_r\in\CaCl(X)\otimes_\Z\Q$. Let us take $m\in\Z_{>0}$ 
such that each $m L_i$ lifts to an element in $\CaCl(X)$. Fix such lifts and fix Cartier 
divisors (denoted also by $m L_i$) whose linear equivalence classes are $m L_i$. 
Note that the lifts $m L_i\in\CaCl(X)$ are not uniquely determined in general. 
\begin{enumerate}
\renewcommand{\theenumi}{\arabic{enumi}}
\renewcommand{\labelenumi}{(\theenumi)}
\item\label{veronese_definition1}
An \emph{$(m\Z_{\geq 0})^r$-graded linear series $V_{m\vec{\bullet}}$ on $X$ 
associated to $L_1,\dots,L_r$} consists of sub-vector spaces 
\[
V_{m\vec{a}}\subset H^0\left(X, \sO_X\left(\vec{a}\cdot m\vec{L}\right)\right)
\]
for all $\vec{a}=(a_1,\dots,a_r)\in\Z_{\geq 0}^r$ 
(where $\vec{a}\cdot m\vec{L}:=\sum_{i=1}^r a_i (m L_i)$) such that 
$V_{\vec{0}}=\C$ and 
$V_{m\vec{a}}\cdot V_{m\vec{a}'}\subset V_{m(\vec{a}+\vec{a}')}$ holds for any 
$\vec{a}$, $\vec{a}'\in\Z_{\geq 0}^r$. 
Under the setting, for any $k\in\Z_{>0}$, the \emph{$k$-th Veronese sub-series 
$V_{k m\vec{\bullet}}$ of $V_{m\vec{\bullet}}$} is defined to be the 
$(k m\Z_{\geq 0})^r$-graded linear series on $X$ associated to 
$L_1,\dots,L_r\in\CaCl(X)\otimes_\Z\Q$ 
defined naturally  by $V_{m\vec{\bullet}}$. 
\item\label{veronese_definition2}
For an $(m\Z_{\geq 0})^r$-graded (resp., $(m'\Z_{\geq 0})^r$-graded) linear series 
$V_{m\vec{\bullet}}$ (resp., $V'_{m'\vec{\bullet}}$) on $X$ associated to 
$L_1,\dots,L_r$, the series $V_{m\vec{\bullet}}$ and $V'_{m'\vec{\bullet}}$ are 
said to be \emph{Veronese equivalent} if there is a positive integer $d\in\Z_{>0}$ 
with $d\in m\Z$ and $d\in m'\Z$ such that 
$(d/m)(m L_i)\sim(d/m')(m' L_i)$ holds for each $i$ and 
\[
V_{\frac{d}{m}m\vec{\bullet}}=V'_{\frac{d}{m'}m'\vec{\bullet}}
\]
holds as $(d\Z_{\geq 0})^r$-graded linear series. 
The Veronese equivalence class of $V_{m\vec{\bullet}}$ is denoted by 
$V_{\vec{\bullet}}$. 
\end{enumerate}
\end{definition}

\begin{definition}[{\cite[\S 4.3]{LM} 
and \cite[Definition 2.11]{AZ}}]\label{bdd-ample_definition}
Let $V_{m\vec{\bullet}}$ be an $(m\Z_{\geq 0})^r$-graded linear series on $X$ 
associated to $L_1,\dots,L_r$ and let 
$V_{\vec{\bullet}}$ be its Veronese equivalence class. 
\begin{enumerate}
\renewcommand{\theenumi}{\arabic{enumi}}
\renewcommand{\labelenumi}{(\theenumi)}
\item\label{bdd-ample_definition1}
Set 
\[
\sS\left(V_{m\vec{\bullet}}\right):=\left\{m\vec{a}\in(m\Z_{\geq 0})^r\,\,|\,\,
V_{m\vec{a}}\neq 0\right\}\,\,\subset\,\,(m\Z_{\geq 0})^r, 
\]
and let $\Supp\left(V_{\vec{\bullet}}\right)\subset\R^r_{\geq 0}$ be the closure 
of the cone in $\R^r_{\geq 0}$ spanned by 
$\sS\left(V_{m\vec{\bullet}}\right)\subset(m\Z_{\geq 0})^r\subset\R^r_{\geq 0}$. 
By Lemma \ref{veronese_lemma}, the cone $\Supp\left(V_{\vec{\bullet}}\right)$ 
is independent of the choice of representatives of $V_{\vec{\bullet}}$ and the 
choices of lifts $m L_1,\dots, m L_r$. We say that 
$V_{\vec{\bullet}}$ \emph{has bounded support} if the set 
\[
\left(\{1\}\times\R_{\geq 0}^{r-1}\right)\cap\Supp\left(V_{\vec{\bullet}}\right)
\]
is bounded. 
\item\label{bdd-ample_definition2}
We say that $V_{m\vec{\bullet}}$ \emph{contains an ample series} if 
the following conditions are satisfied: 
\begin{enumerate}
\renewcommand{\theenumii}{\roman{enumii}}
\renewcommand{\labelenumii}{(\theenumii)}
\item\label{bdd-ample_definition21}
we have $\interior\left(\Supp\left(V_{\vec{\bullet}}\right)\right)\neq \emptyset$, 
\item\label{bdd-ample_definition22}
for any $m\vec{a}\in\interior\left(\Supp\left(V_{\vec{\bullet}}\right)\right)\cap
(m\Z_{\geq 0})^r$, we have $V_{p m\vec{a}}\neq 0$ for any $p\gg 0$, and 
\item\label{bdd-ample_definition23}
there exists an element 
$m\vec{a}_0\in\interior\left(\Supp\left(V_{\vec{\bullet}}\right)\right)\cap
(m\Z_{\geq 0})^r$ and there exists a decomposition $m\vec{a}_0\cdot\vec{L}\sim A+E$ 
with $A$ ample Cartier and $E$ effective Cartier such that 
\[
p E+H^0\left(X, p A\right)\subset V_{p m\vec{a}_0}
\]
holds for any $p\gg 0$.
\end{enumerate}
We say that the class $V_{\vec{\bullet}}$ \emph{contains an ample series} if 
there is a (sufficiently divisible) positive integer $m\in\Z_{>0}$ such that 
a representative $V_{m\vec{\bullet}}$ of $V_{\vec{\bullet}}$ contains an ample 
series (cf.\ Lemma \ref{veronese_lemma}). 
\end{enumerate}
\end{definition}

\begin{definition}[{\cite[\S 1, \S 4.3]{LM} and 
\cite[Definition 2.11]{AZ}}]\label{okounkov_definition}
Let $Y_\bullet$ be an admissible flag on $X$ in the sense of \cite[(1.1)]{LM}, i.e., 
\[
Y_{\bullet}\quad\colon\quad X=Y_0\supsetneq Y_1\supsetneq\cdots\supsetneq
Y_n=\{\text{point}\}
\]
is a sequence of subvarieties such that each $Y_i$ is smooth at the point $Y_n$. 
Let $V_{m\vec{\bullet}}$ be an $(m\Z_{\geq 0})^r$-graded linear series on $X$ 
associated to $L_1,\dots,L_r\in\CaCl(X)\otimes_\Z\Q$ 
which contains an ample series, and let 
$V_{\vec{\bullet}}$ be its Veronese equivalence class. As we have seen in 
\cite[(1.2)]{LM}, the flag $Y_\bullet$ gives a valuation-like function 
\[
\nu_{Y_\bullet}\colon V_{m\vec{a}}\setminus\{0\}\to\Z_{\geq 0}^n
\]
for each $\vec{a}\in\Z_{\geq 0}^r$. 
\begin{enumerate}
\renewcommand{\theenumi}{\arabic{enumi}}
\renewcommand{\labelenumi}{(\theenumi)}
\item\label{okounkov_definition1}
Let us set the sub-semigroup 
\[
\Gamma_{Y_\bullet}\left(V_{m\vec{\bullet}}\right):=\left\{
\left(m\vec{a},\nu_{Y_\bullet}(s)\right)\,\,|\,\,m\vec{a}\in(m\Z_{\geq 0})^r, \,\,
s\in V_{m\vec{a}}\setminus\{0\}\right\}\,\,\subset\,\,
(m\Z_{\geq 0})^r\times\Z_{\geq 0}^n
\]
of $\Z_{\geq 0}^r\times\Z_{\geq 0}^n$. Let 
\[
\Sigma_{Y_\bullet}\left(V_{\vec{\bullet}}\right)\subset 
\R_{\geq 0}^r\times\R_{\geq 0}^n
\]
be the closure of the cone in $\R_{\geq 0}^r\times\R_{\geq 0}^n$ spanned by 
$\Gamma_{Y_\bullet}\left(V_{m\vec{\bullet}}\right)$. Moreover, let us set 
\[
\Delta_{Y_\bullet}\left(V_{\vec{\bullet}}\right):=\left(\{1\}\times
\R_{\geq 0}^{r-1}\times\R_{\geq 0}^n\right)\cap
\Sigma_{Y_\bullet}\left(V_{\vec{\bullet}}\right)\subset \R_{\geq 0}^{r-1+n}.
\]
By Lemma \ref{veronese_lemma}, both $\Sigma_{Y_\bullet}\left(V_{\vec{\bullet}}\right)$ 
and $\Delta_{Y_\bullet}\left(V_{\vec{\bullet}}\right)$ are independent of the choice 
of representatives $V_{m\vec{\bullet}}$ of $V_{\vec{\bullet}}$ for 
$V_{m\vec{\bullet}}$ containing an ample series and of the choices of lifts 
$m L_1,\dots, m L_r$. 
We call it the \emph{Okounkov body of $V_{\vec{\bullet}}$ associated to $Y_{\bullet}$}. 
As in \cite[Definition 2.11]{AZ}, if $V_{\vec{\bullet}}$ has bounded support, then 
$\Delta_{Y_\bullet}\left(V_{\vec{\bullet}}\right)\subset\R_{\geq 0}^{r-1+n}$ is a 
compact convex body. When $L\in\CaCl(X)\otimes_\Z\Q$ is big and 
$V_{\bullet}$ is 
the class of the complete linear series of $L$ (i.e., for a sufficiently divisible 
$m\in\Z_{> 0}$, a representative $V_{m\bullet}$ is given by 
$V_{m p}=H^0(X, pm L)$ for any $p\in\Z_{\geq 0}$), 
then we simply write 
$\Sigma_{Y_\bullet}\left(L\right):=\Sigma_{Y_\bullet}\left(V_{\bullet}\right)$ and 
$\Delta_{Y_\bullet}\left(L\right):=\Delta_{Y_\bullet}\left(V_{\bullet}\right)$. 
\item\label{okounkov_definition2}
For any $l\in m\Z_{>0}$, we set 
\[
h^0\left(V_{l,m\vec{\bullet}}\right):=\sum_{\vec{a}\in\Z_{\geq 0}^{r-1}}
\dim V_{l,m\vec{a}}, 
\]
and 
\[
\vol\left(V_{\vec{\bullet}}\right):=\lim_{l\in m\Z_{>0}}
\frac{h^0\left(V_{l,m\vec{\bullet}}\right)\cdot m^{r-1}}{l^{r-1+n}/(r-1+n)!}\,\,
\in (0,\infty].
\]
By \cite[Remark 2.12]{AZ}, the limit exists. Moreover, 
by Lemma \ref{veronese_lemma}, the value $\vol\left(V_{\vec{\bullet}}\right)$ is 
independent of the choice of representatives $V_{m\vec{\bullet}}$ 
of $V_{\vec{\bullet}}$ for $V_{m\vec{\bullet}}$ containing an ample series and of the 
choices of lifts $m L_1,\dots, m L_r$. 
Moreover, by \cite[Remark 2.12]{AZ}, we have 
\[
\vol\left(V_{\vec{\bullet}}\right)=(r-1+n)!
\cdot\vol\left(\Delta_{Y_\bullet}\left(V_{\vec{\bullet}}\right)\right).
\]
If $V_{\vec{\bullet}}$ has bounded support, then $\vol\left(V_{\vec{\bullet}}\right)
\in(0,\infty)$ holds, since the Okounkov body 
$\Delta_{Y_\bullet}\left(V_{\vec{\bullet}}\right)$ is a compact convex body. 
\end{enumerate}
\end{definition}

\begin{lemma}[{cf.\ \cite[Lemma 2.24]{AZ}}]\label{veronese_lemma}
Let $W_{\vec{\bullet}}$ be a $\Z_{\geq 0}^r$-graded linear series on $X$ 
associated to Cartier divisors $L_1,\dots,L_r$. Let $Y_{\bullet}$ be an admissible 
flag on $X$. Let us take any $\vec{k}=(k_1,\dots,k_r)\in\Z_{>0}^r$. 
Let $W_{\vec{\bullet}}^{(\vec{k})}$ 
be the $\Z_{\geq 0}^r$-graded linear series on $X$ associated to 
$k_1L_1,\dots,k_r L_r$ defined by 
\[
W_{a_1,\dots,a_r}^{(\vec{k})}:=W_{k_1a_1,\dots,k_r a_r}.
\]
\begin{enumerate}
\renewcommand{\theenumi}{\arabic{enumi}}
\renewcommand{\labelenumi}{(\theenumi)}
\item\label{veronese_lemma1}
We have 
\[
\sS\left(W_{\vec{\bullet}}^{(\vec{k})}\right)
=\left\{(a_1,\dots,a_r)\in\Z_{\geq 0}^r\,\,|\,\,
(k_1a_1,\dots,k_r a_r)\in \sS\left(W_{\vec{\bullet}}\right)\right\}. 
\]
Thus, for the linear transform 
\begin{eqnarray*}
f\colon \R^r&\to&\R^r\\
(x_1,\dots,x_r)&\mapsto&(k_1x_1,\dots,k_r x_r),
\end{eqnarray*}
we have 
\[
\Supp\left(W_{\vec{\bullet}}\right)=
f\left(\Supp\left(W_{\vec{\bullet}}^{(\vec{k})}\right)\right).
\] 
In particular, $W_{\vec{\bullet}}$ has bounded support if and only if 
$W_{\vec{\bullet}}^{(\vec{k})}$ has bounded support. 
If $W_{\vec{\bullet}}$ contains an ample series, then so is 
$W_{\vec{\bullet}}^{(\vec{k})}$. 
\item\label{veronese_lemma2}
Assume that $W_{\vec{\bullet}}$ contains an ample series. 
For the linear transform 
\begin{eqnarray*}
g\colon \R^{r+n}&\to&\R^{r+n}\\
(x_1,\dots,x_{r+n})&\mapsto&(k_1x_1,\dots,k_r x_r,x_{r+1},\dots,x_{r+n}),
\end{eqnarray*}
we have 
\[
\Sigma_{Y_\bullet}\left(W_{\vec{\bullet}}\right)
=g\left(\Sigma_{Y_\bullet}\left(W_{\vec{\bullet}}^{(\vec{k})}\right)\right).
\]
Therefore, we have 
\[
\Delta_{Y_\bullet}\left(W_{\vec{\bullet}}\right)
=\bar{g}\left(\Delta_{Y_\bullet}\left(W_{\vec{\bullet}}^{(\vec{k})}\right)\right),
\]
where 
\begin{eqnarray*}
\bar{g}\colon \R^{r-1+n}&\to&\R^{r-1+n}\\
(x_1,\dots,x_{r-1+n})&\mapsto&\left(\frac{k_2}{k_1}x_1,\dots,\frac{k_r}{k_1}x_{r-1},
\frac{1}{k_1}x_r,\dots,\frac{1}{k_1}x_{r-1+n}\right).
\end{eqnarray*}
In particular, we have the equality
\[
\vol\left(W_{\vec{\bullet}}^{(\vec{k})}\right)=\frac{k_1^{r-1+n}}{k_2\cdots k_r}
\vol\left(W_{\vec{\bullet}}\right).
\]
\end{enumerate}
\end{lemma}

\begin{proof}
\eqref{veronese_lemma1}
The equalities on $\sS\left(W_{\vec{\bullet}}^{(\vec{k})}\right)$ and 
$\Supp\left(W_{\vec{\bullet}}^{(\vec{k})}\right)$ are obvious. 
Assume that $W_{\vec{\bullet}}$ contains an ample series. 
Since $\sS\left(W_{\vec{\bullet}}\right)$ generates $\Z^r$ as an abelian group 
(see \cite[Lemma 4.18]{LM}), 
the semigroup $\sS\left(W_{\vec{\bullet}}^{(\vec{k})}\right)$ also generates 
$\Z^r$ as an abelian group. 
From the assumption, there is an element 
\[
\vec{a}\in\interior\left(\Supp\left(W_{\vec{\bullet}}\right)\right)\cap\Z_{\geq 0}^r
\] 
and a decomposition 
\[
\vec{a}\cdot\vec{L}=A+E
\]
with $A$ ample Cartier and $E$ effective Cartier such that 
$mE+H^0\left(X, mA\right)\subset W_{m\vec{a}}$
for any $m\gg 0$. 
After replacing $\vec{a}$ with its positive multiple if necessary, we may assume that 
\[
\vec{b}:=\left(\frac{a_1}{k_1},\dots,\frac{a_r}{k_r}\right)\in\Z_{\geq 0}^r.
\]
Set $\vec{L}^{(\vec{k})}:=(k_1L_1,\dots,k_r L_r)$. Since 
$\vec{b}\cdot\vec{L}^{(\vec{k})}=A+E$ and 
$mE+H^0(X, mA)\subset W_{m\vec{a}}=W^{(\vec{k})}_{m\vec{b}}$ for any 
$m\gg 0$, the graded linear 
series $W_{\vec{\bullet}}^{(\vec{k})}$ also contains an ample series. 

\eqref{veronese_lemma2}
Let us show that 
$\Sigma_{Y_\bullet}\left(W_{\vec{\bullet}}\right)
=g\left(\Sigma_{Y_\bullet}\left(W_{\vec{\bullet}}^{(\vec{k})}\right)\right)$. 
Since the inclusion $\supset$ is trivial, it is enough to show the converse inclusion 
$\subset$. Take any 
\[
(a_1,\dots,a_r,\nu_1,\dots,\nu_n)\in\interior\left(\Sigma_{Y_\bullet}
\left(W_{\vec{\bullet}}\right)\right)\cap\Z^{r+n}.
\]
Since both $\Sigma_{Y_\bullet}\left(W_{\vec{\bullet}}\right)$ 
and $g\left(\Sigma_{Y_\bullet}\left(W_{\vec{\bullet}}^{(\vec{k})}\right)\right)$ 
are closed convex 
cones, it is enough to show that a positive multiple of 
$(a_1/k_1,\dots,a_r/k_r,\nu_1,\dots,\nu_n)$ belongs to 
$\Gamma_{Y_\bullet}\left(W_{\vec{\bullet}}^{(\vec{k})}\right)$. 
By \cite[Lemma 4.20]{LM}, the semigroup 
$\Gamma_{Y_\bullet}\left(W_{\vec{\bullet}}\right)$ generates $\Z^{r+n}$ as an 
abelian group. By \cite[Lemme 1.13]{boucksom}, for any $m\gg 0$, 
we have $m(a_1,\dots,a_r,\nu_1,\dots,\nu_n)\in\Gamma_{Y_\bullet}
\left(W_{\vec{\bullet}}\right)$. Take $m\in \Z_{>0}$ divisible by $k_1\cdots k_r$. 
Then there exists 
\[
s\in W_{ma_1,\dots,ma_r}\setminus\{0\}
=W_{\frac{ma_1}{k_1},\dots,\frac{ma_r}{k_r}}^{(\vec{k})}\setminus\{0\}
\]
such that $\nu_{Y_\bullet}(s)=(m\nu_1,\dots,m\nu_n)$. Thus we get 
\[
m\left(\frac{a_1}{k_1},\dots,\frac{a_r}{k_r},\nu_1,\dots,\nu_n\right)\in\Gamma_{Y_\bullet}
\left(W_{\vec{\bullet}}^{(\vec{k})}\right).
\]
The remaining assertions are trivial from the above. 
\end{proof}

\begin{example}\label{pullback_example}
Let $\sigma\colon\tilde{X}\to X$ be a birational morphism between projective 
varieties. Let $V_{\vec{\bullet}}$ be the Veronese equivalence class of a graded 
linear series on $X$ associated to $L_1,\dots,L_r\in\CaCl(X)\otimes_\Z\Q$. 
Under the natural inclusion $\sO_X\hookrightarrow\sigma_*\sO_{\tilde{X}}$, 
we can naturally consider 
the pullback $\sigma^*V_{\vec{\bullet}}$ of $V_{\vec{\bullet}}$. Obviously, 
the series $V_{\vec{\bullet}}$ has bounded support if and only if the series 
$\sigma^*V_{\vec{\bullet}}$ has bounded support. Moreover, if $X$ is normal, then 
the series $V_{\vec{\bullet}}$ contains an ample series if and only if the series 
$\sigma^*V_{\vec{\bullet}}$ contains an ample series, since 
$\sO_X\simeq\sigma_*\sO_{\tilde{X}}$ holds. (We sometimes denote 
$\sigma^*V_{\vec{\bullet}}$ by $V_{\vec{\bullet}}$ if there is no confusion.)
\end{example}

We will use the following theorem in \S \ref{AZ_section}. 

\begin{thm}[{\cite[Theorem 4.21]{LM}}]\label{LM_thm}
Let $V_{\vec{\bullet}}$ be the Veronese equivalence class of a graded linear series 
on $X$ associated to $L_1,\dots,L_r\in\CaCl(X)\otimes_\Z\Q$ which contains 
an ample series. Let $Y_\bullet$ be an admissible flag on $X$ and let 
\[\xymatrix{
\Gamma_{Y_\bullet}\left(V_{\vec{\bullet}}\right) \ar@{^{(}->}[r] \ar[rd]_{pr}
& \R^r\times\R^n \ar[d]^{p_1} \\
& \R^r
}\]
be the natural projection. 
Take any 
$\vec{a}\in\interior\left(\Supp\left(V_{\vec{\bullet}}\right)\right)\cap\Q_{\geq 0}^r$. 
Let $V_{\vec{a},\bullet}$ be the Veronese equivalence class of the graded linear 
series on $X$ associated to $\vec{a}\cdot\vec{L}$ defined by 
\[
V_{\vec{a},m}:=V_{m\vec{a}}
\]
for any sufficiently divisible $m\in\Z_{\geq 0}$. 
Then the series $V_{\vec{a},\bullet}$ has bounded support and contains an ample 
series, and 
\[
{pr}^{-1}\left(\left\{\vec{a}\right\}\right)=\Delta_{Y_\bullet}\left(V_{\vec{a},\bullet}\right)
\]
holds. 
\end{thm}

\begin{proof}
Since $V_{\vec{a},m\bullet}$ for a sufficiently divisible $m\in\Z_{>0}$ is 
$m\Z_{\geq 0}$-graded, the series obviously has bounded support. 
Moreover, the series contains an ample series by \cite[Lemma 4.18]{LM}. 
The remaining assertion follows directly from \cite[Theorem 4.21]{LM}. 
\end{proof}

\subsection{Filtrations on graded linear series}\label{filtration_subsection}

\begin{definition}\label{filter-vs_definition}
Let $W$ be a finite dimensional vector space over the complex number field. 
A \emph{filtration} $\sF$ on $W$ consists of a family 
$\{\sF^\lambda W\}_{\lambda\in\R}$ of sub-vector spaces of $W$ parametrized by 
$\R$ such that the following conditions are satisfied: 
\begin{enumerate}
\renewcommand{\theenumi}{\roman{enumi}}
\renewcommand{\labelenumi}{(\theenumi)}
\item\label{filter_definition_i}
If $\lambda<\lambda'$, then we have $\sF^{\lambda'}W\subset
\sF^\lambda W$.
\item\label{filter_definition_ii}
For any $\lambda\in\R$, we have $\sF^\lambda W
=\bigcap_{\lambda'<\lambda}\sF^{\lambda'}W$.
\item\label{filter_definition_iii}
We have $\sF^0 W=W$ and $\sF^\lambda W=0$ for $\lambda \gg 0$. 
\end{enumerate}
\end{definition}

\begin{definition}[{\cite[\S 1.3]{BC}, \cite[\S 2.5]{BJ} and 
\cite[\S 2.6]{AZ}}]\label{filter_definition}
Let $V_{\vec{\bullet}}$ be the Veronese equivalence class of a graded linear series 
on $X$ associated to $L_1,\dots,L_r\in\CaCl(X)\otimes_\Z\Q$ which 
has bounded support and contains an ample series.
We say that $\sF$ is a \emph{linearly bounded filtration of $V_{\vec{\bullet}}$} when 
there is a representative $V_{m\vec{\bullet}}$ of $V_{\vec{\bullet}}$ such that 
$\sF$ is a linearly bounded filtration on $V_{m\vec{\bullet}}$. More precisely, 
for any $m\vec{a}\in(m\Z_{\geq 0})^r$, we have a filtration 
$\left\{\sF^{\lambda}V_{m\vec{a}}\right\}_{\lambda\in\R}$ 
of $V_{m\vec{a}}$ 
such that, 
for any $\lambda$, $\lambda'\in\R$ and for any $\vec{a}$, $\vec{a}'\in\Z_{\geq 0}^r$, 
we have $\sF^\lambda V_{m\vec{a}}\cdot\sF^{\lambda'}V_{m\vec{a}'}\subset
\sF^{\lambda+\lambda'}V_{m(\vec{a}+\vec{a}')}$.
Moreover, there exists $C>0$ such that for any 
$\vec{a}=(a_1,\dots,a_r)\in\Z_{\geq 0}^r$ and for any $\lambda\geq C m a_1$, 
we have $\sF^\lambda V_{m\vec{a}}=0$.
We introduce the following notations as in \cite{BC} and \cite{AZ}. 
See also \cite{BJ}. 
\begin{enumerate}
\renewcommand{\theenumi}{\arabic{enumi}}
\renewcommand{\labelenumi}{(\theenumi)}
\item\label{filter_definition1}
For any $l\in m\Z_{>0}$, let us set 
\[
T_l\left(V_{m\vec{\bullet}};\sF\right):=\sup\left\{\lambda\in\R_{\geq 0}\,\,|\,\,
\sF^\lambda V_{l,m\vec{a}}\neq 0 \text{ for some } \vec{a}\in\Z_{\geq 0}^{r-1}\right\}
\]
and 
\[
T\left(V_{\vec{\bullet}};\sF\right):=\sup_{l\in m\Z_{>0}}
\frac{T_l\left(V_{m\vec{\bullet}};\sF\right)}{l}
=\lim_{l\in m\Z_{>0}}
\frac{T_l\left(V_{m\vec{\bullet}};\sF\right)}{l}. 
\]
As in \cite[Lemma 1.4]{BC}, the limit exists. Moreover, by Lemma \ref{filter_lemma}, 
the value $T\left(V_{\vec{\bullet}};\sF\right)$ does not depend on the choice of 
representatives of $V_{\vec{\bullet}}$ and the choices of lifts $m L_1,\dots, m L_r$. 
\item\label{filter_definition2}
For any $t\in\R_{\geq 0}$, let $V_{\vec{\bullet}}^t(=V_{\vec{\bullet}}^{\sF,t})$ be 
the Veronese equivalence class of the graded linear series on $X$ associated to 
$L_1,\dots,L_r$ which is the class of the $(m\Z_{\geq 0})^r$-graded linear series 
$V_{m\vec{\bullet}}^t$ defined by 
\[
V_{m\vec{a}}^t:=\sF^{m a_1 t}V_{m\vec{a}}\quad\left({}^\forall\vec{a}=(a_1,\dots,a_r)
\in\Z_{\geq 0}^r\right).
\]
Obviously, the series $V_{\vec{\bullet}}^t$ has 
bounded support since $V_{\vec{\bullet}}$ is so. As in \cite[Lemma 1.6]{BC} or 
\cite[Lemma 2.21]{AZ}, we have the following: 
\begin{itemize}
\item
If $t>T\left(V_{\vec{\bullet}};\sF\right)$, then $V_{\vec{\bullet}}^t=0$. 
If $t=0$, then $V^0_{\vec{\bullet}}=V_{\vec{\bullet}}$. 
\item
If $t\in[0,T\left(V_{\vec{\bullet}};\sF\right))$, then $V_{\vec{\bullet}}^t$
contains an ample series. 
\end{itemize}
\item\label{filter_definition3}
Let $Y_\bullet$ be an admissible flag on $X$. 
For any $t\in[0,T\left(V_{\vec{\bullet}};\sF\right))$, let us set 
\[
\Delta^t:=\Delta_{Y_\bullet}\left(V_{\vec{\bullet}}^t\right)
\subset\Delta:=\Delta_{Y_\bullet}\left(V_{\vec{\bullet}}\right).
\]
Moreover, let us consider the function 
\begin{eqnarray*}
G:= G_{\sF} \colon \Delta &\to& [0,T\left(V_{\vec{\bullet}};\sF\right)]\\
\vec{x}&\mapsto&\sup\left\{t\in[0,T\left(V_{\vec{\bullet}};\sF\right))\,\,|\,\,
\vec{x}\in\Delta^t\right\}
\end{eqnarray*}
as in \cite[\S 2.5]{BJ}, \cite[Lemma 2.21]{AZ}. The function $G$ is concave 
(see \cite[\S 2.5]{BJ}). Moreover, from the construction, the function does not 
depend on the choice of representatives of $V_{\vec{\bullet}}$. We set 
\begin{eqnarray*}
S\left(V_{\vec{\bullet}};\sF\right):=\frac{1}{\vol(\Delta)}\int_\Delta G(\vec{x})d\vec{x}
=\frac{1}{\vol\left(V_{\vec{\bullet}}\right)}\int_0^{T(V_{\vec{\bullet}};\sF)}
\vol\left(V_{\vec{\bullet}}^t\right)dt.
\end{eqnarray*}
The last equality is easily obtained by Fubini's theorem. 
From the concavity of the function $G$ (cf.\ \cite[Lemma 2.6]{BJ}), 
we can immediately get the inequalities
\[
\frac{1}{r+n}T(V_{\vec{\bullet}};\sF)\leq S(V_{\vec{\bullet}};\sF)\leq 
T(V_{\vec{\bullet}};\sF).
\]
\item\label{filter_definition4}
For any $l\in m\Z_{>0}$ with $h^0\left(V_{l,m\vec{\bullet}}\right)\neq 0$, we set 
\[
S_l\left(V_{m\vec{\bullet}};\sF\right):=\frac{1}{l\cdot h^0\left(V_{l,m\vec{\bullet}}\right)}
\int_0^{T_l\left(V_{m\vec{\bullet}};\sF\right)}\dim\sF^\lambda V_{l,m\vec{\bullet}}
\,\,d\lambda,
\]
where 
\[
\sF^\lambda V_{l,m\vec{\bullet}}:=\bigoplus_{\vec{a}\in\Z_{\geq 0}^{r-1}}
\sF^\lambda V_{l,m\vec{a}}.
\]
Then, by \cite[Lemma 2.9]{BJ} or \cite[Lemma 2.21]{AZ}, we have 
\[
\lim_{l\in m\Z_{>0}}S_l\left(V_{m\vec{\bullet}};\sF\right)
=S\left(V_{\vec{\bullet}};\sF\right).
\]
Indeed, we have 
\[
S_l\left(V_{m\vec{\bullet}};\sF\right)
=\frac{l^{r-1+n}/(r-1+n)!}{m^{r-1}\cdot h^0\left(V_{l,m\vec{\bullet}}\right)}\cdot\int_0^{\frac{T_l\left(V_{m\vec{\bullet}};\sF\right)}{l}}
\frac{m^{r-1}\cdot h^0\left(V^t_{l,m\vec{\bullet}}\right)}{l^{r-1+n}/(r-1+n)!}dt. 
\]
\end{enumerate}
\end{definition}

\begin{example}\label{filter_example}
Let $V_{\vec{\bullet}}$ be the Veronese equivalence class of a graded linear series 
on $X$ associated to $L_1,\dots,L_r\in\CaCl(X)\otimes_\Z\Q$ which 
has bounded support and contains an ample series. 
\begin{enumerate}
\renewcommand{\theenumi}{\arabic{enumi}}
\renewcommand{\labelenumi}{(\theenumi)}
\item\label{filter_example1}
For any linearly bounded filtration $\sF$ on $V_{\vec{\bullet}}$ and for any 
$\mu\in\R_{>0}$, we can naturally consider the linearly bounded filtration $\sF^\mu$ 
on $V_{\vec{\bullet}}$ defined by $(\sF^\mu)^\lambda V_{\vec{a}}
:=\sF^{\mu\lambda}V_{\vec{a}}$. It is obvious that 
\begin{itemize}
\item
$T(V_{\vec{\bullet}};\sF^\mu)=\mu^{-1}\cdot T(V_{\vec{\bullet}};\sF)$, 
\item
$G_{\sF^\mu}=\mu^{-1}\cdot G_{\sF}$ for any admissible flag on $X$, and 
\item
$S(V_{\vec{\bullet}};\sF^\mu)=\mu^{-1}\cdot S(V_{\vec{\bullet}};\sF)$.
\end{itemize}
\item\label{filter_example2}
Let $E$ be any prime divisor over the normalization of $X$. Then we can naturally 
define the linearly bounded filtration $\sF_E$ on $V_{\vec{\bullet}}$ with 
\[
\sF_E^\lambda V_{\vec{a}}:=\{s\in V_{\vec{a}}\,\,|\,\,\ord_E(s)\geq \lambda\}. 
\]
Moreover, we write $T(V_{\vec{\bullet}};E):=T(V_{\vec{\bullet}};\sF_E)$ and 
$S(V_{\vec{\bullet}};E):=S(V_{\vec{\bullet}};\sF_E)$. 
When $L\in\CaCl(X)\otimes_\Z\Q$ is big and $V_\bullet$ is the class of 
the complete linear series of $L$, then the value $S(V_\bullet;E)$ 
(resp., $T(V_\bullet;E)$) coincides 
with the value $S_L(E)$ (resp., $\tau_L(E)$) in Definition \ref{volume_definition}, 
even when $X$ is non-normal. See \cite[Proposition 2.2.43]{L1}.
\item\label{filter_example3}
Let $\sigma\colon\tilde{X}\to X$ be a birational morphism between normal 
projective varieties. As we have seen in Example \ref{pullback_example}, the 
class $\sigma^*V_{\vec{\bullet}}$ has bounded support and contains an ample 
series. For any prime divisor $E$ over $X$, we can naturally identify 
the filtration $\sF_E$ on $V_{\vec{\bullet}}$ and the filtration $\sF_E$ on 
$\sigma^*V_{\vec{\bullet}}$. In particular, we have the equalities 
$T(\sigma^*V_{\vec{\bullet}};E)=T(V_{\vec{\bullet}};E)$ 
and $S(\sigma^*V_{\vec{\bullet}};E)=S(V_{\vec{\bullet}};E)$. 
\end{enumerate}
\end{example}

\begin{lemma}[{cf.\ \cite[Lemma 2.24]{AZ}}]\label{filter_lemma}
Let $W_{\vec{\bullet}}$, $\vec{k}$, $Y_\bullet$ and $W_{\vec{\bullet}}^{(\vec{k})}$
be as in Lemma \ref{veronese_lemma}. 
Assume moreover that $W_{\vec{\bullet}}$ has bounded support and contains an 
ample series. 
Let $\sF$ be a linearly bounded filtration on $W_{\vec{\bullet}}$. 
The filtration $\sF$ naturally induces the filtration $\sF$ on 
$W_{\vec{\bullet}}^{(\vec{k})}$ defined by 
\[
\sF^\lambda W_{a_1,\dots,a_r}^{(\vec{k})}:=\sF^\lambda W_{k_1a_1,\dots,k_r a_r}.
\]
Then we have 
\[
G_{\sF}^{W_{\vec{\bullet}}^{(\vec{k})}}=k_1\cdot G_{\sF}^{W_{\vec{\bullet}}}
\circ\bar{g}\,\,\Big|_{\Delta_{Y_\bullet}\left(W_{\vec{\bullet}}^{(\vec{k})}\right)},
\]
where $\bar{g}$ be as in Lemma \ref{veronese_lemma}. In particular, we have 
\[
T\left(W_{\vec{\bullet}}^{(\vec{k})};\sF\right)=
k_1\cdot T\left(W_{\vec{\bullet}};\sF\right),\quad\quad
S\left(W_{\vec{\bullet}}^{(\vec{k})};\sF\right)=
k_1\cdot S\left(W_{\vec{\bullet}};\sF\right).
\]
\end{lemma}

\begin{proof}
From the definition, we have 
\[
W_{\vec{\bullet}}^{(\vec{k}),\sF^{k_1},t}=\left(W_{\vec{\bullet}}^{\sF,t}\right)^{(\vec{k})}.
\]
Thus, by Lemma \ref{veronese_lemma}, we get 
\[
\bar{g}\left(\Delta_{Y_\bullet}\left(W_{\vec{\bullet}}^{(\vec{k})}\right)^{\sF,t}\right)
=\bar{g}\left(\Delta_{Y_\bullet}
\left(W_{\vec{\bullet}}^{(\vec{k})}\right)^{\sF^{k_1},t/k_1}\right)
=\Delta_{Y_\bullet}\left(W_{\vec{\bullet}}\right)^{\sF,t/k_1}
\]
for any $t/k_1\in[0,T(W_{\vec{\bullet}};\sF))$. 
This implies that 
$G_{\sF}^{W_{\vec{\bullet}}^{(\vec{k})}}=k_1\cdot G_{\sF}^{W_{\vec{\bullet}}}
\circ\bar{g}$. Moreover, 
\begin{eqnarray*}
&&S\left(W_{\vec{\bullet}}^{(\vec{k})};\sF\right)=
\frac{1}{\vol\left(\Delta_{Y_\bullet}\left(W_{\vec{\bullet}}^{(\vec{k})}\right)\right)}
\int_{\Delta_{Y_\bullet}\left(W_{\vec{\bullet}}^{(\vec{k})}\right)}
G_{\sF}^{W_{\vec{\bullet}}^{(\vec{k})}}(\vec{x})d\vec{x} \\
&=&\frac{k_2\cdots k_r}{k_1^{r-1+n}}\cdot
\frac{1}{\vol\left(\Delta_{Y_\bullet}(W_{\vec{\bullet}})\right)}
\int_{\Delta_{Y_\bullet}(W_{\vec{\bullet}})}k_1\cdot G_{\sF}^{W_{\vec{\bullet}}}(\vec{y})
\cdot\frac{k_1^{r-1+n}}{k_2\cdots k_r}d\vec{y}
=k_1\cdot S\left(W_{\vec{\bullet}};\sF\right)
\end{eqnarray*}
holds. 
\end{proof}

\begin{definition}[{see \cite[Lemma 2.21]{AZ}}]\label{delta-AZ_definition}
Take an effective $\Q$-Weil divisor 
$\Delta$ on $X$. 
Let $V_{\vec{\bullet}}$ be the Veronese equivalence class of a graded linear 
series on $X$ associated to $L_1,\dots,L_r\in\CaCl(X)\otimes_\Z\Q$ 
which has bounded support and contains an ample series. 
Take any scheme-theoretic point $\eta\in X$ with $(X,\Delta)$ klt at $\eta$. 
We set 
\begin{eqnarray*}
\alpha_\eta\left(X,\Delta; V_{\vec{\bullet}}\right)
&:=&\alpha_\eta\left(V_{\vec{\bullet}}\right)
:=\inf_{\substack{E: \text{ prime divisor}\\ \text{over $X$ with }\eta\in c_X(E)}}
\frac{A_{X,\Delta}(E)}{T\left(V_{\vec{\bullet}};E\right)},\\
\delta_\eta\left(X,\Delta; V_{\vec{\bullet}}\right)
&:=&\delta_\eta\left(V_{\vec{\bullet}}\right)
:=\inf_{\substack{E: \text{ prime divisor}\\ \text{over $X$ with }\eta\in c_X(E)}}
\frac{A_{X,\Delta}(E)}{S\left(V_{\vec{\bullet}};E\right)}.
\end{eqnarray*}
If $(X,\Delta)$ is a klt pair, then we set
\begin{eqnarray*}
\alpha\left(X,\Delta; V_{\vec{\bullet}}\right)&:=&\alpha\left(V_{\vec{\bullet}}\right)
:=\inf_{\substack{E: \text{ prime divisor}\\ \text{over $X$}}}
\frac{A_{X,\Delta}(E)}{T\left(V_{\vec{\bullet}};E\right)}, \\
\delta\left(X,\Delta; V_{\vec{\bullet}}\right)&:=&\delta\left(V_{\vec{\bullet}}\right)
:=\inf_{\substack{E: \text{ prime divisor}\\ \text{over $X$}}}
\frac{A_{X,\Delta}(E)}{S\left(V_{\vec{\bullet}};E\right)}.
\end{eqnarray*}
By Definition \ref{filter_definition} \eqref{filter_definition3}, we have 
\begin{eqnarray*}
\alpha_\eta\left(V_{\vec{\bullet}}\right)\leq
&\delta_\eta\left(V_{\vec{\bullet}}\right)&\leq
(r+n)\alpha_\eta\left(V_{\vec{\bullet}}\right), \\
\alpha\left(V_{\vec{\bullet}}\right)\leq
&\delta\left(V_{\vec{\bullet}}\right)&\leq
(r+n)\alpha\left(V_{\vec{\bullet}}\right).
\end{eqnarray*}
When $L\in\CaCl(X)\otimes_\Z\Q$ is big and $V_{\bullet}$ is the class of the 
complete linear series of $L$, then the value 
$\delta_\eta\left(X,\Delta; V_{\bullet}\right)$
(resp., $\delta\left(X,\Delta; V_{\bullet}\right)$)
is nothing but the value 
$\delta_\eta\left(X,\Delta; L\right)$
(resp., $\delta\left(X,\Delta; L\right)$)
in Definition \ref{delta_definition}. 
If $(X,\Delta)$ is a log Fano pair and $V_{\bullet}$ is the class of the 
complete linear series of $-(K_X+\Delta)$, then the value 
$\alpha_\eta\left(X,\Delta; V_{\bullet}\right)$
(resp., $\alpha\left(X,\Delta; V_{\bullet}\right)$)
is nothing but the value 
$\alpha_\eta\left(X,\Delta\right)$
(resp., $\alpha\left(X,\Delta\right)$)
in Definition \ref{alpha_definition} (see \cite[Theorem C]{BJ}). 
We remark that, although we do not use it in the rest of paper, the above values 
are positive (see Proposition \ref{positivity_proposition}).
\end{definition}

We will use the following proposition in \S \ref{AZ_section}. 

\begin{proposition}\label{barycenter_proposition}
Let $V_{\vec{\bullet}}$ be the Veronese equivalence class of a graded linear series 
on $X$ associated to $L_1,\dots,L_r\in\CaCl(X)\otimes_\Z\Q$ 
which has bounded support and contains 
an ample series, and let $Y_\bullet$ be an admissible flag on $X$. 
Then $Y_1\subset X$ naturally gives a prime divisor over the normalization 
of $X$. Let us set $\sF:=\sF_{Y_1}$, and let us consider the Okounkov body 
$\Delta:=\Delta_{Y_\bullet}(V_{\vec{\bullet}})$ and let 
$G:=G_{\sF}\colon\Delta\to\R$ be as in Definition \ref{filter_definition}. 
\begin{enumerate}
\renewcommand{\theenumi}{\arabic{enumi}}
\renewcommand{\labelenumi}{(\theenumi)}
\item\label{barycenter_proposition1}
For any $t\in[0,T(V_{\vec{\bullet}};Y_1))$, we have 
\[
\Delta^t=\Delta\cap\left\{\vec{x}=(x_1,\dots,x_{r-1+n})\in\R_{\geq 0}^{r-1+n}
\,\,|\,\,x_r\geq t\right\}. 
\]
\item\label{barycenter_proposition2}
The restriction map
$G|_{\interior(\Delta)}\colon\interior(\Delta)\to\R$ is equal to the composition 
\[
\interior(\Delta)\hookrightarrow\R^{r-1+n}\xrightarrow{p_r}\R,
\]
where $p_r$ is the $r$-th projection. In particular, the value $T(V_{\vec{\bullet}};Y_1)$ 
is the maximum of the closed area $p_r(\Delta)\subset\R$, and the value 
$S(V_{\vec{\bullet}};Y_1)$ is the $r$-th coordinate of the barycenter of $\Delta$. 
\end{enumerate}
\end{proposition}

\begin{proof}
Fix a representative $V_{m\vec{\bullet}}$ of $V_{\vec{\bullet}}$ which contains an 
ample series. 

\eqref{barycenter_proposition2}
Since $G|_{\interior(\Delta)}$ is continuous, it is enough to show that 
$G(\vec{x})=\nu_1$ for any $\vec{x}=(\vec{a},\vec{\nu})=(a_1,\dots,a_{r-1},
\nu_1,\dots,\nu_n)\in\interior(\Delta)\cap\Q_{\geq 0}^{r-1+n}$. 
By \cite[Lemme 1.13]{boucksom}, there exists $l\in m\Z_{>0}$ such that 
$l(1,\vec{x})\in\Gamma_{Y_\bullet}(V_{m\vec{\bullet}})$, i.e., there exists a section 
$s\in V_{l(1,\vec{a})}\setminus\{0\}$ such that $\nu_{Y_\bullet}(s)=l\vec{\nu}$ holds. 
Since $s\in\sF^{l\nu_1}V_{l(1,\vec{a})}=V^{\nu_1}_{l(1,\vec{a})}$, we have 
$l(1,\vec{x})\in\Gamma_{Y_\bullet}\left(V_{m\vec{\bullet}}^{\nu_1}\right)$. 
Thus we have $\vec{x}\in\Delta^{\nu_1}$. This implies the inequality 
$G(\vec{x})\geq \nu_1$. 

Assume that $G(\vec{x})> \nu_1$. Take any $G(\vec{x})>\nu'_1> \nu_1$. By 
the concavity of $G$, we have $\vec{x}\in\interior\left(\Delta^{\nu'_1}\right)$. 
Again by \cite[Lemme 1.13]{boucksom}, there exists $l'\in m\Z_{>0}$ such that 
$l'(1,\vec{x})\in\Gamma_{Y_\bullet}\left(V_{m\vec{\bullet}}^{\nu'_1}\right)$, i.e., 
there exists a section $s'\in\sF^{l'\nu'_1}V_{l'(1,\vec{a})}\setminus\{0\}$ such that 
$\nu_{Y_\bullet}(s')=l'\vec{\nu}$ holds. Thus we get $\ord_{Y_1}(s')=l'\nu_1$. However, 
since $s'\in \sF^{l'\nu'_1}V_{l'(1,\vec{a})}$, we have 
$\ord_{Y_1}(s')\geq l'\nu'_1>l'\nu_1$, a contradiction. Thus we get 
$G(\vec{x})=\nu_1$. In particular, we get 
\[
S(V_{\vec{\bullet}};\sF)=\frac{1}{\vol(\Delta)}\int_\Delta x_r d\vec{x}.
\]
The value is nothing but the $r$-th coordinate of the barycenter of $\Delta$. 

\eqref{barycenter_proposition1}
As in Definition \ref{filter_definition}, for any $t\in[0,T(V_{\vec{\bullet}};Y_1))$, 
$\Delta^t\subset\Delta$ is a compact convex body with 
$\interior(\Delta^t)\neq\emptyset$. Thus, it is enough to show 
\[
\interior(\Delta)\cap\Delta^t
=\interior(\Delta)\cap\left\{\vec{x}=(x_1,\dots,x_{r-1+n})\in\R_{\geq 0}^{r-1+n}
\,\,|\,\,x_r\geq t\right\}. 
\]
The above is obvious from \eqref{barycenter_proposition2}. 
\end{proof}

\begin{corollary}[{cf.\ \cite[Theorem 3.2]{FO}}]\label{barycenter_corollary}
Under the assumption in Proposition \ref{barycenter_proposition}, let 
$U_r\in\R_{\geq 0}$ be the minimum of the closed area 
$p_r(\Delta)\subset\R_{\geq 0}$. Let us set $T_r:=T(V_{\vec{\bullet}};Y_1)$ just 
for simplicity. Then we have the inequalities 
\[
U_r+\frac{T_r-U_r}{r+n}\leq S(V_{\vec{\bullet}};Y_1)\leq
T_r-\frac{T_r-U_r}{r+n}. 
\]
For example, if $V_\bullet$ is the class of the complete linear series of 
a big $L\in\CaCl(X)\otimes_\Z\Q$ and if $Y_\bullet$ is an admissible flag 
with $Y_n\not\in\B_-(L)$, then we have 
\[
\frac{1}{n+1}T(V_\bullet;Y_1)\leq S(V_\bullet;Y_1)\leq
\frac{n}{n+1}T(V_\bullet;Y_1), 
\]
where $\B_-(L)$ is the \emph{restricted base locus} of $L$ 
$($see \cite{ELMNP} for the definition$)$. 
\end{corollary}

\begin{proof}
The first inequalities follow immediately from Proposition \ref{barycenter_proposition} 
and the standard fact of the barycenters of convex bodies (see \cite{hammer}). 
For the second inequalities, when $Y_n\not\in\B_-(L)$, then $\Delta$ contains 
the origin by \cite[Theorem 4.2]{CHPW}. Thus we have $U_r=0$. Since $r=1$, 
we get the assertion. 
\end{proof}

For example, if the above $L$ in Corollary \ref{barycenter_corollary} 
is nef and big, then the condition  
$Y_n\not\in\B_-(L)$ is always satisfied, 
since we have $\B_-(L)=\emptyset$ (see \cite{ELMNP}). 

\begin{corollary}[{cf.\ \cite[Proposition 2.1]{pltK} and 
\cite[Proposition 3.11]{BJ}}]\label{a-d_corollary}
Assume that $L\in\CaCl(X)\otimes_\Z\Q$ is big, and a prime divisor $E$ over 
the normalization of $X$ which satisfies that $c_X(E)\not\subset\B_-(L)$. 
Then we have the inequalities
\[
\frac{1}{n+1}\tau_L(E)\leq S_L(E)\leq\frac{n}{n+1}\tau_L(E). 
\]
\end{corollary}

\begin{proof}
Take a resolution $\sigma\colon\tilde{X}\to X$ of singularities with 
$E\subset\tilde{X}$. Note that 
$\B_-\left(\sigma^*L\right)\subset\sigma^{-1}\left(\B_-(L)\right)$
holds by the proof of \cite[Proposition 2.5]{lehmann}. Thus we have 
$E\not\subset\B_-\left(\sigma^*L\right)$. We can take an admissible flag 
${\tilde{Y}}_\bullet$ on $\tilde{X}$ with ${\tilde{Y}}_1=E$ and 
${\tilde{Y}}_n\not\in\B_-\left(\sigma^*L\right)$. Thus we get the assertion 
by Corollary \ref{barycenter_corollary}. 
\end{proof}

\subsection{Refinements}\label{refinement_subsection}

\begin{definition}[{cf.\ \cite[Example 2.15]{AZ}}]\label{refinement_definition}
Assume that $X$ is normal. Let $Y\subset X$ be a prime \emph{$\Q$-Cartier} 
divisor. Let $V_{m\vec{\bullet}}$ be an $(m\Z_{\geq 0})^r$-graded linear series 
on $X$ associated to $L_1,\dots,L_r\in\CaCl(X)\otimes_\Z\Q$. 
We assume that $m Y$ is Cartier. Let us define the $(m\Z_{\geq 0})^{r+1}$-graded 
linear series $V_{m\vec{\bullet}}^{(Y)}$ on $Y$ associated to 
$L_1|_Y,\dots,L_r|_Y, -Y|_Y$ as follows. 
(We note that $m L_1|_Y,\dots,m L_r|_Y, -m Y|_Y\in\CaCl(Y)$.)
For any $m(\vec{a},j)\in(m\Z_{\geq 0})^{r+1}$, we set: 
\[
V^{(Y)}_{m(\vec{a},j)}:=\Image\left(V_{m\vec{a}}\cap\left(m j Y+
H^0\left(X,m\vec{a}\cdot\vec{L}-m j Y\right)\right)\xrightarrow{\rest}
H^0\left(Y, m\vec{a}\cdot\vec{L}|_Y-m j Y|_Y\right)\right). 
\]
We call the Veronese equivalence class $V_{\vec{\bullet}}^{(Y)}$ of 
$V_{m\vec{\bullet}}^{(Y)}$ the \emph{refinement of $V_{\vec{\bullet}}$ by $Y$}. 
By Lemma \ref{refinement_lemma}, if 
$V_{\vec{\bullet}}$ has bounded support (resp., contains an ample series), then 
so is $V_{\vec{\bullet}}^{(Y)}$. 

When $V_{\vec{\bullet}}$ contains an ample 
series, by Lemmas \ref{veronese_lemma} and \ref{refinement_lemma}, 
for any admissible flag $Y_\bullet$ on $X$ with $Y_1=Y$, we have 
\[
\Sigma_{Y'_\bullet}\left(V_{\vec{\bullet}}^{(Y)}\right)
=\Sigma_{Y_\bullet}\left(V_{\vec{\bullet}}\right)
\]
under the natural identification $\R^{(r+1)+(n-1)}=\R^{r+n}$, where 
\[
Y'_{\bullet}\quad\colon\quad Y=Y_1\supsetneq Y_2\supsetneq\cdots\supsetneq
Y_n
\]
is the natural admissible flag on $Y$ induced by $Y_\bullet$. In particular, we have 
$\Delta_{Y'_\bullet}\left(V_{\vec{\bullet}}^{(Y)}\right)=
\Delta_{Y_\bullet}(V_{\vec{\bullet}})$ 
and $\vol\left(V^{(Y)}_{\vec{\bullet}}\right)=
\vol\left(V_{\vec{\bullet}}\right)$. 
\end{definition}

\begin{lemma}[{cf.\ \cite[Example 2.15 and Lemma 2.24]{AZ}}]\label{refinement_lemma}
Let $W_{\vec{\bullet}}$ be a $\Z_{\geq 0}^r$-graded linear series on an 
$n$-dimensional normal projective variety $X$ associated to Cartier divisors 
$L_1,\dots,L_r$. Let $Y\subset X$ be a prime divisor such that $e Y$ is Cartier 
for some $e\in\Z_{>0}$. 
Let $Y_{\bullet}$ be an admissible 
flag of $X$ with $Y=Y_1$. As in Definition \ref{refinement_definition}, we can naturally 
define the admissible flag $Y'_\bullet$ on $Y$ given by $Y_\bullet$. 
Let $W_{\vec{\bullet}}^{(Y,e)}$ be the $\Z_{\geq 0}^{r+1}$-graded linear series on $Y$ 
associated to $L_1|_Y,\dots,L_r|_Y$, $-eY|_Y$ defined by 
\begin{eqnarray*}
W_{\vec{a},j}^{(Y,e)}:=\Image\left(W_{\vec{a}}\cap\left(jeY
+H^0\left(X,\vec{a}\cdot\vec{L}-jeY\right)\right)\xrightarrow{\rest}
H^0\left(Y, \vec{a}\cdot\vec{L}|_Y-jeY|_Y\right)\right). 
\end{eqnarray*}
\begin{enumerate}
\renewcommand{\theenumi}{\arabic{enumi}}
\renewcommand{\labelenumi}{(\theenumi)}
\item\label{refinement_lemma1}
If $W_{\vec{\bullet}}$ has bounded support $($resp., contains an ample series$)$, 
then so is $W_{\vec{\bullet}}^{(Y,e)}$. 
\item\label{refinement_lemma2}
Assume that $W_{\vec{\bullet}}$ contains an ample series. Then we have 
\[
h\left(\Sigma_{Y'_\bullet}\left(W_{\vec{\bullet}}^{(Y,e)}\right)\right)
=\Sigma_{Y_\bullet}\left(W_{\vec{\bullet}}\right),
\]
where $h$ is defined by 
\begin{eqnarray*}
h\colon \R^{(r+1)+(n-1)}&\to&\R^{r+n}\\
(x_1,\dots,x_{r+n})&\mapsto&(x_1,\dots,x_r,e x_{r+1},x_{r+2},\dots,x_{r+n}).
\end{eqnarray*}
In particular, we have the equality 
\[
\vol\left(W_{\vec{\bullet}}^{(Y,e)}\right)=\frac{1}{e}\vol\left(W_{\vec{\bullet}}\right).
\]
\end{enumerate}
\end{lemma}

\begin{proof}
\eqref{refinement_lemma1}
Assume that $W_{\vec{\bullet}}$ has bounded support. 
There exists a positive integer $M>0$ such that $W_{\vec{a}}=0$ for any 
$\vec{a}=(a_1,\dots,a_r)\in\Z_{\geq 0}^r$ with $a_i\geq M a_1$ 
for some $2\leq i\leq r$. Take an ample Cartier divisor $H$ on $X$. Let $N>0$ be 
a sufficiently big positive integer satisfying 
\[
\left(L_1+x_2L_2+\cdots +x_r L_r-N e Y\right)\cdot H^{\cdot n-1}<0
\]
for any $x_2,\dots,x_r\in[0,M]$. Then, we can immediately show that 
$W_{\vec{a},j}^{(Y,e)}=0$ for any $(\vec{a},j)\in\Z_{\geq 0}^{r+1}$ with 
$a_i\geq M a_1$ for some $2\leq i\leq r$ or $j\geq N a_1$. Thus 
$W_{\vec{\bullet}}^{(Y,e)}$ also has bounded support. 

Assume that $W_{\vec{\bullet}}$ contains an ample series. 
Take $\vec{x}_1,\dots,\vec{x}_r\in\interior
\left(\Supp\left(W_{\vec{\bullet}}\right)\right)\cap\Z_{\geq 0}^{r}$ such that 
$\vec{x}_1,\dots,\vec{x}_r$ form a basis of $\Z^r$. 
Since $W_{\vec{\bullet}}$ contains an ample series, we have 
$m\vec{x}_i\in\sS(W_{\vec{\bullet}})$ for any $1\leq i\leq r$ and for any $m\gg 0$. 
By \cite[Lemma 4.18]{LM}, there exists 
$m_0\in\Z_{>0}$ and there exists a decomposition 
\[
m_0\vec{x}_i\cdot\vec{L}=A_i+E_i
\]
for any $1\leq i\leq r$ with $A_i$ ample and $E_i$ effective such that 
$kE_i+H^0(kA_i)\subset W_{km_0\vec{x}_i}$ for any $k\in\Z_{>0}$. 
After replacing $m_0$ sufficiently divisible, we may further assume that 
$A_i(-jY)$ is globally generated for any $1\leq i\leq r$ and for any $j=0,1,\dots, 2e$. 
Set $c_i:=\ord_YE_i$. For any $k\in\Z_{>0}$ and for any $1\leq i\leq r$, we have 
\[
\left(km_0\vec{x}_i,\lceil kc_i/e\rceil\right),\,\,\,\left(km_0\vec{x}_i,
\lceil kc_i/e\rceil+1\right)\in \sS\left(W_{\vec{\bullet}}^{(Y,e)}\right).
\]
Moreover, for any $m\gg 0$, we have $(m\vec{x}_i,c)\in
\sS\left(W_{\vec{\bullet}}^{(Y,e)}\right)$ for some $c\in\Z_{\geq 0}$. 
Therefore $\sS\left(W_{\vec{\bullet}}^{(Y,e)}\right)$ generates $\Z^{r+1}$ as an 
abelian group. 

Let us consider the condition (iii) in \cite[Definition 4.17]{LM}. 
Take any element 
$\vec{x}\in\interior\left(\Supp\left(W_{\vec{\bullet}}\right)\right)\cap
\left(\{1\}\times\Q^{r-1}_{\geq 0}\right)$. 
By \cite[lemma 4.18]{LM}, there is a sufficiently divisible $m\in \Z_{>0}$ and a 
decomposition $m\vec{x}\cdot\vec{L}=A+E$ with $A$ 
ample and $E$ effective such that $kE+H^0(X, kA)\subset W_{km\vec{x}}$ holds 
for any $k\in\Z_{>0}$. 
Moreover, we may further assume that, for any $l\in\{0,1,2\}$, 
$A-leY$ is very ample and the restriction homomorphism 
\[
H^0(X, k(A-leY))\to H^0(Y, k(A-leY)|_Y)
\]
is surjective for any $k\in\Z_{>0}$. Let us set $c:=\ord_YE$. 
For any $k\in e\Z_{\gg 0}$, the restriction homomorphism  
\[
kE+kleY+H^0\left(X, k(A-leY)\right)\to
k(E-cY)|_Y+H^0\left(Y, k(A-leY)|_Y\right)
\]
is surjective for any $l\in\{0,1,2\}$. Thus we have
\[
k(E-cY)|_Y+H^0\left(Y, k(A-leY)|_Y\right)\subset W^{(Y,e)}_{km\vec{x},k(c+le)}.
\]
In particular, we have 
\[
\left(\vec{x}, \frac{1}{m}(c+ye)\right)\in\Supp\left(W_{\vec{\bullet}}^{(Y,e)}\right)
\]
for any $y\in[0,2]$. 
If we take $\vec{x}$ generally, then we have 
\[
\left(\vec{x}, \frac{1}{m}(c+e)\right)\in\interior\left(
\Supp\left(W_{\vec{\bullet}}^{(Y,e)}\right)\right)\cap\Q^{r+1}_{\geq 0}.
\]
The decomposition 
\[
m\left(\vec{x}\cdot\vec{L}-\frac{1}{m}(c+e)Y|_Y\right)\sim_\Q
(A-eY)|_Y+(E-cY)|_Y
\]
with $(A-eY)|_Y$ ample and $(E-cY)|_Y$ effective satisfies that, for any sufficiently 
divisible $k\in\Z_{>0}$, we have the condition (iii) in \cite[Definition 4.17]{LM}. 
Therefore $W^{(Y,e)}_{\vec{\bullet}}$ contains an ample series. 

\eqref{refinement_lemma2}
Take any element 
\[
(a_1,\dots,a_r,\nu_1,\nu_2,\dots,\nu_n)\in\Gamma_{Y'_\bullet}\left(
W_{\vec{\bullet}}^{(Y,e)}\right).
\]
There is a nonzero element 
$s_1\in W^{(Y,e)}_{a_1,\dots,a_r,\nu_1}$ such that $\nu_{Y'_\bullet}(s_1)=
(\nu_2,\dots,\nu_n)$. From the definition of $W^{(Y,e)}_{\vec{\bullet}}$, 
there is a nonzero element $s\in W_{a_1,\dots,a_r}$ such that 
$\nu_1(s)=\nu_1\cdot e$ and the image of $s$ with respects to the 
restriction homomorphism 
\[
W_{a_1,\dots,a_r}\cap\left(\nu_1\cdot e Y+
H^0\left(X,\vec{a}\cdot\vec{L}-\nu_1\cdot eY\right)\right)
\to H^0\left(Y, \vec{a}\cdot\vec{L}|_Y-\nu_1\cdot eY|_Y\right)
\]
is equal to $s_1$. Since $\nu_{Y_\bullet}(s)=(\nu_1\cdot e,\nu_2,\dots,\nu_n)$, 
we have 
\[
(a_1,\dots,a_r,\nu_1\cdot e,\nu_2,\dots,\nu_n)\in\Gamma_{Y_\bullet}
\left(W_{\vec{\bullet}}\right). 
\]
This gives the inclusion 
$h\left(\Sigma_{Y'_\bullet}\left(W_{\vec{\bullet}}^{(Y,e)}\right)\right)
\subset\Sigma_{Y_\bullet}\left(W_{\vec{\bullet}}\right)$. 
For the converse inclusion, since both are closed convex cones, it is enough to 
prove that there is some $m\in\Z_{>0}$ such that 
$m(a_1,\dots,a_r,\nu_1,\dots,\nu_n)\in 
h\left(\Sigma_{Y'_\bullet}\left(W_{\vec{\bullet}}^{(Y,e)}\right)\right)$
holds for any element
\[
(a_1,\dots,a_r,\nu_1,\dots,\nu_n)\in\interior\left(\Sigma_{Y_\bullet}
\left(W_{\vec{\bullet}}\right)\right)\cap\Z^{n+r}. 
\]
For any sufficiently divisible $m\in e\Z_{>0}$, we have 
\[
m(a_1,\dots,a_r,\nu_1,\dots,\nu_n)\in \Gamma_{Y_\bullet}\left(W_{\vec{\bullet}}\right).
\]
Thus there is a nonzero element 
$s\in W_{ma_1,\dots,ma_r}$ such that $\nu_{Y_\bullet}(s)=(m\nu_1,m\nu_2,\dots,
m\nu_n)$. The section $s$ vanishes along $Y$ exactly $m\nu_1$ times. Thus 
the image $s_1$ of $s$ with respects to the restriction homomorphism 
\[
W_{ma_1,\dots,ma_r}\cap\left(\frac{m\nu_1}{e}\cdot eY
+H^0\left(X, m\vec{a}\cdot\vec{L}-\frac{m\nu_1}{e}\cdot eY\right)\right)
\to H^0\left(Y, m\vec{a}\cdot\vec{L}|_Y-\frac{m\nu_1}{e}\cdot eY|_Y\right)
\]
gives a nonzero element in $W^{(Y,e)}_{ma_1,\dots,ma_r,\frac{m\nu_1}{e}}$. 
From the definition of $\nu_{Y_\bullet}(s)$, we have 
$\nu_{Y'_\bullet}(s_1)=(m\nu_2,\dots,m\nu_n)$. 
This means that the element 
$\left(ma_1,\dots,ma_r,\frac{m\nu_1}{e},m\nu_2,\dots,m\nu_n\right)$
belongs to $\Gamma_{Y'_\bullet}\left(W_{\vec{\bullet}}^{(Y,e)}\right)$. 
Thus we get the assertion. 
\end{proof}

\begin{remark}\label{refinement_remark}
Let $\sigma\colon\tilde{X}\to X$ be a birational morphism between normal projective 
varieties, let $Y\subset X$ be a prime $\Q$-Cartier divisor on $X$ such that 
$\tilde{Y}:=\sigma^{-1}_*Y$ is also $\Q$-Cartier. Let us set 
$\sigma^*Y=:\tilde{Y}+\Sigma$. Take any $(m\Z_{\geq 0})^r$-graded linear series 
$V_{m\vec{\bullet}}$ on $X$ associated to $L_1,\dots,L_r\in\CaCl(X)\otimes_\Z\Q$. 
Assume moreover that both $m Y$ and $m\tilde{Y}$ are Cartier. Let us compare 
$(\sigma|_{\tilde{Y}})^*\left(V_{m\vec{\bullet}}^{(Y)}\right)$ and 
$\left(\sigma^*V_{m\vec{\bullet}}\right)^{(\tilde{Y})}$. 

Take any $(\vec{a},j)\in(m\Z_{\geq 0})^{r+1}$. We note that the inclusion 
\[
H^0\left(\tilde{X},\sigma^*\left(\vec{a}\cdot\vec{L}-j Y\right)\right)
\xrightarrow{\cdot j\Sigma}
H^0\left(\tilde{X},\sigma^*\left(\vec{a}\cdot\vec{L}\right)-j\tilde{Y}\right)
\]
is an isomorphism. Moreover, we have the following commutative diagram: 
\[
\xymatrix{
V_{\vec{a}}\cap\left(jY+H^0\left(X,\vec{a}\cdot\vec{L}-j Y\right)\right)
 \ar[r]^-{\rest_Y}
\ar[d]_{\sigma^*\simeq}
& H^0\left(Y, \vec{a}\cdot\vec{L}|_Y-j Y|_Y\right) 
\ar@{_{(}-{>}}[d]^{\sigma|_{\tilde{Y}}^*} \\
\sigma^*V_{\vec{a}}\cap\left(j\sigma^*Y+
H^0\left(\tilde{X},\sigma^*\left(\vec{a}\cdot\vec{L}-j Y\right)\right)\right)
 \ar[r]^-{\rest_{\tilde{Y}}} \ar[d]_{\simeq}
& H^0\left(\tilde{Y}, \sigma^*\left(\vec{a}\cdot\vec{L}|_Y-j Y|_Y\right)\right) 
\ar@{_{(}-{>}}[d]^{\cdot j\Sigma|_{\tilde{Y}}} \\
\sigma^*V_{\vec{a}}\cap\left(j\tilde{Y}+H^0\left(
\tilde{X}, \sigma^*\left(\vec{a}\cdot\vec{L}\right)-j\tilde{Y}\right)\right) 
\ar[r]^-{\rest_{\tilde{Y}}}& 
H^0\left(\tilde{Y}, \sigma^*\left(\vec{a}\cdot\vec{L}|_Y\right)
-j\tilde{Y}|_{\tilde{Y}}\right).
}
\]
This implies that 
\[
\left(\sigma^*V_{\vec{a},j}\right)^{(\tilde{Y})}=(\sigma|_{\tilde{Y}})^*V^{(Y)}_{\vec{a},j}
+j\left(\Sigma|_{\tilde{Y}}\right)
\]
for any $(\vec{a},j)\in(m\Z_{\geq 0})^{r+1}$. 
\end{remark}

\begin{remark}\label{plt_remark}
In this paper, we essentially consider only the linear equivalence classes of 
\emph{Cartier} divisors by taking Veronese sub-series. 
However, although we do not treat in this paper, on normal projective varieties $X$, 
it is important to consider the linear equivalence classes of $\Q$-Cartier 
$\Q$-divisors in order to consider the theory of graded linear series. 
In fact, for considering the proof of Theorem \ref{AZ_thm} by the authors in \cite{AZ}, 
it is essential to consider the refinements of 
$\Z_{\geq 0}^r$-graded linear series on $X$ associated to Cartier divisors by 
possibly non-Cartier prime $\Q$-Cartier divisors $Y$ on $X$ such that 
the linear equivalence classes $-Y|_Y$ of $\Q$-Cartier divisors are well-behaved 
(cf.\ Definition \ref{plt_definition}). 
See \cite{AZ} for detail. See also Theorem \ref{AZ-R_thm}. 
\end{remark}

\begin{definition}[{\cite[Definition 1.1]{pltK} and \cite[\S 2.3]{AZ}}]\label{plt_definition}
Let $(X,\Delta)$ be a (possibly non-projective) 
klt pair with $\Delta$ effective $\Q$-Weil divisor. 
A prime divisor $Y$ over $X$ is said to be \emph{plt-type} over $(X, \Delta)$ 
if there is a projective birational morphism $\sigma\colon\tilde{X}\to X$ 
between normal varieties with $Y\subset\tilde{X}$ prime divisor such that 
$-Y$ is a $\sigma$-ample $\Q$-Cartier divisor on $\tilde{X}$ and the pair 
$(\tilde{X}, \tilde{\Delta}+Y)$ is a plt pair, where the $\Q$-Weil divisor $\tilde{\Delta}$ 
on $\tilde{X}$ is defined to be the equation 
\[
K_{\tilde{X}}+\tilde{\Delta}+\left(1-A_{X,\Delta}(Y)\right)Y
=\sigma^*\left(K_X+\Delta\right).
\]
The morphism $\sigma$ is uniquely determined by $Y$. We call the morphism 
the \emph{plt-blowup associated to $Y$}. We can naturally take the klt pair 
$(Y, \Delta_Y)$ defined by 
\[
K_Y+\Delta_Y:=\left(K_{\tilde{X}}+\tilde{\Delta}+Y\right)\Big|_Y.
\]
We note that, although we do not treat it in this paper, 
we can canonically define the linear equivalence class of a 
$\Q$-Cartier $\Q$-divisor $-Y|_Y$ by 
\cite[Definitions A.2 and A.4]{HLS} (see also \cite[Lemma 2.7]{AZ}). 
\end{definition}

The following theorem is very important in this paper. For the proof, 
see \cite[Theorem 3.3]{AZ}, or see \S \ref{adj_subsection} for an alternative proof. 
Note that we can easily reduce to the case $L_1,\dots,L_r\in\CaCl(X)$ 
by Lemmas \ref{filter_lemma} and \ref{refinement_lemma}. 
We remark that \cite[Theorem 3.3]{AZ} treats 
much more general situations.

\begin{thm}[{\cite[Theorem 3.3]{AZ}, see also Theorem \ref{AZ-R_thm}}]\label{AZ_thm}
Let $(X,\Delta)$ be a projective klt pair with $\Delta$ effective $\Q$-Weil divisor, 
let $\eta\in X$ be a scheme-theoretic point, let $Y$ be a plt-type prime divisor 
over $(X,\Delta)$ with the associated plt-blowup $\sigma\colon\tilde{X}\to X$ 
satisfying $\eta\in c_X(Y)$, 
and let $V_{\vec{\bullet}}$ be the Veronese equivalence class of a graded linear series 
on $X$ associated to $L_1,\dots,L_r\in\CaCl(X)\otimes_\Z\Q$ which has bounded 
support and contains an ample series. Let 
$W_{\vec{\bullet}}$ be the refinement of $\sigma^*V_{\vec{\bullet}}$ by 
$Y\subset\tilde{X}$. Let $\Delta_Y$ on $Y$ be as in Definition \ref{plt_definition}. 
Then we have the inequality
\[
\delta_\eta\left(X,\Delta;V_{\vec{\bullet}}\right)
\geq \min\left\{\frac{A_{X,\Delta}(Y)}{S(V_{\vec{\bullet}};Y)},\quad
\inf_{\eta'}\delta_{\eta'}\left(Y, \Delta_Y;W_{\vec{\bullet}}\right)
\right\}, 
\]
where the infimum runs over all scheme-theoretic points $\eta'\in Y\subset\tilde{X}$
with $\sigma(\eta')=\eta$. 
\end{thm}

\section{Ahmadinezhad--Zhuang's theory on Mori dream spaces}\label{AZ_section}

We calculate the values in \S \ref{AZR_section} when $X$ is a \emph{Mori 
dream space} (see\cite{HK}). Since any log Fano pair is a Mori dream space 
\cite[Corollary 1.3.2]{BCHM}, we can apply the computations in \S \ref{AZ_section} 
for various situations in order to evaluate local $\delta$-invariants for log Fano pairs. 
Many statements in this section are similar to the statements in 
\cite{FANO}. However, the situations we consider are more complicated than 
the situations in \cite{FANO}. 

In this section, we fix: 
\begin{itemize}
\item
an $n$-dimensional Mori dream space $X$ (in the sense of \cite[Definition 1.10]{HK}), 
\item
a big $\Q$-divisor $L$ on $X$, 
\item
the Veronese equivalence class $V_\bullet$ of the complete linear series of $L$, 
\item
a prime divisor 
$Y\subset X$ (note that $Y$ is $\Q$-Cartier since $X$ is $\Q$-factorial), and
\item
the refinement $W_{\bullet,\bullet}$ of $V_\bullet$ by $Y$, i.e., 
$W_{\bullet,\bullet}=V_{\vec{\bullet}}^{(Y)}$. 
\end{itemize}
Moreover, let us set 
\begin{eqnarray*}
\tau_-&:=&\ord_Y N_\sigma(X,L), \\
\tau_+&:=&\max\left\{u\in\R_{\geq 0}\,\,|\,\,L-uY\text{ is pseudo-effective}\right\},
\end{eqnarray*}
where $N_\sigma(X, L)$ is the negative part of the Nakayama--Zariski decomposition 
of $L$ (see \cite[Chapter III]{N}).

\subsection{Basics of Mori dream spaces}\label{MDS_subsection}

We recall basic theories of Mori dream spaces and Nakayama--Zariski decompositions.

\begin{lemma}[{cf.\ \cite{okawa}}]\label{okawa_lemma}
\begin{enumerate}
\renewcommand{\theenumi}{\arabic{enumi}}
\renewcommand{\labelenumi}{(\theenumi)}
\item\label{okawa_lemma1}
The values $\tau_-$, $\tau_+$ are rational numbers with $\tau_-<\tau_+$. 
\item\label{okawa_lemma2}
If $u\in[0,\tau_-)$, then we have $Y\subset\Supp N_\sigma\left(X,L-u Y\right)$. 
If $u\in[\tau_-,\tau_+]$, then we have 
$Y\not\subset\Supp N_\sigma\left(X,L-u Y\right)$. 
\item\label{okawa_lemma3}
There exists 
\begin{itemize}
\item
a finite sequence $\tau_-=\tau_0<\cdots<\tau_I=\tau_+$ of rational numbers, and 
\item
a finite set $\left\{X_1,\dots,X_I\right\}$ of small $\Q$-factorial modifications of $X$
\end{itemize}
such that, for any $1\leq i\leq I$ and for any $u\in[\tau_{i-1},\tau_i]$, we have
the following:
\begin{enumerate}
\renewcommand{\theenumii}{\roman{enumii}}
\renewcommand{\labelenumii}{(\theenumii)}
\item\label{okawa_lemma31}
the positive part $P_\sigma\left(X_i,(L-u Y)_{X_i}\right)$ is semiample on $X_i$, where 
$(L-uY)_{X_i}$ is the strict transform of $L-uY$ on $X_i$, 
\item\label{okawa_lemma32}
both $N_\sigma\left(X_i,(L-\tau_{i-1} Y)_{X_i}\right)$ and 
$N_\sigma\left(X_i,(L-\tau_i Y)_{X_i}\right)$ are $\Q$-divisors, and we have 
\begin{eqnarray*}
&&N_\sigma\left(X_i,(L-u Y)_{X_i}\right)\\
&=&\frac{\tau_i-u}{\tau_i-\tau_{i-1}}N_\sigma\left(X_i,(L-\tau_{i-1} Y)_{X_i}\right)
+\frac{u-\tau_i}{\tau_i-\tau_{i-1}}N_\sigma\left(X_i,(L-\tau_i Y)_{X_i}\right),
\end{eqnarray*}
and
\item\label{okawa_lemma33}
if $u\in(\tau_{i-1},\tau_i]$ and $u<\tau_+$, 
then $P_\sigma\left(X_i,(L-u Y)_{X_i}\right)|_{Y_i}$ is 
semiample and big on $Y_i$, where $Y_i$ is the strict transform of $Y$ on $X_i$.
\end{enumerate}
\end{enumerate}
\end{lemma}

\begin{proof}
\eqref{okawa_lemma1} The properties $\tau_-$, $\tau_+\in\Q$ are trivial 
(see \cite[\S 2.3]{okawa}). By \cite[Chapter III, Lemma 1.4 (4)]{N}, the 
$\Q$-divisor $L-\tau_-Y$ is big with $\ord_Y N_\sigma(L-\tau_- Y)=0$. 
Thus we have $\tau_-<\tau_+$. 

\eqref{okawa_lemma2} Trivial from 
\cite[Chapter III, Lemma 1.8 and Corollary 1.9]{N}. 

\eqref{okawa_lemma3} The properties \eqref{okawa_lemma31} and 
\eqref{okawa_lemma32} are direct corollaries of \cite[Proposition 2.13]{okawa}. 
Let us consider \eqref{okawa_lemma33}. Since 
$P_\sigma\left(X_i, (L-u Y)_{X_i}\right)$ is semiample and big, it is enough to show 
that $P_\sigma\left(X_i, (L-u Y)_{X_i}\right)|_{Y_i}$ is big. Assume not. 
We may assume that $u\in\Q$. 
Since $P_\sigma\left(X_i, (L-u Y)_{X_i}\right)$ is semiample, 
as in \cite[\S 2.3]{okawa}, there is a projective 
birational morphism $\mu\colon X_i\to X'$ and an ample $\Q$-divisor $A$ on $X'$ 
such that we have $P_\sigma\left(X_i, (L-u Y)_{X_i}\right)=\mu^*A$ and 
the $\Q$-divisor $N_\sigma\left(X_i, (L-u Y)_{X_i}\right)$ is $\mu$-exceptional. 
From the assumption, $Y_i$ is also $\mu$-exceptional. Therefore, for any 
$0<\varepsilon\ll 1$ and for any sufficiently divisible $m\in\Z_{>0}$, we have 
\begin{eqnarray*}
H^0\left(X_i,m(L-u Y)_{X_i}\right)&\simeq& H^0\left(X_i,m\mu^*A\right)\\
\simeq H^0\left(X_i,m(\mu^*A+N_\sigma\left(X_i, (L-u Y)_{X_i}\right)+\varepsilon 
Y\i)\right)&=&H^0\left(X_i, m(L-(u-\varepsilon)Y)_{X_i}\right).
\end{eqnarray*}
This implies that $Y\subset\Supp N_\sigma\left(X, L-(u-\varepsilon)Y\right)$, 
a contradiction. 
\end{proof}

\begin{notation}\label{MDS1_notation}
Let us fix a common resolution 
$\sigma_i\colon\tilde{X}\to X_i$ of $X_0,X_1,\dots,X_I$ 
with $\tilde{X}$ normal and $\Q$-factorial, where $X_0:=X$ and 
$X_1,\dots,X_I$ are in Lemma \ref{okawa_lemma} \eqref{okawa_lemma3}. 
We set $\sigma:=\sigma_0$ and $\tilde{Y}:=\sigma^{-1}_*Y$. 
Moreover, for any $u\in[0,\tau_+]$, let 
\[
\sigma^*(L-u Y)=P(u)+N(u)
\]
be the Nakayama--Zariski decomposition of $\sigma^*(L-u Y)$, i.e., 
\begin{eqnarray*}
P(u)&:=&P_\sigma\left(\tilde{X},\sigma^*(L-u Y)\right),\\
N(u)&:=&N_\sigma\left(\tilde{X},\sigma^*(L-u Y)\right).
\end{eqnarray*}
\end{notation}

We remark that, if $X$ is a smooth Fano threefold, then there is no small 
$\Q$-factorial modification of $X$ by \cite{mori}. Thus we have $I=0$ 
and we can take $\sigma\colon\tilde{X}\to X$ as the identity morphism 
when $X$ is a smooth Fano threefold. See also \cite{FANO}.

\begin{remark}\label{NZ_resol_remark}
By \cite[Chapter III, Lemma 2.5]{N} and Lemma \ref{okawa_lemma}, we have: 
\begin{itemize}
\item
$P(u)$ is semiample 
and $\tilde{Y}\not\subset\Supp N(u)$ for any $u\in[\tau_-,\tau_+]$, 
and
\item
we have
\[
N(u)=\frac{\tau_i-u}{\tau_i-\tau_{i-1}}N(\tau_{i-1})
+\frac{u-\tau_{i-1}}{\tau_i-\tau_{i-1}}N(\tau_i)
\]
for any $1\leq i\leq I$ and for any $u\in[\tau_{i-1},\tau_i]$, and 
\item
the $\R$-divisor $P(u)|_{\tilde{Y}}$ is semiample and big for any $u\in(\tau_-,\tau_+)$. 
\end{itemize}
\end{remark}

\begin{lemma}\label{gordan_lemma}
There exists $($a sufficiently divisible$)$ $m_0\in\Z_{>0}$ such that: 
\begin{itemize}
\item
we have $m_0\tau_i\in\Z_{\geq 0}$ for any $0\leq i\leq I$, where $\tau_i$ is 
as in Lemma \ref{okawa_lemma} \eqref{okawa_lemma3}, and 
\item
for any $(a,j)\in\left(m_0\Z_{\geq 0}\right)^2\setminus\{(0,0)\}$ with 
$\tau_-\leq j/a\leq \tau_+$, both 
\[
a N\left(\frac{j}{a}\right)\quad\text{and}\quad a P\left(\frac{j}{a}\right)
\]
are Cartier divisors. 
\end{itemize}
\end{lemma}

\begin{proof}
Fix any $1\leq i\leq I$ and set 
\[
\sC_i:=\Cone\left((1,\tau_{i-1}),(1,\tau_i)\right)\subset\R^2_{\geq 0}.
\]
By Gordan's lemma, there exist 
\[
(a_1,j_1),\dots,(a_N,j_N)\in\sC_i\cap\Z_{\geq 0}^2
\]
such that, every element $(a,j)\in\sC_i\cap\Z_{\geq 0}^2$ can be expressed by a 
$\Z_{\geq 0}$-linear sum 
\[
(a,j)=\sum_{k=1}^N c_k(a_k,j_k)
\]
of $(a_1,j_1),\dots,(a_N,j_N)$. Take $m_0\in\Z_{> 0}$ such that 
\[
m_0 a_1 N\left(\frac{j_1}{a_1}\right),\dots,m_0 a_N N\left(\frac{j_N}{a_N}\right)
\]
are Cartier divisors. Then, for any $(a,j)\in\sC_i\cap\Z_{\geq 0}^2$, the $\Q$-divisor 
\[
m_0 a\cdot N\left(\frac{m_0 j}{m_0 a}\right)
=m_0\cdot a N\left(\frac{j}{a}\right)
=m_0\sum_{k=1}^N c_k a_k N\left(\frac{j_k}{a_k}\right)
\]
is a Cartier divisor by Remark \ref{NZ_resol_remark}. 
\end{proof}

\begin{definition}\label{divisorial-restriction_definition}
Under Notation \ref{MDS1_notation}, let us set the Veronese equivalent 
class $V_{\bullet,\bullet}^{\tilde{Y}}$ 
of $(m_0\Z_{\geq 0})^2$-graded linear series $V^{\tilde{Y}}_{m_0\bullet,m_0\bullet}$ 
on $\tilde{Y}$ associated to $\sigma^*(L|_Y)$, $\sigma^*(-Y|_Y)$ 
(for a sufficiently divisible $m_0\in\Z_{>0}$ as in Lemma \ref{gordan_lemma}) 
defined by: 
\[
V_{a,j}^{\tilde{Y}}:=\begin{cases}
a N\left(\frac{j}{a}\right)|_{\tilde{Y}}
+H^0\left(\tilde{Y},a P\left(\frac{j}{a}\right)|_{\tilde{Y}}\right) 
& \text{if }j\in[a\tau_-,a\tau_+],\\
0 & \text{otherwise},
\end{cases}
\]
for any $(a,j)\in(m_0\Z_{\geq 0})^2$. We call it the \emph{divisorial restriction 
of $V_\bullet$ by $\tilde{Y}\subset\tilde{X}$}. It is obvious that 
$V^{\tilde{Y}}_{\bullet,\bullet}$ 
has bounded support with 
\[
\Supp\left(V_{\bullet,\bullet}^{\tilde{Y}}\right)=\Cone\left(
(1,\tau_-), (1,\tau_+)\right)\subset\R_{\geq 0}^2, 
\]
and contains an ample series 
(see also Remark \ref{divisorial-restriction_remark}). 
\end{definition}

\begin{remark}\label{divisorial-restriction_remark}
For any sufficiently divisible $a$, $j\in m_0 \Z_{\geq 0}$, we have 
$W_{a,j}\subset V^{\tilde{Y}}_{a,j}$ as linear series on $\tilde{Y}$, 
where we regard $W_{\bullet,\bullet}$ as $(\sigma|_{\tilde{Y}})^*W_{\bullet,\bullet}$ 
(see Example \ref{pullback_example}). In fact, when $j>a\tau_+$, then 
\[
H^0\left(\tilde{X},\sigma^*(a L-j Y)\right)=0.
\]
Thus we have $W_{a,j}= V^{\tilde{Y}}_{a,j}=0$. When $j\leq a\tau_+$, then we have 
\[
H^0\left(\tilde{X},\sigma^*(a L-j Y)\right)
=a N\left(\frac{j}{a}\right)+H^0\left(\tilde{X}, a P\left(\frac{j}{a}\right)\right).
\]
Thus, when $j\in[0, a\tau_-)$, then the restriction homomorphism is the zero map 
by Lemma \ref{okawa_lemma} \eqref{okawa_lemma2} and thus 
$W_{a,j}= V^{\tilde{Y}}_{a,j}=0$. 
When $j\in[a\tau_-,a\tau_+]$, then the restriction homomorphism 
factors through
\[
j\sigma^*Y+
a N\left(\frac{j}{a}\right)+H^0\left(\tilde{X}, a P\left(\frac{j}{a}\right)\right)
\to V^{\tilde{Y}}_{a,j}\subset 
H^0\left(\tilde{Y}, a \sigma^*(L|_Y-j Y|_Y)\right).
\]
Moreover, from the above, we get the inequality 
\[
\dim\left(V^{\tilde{Y}}_{a,j}/W_{a,j}\right)\leq 
h^1\left(\tilde{X}, aP\left(\frac{j}{a}\right)-\tilde{Y}\right).
\]
\end{remark}

We recall the following easy proposition given in \cite{FANO}: 

\begin{proposition}[{see \cite{FANO}}]\label{vanish_proposition}
Let $Z$ be an $n$-dimensional projective variety, let $a\in\Z_{>0}$ and 
let $A$, $B$ be Cartier divisors on $Z$ with $A$ nef and big and $A+a B$ nef. 
Then, for any coherent sheaf $\sF$ on $Z$ and for any $i>0$, we have 
\[
\sum_{j=0}^{m a}h^i\left(Z, \sF\otimes\sO_Z(m A+j B)\right)=O(m^{n-i})
\]
as $m\to\infty$. 
\end{proposition}

\begin{proof}
This is an easy consequence of Takao Fujita's vanishing theorem (see 
\cite[Corollary 3.9.3]{fujino}). 
For the proof, see \cite{FANO}. 
\end{proof}

\subsection{Refinements on Mori dream spaces}
\label{AZ1_subsection}

In \S \ref{AZ1_subsection}, we show the following slight generalization of the formula 
obtained in \cite{FANO}. 

\begin{thm}[{cf.\ \cite{FANO}}]\label{AZ1_thm}
Let $V^{\tilde{Y}}_{\bullet,\bullet}$ be the divisorial restriction of $V_\bullet$ 
by $\tilde{Y}\subset\tilde{X}$, where $\sigma\colon\tilde{X}\to X$ be 
as in Notation \ref{MDS1_notation}. 
\begin{enumerate}
\renewcommand{\theenumi}{\arabic{enumi}}
\renewcommand{\labelenumi}{(\theenumi)}
\item\label{AZ1_thm1}
We have
\[
\vol(L)=\vol\left(W_{\bullet,\bullet}\right)
=\vol\left(V^{\tilde{Y}}_{\bullet,\bullet}\right). 
\]
\item\label{AZ1_thm2}
For any prime divisor $E$ over the normalization of $Y$, we have 
\begin{eqnarray*}
&&S\left(W_{\bullet,\bullet};E\right)
=S\left(V^{\tilde{Y}}_{\bullet,\bullet};E\right)\\
&=&\frac{n}{\vol(L)}\int_{\tau_-}^{\tau_+}
\left(\left(P(u)|_{\tilde{Y}}\right)^{\cdot n-1}\cdot 
\ord_E\left(N(u)|_{\tilde{Y}}\right)+\int_0^\infty
\vol_{\tilde{Y}}\left(P(u)|_{\tilde{Y}}-vE\right)dv\right)du.
\end{eqnarray*}
\end{enumerate}
\end{thm}

\begin{proof}
The proof is essentially same as the proof in the paper \cite{FANO}. 
Fix a sufficiently divisible $m_0\in\Z_{>0}$ as in Lemma \ref{gordan_lemma}. 
Let us consider the differences between the two representatives 
$W_{m_0\bullet,m_0\bullet}$ and $V_{m_0\bullet,m_0\bullet}^{\tilde{Y}}$. 

\eqref{AZ1_thm1}
We already know the equality 
$\vol\left(V_\bullet\right)=\vol\left(W_{\bullet,\bullet}\right)$ by Definition 
\ref{okounkov_definition}. Moreover, 
$\vol(L)=\vol\left(V_\bullet\right)$ holds from the definition of $\vol(L)$. Thus, 
it is enough to show the equality 
$\vol\left(W_{\bullet,\bullet}\right)=\vol\left(V^{\tilde{Y}}_{\bullet,\bullet}\right)$. 
For any $a\in m_0^2\Z_{>0}$, by Remark \ref{divisorial-restriction_remark}, we have
\begin{eqnarray*}
0&\leq& 
h^0\left(V^{\tilde{Y}}_{a,m_0\bullet}\right)-h^0\left(W_{a,m_0\bullet}\right)
=\sum_{j\in m_0\Z_{\geq 0}}h^0\left(V^{\tilde{Y}}_{a,j}/W_{a,j}\right)\\
&\leq&\sum_{i=1}^I\sum_{j\in[a\tau_{i-1},a\tau_i]\cap m_0\Z}
h^0\left(V^{\tilde{Y}}_{a,j}/W_{a,j}\right)\\
&\leq&\sum_{i=1}^I\sum_{j\in[a\tau_{i-1},a\tau_i]\cap m_0\Z}
h^1\left(\tilde{X},a P\left(\frac{j}{a}\right)-\tilde{Y}\right).
\end{eqnarray*}
By Remark \ref{NZ_resol_remark} and Lemma \ref{gordan_lemma}, for any 
$1\leq i\leq I$, there exists a Cartier divisor $B^i_{\tilde{X}}$ on $\tilde{X}$ such that 
\[
a P\left(\frac{j}{a}\right)=a P(\tau_{i-1})+\frac{j-a\tau_{i-1}}{m_0}B^i_{\tilde{X}}
\]
holds for any $j\in[a\tau_{i-1},a \tau_i]\cap m_0\Z$. 
Therefore, by Proposition \ref{vanish_proposition}, we have 
\[
0\leq h^0\left(V^{\tilde{Y}}_{a,m_0\bullet}\right)-h^0\left(W_{a,m_0\bullet}\right)\leq 
O(a^{n-1})\quad(\text{as }a\in m_0^2\Z_{>0}\to\infty),
\]
since $P(\tau_i)$ is nef for any $0\leq i\leq I$ and also big when $i<I$. 
Thus we get the equality 
$\vol\left(W_{\bullet,\bullet}\right)=\vol\left(V^{\tilde{Y}}_{\bullet,\bullet}\right)$. 

\eqref{AZ1_thm2}
We can take $\tau'\in\Z_{>0}$ such that 
\[
\vol_{\tilde{Y}}\left(\tilde{Y}, P(u)|_{\tilde{Y}}-v E\right)=0
\]
for any $u\in[\tau_-,\tau_+]$ and for any $v\geq \tau'$. 
For any $t\in\R_{\geq 0}$, the graded linear series $V^{\tilde{Y},t}_{\bullet,\bullet}$ 
and $W^{t}_{\bullet,\bullet}$ in Definition \ref{filter_definition} satisfy 
\[
W^{t}_{a,j}=\sF_E^{a t}W_{a,j}\subset \sF_E^{a t}V^{\tilde{Y}}_{a,j}
=V^{\tilde{Y}, t}_{a,j}
\]
for any $(a,j)\in(m_0\Z_{\geq 0})^2$. Moreover, since we have the natural inclusion 
\[
V^{\tilde{Y},t}_{a,j}/W^{t}_{a,j}\hookrightarrow V^{\tilde{Y}}_{a,j}/W_{a,j}, 
\]
we have 
\begin{eqnarray*}
0\leq h^0\left(V^{\tilde{Y},t}_{a,m_0\bullet}\right)-h^0\left(W^t_{a,m_0\bullet}\right)
\leq\sum_{j\in m_0\Z_{\geq 0}}h^0\left( V^{\tilde{Y}}_{a,j}/W_{a,j}\right)
\leq O(a^{n-1})\quad(\text{as }a\in m_0^2\Z_{>0}\to\infty)
\end{eqnarray*}
by \eqref{AZ1_thm1}. Thus we get 
$\vol\left(W^{t}_{\bullet,\bullet}\right)
=\vol\left(V^{\tilde{Y},t}_{\bullet,\bullet}\right)$. 

For any $a\in m_0^2\Z_{>0}$, we have
\begin{eqnarray*}
S_a\left(V^{\tilde{Y}}_{m_0\vec{\bullet}};\sF_E\right)&=&
\frac{1}{a h^0\left(V^{\tilde{Y}}_{a,m_0\bullet}\right)}
\sum_{j\in[a\tau_-,a\tau_+]\cap m_0\Z}\Biggl(
a\cdot\ord_E\left(N\left(\frac{j}{a}\right)|_{\tilde{Y}}\right)\cdot
h^0\left(\tilde{Y}, a P\left(\frac{j}{a}\right)|_{\tilde{Y}}\right)\\
&&+\sum_{k=0}^{a\tau'}h^0\left(\tilde{Y}, a 
P\left(\frac{j}{a}\right)|_{\tilde{Y}}-k E\Biggr)
\right).
\end{eqnarray*}
This implies that 
\begin{eqnarray*}
&&S\left(V^{\tilde{V}}_{\bullet,\bullet};E\right)\cdot\frac{\vol(L)}{n}
=S\left(V^{\tilde{V}}_{\bullet,\bullet};E\right)\cdot
\lim_{a\in m_0^2\Z_{>0}}\frac{m_0\cdot 
h^0\left(V^{\tilde{Y}}_{a,m_0\bullet}\right)}{na^n/n!}\\
&=&\lim_{a\in m_0^2\Z_{>0}}\Biggl(\frac{m_0}{a}\sum_{j\in[a\tau_-,a\tau_+]\cap m_0\Z}
\ord_E\left(N\left(\frac{j}{a}\right)|_{\tilde{Y}}\right)
\frac{h^0\left(\tilde{Y}, a P\left(\frac{j}{a}\right)|_{\tilde{Y}}\right)}{a^{n-1}/(n-1)!}\\
&&+\frac{m_0}{a^2}\sum_{j\in[a\tau_-,a\tau_+]\cap m_0\Z}\sum_{k=0}^{a\tau'}
\frac{h^0\left(\tilde{Y}, a P\left(\frac{j}{a}\right)|_{\tilde{Y}}-k E\right)}
{a^{n-1}/(n-1)!}\Biggr)\\
&=&\int_{\tau_-}^{\tau_+}
\left(\left(P(u)|_{\tilde{Y}}\right)^{\cdot n-1}\cdot 
\ord_E\left(N(u)|_{\tilde{Y}}\right)+\int_0^{\tau'}
\vol_{\tilde{Y}}\left(P(u)|_{\tilde{Y}}-vE\right)dv\right)du.
\end{eqnarray*}
Thus we have completed the proof of \eqref{AZ1_thm2}. 
\end{proof}

\begin{corollary}\label{AZ1_corollary}
Assume that there is an effective $\Q$-divisor $\Delta$ on $X$ such that 
the pair $(X, \Delta+Y)$ is a plt pair.
Set 
\[
K_Y+\Delta_Y:=(K_X+\Delta+Y)|_Y. 
\]
Then we have 
\[
\delta\left(Y,\Delta_Y; W_{\bullet,\bullet}\right)
=\delta\left(Y,\Delta_Y; V^{\tilde{Y}}_{\bullet,\bullet}\right)
\]
and 
\[
\delta_\eta\left(Y,\Delta_Y; W_{\bullet,\bullet}\right)
=\delta_\eta\left(Y,\Delta_Y; V^{\tilde{Y}}_{\bullet,\bullet}\right)
\]
for any scheme-theoretic point $\eta\in Y$, where we regard 
$V^{\tilde{Y}}_{\bullet,\bullet}$ as the Veronese equivalent class of 
a graded linear series \emph{on} $Y$ under the isomorphism 
$\sO_Y\simeq\left(\sigma|_{\tilde{Y}}\right)_*\sO_{\tilde{Y}}$. 
In particular, we have the inequality
\[
\delta_\eta\left(X,\Delta;L\right)\geq\min\left\{\frac{A_{X,\Delta}(Y)}{S_L(Y)},
\,\,\delta_\eta\left(Y,\Delta_Y; V^{\tilde{Y}}_{\bullet,\bullet}\right)
\right\}.
\]
\end{corollary}

\begin{proof}
Follows immediately from Theorems \ref{AZ_thm} and \ref{AZ1_thm}. 
\end{proof}

\subsection{Taking the refinements twice on $3$-dimensional Mori 
dream spaces}\label{AZ2_subsection}

In \S \ref{AZ2_subsection}, 
we consider a slight generalization of the result in \cite{FANO}. 
The authors in \cite{FANO} only consider the case $X$ is a smooth Fano threefold. 
The case is relatively easy since the movable cone of $X$ is 
equal to the nef cone of $X$. 
However, in order to consider Theorem \ref{main_thm}, we must consider 
(weighted)  blowups of smooth Fano threefolds. Thus we must consider 
more complicated situations than \cite{FANO}. 

In \S \ref{AZ2_subsection}, we further assume that $n=3$. Moreover, 
the prime divisors $\tilde{Y}$ and $Y$ in Notation \ref{MDS1_notation} are 
assumed to be normal. In this case, the series 
$V_{\bullet,\bullet}^{\tilde{Y}}$ on $\tilde{Y}$ 
can be regarded as a series \emph{on $Y$}. 
Moreover, let us fix: 
\begin{itemize}
\item
a projective birational morphism $\nu\colon Y'\to Y$ with $Y'$ normal, 
\item
a prime $\Q$-Cartier divisor $C\subset Y'$ such that $C$ is a smooth 
projective curve, and 
\item
a common resolution
\[\xymatrix{
 & \bar{Y} \ar[dr]^\theta \ar[dl]_\gamma & \\
 Y' \ar[dr]_\nu & & \tilde{Y} \ar[dl]^{\sigma|_{\tilde{Y}}} \\
 & Y & 
}\]
with $\bar{Y}$ normal and $\Q$-factorial.
\end{itemize}
Let us set $\bar{C}:=\gamma^{-1}_*C$ and 
$\gamma^*C=:\bar{C}+\Sigma$. Let 
$W_{\bullet,\bullet,\bullet}^{Y',C}$ (resp., 
$W_{\bullet,\bullet,\bullet}^{\bar{Y},\bar{C}}$) be the refinement of 
$V_{\bullet,\bullet}^{\tilde{Y}}$ on $Y'$ by $C\subset Y'$
(resp., $V_{\bullet,\bullet}^{\tilde{Y}}$ on $\bar{Y}$ by $\bar{C}\subset\bar{Y}$).

\begin{remark}\label{refinement-AZ_remark}
By Remark \ref{refinement_remark}, for any sufficiently divisible $a,j,k\in\Z_{\geq 0}$, 
we have 
\[
W^{\bar{Y},\bar{C}}_{a,j,k}=W^{Y',C}_{a,j,k}+k\left(\Sigma|_{\bar{C}}\right), 
\]
where we regard $W_{\bullet,\bullet,\bullet}^{Y',C}$ as a series on $\bar{C}$ 
under the isomorphism $\gamma|_{\bar{C}}\colon\bar{C}\to C$. 
In particular, we have 
$\Supp\left(W_{\bullet,\bullet,\bullet}^{\bar{Y},\bar{C}}\right)=
\Supp\left(W_{\bullet,\bullet,\bullet}^{Y', C}\right)$. 
\end{remark}

\begin{notation}\label{MDS2_notation}
\begin{enumerate}
\renewcommand{\theenumi}{\arabic{enumi}}
\renewcommand{\labelenumi}{(\theenumi)}
\item\label{MDS2_notation1}
For any $u\in[\tau_-, \tau_+]$, since $\tilde{Y}\not\subset\Supp N(u)$ 
(see Remark \ref{NZ_resol_remark}), we can set 
\[
\theta^*\left(N(u)|_{\tilde{Y}}\right)=: d(u) \bar{C}+N'(u), 
\]
where $d(u):=\ord_{\bar{C}}\left(\theta^*\left(N(u)|_{\tilde{Y}}\right)\right)$. 
(By Lemma \ref{okawa_lemma}, $\theta^*\left(N(u)|_{\tilde{Y}}\right)$ 
is a $\Q$-divisor when $u\in\Q$.)
\item\label{MDS2_notation2}
For any $u\in[\tau_-,\tau_+]$, since $P(u)|_{\tilde{Y}}$ is semiample, we can define 
\[
t(u):=\max\left\{t\in\R_{\geq 0}\,\,\Big|\,\,
\theta^*\left(P(u)|_{\tilde{Y}}\right)-t\bar{C}\text{ is 
pseudo-effective}\right\}.
\]
For any $v\in[0,t(u)]$, let us set 
\begin{eqnarray*}
P(u,v) &:=& P_\sigma\left(\bar{Y}, \theta^*\left(P(u)|_{\tilde{Y}}\right)
-v\bar{C}\right),\\ 
N(u,v) &:=& N_\sigma\left(\bar{Y}, \theta^*\left(P(u)|_{\tilde{Y}}\right)
-v\bar{C}\right),
\end{eqnarray*}
and define 
\[
\bar{\Delta}^{\Supp}:=\left\{(u,v)\in\R^2_{\geq 0}\,\,\Big|\,\,
u\in[\tau_-,\tau_+], \,\,v\in[d(u),d(u)+t(u)]\right\}\subset\R^{2}_{\geq 0}. 
\]
\end{enumerate}
\end{notation}

\begin{lemma}\label{Delta_lemma}
\begin{enumerate}
\renewcommand{\theenumi}{\arabic{enumi}}
\renewcommand{\labelenumi}{(\theenumi)}
\item\label{Delta_lemma1}
The function $d\colon [\tau_-,\tau_+]\to\R_{\geq 0}$ is continuous and convex. 
\item\label{Delta_lemma2}
The function $d+t\colon [\tau_-,\tau_+]\to\R_{\geq 0}$ is continuous and concave. 
\end{enumerate}
In particular, $\bar{\Delta}^{\Supp}\subset\R^2_{\geq 0}$ is a closed and 
convex set. 
\end{lemma}

\begin{proof}
\eqref{Delta_lemma1}
The continuity is trivial by Lemma \ref{okawa_lemma}. The convexity follows from the 
following: 

\begin{claim}\label{Delta_claim}
For all $\tau_-\leq u<u'\leq \tau_+$ and for all $s\in[0,1]$, we have 
\[
N\left((1-s)u+su'\right)\leq (1-s)N(u)+s N(u'). 
\]
\end{claim}

\begin{proof}[Proof of Claim \ref{Delta_claim}]
Trivial since 
\begin{eqnarray*}
N\left((1-s)u+su'\right)&=&N_\sigma\left(\tilde{X}, (1-s)\sigma^*(L-u Y)
+s \sigma^*(L-u' Y)\right)\\
&\leq& 
(1-s) N_\sigma\left(\tilde{X}, \sigma^*(L-u Y)\right)
+s N_\sigma\left(\tilde{X}, \sigma^*(L-u' Y)\right)
\end{eqnarray*}
(see \cite[Chapter III, Definition 1.1]{N}). 
\end{proof}

\eqref{Delta_lemma2}
The continuity of $d+t$ is trivial by \eqref{Delta_lemma1}. Let us show the concavity. 
Take any $\tau_-\leq u<u'\leq \tau_+$ and $s\in[0,1]$. Note that 
\[
\theta^*\left(\sigma^*(L-u Y)|_{\tilde{Y}}\right)-N'(u)=
\theta^*\left(P(u)|_{\tilde{Y}}\right)+d(u)\bar{C}.
\]
Thus 
\[
\theta^*\left(\sigma^*(L-u Y)|_{\tilde{Y}}\right)-N'(u)-(d(u)+t(u))\bar{C}
\]
is pseudo-effective. By Claim \ref{Delta_claim}, we have 
\begin{eqnarray*}
&&\theta^*\left(\sigma^*\left(L-\left((1-s)u+s u'\right)Y\right)|_{\tilde{Y}}\right)
-\left((1-s)N'(u)+sN'(u')\right)\\
&&-\left((1-s)(d(u)+t(u))+s(d(u')+t(u'))\right)\bar{C}\\
&\leq&\theta^*\left(\sigma^*\left(L-\left((1-s)u+s'u\right)Y\right)|_{\tilde{Y}}\right)
-N'\left((1-s)u+su'\right)\\
&&-\left((1-s)(d(u)+t(u))+s(d(u')+t(u'))\right)\bar{C}. 
\end{eqnarray*}
This implies that 
\begin{eqnarray*}
\theta^*\left(P\left((1-s)u+s u'\right)|_{\tilde{Y}}\right)-
\left((1-s)\left(d(u)+t(u)\right)+s\left(d(u')+t(u')\right)
-d\left((1-s)u+s u'\right)\right)\bar{C}
\end{eqnarray*}
is pseudo-effective. Thus we get the assertion. 
\end{proof}

\begin{lemma}\label{interior_lemma}
\begin{enumerate}
\renewcommand{\theenumi}{\arabic{enumi}}
\renewcommand{\labelenumi}{(\theenumi)}
\item\label{interior_lemma1}
For any $u\in[\tau_-,\tau_+]$, we have $N(u,0)=0$. In particular, for any 
$v\in[0,t(u)]$, we have $\bar{C}\not\subset \Supp N(u,v)$. 
\item\label{interior_lemma2}
For any $(u,v)\in\interior\left(\bar{\Delta}^{\Supp}\right)$, we have 
$\left(P(u,v)\cdot\bar{C}\right)>0$. 
\end{enumerate}
\end{lemma}

\begin{proof}
\eqref{interior_lemma1}
Since $\theta^*\left(P(u)|_{\tilde{Y}}\right)$ is nef, we have $N(u,0)=0$. If 
$\bar{C}\subset\Supp N(u,v)$, then, by \cite[Chapter III, Corollary 1.9]{N}, 
we must have $\bar{C}\subset\Supp N(u,0)$. This leads to a contradiction. 

\eqref{interior_lemma2}
Assume that $\left(P(u,v)\cdot\bar{C}\right)=0$. The Hodge index theorem implies that 
the intersection matrix of the support of $N(u,v)+\bar{C}$ is negative definite. 
Thus, for any $0<\varepsilon \ll 1$, 
\[
\theta^*\left(P(u)|_{\tilde{Y}}\right)-(v-\varepsilon)C
=P(u,v)+\left(N(u,v)+\varepsilon \bar{C}\right)
\]
gives the Zariski decomposition. This leads to a contradiction to \eqref{interior_lemma1}.
\end{proof}

\begin{proposition}\label{interior_proposition}
\begin{enumerate}
\renewcommand{\theenumi}{\arabic{enumi}}
\renewcommand{\labelenumi}{(\theenumi)}
\item\label{interior_proposition1}
We have 
\begin{eqnarray*}
{\bar{\Delta}}^{\Supp}=
\Supp\left(W^{\bar{Y},\bar{C}}_{\bullet,\bullet,\bullet}\right)\cap
\left(\{1\}\times\R^2_{\geq 0}\right)=
\Supp\left(W^{Y',C}_{\bullet,\bullet,\bullet}\right)\cap
\left(\{1\}\times\R^2_{\geq 0}\right).
\end{eqnarray*}
\item\label{interior_proposition2}
Take any closed point $p\in C$ and let us consider the Okounkov body 
$\Delta_{C_\bullet}\left(W_{\bullet,\bullet,\bullet}^{Y',C}\right)\subset\R_{\geq 0}^3$ 
of $W_{\bullet,\bullet,\bullet}^{Y',C}$ associated to the admissible flag 
\[
C_{\bullet}\quad\colon\quad C\supsetneq \{ p\}.
\]
Let 
\[
pr\colon\Delta_{C_\bullet}\left(W_{\bullet,\bullet,\bullet}^{Y',C}\right)
\to\bar{\Delta}^{\Supp}
\]
be the natural projection as in Theorem \ref{LM_thm}. Then, for any 
$(u,v)\in\interior\left(\bar{\Delta}^{\Supp}\right)$, the inverse image 
$pr^{-1}\left((u,v)\right)\subset\R^1$ is equal to the closed area
\begin{eqnarray*}
&&\Bigl[
\ord_p\left(\left(N'(u)+N(u,v-d(u))-v\Sigma\right)|_{\bar{C}}\right),\\
&&\ord_p\left(\left(N'(u)+N(u,v-d(u))-v\Sigma\right)|_{\bar{C}}\right)
+\left(P(u,v-d(u))\cdot\bar{C}\right)\Bigr]\subset\R_{\geq 0}.
\end{eqnarray*}
\end{enumerate}
\end{proposition}

\begin{proof}
\eqref{interior_proposition1}
Take any $(u,v)\in\R^2_{\geq 0}\setminus\bar{\Delta}^{\Supp}$ 
with $(u,v)\in\Q^2$. If $u\not\in[\tau_-,\tau_+]$, then $V_{m,mu}^{\tilde{Y}}=0$ 
for $m\in\Z_{>0}$ sufficiently divisible. Thus we have 
$W^{\bar{Y},\bar{C}}_{m,mu,mv}=0$. If $u\in[\tau_-,\tau_+]$ but 
$v\not\in[d(u), d(u)+t(u)]$, then $W^{\bar{Y}, \bar{C}}_{m,mu,mv}=0$ 
for $m\in\Z_{>0}$ sufficiently divisible. Indeed, 
\begin{itemize}
\item
if $v<d(u)$, then the homomorphism 
\begin{eqnarray*}
&&\theta^*V^{\tilde{Y}}_{m,mu}\cap 
\left(mv\bar{C}+H^0\left(\bar{Y}, m\theta^*\left(\sigma^*(L-u Y)|_{\tilde{Y}}\right)
-mv\bar{C}\right)\right)\\
&\xrightarrow{\rest}&
H^0\left(\bar{C}, m\theta^*\sigma^*(L-u Y)|_{\bar{C}}-mv\bar{C}|_{\bar{C}}\right)
\end{eqnarray*}
is a zero map since any member in 
$\theta^*V^{\tilde{Y}}_{m,mu}$ vanishes along $\bar{C}$ 
of order at least $m d(v)>mv$, 
\item
if $v>d(u)+t(u)$, then 
\[
\theta^*V^{\tilde{Y}}_{m,mu}\cap \left(mv\bar{C}+
H^0\left(\bar{Y}, m\theta^*\left(\sigma^*(L-u Y)|_{\bar{Y}}\right)
-mv\bar{C}\right)\right)=0
\]
since 
\[
\ord_{\bar{C}}\left(\theta^*\left(m N(u)|_{\tilde{Y}}\right)\right)=m d(u)
\]
and 
\[
m\left(\theta^*\left(P(u)|_{\tilde{Y}}\right)-(v-d(u))\bar{C}\right)
\]
is not pseudo-effective.
\end{itemize}
Thus we get the inclusion 
${\bar{\Delta}}^{\Supp}\supset
\Supp\left(W_{\bullet,\bullet,\bullet}^{\bar{Y},\bar{C}}\right)\cap
\left(\{1\}\times\R_{\geq 0}^2\right)$. 

Take any $(u,v)\in\interior\left(\bar{\Delta}^{\Supp}\right)\cap\Q^2$. 
For any sufficiently divisible $m\in\Z_{>0}$, since 
\begin{eqnarray*}
&&\theta^*V^{\tilde{Y}}_{m,mu}\cap\left(mv \bar{C}
+H^0\left(\bar{Y}, m\left(\theta^*\left(\sigma^*(L-u Y)|_{\tilde{Y}}\right)
-v\bar{C}\right)\right)\right)\\
&=&\left(m\left(d(u)\bar{C}+N'(u)\right)+H^0\left(
\bar{Y}, m \theta^*\left(P(u)|_{\tilde{Y}}\right)\right)\right)\\
&&\cap\left(mv\bar{C}+H^0\left(\bar{Y}, 
m\left(\theta^*\left(\sigma^*(L-u Y)|_{\tilde{Y}}\right)-v\bar{C}\right)\right)\right)\\
&=&m\left(v\bar{C}+N'(u)+N(u,v-d(u))\right)
+H^0\left(\bar{Y}, m P(u, v-d(u))\right),
\end{eqnarray*}
we have 
\begin{eqnarray*}
W^{\bar{Y}, \bar{C}}_{m,mu,mv}&=&
\Image\Biggl(m\left(v\bar{C}+
N'(u)+N(u,v-d(u))\right)+H^0\left(\bar{Y}, m P(u,v-d(u))\right)\\
&&\xrightarrow{\rest}m\left(N'(u)|_{\bar{C}}+N(u,v-d(u))|_{\bar{C}}\right)
+H^0\left(\bar{C},m P(u,v-d(u))|_{\bar{C}}\right)\Biggr).
\end{eqnarray*}
The above homomorphism $\rest$ satisfies that 
\[
\coker(\rest)\subset H^1\left(\bar{Y}, m P(u,v-d(u))-\bar{C}\right). 
\]
Since $P(u, v-d(u))$ is a nef and big $\Q$-divisor, we get 
$\dim\left(\coker(\rest)\right)\leq O(1)$ as $m\to\infty$ 
by Takao Fujita's vanishing theorem \cite[Corollary 3.9.3]{fujino}. 
Moreover, let us recall that 
\[
W_{m,mu,m v}^{\bar{Y},\bar{C}}=W_{m,mu,m v}^{Y',C}+m v\left(\Sigma|_{\bar{C}}\right). 
\]
Thus, the Okounkov body of the series $W_{(1,u,v),\bullet}^{Y',C}$ (given 
in Theorem \ref{LM_thm}) associated to $C_\bullet$ is equal to the closed area 
given in the assertion of Proposition \ref{interior_proposition} 
\eqref{interior_proposition2}. By Lemma \ref{interior_lemma}, 
we have $\left(P(u,v-d(u))\cdot\bar{C}\right)>0$. 
Thus we get the inclusion 
${\bar{\Delta}}^{\Supp}\subset\Supp\left(W_{\bullet,\bullet,\bullet}^{Y',C}\right)
\cap\left(\{1\}\times\R_{\geq 0}^2\right)$. Thus we get the assertion. 

\eqref{interior_proposition2}
We may assume that $(u,v)\in\interior\left(\bar{\Delta}^{\Supp}\right)\cap\Q^2$. 
The assertion follows immediately from Theorem \ref{LM_thm} and 
the proof of Proposition 
\ref{interior_proposition} \eqref{interior_proposition1}.
\end{proof}

\begin{definition}[{cf.\ \cite{FANO}}]\label{fixed-order_definition}
Under Notation, \ref{MDS2_notation}, let us set 
\begin{eqnarray*}
&&F_p\left(W_{\bullet,\bullet,\bullet}^{Y',C}\right)\\
&:=&\frac{6}{\vol(L)}\int_{\tau_-}^{\tau_+}\int_0^{t(u)}
\left(P(u,v)\cdot\bar{C}\right)\cdot\ord_p\left(\left(
N'(u)+N(u,v)-(v+d(u))\Sigma\right)|_{\bar{C}}
\right)dvdu
\end{eqnarray*}
for any closed point $p\in C$. 
\end{definition}

Theorem \ref{AZ1_thm} and the following Theorem \ref{AZ2_thm} are crucial 
in this paper. 

\begin{thm}[{cf.\ \cite{FANO}}]\label{AZ2_thm}
Under Notation \ref{MDS2_notation}, we have 
\[
S\left(W^{Y',C}_{\bullet,\bullet,\bullet}; p\right)
=\frac{3}{\vol(L)}\int_{\tau_-}^{\tau_+}\int_0^{t(u)}
\left(\left(P(u,v)\cdot\bar{C}\right)\right)^2dvdu
+F_p\left(W^{Y',C}_{\bullet,\bullet,\bullet}\right)
\]
for any closed point $p\in C$. 
\end{thm}

\begin{proof}
As in Proposition \ref{interior_proposition}, let us consider the Okounkov body 
$\Delta_{C_\bullet}\left(W_{\bullet,\bullet,\bullet}^{Y',C}\right)$ and the projection 
$pr\colon\Delta_{C_\bullet}\left(W_{\bullet,\bullet,\bullet}^{Y',C}\right)
\to\bar{\Delta}^{\Supp}$. 
From Proposition \ref{barycenter_proposition}, we have
\begin{eqnarray*}
S\left(W_{\bullet,\bullet,\bullet}^{Y',C};p\right)
&=&\frac{1}{\vol\left(\Delta_{C_\bullet}\left(
W_{\bullet,\bullet,\bullet}^{Y',C}\right)\right)}
\int_{\Delta_{C_\bullet}\left(W_{\bullet,\bullet,\bullet}^{Y',C}\right)} 
x_3\,\, d\vec{x} \\
&=&\frac{6}{\vol(L)}\int_{(u,v)\in\Delta^{\Supp}}\int_{x\in pr^{-1}\left((u,v)\right)}
x\,\, dx dvdu\\
&=&\frac{6}{\vol(L)}\int_{\tau_-}^{\tau_+}\int_{d(u)}^{d(u)+t(u)}
\Biggl(\frac{1}{2}\left(\left(P(u,v-d(u))\cdot\bar{C}\right)\right)^2\\
&&+\left(P(u,v-d(u))\cdot\bar{C}\right)\ord_p
\left(\left(N'(u)+N(u,v-d(u))-v\Sigma\right)|_{\bar{C}}\right)\Biggr)dvdu.
\end{eqnarray*}
Thus we get the assertion. 
\end{proof}

As a consequence, we get the following corollary. We frequently use it 
in order to prove Theorem \ref{mainthm}. 

\begin{corollary}\label{reduction_corollary}
Assume that there exists a projective klt pair $(X',\Delta')$ with $\Delta'$ 
effective $\Q$-Weil divisor and a big $\Q$-Cartier $\Q$-divisor $L'$ on $X'$ 
such that $Y$ is plt-type over $(X',\Delta')$, the associated plt-blowup 
is equal to $\mu\colon X\to X'$, and $L=\mu^*L'$. Set 
\begin{eqnarray*}
K_X+\Delta+\left(1-A_{X',\Delta'}(Y)\right)Y&:=&
\mu^*\left(K_{X'}+\Delta'\right), \\
K_Y+\Delta_Y &:=&(K_X+\Delta+Y)|_Y.
\end{eqnarray*}
Assume moreover that $\nu\colon Y'\to Y$ is the plt-blowup of the plt-type 
prime divisor $C$ over $(Y,\Delta_Y)$. Set 
\begin{eqnarray*}
K_{Y'}+\Delta_{Y'}+\left(1-A_{Y,\Delta_Y}(C)\right)C&:=&
\nu^*\left(K_Y+\Delta_Y\right), \\
K_C+\Delta_C &:=&\left(K_{Y'}+\Delta_{Y'}+C\right)|_C.
\end{eqnarray*}
\begin{enumerate}
\renewcommand{\theenumi}{\arabic{enumi}}
\renewcommand{\labelenumi}{(\theenumi)}
\item\label{reduction_corollary1}
For any closed point $q\in Y$ with $q\in c_Y(C)$, we have 
\[
\delta_q\left(Y,\Delta_Y;V_{\bullet,\bullet}^{\tilde{Y}}\right)
\geq \min\left\{\frac{A_{Y,\Delta_Y}(C)}{S\left(V_{\bullet,\bullet}^{\tilde{Y}};C\right)}, 
\,\,\inf_{\substack{p\in C\\ \nu(p)=q}}
\frac{A_{C,\Delta_C}(p)}{S\left(W_{\bullet,\bullet,\bullet}^{Y',C};p\right)}
\right\}.
\]
\item\label{reduction_corollary2}
For any closed point $q'\in X'$ with $q'\in c_{X'}(Y)$, we have
\[
\delta_{q'}\left(X',\Delta';L'\right)
\geq \min\left\{\frac{A_{X',\Delta'}(Y)}{S_{L'}(Y)}, 
\,\,\inf_{\substack{q\in Y\\ \mu(q)=q'}}
\delta_q\left(Y,\Delta_Y;V_{\bullet,\bullet}^{\tilde{Y}}\right)
\right\}.
\]
\end{enumerate}
\end{corollary}

\begin{proof}
Follows immediately from Theorem \ref{AZ_thm} and Corollary 
\ref{AZ1_corollary}.
\end{proof}

\section{Fano threefolds of No.\ 3.11}\label{311_section}

Let us explain the family No.\ 3.11, i.e., the family No.\ 11 in Table $3$ in 
Mori--Mukai's list \cite{MoMu}. 

Set $P:=\pr^3$, let $\sC^P\subset P$ be a smooth curve given by the complete 
intersection of two quadric surfaces. Take a point $p^P\in\sC^P$ and 
let $l^P\subset P$ be the tangent line of $\sC^P$ at $p^P$. 
Let us consider the blowup 
\[
\sigma^V\colon V\to P
\]
of $P$ at $p^P$ and let $E_2^V\subset V$ be the exceptional divisor. 
Set $\sC^V:=(\sigma^V)_*^{-1}\sC^P$, $l^V:=(\sigma^V)_*^{-1}l^P$. We know that 
\[
V\simeq\pr_{\pr^2}(\sO\oplus\sO(1))\xrightarrow{\pi^V}\pr^2
\]
and $l^V$ is a fiber of $\pi^V$. Moreover, since $\sC^P$ is the complete intersection 
of two quadrics, the restriction morphism 
\[
\pi^V|_{\sC^V}\colon\sC^V\to\pi^V\left(\sC^V\right)
\]
is an isomorphism. Let us set 
$q^{\sC}:=\pi^V(l^V)\in\pr^2$ and $\sC:=\pi^V\left(\sC^V\right)\subset\pr^2$. 
Then $q^{\sC}\in\sC$, and $\sC\subset\pr^2$ is a smooth cubic curve. 
Set $E_3^V:=(\pi^V)^{-1}(\sC)\subset V$. 

Let $\sigma_1\colon X\to V$ be the blowup along $\sC^V$ and let 
$E_1\subset X$ be the exceptional divisor. Let us set 
\begin{eqnarray*}
E_2&:=&(\sigma_1)^{-1}_* E_2^V, \\
E_3&:=&(\sigma_1)^{-1}_* E_3^V,\\
l&:=&(\sigma_1)^{-1}_*l^V.
\end{eqnarray*}
By \cite{MM84}, $X$ is a smooth Fano threefold such that $\rho(X)=3$, 
$(-K_X)^{\cdot 3}=28$ and $B_3(X)=2$. 
Since the pair $\left(V, E_2^V+E_3^V\right)$ is a log smooth pair and 
$\sC^V\subset E_3^V$ is a smooth curve intersecting with $E_2^V$ transversely 
at one point, the pair $\left(X, E_1+E_2+E_3\right)$ is also a log smooth pair. 
Set $q:=E_1\cap E_2\cap E_3$. Moreover, let $V_\bullet$ be the Veronese 
equivalence class of the complete linear series of $-K_X$ on $X$. 

By \cite[\S III-3]{matsuki}, there exists the commutative diagram: 
\[\xymatrix{
& \Bl_{{\sC}^P}P \ar[dr] \ar[dl] & \\
P=\pr^3 &  & \pr^1 \\
& X \ar[dl]_{\sigma_1} \ar[uu]_{\sigma_2} \ar[dr]^{\sigma_3} & \\
V \ar[uu]^{\sigma^V} \ar[dr]_{\pi^V} & & \pr^1\times\pr^2 
\ar[uu]_{\operatorname{pr}_1} 
\ar[dl]^{\operatorname{pr}_2} \\
& \pr^2 & 
}\]
where $\sigma_i$ is the birational morphism whose exceptional set is $E_i$ for 
$1\leq i\leq 3$, and $\Bl_{\sC^P}P$  is the blowup of $P$ along $\sC^P$. 
For $1\leq i\leq 3$, 
let $H_i\in\Pic(X)$ be the pull-back of $\sO_{\pr^i}(1)$ on $\pr^i$ and let 
$l_i\subset E_i$ be a curve contracted by $\sigma_i$. 
By \cite[\S III-3]{matsuki} and \cite[\S 10]{div-stability}, we have 
\begin{eqnarray*}
\Pic(X)&=&\Z[H_1]\oplus \Z[H_2]\oplus \Z[H_3], \\
\left(H_i\cdot l_j\right) &=& \delta_{ij} \quad (1\leq i,j\leq 3),\\
E_1&\sim& -H_1+H_2+H_3, \\
E_2&\sim& -H_2+H_3, \\
E_3&\sim& H_1+2H_2-H_3, \\
-K_X&\sim& H_1+H_2+H_3, \\
\Nef(X)&=&\R_{\geq 0}[H_1]+\R_{\geq 0}[H_2]+\R_{\geq 0}[H_3],\\
\Eff(X)&=&\R_{\geq 0}[H_1]+\R_{\geq 0}[E_1]+\R_{\geq 0}[E_2]+\R_{\geq 0}[E_3],
\end{eqnarray*}
where $\Nef(X)$ is the nef cone of $X$ and $\Eff(X)$ is the pseudo-effective 
cone of $X$. Note that $E_3\sim -E_1-3E_2+3H_3$. Thus the divisor 
$(\sigma^V\circ\sigma_1)_*E_3$ is the cubic surface with 
$\sC^P\subset(\sigma^V\circ\sigma_1)_*E_3$ and vanishes at $p^P$ of order 
at least $3$. 
We can easily check that 
\[\begin{matrix}
H_1^{\cdot 2}\equiv 0, & \left(H_1\cdot H_2^{\cdot 2}\right)=1, & 
\left(H_1\cdot H_2\cdot H_3\right)=2, & \left(H_1\cdot H_3^{\cdot 2}\right)=2,\\
\left(H_2^{\cdot 3}\right)=0, & \left(H_2^{\cdot 2}\cdot H_3\right)=1, &
\left(H_2\cdot H_3^{\cdot 2}\right)=1, & \left(H_3^{\cdot 3}\right)=1.
\end{matrix}\]

\begin{remark}\label{PCSD_remark}
\begin{enumerate}
\renewcommand{\theenumi}{\arabic{enumi}}
\renewcommand{\labelenumi}{(\theenumi)}
\item\label{PCSD_remark1}
By \cite{PCS}, for any such $X$, we have $\Aut^0(X)=\{1\}$. 
\item\label{PCSD_remark2}
By \cite[\S 10]{div-stability}, for any prime divisor $D$ \emph{on} $X$, 
the inequality 
\[
\frac{A_X(D)}{S_X(D)}>1
\]
holds. 
\end{enumerate}
\end{remark}

\begin{remark}\label{infl_remark}
There are two possibilities: 
\begin{enumerate}
\renewcommand{\theenumi}{\Alph{enumi}}
\renewcommand{\labelenumi}{(\theenumi)}
\item\label{infl_remark1}
The point $q^{\sC}\in\sC$ is not an inflection point of $\sC\subset\pr^2$. 
\item\label{infl_remark2}
The point $q^{\sC}\in\sC$ is an inflection point of $\sC\subset\pr^2$. 
\end{enumerate}
\end{remark}

In fact, we have explicit examples. The following examples 
provided by Cheltsov and Shramov, especially Example \ref{CS_example} 
\eqref{CS_example2}, are very important in this paper. 

\begin{example}[{Cheltsov and Shramov, see also \cite{FANO}}]\label{CS_example}
\begin{enumerate}
\renewcommand{\theenumi}{\Alph{enumi}}
\renewcommand{\labelenumi}{(\theenumi)}
\item\label{CS_example1}
Set $G:=\boldsymbol{\mu}_4$ acting $P=\pr^3_{xyzt}$ with 
\[
\begin{bmatrix}x\\y\\z\\t\end{bmatrix}
\mapsto
\begin{bmatrix}x\\ \sqrt{-1}y\\-z\\ -\sqrt{-1}t\end{bmatrix}.
\]
Let us set 
\[
\sC^P:=(xy+zt=0)\cap (x^2+z^2+yt=0), \quad\quad 
p^P:=\begin{bmatrix}0\\0\\0\\1\end{bmatrix}\in\sC^P.
\]
Then, since $p^P$ and $\sC^P$ are $G$-invariant, we have $G\subset\Aut(X)$. 
Moreover, the cubic curve $\sC\subset\pr^2_{xyz}$ is defined by the equation 
\[
\left(y(xy+zt)-z(x^2+z^2+yt)=\right)xy^2-x^2z-z^3=0,
\]
and $q^\sC$ corresponds to the point 
\[\begin{bmatrix}1\\0\\0\end{bmatrix}\in\pr^2_{xyz}.\]
Clearly, $q^\sC\in\sC$ is not an inflection point. 
Thus the $X$ satisfies Remark \ref{infl_remark} \eqref{infl_remark1}. 
\item\label{CS_example2}
Set $G:=\boldsymbol{\mu}_2 \times\boldsymbol{\mu}_3$ acting $P=\pr^3_{xyzt}$ with 
\[
\boldsymbol{\mu}_2\colon\begin{bmatrix}x\\y\\z\\t\end{bmatrix}
\mapsto
\begin{bmatrix}-x\\ y\\z\\ t\end{bmatrix},\quad\quad
\boldsymbol{\mu}_3\colon\begin{bmatrix}x\\y\\z\\t\end{bmatrix}
\mapsto
\begin{bmatrix}x\\ y\\\omega z\\ \omega^2 t\end{bmatrix},\]
where $\omega:=e^{2\pi\sqrt{-1}/3}$. Let us set 
\[
\sC^P:=(yz+t^2=0)\cap (x^2+y^2+zt=0), \quad\quad 
p^P:=\begin{bmatrix}0\\0\\1\\0\end{bmatrix}\in\sC^P.
\]
Then, since $p^P$ and $\sC^P$ are $G$-invariant, we have $G\subset\Aut(X)$. 
Moreover, the cubic curve $\sC\subset\pr^2_{xyt}$ is defined by the equation 
\[
\left(t(yz+t^2)-y(x^2+y^2+zt)=\right)t^3-x^2y-y^3=0,
\]
and $q^\sC$ corresponds to the point 
\[\begin{bmatrix}1\\0\\0\end{bmatrix}\in\pr^2_{xyt}.\]
Clearly, $q^\sC\in\sC$ is an inflection point. 
\end{enumerate}
\end{example}

\begin{remark}\label{G-inv_remark}
We see the action $G\curvearrowright X$ in Example \ref{CS_example} 
\eqref{CS_example2}. 
\begin{enumerate}
\renewcommand{\theenumi}{\arabic{enumi}}
\renewcommand{\labelenumi}{(\theenumi)}
\item\label{G-inv_remark1}
The set of $G$-invariant points in $P=\pr^3_{xyzt}$ is 
\[
\left\{
\begin{bmatrix}1\\0\\0\\0\end{bmatrix},\,\, 
\begin{bmatrix}0\\1\\0\\0\end{bmatrix},\,\, 
\begin{bmatrix}0\\0\\1\\0\end{bmatrix}=p^P,\,\, 
\begin{bmatrix}0\\0\\0\\1\end{bmatrix}
\right\}. 
\]
We note that 
\[
\begin{bmatrix}1\\0\\0\\0\end{bmatrix},\,\, 
\begin{bmatrix}0\\1\\0\\0\end{bmatrix},\,\, 
\begin{bmatrix}0\\0\\0\\1\end{bmatrix}\not\in\sC^P.
\]
Let $p_x$, $p_y$, $p_t\in X$ be the inverse images of 
\[
\begin{bmatrix}1\\0\\0\\0\end{bmatrix},\,\, 
\begin{bmatrix}0\\1\\0\\0\end{bmatrix},\,\, 
\begin{bmatrix}0\\0\\0\\1\end{bmatrix}\in P,
\]
respectively. Then, obviously, we have 
\[
\left\{p\in X\setminus E_2\,\,|\,\,\text{$G$-invariant}\right\}=\{p_x, p_y, p_t\}. 
\]
Moreover, we have $p_x\in E_3\setminus E_1$ and 
$p_y$, $p_t\in X\setminus (E_1\cup E_3)$. 
\item\label{G-inv_remark2}
Let $H_x$, $H_y$, $H_z$, $H_t\subset X$ be the strict transforms of the planes 
$(x=0)$, $(y=0)$, $(z=0)$, $(t=0)\subset P=\pr^3_{xyzt}$, respectively. 
Then we have: 
\begin{itemize}
\item
a prime divisor $D\in|H_2|$ is $G$-invariant if and only if $D=H_x$, $H_y$ or 
$H_t$, and 
\item
a prime divisor $D\in |H_3|$ is $G$-invariant if and only if $D=H_z$. 
\end{itemize}
\end{enumerate}
\end{remark}

\begin{remark}\label{H1_remark}
\begin{enumerate}
\renewcommand{\theenumi}{\arabic{enumi}}
\renewcommand{\labelenumi}{(\theenumi)}
\item\label{H1_remark1}
The members $Q\in|H_1|$ are characterized by the strict transforms of the 
quadrics $Q^P\subset P$ passing through $\sC^P$. Note that 
$Q^P\simeq\pr^1\times\pr^1$ or $\pr(1,1,2)$, 
and $\sC^P\cap\Sing Q^P=\emptyset$. 
\begin{itemize}
\item
If $Q$ is smooth, then $Q$ is isomorphic to the del Pezzo surface of degree $7$, 
and the union of the negative curves on $Q$ is equal to the set 
$(E_2\cup E_3)\cap Q$. 
\item
If $Q$ is singular, then $Q$ has exactly $2$ negative curves with the self 
intersection numbers $-1/2$ and $-1$. 
\end{itemize}
\item\label{H1_remark2}
For any closed point $p\in X$, there uniquely exists $Q_p\in |H_1|$ with $p\in Q_p$, 
since $|H_1|$ induces the del Pezzo fibration 
$\operatorname{pr}_1\circ\sigma_3\colon X\to\pr^1$. 
\end{enumerate}
\end{remark}

\begin{example}\label{CS-H1_example}
If $X$ is in Example \ref{CS_example} \eqref{CS_example2}, then 
we can write 
\[
|H_1|=\left\{Q_\lambda\,\,|\,\,\lambda\in\pr^1\right\}, 
\]
where $Q_\lambda$ is the strict transform of the quadric 
\[
Q_\lambda^P:=\left(yz+t^2-\lambda(x^2+y^2+zt)=0\right)\subset P=\pr^3_{xyzt}.
\]
We have the following properties:
\begin{itemize}
\item
The divisor $Q_\lambda$ is $G$-invariant if and only if $\lambda=0$ or $\infty$. 
\item
The divisor $Q_\lambda$ is singular if and only if 
\[
\lambda=0, \,\,1,\,\,\omega\,\,\text{ or }\,\,\omega^2
\]
(recall that $\omega=e^{2\pi\sqrt{-1}/3}$). 
\item
If $p=p_t$, then the $Q_p$ in Remark \ref{H1_remark} \eqref{H1_remark2} is equal to 
$Q_\infty$. If $p=p_x$ or $p_y$, then the $Q_p$ 
in Remark \ref{H1_remark} \eqref{H1_remark2} is equal to 
$Q_0$. Moreover, $Q_0$ is singular at $p_x$. Since $l^P=(y=t=0)\subset P$, 
the curve $l\subset Q_0$ is the negative curve with 
$\left(l^{\cdot 2}\right)=-1/2$. 
\end{itemize}
\end{example}

\section{Local $\delta$-invariants for general points}\label{local_general_section}

\begin{proposition}\label{5651_proposition}
Let $X$ be as in \S \ref{311_section}. Take a closet point 
$p\in X\setminus(E_2\cup E_3)$. If the divisor $Q_p\in |H_1|$ in 
Remark \ref{H1_remark} \eqref{H1_remark2} is smooth, then we have 
the inequality
\[
\delta_p(X)\geq\frac{56}{51}.
\]
\end{proposition}

\begin{proof}
We note that $\tau_X(Q_p)=2$. For any $u\in[0,2]$, let us set 
\begin{eqnarray*}
P(u)&:=&P_\sigma(X, -K_X-uQ_p),\\
N(u)&:=&N_\sigma(X, -K_X-uQ_p).
\end{eqnarray*}
By \S \ref{311_section}, we have the following: 
\begin{itemize}
\item
If $u\in[0,1]$, then 
\begin{eqnarray*}
N(u)&=&0,\\
P(u)&\sim_\R&(1-u)H_1+H_2+H_3.
\end{eqnarray*}
\item
If $u\in[1,2]$, then 
\begin{eqnarray*}
N(u)&=&(u-1)E_1,\\
P(u)&\sim_\R&(2-u)H_2+(2-u)H_3.
\end{eqnarray*}
\end{itemize}
Therefore we get 
\[
S_X(Q_p)=\frac{1}{28}\left(\int_0^1\left((1-u)H_1+H_2+H_3\right)^{\cdot 3}du
+\int_1^2\left((2-u)H_2+(2-u)H_3\right)^{\cdot 3}du\right)=\frac{11}{16}.
\]
Let $e_0$, $e_1$, $e_2\subset Q_p$ be mutually distinct $(-1)$-curves with 
$e_1\cap e_2=\emptyset$. Set $\sC^Q:=E_1|_{Q_p}$. Then $\sC^Q$ is 
a smooth curve with 
\[
\sC^Q\sim 3e_0+2e_1+2e_2. 
\]
Since $H_1|_{Q_p}\sim 0$, $H_2|_{Q_p}\sim e_0+e_1+e_2$ and 
$H_3|_{Q_p}\sim 2e_0+e_1+e_2$, we have the following: 
\begin{itemize}
\item
If $u\in[0,1]$, then 
\begin{eqnarray*}
N(u)|_{Q_p}&=&0,\\
P(u)|_{Q_p}&\sim_\R&3e_0+2e_1+2e_2.
\end{eqnarray*}
\item
If $u\in[1,2]$, then 
\begin{eqnarray*}
N(u)|_{Q_p}&=&(u-1)\sC^Q,\\
P(u)|_{Q_p}&\sim_\R&(2-u)(3e_0+2e_1+2e_2).
\end{eqnarray*}
\end{itemize}
Take the smooth curve $B\subset Q_p$ with $B\sim e_0+e_1$ and $p\in B$. 
After replacing $e_1$ and $e_2$ if necessary, we may assume that 
$\ord_p\left(\sC^Q|_B\right)\leq 1$. 
Let us set 
\begin{eqnarray*}
P(u,v)&:=&P_\sigma\left(Q_p, P(u)|_{Q_p}-vB\right), \\
N(u,v)&:=&N_\sigma\left(Q_p, P(u)|_{Q_p}-vB\right).
\end{eqnarray*}

\begin{itemize}
\item
Assume that $u\in[0,1]$. 
\begin{itemize}
\item
If $v\in[0,1]$, then we have 
\begin{eqnarray*}
N(u,v)&=&0, \\
P(u,v)&\sim_\R&(3-v)e_0+(2-v)e_1+2e_2,
\end{eqnarray*}
and $(P(u,v))^{\cdot 2}=7-4v$.
\item
If $v\in[1,2]$, then we have 
\begin{eqnarray*}
N(u,v)&=&(v-1)e_2, \\
P(u,v)&\sim_\R&(3-v)e_0+(2-v)e_1+(3-v)e_2,
\end{eqnarray*}
and $(P(u,v))^{\cdot 2}=(2-v)(4-v)$.
\end{itemize}
\item
Assume that $u\in[1,2]$. 
\begin{itemize}
\item
If $v\in[0,2-u]$, then we have 
\begin{eqnarray*}
N(u,v)&=&0, \\
P(u,v)&\sim_\R&(6-3u-v)e_0+(4-2u-v)e_1+(4-2u)e_2,
\end{eqnarray*}
and $(P(u,v))^{\cdot 2}=(2-u)(14-7u-4v)$.
\item
If $v\in[2-u,4-2u]$, then we have 
\begin{eqnarray*}
N(u,v)&=&(-2+u+v)e_2, \\
P(u,v)&\sim_\R&(6-3u-v)e_0+(4-2u-v)e_1+(6-3u-v)e_2,
\end{eqnarray*}
and $(P(u,v))^{\cdot 2}=(4-2u-v)(8-4u-v)$.
\end{itemize}
\end{itemize}
Hence, by Theorem \ref{AZ1_thm}, we get
\begin{eqnarray*}
&&S\left(V^{Q_p}_{\bullet,\bullet};B\right)\\
&=&\frac{3}{28}\Biggl(
\int_0^1\left(\int_0^1(7-4v)dv+\int_1^2(2-v)(4-v)dv\right)du\\
&&+\int_1^2\left(\int_0^{2-u}(2-u)(14-7u-4v)dv
+\int_{2-u}^{4-2u}(4-2u-v)(8-4u-v)dv\right)du\Biggr)\\
&=&\frac{95}{112}.
\end{eqnarray*}
Moreover, by Theorem \ref{AZ2_thm}, we get 
\begin{eqnarray*}
S\left(W^{Q_p, B}_{\bullet,\bullet,\bullet};p\right)
&\leq&\frac{3}{28}\Biggl(\int_0^1\left(
\int_0^1 2^2dv+\int_1^2 (3-v)^2dv\right)du\\
&&+\int_1^2\left(\int_0^{2-u}(4-2u)^2dv
+\int_{2-u}^{4-2u}(6-3u-v)^2dv\right)du\Biggr)\\
&&+\frac{6}{28}\left(\int_1^2
\left(\int_0^{2-u}(u-1)(4-2u)dv+\int_{2-u}^{4-2u}(u-1)(6-3u-v)dv\right)du\right)\\
&=&\frac{95}{112}+\frac{1}{16}=\frac{51}{56}.
\end{eqnarray*}
Therefore, we get the inequality 
\begin{eqnarray*}
\delta_p(X)&\geq&\min\left\{\frac{A_X(Q_p)}{S_X(Q_p)},\,\,\, 
\frac{A_{Q_p}(B)}{S\left(V^{Q_p}_{\bullet,\bullet}; B\right)},\,\,\, 
\frac{A_B(p)}{S\left(W^{Q_p,B}_{\bullet,\bullet,\bullet};p\right)}\right\}
\geq\min\left\{\frac{16}{11},\,\,\frac{112}{95},\,\,\frac{56}{51}\right\}=\frac{56}{51}
\end{eqnarray*}
by Corollary \ref{reduction_corollary}.
\end{proof}

\begin{corollary}\label{5651_corollary}
Let $G\curvearrowright X$ be as in Example \ref{CS_example} \eqref{CS_example2}. 
Then we have 
\[
\delta_{p_t}(X)\geq \frac{56}{51},
\]
where $p_t\in X$ is given in Remark \ref{G-inv_remark}. 
\end{corollary}

\begin{proof}
Trivial from Remark \ref{G-inv_remark} \eqref{G-inv_remark1} and 
Example \ref{CS-H1_example}.
\end{proof}

\begin{proposition}\label{112107_proposition}
Let $X$ be as in \S \ref{311_section}. Take a closed point 
$p_0\in X\setminus(E_1\cup E_2)$. Assume that there is a smooth member 
$S\in |H_3|$ with $-K_S$ ample such that 
\begin{itemize}
\item
$p_0\in S$, and 
\item
any $(-1)$-curve in $S$ does not pass through $p_0$.
\end{itemize}
Then we have the inequality 
\[
\delta_{p_0}(X)\geq \frac{112}{107}. 
\]
\end{proposition}

\begin{proof}
Let $S^P\subset P$ be the strict transform of $S$ on $P$. Then $S^P$ is a plane 
with $p^P\not\in S^P$,
$\sC^P\cap S^P=\{p_1,\dots, p_4\}$, 
and $S$ is the blowup of $S^P$ along the points $p_1,\dots, p_4$ 
from the assumption. 
Let $\varepsilon\colon\tilde{S}\to S^P$ be the composition of 
the blowup $\varepsilon_0\colon\tilde{S}\to S$ at $p_0\in S$ and the natural morphism 
$S\to S^P$. From the assumption, $\tilde{S}$ is a smooth del Pezzo surface of 
degree $4$. Let $e_0,\dots,e_4\subset\tilde{S}$ be the $\varepsilon$-exceptional 
curves with $\varepsilon(e_i)=p_i$, let $l_{ij}\subset\tilde{S}$ $(0\leq i<j\leq 4)$ be 
the strict transform of the line passing through $p_i$ and $p_j$, and let 
$\tilde{C}\subset\tilde{S}$ be the strict transform of the conic passing through 
$p_0,\dots,p_4$. Moreover, let $\tilde{\sC}\subset\tilde{S}$ be the 
strict transform of $E_3|_S\subset S$. 
Note that $\tau_X(S)=3/2$. For $u\in[0,3/2]$, let us set 
\begin{eqnarray*}
P(u)&:=&P_\sigma(X, -K_X-uS),\\
N(u)&:=&N_\sigma(X, -K_X-uS).
\end{eqnarray*}
Then, 
\begin{itemize}
\item
if $u\in[0,1]$, then 
\begin{eqnarray*}
N(u)&=&0,\\
P(u)&\sim_\R&H_1+H_2+(1-u)H_3,
\end{eqnarray*}
\item
If $u\in[1,3/2]$, then 
\begin{eqnarray*}
N(u)&=&(u-1)E_3,\\
P(u)&\sim_\R&(2-u)H_1+(3-2u)H_2.
\end{eqnarray*}
\end{itemize}
Thus we get
\begin{eqnarray*}
S_X(S)&=&\frac{1}{28}\left(\int_0^1\left(H_1+H_2+(1-u)H_3\right)^{\cdot 3}du
+\int_1^{\frac{3}{2}}\left((2-u)H_1+(3-2u)H_2\right)^{\cdot 3}du\right)\\
&=&\frac{227}{448}.
\end{eqnarray*}
Let us set 
\begin{eqnarray*}
P(u,v)&:=&P_\sigma\left(\tilde{S}, \varepsilon_0^*\left(P(u)|_{S}\right)-ve_0\right), \\
N(u,v)&:=&N_\sigma\left(\tilde{S}, \varepsilon_0^*\left(P(u)|_{S}\right)-ve_0\right).
\end{eqnarray*}
Note that 
\[
\varepsilon_0^*\left(P(u)|_{S}\right)\sim_\R
\begin{cases}
\varepsilon^*\sO(4-u)-(e_1+\cdots+e_4) & \text{if }u\in[0,1],\\ 
\varepsilon^*\sO(7-4u)-(2-u)(e_1+\cdots+e_4) & \text{if }u\in\left[1,\frac{3}{2}\right].
\end{cases}
\]

\begin{itemize}
\item
Assume that $u\in[0,1]$. 
\begin{itemize}
\item
If $v\in[0,3-u]$, then 
\begin{eqnarray*}
N(u,v)&=&0, \\
P(u,v)&\sim_\R&\varepsilon^*\sO(4-u)-ve_0-(e_1+\cdots+e_4),
\end{eqnarray*}
and $(P(u,v))^{\cdot 2}=(4-u)^2-4-v^2$. 
\item
If $v\in[3-u,4-2u]$, then 
\begin{eqnarray*}
N(u,v)&=&(-3+u+v)(l_{01}+\cdots+l_{04}), \\
P(u,v)&\sim_\R&\varepsilon^*\sO(16-5u-4v)-(12-4u-3v)e_0
-(4-u-v)(e_1+\cdots+e_4),
\end{eqnarray*}
and $(P(u,v))^{\cdot 2}=(4-u-v)(12-5u-3v)$. 
\item
If $v\in\left[4-2u,\frac{8-3u}{2}\right]$, then 
\begin{eqnarray*}
N(u,v)&=&(-3+u+v)(l_{01}+\cdots+l_{04})+(-4+2u+v)\tilde{C}, \\
P(u,v)&\sim_\R&\varepsilon^*\sO(24-9u-6v)-(16-6u-4v)e_0
-(8-3u-2v)(e_1+\cdots+e_4),
\end{eqnarray*}
and $(P(u,v))^{\cdot 2}=(8-3u-2v)^2$. 
\end{itemize}
\item
Assume that $u\in\left[1,\frac{3}{2}\right]$. 
\begin{itemize}
\item
If $v\in[0,6-4u]$, then 
\begin{eqnarray*}
N(u,v)&=&0, \\
P(u,v)&\sim_\R&\varepsilon^*\sO(7-4u)-ve_0-(2-u)(e_1+\cdots+e_4),
\end{eqnarray*}
and $(P(u,v))^{\cdot 2}=(7-4u)^2-4(2-u)^2-v^2$. 
\item
If $v\in[6-4u,5-3u]$, then 
\begin{eqnarray*}
N(u,v)&=&(-6+4u+v)\tilde{C}, \\
P(u,v)&\sim_\R&\varepsilon^*\sO(19-12u-2v)-(6-4u)e_0
-(8-5u-v)(e_1+\cdots+e_4),
\end{eqnarray*}
and $(P(u,v))^{\cdot 2}=(3-2u)(23-14u-4v)$. 
\item
If $v\in\left[5-3u,\frac{13-8u}{2}\right]$, then 
\begin{eqnarray*}
N(u,v)&=&(-5+3u+v)(l_{01}+\cdots+l_{04})+(-6+4u+v)\tilde{C}, \\
P(u,v)&\sim_\R&\varepsilon^*\sO(39-24u-6v)-(26-16u-4v)e_0
-(13-8u-2v)(e_1+\cdots+e_4),
\end{eqnarray*}
and $(P(u,v))^{\cdot 2}=(13-8u-2v)^2$. 
\end{itemize}
\end{itemize}
Note that $\ord_{e_0}(E_3|_S)\leq 1$. Therefore, we have 
\begin{eqnarray*}
&&S\left(V^S_{\bullet,\bullet};e_0\right)\\
&\leq&
\frac{3}{28}\Biggl(
\int_0^1\biggl(\int_0^{3-u}\left((4-u)^2-4-v^2\right)dv\\
&&+\int_{3-u}^{4-2u}(4-u-v)(12-5u-3v)dv
+\int_{4-2u}^{\frac{8-3u}{2}}(8-3u-2v)^2dv\biggr)du\\
&+&\int_1^{\frac{3}{2}}\biggl((u-1)\left((7-4u)^2-4(2-u)^2\right)
+\int_0^{6-4u}\left((7-4u)^2-4(2-u)^2-v^2\right)dv\\
&&+\int_{6-4u}^{5-3u}(3-2u)(23-14u-4v)dv
+\int_{5-3u}^{\frac{13-8u}{2}}(13-8u-2v)^2dv\biggr)du\Biggr)
=\frac{107}{56}.
\end{eqnarray*}
Moreover, for any $p\in e_0$, since $l_{01}|_{e_0},\dots,l_{04}|_{e_0}$ are mutually 
distinct, and $\ord_p\left(l_{0i}|_{e_0}\right)$, $\ord_p\left(\tilde{C}|_{e_0}\right)$, 
$\ord_p\left(\tilde{\sC}|_{e_0}\right)\leq 1$, 
we have 
\begin{eqnarray*}
&&F_p\left(W^{\tilde{S},e_0}_{\bullet,\bullet,\bullet}\right)\\
&\leq&\frac{6}{28}\Biggl(\int_0^1\biggl(\int_{3-u}^{4-2u}(12-4u-3v)(-3+u+v)dv\\
&&+\int_{4-2u}^{\frac{8-3u}{2}}(16-6u-4v)(-3+u+v)dv\biggr)du\\
&&+\int_1^{\frac{3}{2}}\int_{5-3u}^{\frac{13-8u}{2}}(26-16u-4v)(-5+3u+v)dvdu\Biggr)\\
&+&\frac{6}{28}\Biggl(\int_0^1\int_{4-2u}^{\frac{8-3u}{2}}(16-6u-4v)(-4+2u+v)dvdu\\
&&+\int_1^{\frac{3}{2}}\biggl(\int_{6-4u}^{5-3u}(6-4u)(-6+4u+v)dv
+\int_{5-3u}^{\frac{13-8u}{2}}(26-16u-4v)(-6+4u+v)dv\biggr)du\Biggr)\\
&+&\frac{6}{28}\Biggl(\int_1^{\frac{3}{2}}\biggl(\int_0^{6-4u}v(u-1)dv
+\int_{6-4u}^{5-3u}(6-4u)(u-1)dv\\
&&+\int_{5-3u}^{\frac{13-8u}{2}}(26-16u-4v)(u-1)dv\biggr)du\Biggr)
=\frac{27}{448}+\frac{5}{448}+\frac{1}{64}=\frac{39}{448}.
\end{eqnarray*}
Therefore, we get 
\begin{eqnarray*}
&&S\left(W^{\tilde{S},e_0}_{\bullet,\bullet,\bullet};p\right)\\
&\leq&\frac{39}{448}+\frac{3}{28}\Biggl(\int_0^1\biggl(\int_0^{3-u}v^2dv
+\int_{3-u}^{4-2u}(12-4u-3v)^2dv+\int_{4-2u}^{\frac{8-3u}{2}}(16-6u-4v)^2dv\biggr)du\\
&&+\int_1^{\frac{3}{2}}\int_0^{6-4u}v^2dv+\int_{6-4u}^{5-3u}(6-4u)^2dv
+\int_{5-3u}^{\frac{13-8u}{2}}(26-16u-4v)^2dv\biggr)du\Biggr)
=\frac{407}{448}.
\end{eqnarray*}
As a consequence, we get the inequality
\begin{eqnarray*}
\delta_{p_0}(X)\geq \min\left\{\frac{A_X(S)}{S_X(S)},\,\,\,
\frac{A_S(e_0)}{S\left(V^S_{\bullet,\bullet};e_0\right)},\,\,\,
\inf_{p\in e_0}\frac{A_{e_0}(p)}{S\left(W^{\tilde{S},e_0}_{\bullet,\bullet,\bullet};
p\right)}\right\}\geq\min\left\{\frac{448}{227},\,\,\frac{112}{107},\,\,
\frac{448}{407}\right\}=\frac{112}{107}
\end{eqnarray*}
by Corollary \ref{reduction_corollary}.
\end{proof}

\begin{corollary}\label{112107_corollary}
Let $G\curvearrowright X$ be as in Example \ref{CS_example} \eqref{CS_example2}. 
If $p_0=p_y$ or $p_0\in l\setminus\{p_x,q\}$, then we have 
\[
\delta_{p_0}(X)\geq\frac{112}{107}.
\]
\end{corollary}

\begin{proof}
Assume that $p_0=p_y$. Take the plane 
\[
S^P:=(t-\sqrt[3]{2}x=0)\subset P. 
\]
Then, under the natural isomorphism $S^P\simeq\pr^2_{xyz}$, the point $p_0$ 
corresponds to the point 
\[
\begin{bmatrix}0\\1\\0\end{bmatrix}\in\pr^2_{xyz},
\]
and the intersection $S^P\cap \sC^P$ is determined by the equations 
\[
yz+\sqrt[3]{4}x^2=0, \quad\quad x^2+y^2+\sqrt[3]{2}xz=0.
\]
Hence the points $p_1,\dots,p_4\in\pr^2$ in Proposition \ref{112107_proposition} 
corresponds to the points 
\[
\begin{bmatrix}0\\0\\1\end{bmatrix},\quad
\begin{bmatrix}1\\1\\-\sqrt[3]{4}\end{bmatrix},\quad
\begin{bmatrix}4\\2(-1+\sqrt{-7})\\\sqrt[3]{4}(1+\sqrt{-7})\end{bmatrix},\quad
\begin{bmatrix}4\\2(-1-\sqrt{-7})\\\sqrt[3]{4}(1-\sqrt{-7})\end{bmatrix}.
\]
We can easily check that no $3$ points among $p_0,\dots,p_4$ are collinear. 

Now assume that $p_0\in l\setminus\{p_x,q\}$.
There exists $c\in\C^*$ such that we can write 
\[
p_0=\begin{bmatrix}1\\0\\c\\0\end{bmatrix}\in\pr^2_{xyzt}.
\]
Take a general plane 
\[
S^P=(z=cx+ay+bt)\subset P
\]
passing through $p_0$ (for $a$, $b\in \C^*$ general). 
Then, under the natural isomorphism $S^P\simeq\pr^2_{xyt}$, the point 
$p_0$ corresponds to the point 
\[
\begin{bmatrix}1\\0\\0\end{bmatrix}\in\pr^2_{xyt},
\]
and the intersection $S^P\cap \sC^P$ is determined by the equations 
\[
y(cx+ay+bt)+t^2=0, \quad\quad x^2+y^2+(cx+ay+bt)t=0.
\]
By Bertini's theorem, $S^P\cap \sC^P$ consists of $4$ distinct points 
$\{p_1,\dots,p_4\}$. We can write 
\[
p_i=\begin{bmatrix}x_i\\1\\t_i\end{bmatrix}\quad(1\leq i\leq 4), 
\]
where $t_1,\dots,t_4$ are the roots of the polynomial 
\[
f(t):=t^4+(2b-c^2)t^3+(2a+b^2)t^2+2abt+a^2+c^2.
\]
Since the discriminant 
\begin{eqnarray*}
&&c^4(256a^2+256a^5-768a^3b-128ab^2-128a^4b^2
+640a^2b^3+16b^4+16a^3b^4-176ab^5+16b^7\\
&&+256c^2+544a^3c^2-768abc^2+144a^4bc^2-648a^2b^2c^2
+288b^3c^2-4a^3b^3c^2
+192ab^4c^2\\
&&-4b^6c^2
+288ac^4-27a^4c^4+360a^2bc^4
-504b^2c^4-36ab^3c^4-54a^2c^6+216bc^6-27c^8)
\end{eqnarray*}
of $f(t)$ is nonzero for general $a$, $b\in\C^*$, the values $t_1,\dots,t_4$ are 
mutually distinct. Hence 
$p_0$, $p_i$, $p_j$ for $1\leq i<j\leq 4$ are not collinear. 
Moreover, since $\{p_0,\dots,p_4\}\subset Q_0^P\cap S^P$ and 
$Q_0^P\cap S^P$ is a smooth conic, no $3$ points 
among $\{p_0,\dots,p_4\}$ are not collinear. 
\end{proof}

\begin{remark}
Let $G\curvearrowright X$ be as in Example \ref{CS_example} \eqref{CS_example2}. 
For any smooth $S\in|H_3|$ with $p_x\in S$, the point $p_x$ is contained in 
a $(-1)$-curve in $S$. Indeed, the intersection $S^P\cap Q_0^P$ must be the 
union of $2$ lines passing through $p_x$. Note that, in $Q_0^P$, the curve $\sC^P$ 
tangents to $l^P$ at $p^P$. We will discuss the case in \S \ref{special-I_section}. 
\end{remark}

\section{Local $\delta$-invariants for points in $E_2$}\label{local_E2_section}

In this section, we prove the following: 

\begin{proposition}\label{112109_proposition}
Let $X$ be as in \S \ref{311_section}. Take any closed point 
$p\in E_2\setminus\{q\}$. Then we have the inequality 
\[
\delta_p(X)\geq\frac{112}{109}.
\]
\end{proposition}

\begin{proof}
The divisor $E_2$ is isomorphic to the Hirzebruch surface 
$\pr_{\pr^1}(\sO\oplus\sO(1))$. Let $s\subset E_2$ be the $(-1)$-curve 
and let $l_2\subset E_2$ be the fiber of $E_2/\pr^1$ with $p\in l_2$. 
Note that $q\in s$. Let us set $\sC^E:=E_3|_{E_2}$. Then $\sC^E$ is a smooth curve 
with $\sC^E\sim 2s+3l_2$. Note that $\tau_X(E_2)=2$. 

For any $u\in[0,2]$, let us set 
\begin{eqnarray*}
P(u)&:=&P_\sigma(X, -K_X-uE_2),\\
N(u)&:=&N_\sigma(X, -K_X-uE_2).
\end{eqnarray*}
Then we have the following: 
\begin{itemize}
\item
If $u\in[0,1]$, then 
\begin{eqnarray*}
N(u)&=&0,\\
P(u)&\sim_\R&H_1+(1+u)H_2+(1-u)H_3.
\end{eqnarray*}
\item
If $u\in[1,2]$, then 
\begin{eqnarray*}
N(u)&=&(u-1)E_3,\\
P(u)&\sim_\R&(2-u)H_1+(3-u)H_2.
\end{eqnarray*}
\end{itemize}
Therefore we get 
\begin{eqnarray*}
S_X(E_2)&=&\frac{1}{28}\Biggl(\int_0^1\left(H_1+(1+u)H_2+(1-u)H_3\right)^{\cdot 3}du\\
&&+\int_1^2\left((2-u)H_1+(3-u)H_2\right)^{\cdot 3}du\Biggr)=\frac{51}{56}.
\end{eqnarray*}
Note that, 
\begin{itemize}
\item
if $u\in[0,1]$, then 
\begin{eqnarray*}
N(u)|_{E_2}&=&0,\\
P(u)|_{E_2}&\sim_\R&(1+u)s+(2+u)l_2,
\end{eqnarray*}
\item
if $u\in[1,2]$, then 
\begin{eqnarray*}
N(u)|_{E_2}&=&(u-1)\sC^E,\\
P(u)|_{E_2}&\sim_\R&(3-u)s+(5-2u)l_2.
\end{eqnarray*}
\end{itemize}
Let us set 
\begin{eqnarray*}
P(u,v)&:=&P_\sigma\left(E_2, P(u)|_{E_2}-vl_2\right), \\
N(u,v)&:=&N_\sigma\left(E_2, P(u)|_{E_2}-vl_2\right).
\end{eqnarray*}

\begin{itemize}
\item
Assume that $u\in[0,1]$. 
\begin{itemize}
\item
If $v\in[0,1]$, then we have 
\begin{eqnarray*}
N(u,v)&=&0, \\
P(u,v)&\sim_\R&(1+u)s+(2+u-v)l_2,
\end{eqnarray*}
and $(P(u,v))^{\cdot 2}=(1+u)(3+u-2v)$.
\item
If $v\in[1,2+u]$, then we have 
\begin{eqnarray*}
N(u,v)&=&(v-1)s, \\
P(u,v)&\sim_\R&(2+u-v)s+(2+u-v)l_2,
\end{eqnarray*}
and $(P(u,v))^{\cdot 2}=(2+u-v)^2$.
\end{itemize}
\item
Assume that $u\in[1,2]$. 
\begin{itemize}
\item
If $v\in[0,2-u]$, then we have 
\begin{eqnarray*}
N(u,v)&=&0, \\
P(u,v)&\sim_\R&(3-u)s+(5-2u-v)l_2,
\end{eqnarray*}
and $(P(u,v))^{\cdot 2}=(3-u)(7-3u-2v)$.
\item
If $v\in[2-u,5-2u]$, then we have 
\begin{eqnarray*}
N(u,v)&=&(-2+u+v)s, \\
P(u,v)&\sim_\R&(5-2u-v)s+(5-2u-v)l_2,
\end{eqnarray*}
and $(P(u,v))^{\cdot 2}=(5-2u-v)^2$.
\end{itemize}
\end{itemize}
Hence we get
\begin{eqnarray*}
S\left(V^{E_2}_{\bullet,\bullet};l_2\right)
&=&\frac{3}{28}\Biggl(
\int_0^1\left(\int_0^1(1+u)(3+u-2v)dv+\int_1^{2+u}(2+u-v)^2dv\right)du\\
&&+\int_1^2\left(\int_0^{2-u}(3-u)(7-3u-2v)dv
+\int_{2-u}^{5-2u}(5-2u-v)^2dv\right)du\Biggr)\\
&=&\frac{25}{28}.
\end{eqnarray*}
Moreover, we have 
\begin{eqnarray*}
F_{s|_{l_2}}\left(W^{E_2,l_2}_{\bullet,\bullet,\bullet}\right)&=&\frac{6}{28}
\Biggl(\int_0^1\int_1^{2+u}(2+u-v)(v-1)dvdu\\
&&+\int_1^2\int_{2-u}^{5-2u}(5-2u-v)(-2+u+v)dvdu\Biggr)=\frac{15}{56}, \\
F_{\red(\sC^E|_{l_2})}\left(W^{E_2,l_2}_{\bullet,\bullet,\bullet}\right)&=&
\ord_p\left(\sC^E|_{l_2}\right)\cdot\frac{6}{28}
\Biggl(\int_1^2\biggl(\int_0^{2-u}(3-u)(u-1)dv\\
&&+\int_{2-u}^{5-2u}(5-2u-v)(u-1)dv\biggr)du\Biggr)=\ord_p\left(\sC^E|_{l_2}\right)
\cdot\frac{17}{112}.
\end{eqnarray*}
Since $\ord_p\left(\sC^E|_{l_2}\right)\leq 2$ and $p\neq q$, we have
\[
F_p\left(W^{E_2,l_2}_{\bullet,\bullet,\bullet}\right)
\leq\max\left\{\frac{15}{56},\,\,2\cdot\frac{17}{112}\right\}=\frac{17}{56}.
\]
Therefore we get 
\begin{eqnarray*}
S\left(W^{E_2, l_2}_{\bullet,\bullet,\bullet};p\right)
&\leq&\frac{17}{56}
+\frac{3}{28}\Biggl(\int_0^1\left(
\int_0^1 (1+u)^2dv+\int_1^{2+u} (2+u-v)^2dv\right)du\\
&&+\int_1^2\left(\int_0^{2-u}(3-u)^2dv
+\int_{2-u}^{5-2u}(5-2u-v)^2dv\right)du\Biggr)\\
&=&\frac{17}{56}+\frac{75}{112}=\frac{109}{112}.
\end{eqnarray*}
Therefore, we get the inequality 
\begin{eqnarray*}
\delta_p(X)&\geq&\min\left\{\frac{A_X(E_2)}{S_X(E_2)},\,\,\, 
\frac{A_{E_2}(l_2)}{S\left(V^{E_2}_{\bullet,\bullet}; l_2\right)},\,\,\, 
\frac{A_{l_2}(p)}{S\left(W^{E_2,l_2}_{\bullet,\bullet,\bullet};p\right)}\right\}
=\min\left\{\frac{56}{51},\,\,\frac{28}{25},\,\,\frac{112}{109}\right\}=\frac{112}{109}
\end{eqnarray*}
by Corollary \ref{reduction_corollary}.
\end{proof}

If $p=q$, then the value $\delta_q\left(V^{E_2}_{\bullet,\bullet}\right)$ cannot be big. 
We do not use the following Proposition \ref{112113_proposition} later. 
However, we can recognize the importance of 
the arguments in \S \ref{special-II_section} from 
Proposition \ref{112113_proposition} \eqref{112113_proposition2}. 

\begin{proposition}\label{112113_proposition}
Let $X$ be as in \S \ref{311_section}. 
\begin{enumerate}
\renewcommand{\theenumi}{\arabic{enumi}}
\renewcommand{\labelenumi}{(\theenumi)}
\item\label{112113_proposition1}
If $X$ satisfies Remark \ref{infl_remark} \eqref{infl_remark1}, then we have 
$\delta_q\left(V^{E_2}_{\bullet,\bullet}\right)=\frac{112}{111}$.
\item\label{112113_proposition2}
If $X$ satisfies Remark \ref{infl_remark} \eqref{infl_remark2}, then we have 
$\delta_q\left(V^{E_2}_{\bullet,\bullet}\right)=\frac{112}{113}$.
\end{enumerate}
\end{proposition}

\begin{proof}
Let $\varepsilon\colon\tilde{E}_2\to E_2$ be the blowup at $q$ and let 
$\tilde{e}_1\subset\tilde{E}_2$ be the exceptional divisor. 
There are exactly $3$ negative curves on $\tilde{E}_2$: 
\begin{itemize}
\item
the strict transform $\tilde{l}_2$ of the fiber $l_2$ of $E_2/\pr^1$ passing through 
the point $q$, 
\item
the curve $\tilde{e}_1$, and 
\item
the strict transform $\tilde{s}$ of the $(-1)$-curve $s\subset E_2$. 
\end{itemize}
The intersection form of $\tilde{l}_2$, $\tilde{e}_1$ and $\tilde{s}$ on $\tilde{E}_2$ is 
given by the symmetric matrix 
\[\begin{pmatrix}
-1&1&0\\
1&-1&1\\
0&1&-2
\end{pmatrix}.\]
Let $\tilde{\sC}^E\subset\tilde{E}_2$ be the strict transform of $\sC^E:=E_3|_{E_2}$. 
Note that $\tilde{\sC}^E\sim 3\tilde{l}_2+4\tilde{e}_1+2\tilde{s}$. 

\eqref{112113_proposition1}
In this case, 
\[
\tilde{l}_2|_{\tilde{e}_1},\quad\tilde{\sC}^E|_{\tilde{e}_1},\quad\tilde{s}|_{\tilde{e}_1}
\]
are mutually distinct reduced points. Let us denote them by 
$p_l$, $p_\sC$, $p_s$, respectively. 
Let $P(u)$, $N(u)$ $(u\in[0,2])$ be as in the proof of 
Proposition \ref{112109_proposition}. Let us set 
\begin{eqnarray*}
P(u,v)&:=&P_\sigma\left(\tilde{E}_2, \varepsilon^*\left(P(u)|_{E_2}\right)
-v\tilde{e}_1\right), \\
N(u,v)&:=&N_\sigma\left(\tilde{E}_2, \varepsilon^*\left(P(u)|_{E_2}\right)
-v\tilde{e}_1\right).
\end{eqnarray*}

\begin{itemize}
\item
Assume that $u\in[0,1]$. 
\begin{itemize}
\item
If $v\in[0,1]$, then we have 
\begin{eqnarray*}
N(u,v)&=&0, \\
P(u,v)&\sim_\R&(2+u)\tilde{l}_2+(3+2u-v)\tilde{e}_1+(1+u)\tilde{s},
\end{eqnarray*}
and $(P(u,v))^{\cdot 2}=3+4u+u^2-v^2$.
\item
If $v\in[1,1+u]$, then we have 
\begin{eqnarray*}
N(u,v)&=&\frac{v-1}{2}\tilde{s}, \\
P(u,v)&\sim_\R&(2+u)\tilde{l}_2+(3+2u-v)\tilde{e}_1+\frac{3+2u-v}{2}\tilde{s},
\end{eqnarray*}
and 
\[
(P(u,v))^{\cdot 2}=\frac{7}{2}+4u+u^2-v-\frac{v^2}{2}.
\]
\item
If $v\in[1+u,3+2u]$, then we have 
\begin{eqnarray*}
N(u,v)&=&(-1-u+v)\tilde{l}_2+\frac{v-1}{2}\tilde{s}, \\
P(u,v)&\sim_\R&(3+2u-v)\tilde{l}_2+(3+2u-v)\tilde{e}_1+\frac{3+2u-v}{2}\tilde{s},
\end{eqnarray*}
and 
\[
(P(u,v))^{\cdot 2}=\frac{1}{2}(3+2u-v)^2.
\]
\end{itemize}
\item
Assume that $u\in[1,2]$. 
\begin{itemize}
\item
If $v\in[0,2-u]$, then we have 
\begin{eqnarray*}
N(u,v)&=&0, \\
P(u,v)&\sim_\R&(5-2u)\tilde{l}_2+(8-3u-v)\tilde{e}_1+(3-u)\tilde{s},
\end{eqnarray*}
and $(P(u,v))^{\cdot 2}=21-16u+3u^2-v^2$.
\item
If $v\in[2-u,3-u]$, then we have 
\begin{eqnarray*}
N(u,v)&=&\frac{-2+u+v}{2}\tilde{s}, \\
P(u,v)&\sim_\R&(5-2u)\tilde{l}_2+(8-3u-v)\tilde{e}_1+\frac{8-3u-v}{2}\tilde{s},
\end{eqnarray*}
and 
\[
(P(u,v))^{\cdot 2}=23+\frac{7u^2}{2}+u(-18+v)-2v-\frac{v^2}{2}.
\]
\item
If $v\in[3-u,8-3u]$, then we have 
\begin{eqnarray*}
N(u,v)&=&(-3+u+v)\tilde{l}_2+\frac{-2+u+v}{2}\tilde{s}, \\
P(u,v)&\sim_\R&(8-3u-v)\tilde{l}_2+(8-3u-v)\tilde{e}_1+\frac{8-3u-v}{2}\tilde{s},
\end{eqnarray*}
and 
\[
(P(u,v))^{\cdot 2}=\frac{1}{2}(8-3u-v)^2.
\]
\end{itemize}
\end{itemize}
Hence we get
\begin{eqnarray*}
&&S\left(V^{E_2}_{\bullet,\bullet};\tilde{e}_1\right)\\
&=&\frac{3}{28}\Biggl(
\int_0^1\biggl(\int_0^1(3+4u+u^2-v^2)dv\\
&&+\int_1^{1+u}\left(\frac{7}{2}+4u+u^2-v-\frac{v^2}{2}\right)dv
+\int_{1+u}^{3+2u}\frac{1}{2}(3+2u-v)^2dv\biggr)du\\
&+&\int_1^2\biggl((u+1)(21-16u+3u^2)+\int_0^{2-u}(21-16u+3u^2-v^2)dv\\
&&+\int_{2-u}^{3-u}\left(23+\frac{7u^2}{2}+u(-18+v)-2v-\frac{v^2}{2}\right)dv
+\int_{3-u}^{8-3u}\frac{1}{2}(8-3u-v)^2dv\biggr)du\Biggr)\\
&=&\frac{111}{56}.
\end{eqnarray*}
This implies that 
\[
\delta_q\left(V^{E_2}_{\bullet,\bullet}\right)\leq
\frac{A_{E_2}\left(\tilde{e}_1\right)}{S\left(V^{E_2}_{\bullet,\bullet};\tilde{e}_1\right)}
=\frac{112}{111}.
\]
Moreover, we have 
\begin{eqnarray*}
F_{p_l}\left(W^{\tilde{E}_2,\tilde{e}_1}_{\bullet,\bullet,\bullet}\right)
&=&\frac{6}{28}\Biggl(\int_0^1\int_{1+u}^{3+2u}\frac{3+2u-v}{2}(-1-u+v)dvdu\\
&&+\int_1^2\int_{3-u}^{8-3u}\frac{8-3u-v}{2}(-3+u+v)dvdu\Biggr)=\frac{15}{32}, 
\end{eqnarray*}
\begin{eqnarray*}
F_{p_\sC}\left(W^{\tilde{E}_2,\tilde{e}_1}_{\bullet,\bullet,\bullet}\right)
&=&\frac{6}{28}\Biggl(\int_1^2\biggl(\int_0^{2-u}v(u-1)dv
+\int_{2-u}^{3-u}\frac{2-u+v}{2}(u-1)dv\\
&&+\int_{3-u}^{8-3u}\frac{8-3u-v}{2}(u-1)dv\biggr)du\Biggr)=\frac{17}{112}, 
\end{eqnarray*}
\begin{eqnarray*}
F_{p_s}\left(W^{\tilde{E}_2,\tilde{e}_1}_{\bullet,\bullet,\bullet}\right)
&=&\frac{6}{28}\Biggl(\int_0^1\biggl(\int_1^{1+u}\frac{1+v}{2}\cdot\frac{v-1}{2}dv
+\int_{1+u}^{3+2u}\frac{3+2u-v}{2}\cdot\frac{v-1}{2}dv\biggr)du\\
&&+\int_1^2\biggl(\int_{2-u}^{3-u}\frac{2-u+v}{2}\cdot\frac{-2+u+v}{2}dv\\
&&+\int_{3-u}^{8-3u}\frac{8-3u-v}{2}\cdot
\frac{-2+u+v}{2}dv\biggr)du\Biggr)=\frac{115}{224}. 
\end{eqnarray*}
Thus we get 
\begin{eqnarray*}
&&S\left(W^{\tilde{E}_2, \tilde{e}_1}_{\bullet,\bullet,\bullet};p\right)\\
&\leq&\frac{115}{224}
+\frac{3}{28}\Biggl(\int_0^1\biggl(
\int_0^1 v^2dv+\int_1^{1+u} \left(\frac{1+v}{2}\right)^2dv
+\int_{1+u}^{3+2u}\left(\frac{3+2u-v}{2}\right)^2dv\biggr)du\\
&&+\int_1^2\biggl(\int_0^{2-u}v^2dv+\int_{2-u}^{3-u}\left(\frac{2-u+v}{2}\right)^2dv
+\int_{3-u}^{8-3u}\left(\frac{8-3u-v}{2}\right)^2dv\biggr)du\Biggr)\\
&=&\frac{115}{224}+\frac{95}{224}=\frac{15}{16}
\end{eqnarray*}
for any $p\in\tilde{e}_1$. 
As a consequence, we get the inequality
\begin{eqnarray*}
\delta_q\left(V^{E_2}_{\bullet,\bullet}\right)\geq \min\left\{
\frac{A_{E_2}\left(\tilde{e}_1\right)}{S\left(V^{\tilde{E}_2}_{\bullet,\bullet};
\tilde{e}_1\right)},\,\,\,
\inf_{p\in \tilde{e}_1}\frac{A_{\tilde{e}_1}(p)}
{S\left(W^{\tilde{E}_2,\tilde{e}_1}_{\bullet,\bullet,\bullet};p\right)}
\right\}\geq\min\left\{\frac{112}{111},\,\,
\frac{16}{15}\right\}=\frac{112}{111}
\end{eqnarray*}
by Corollary \ref{reduction_corollary}.

\eqref{112113_proposition2}
In this case, we have $\tilde{l}_2\cap\tilde{e}_1\cap\tilde{\sC}^E\neq\emptyset$. 
Let 
\[
\varepsilon\colon\hat{E}_2\to\tilde{E}_2
\]
be the blowup at $\tilde{l}_2\cap\tilde{e}_1\cap\tilde{\sC}^E$ and let 
$\hat{e}_2\subset\hat{E}_2$ be the exceptional divisor. Then 
there are exactly $4$ negative curves on $\tilde{E}_2$:
\begin{itemize}
\item
the strict transform $\hat{l}_2$ of $\tilde{l}_2$, 
\item
the curve $\hat{e}_2$, 
\item
the strict transform $\hat{e}_1$ of $\tilde{e}_1$, and 
\item
the strict transform $\hat{s}$ of $\tilde{s}$. 
\end{itemize}
The intersection form of $\hat{l}_2$, $\hat{e}_2$, $\hat{e}_1$ and $\hat{s}$ 
on $\hat{E}_2$ is given by the symmetric matrix
\[
\begin{pmatrix}
-2&1&0&0\\
1&-1&1&0\\
0&1&-2&1\\
0&0&1&-2
\end{pmatrix}.
\]
Moreover, we can contract $\hat{e}_1$ and gives the commutative diagram 
\[\xymatrix{
\hat{E}_2 \ar[rr]^{\gamma} \ar[dr]_{\hat{\varepsilon}} & & E'_2 \ar[dl]^{\varepsilon'}\\
&E_2,&
}\]
where $\hat{\varepsilon}:=\varepsilon\circ\varepsilon_2$, 
$\gamma$ is the contraction 
of $\hat{e}_1$ and $\varepsilon'$ is the extraction of $e'_2:=\gamma_*\hat{e}_2$. 
Clearly, the morphism $\varepsilon'$ is a plt-blowup. 
Let us set $\hat{\sC}^E:=(\varepsilon_2)^{-1}_*\tilde{\sC}^E$. Then 
\[
\hat{p}_l:=\hat{l}_2|_{\hat{e}_2},\quad
\hat{p}_{\sC}:=\hat{\sC}^E|_{\hat{e}_2},\quad
\hat{p}_e:=\hat{e}_1|_{\hat{e}_2}
\]
are mutually distinct reduced points. Let $p'_l$, $p'_{\sC}$, $p'_e\in e'_2$ 
be the images of those points, respectively. 
We have
\begin{eqnarray*}
\gamma^*e'_2&=&\hat{e}_2+\frac{1}{2}\hat{e}_1,\\
\left(K_{E'_2}+e'_2\right)|_{e'_2}&=&K_{e'_2}+\frac{1}{2}p'_e.
\end{eqnarray*}
Let $P(u)$, $N(u)$ $(u\in[0,2])$ be as in the proof of 
Proposition \ref{112109_proposition}. Let us set 
\begin{eqnarray*}
P(u,v)&:=&P_\sigma\left(\hat{E}_2, {\hat{\varepsilon}}^*\left(P(u)|_{E_2}\right)
-v\hat{e}_2\right), \\
N(u,v)&:=&N_\sigma\left(\hat{E}_2, {\hat{\varepsilon}}^*\left(P(u)|_{E_2}\right)
-v\hat{e}_2\right).
\end{eqnarray*}

\begin{itemize}
\item
Assume that $u\in[0,1]$. 
\begin{itemize}
\item
If $v\in[0,1+u]$, then we have 
\begin{eqnarray*}
N(u,v)&=&\frac{v}{2}\hat{e}_1, \\
P(u,v)&\sim_\R&(2+u)\hat{l}_2+(5+3u-v)\hat{e}_2
+\frac{6+4u-v}{2}\hat{e}_1+(1+u)\hat{s},
\end{eqnarray*}
and 
\[
(P(u,v))^{\cdot 2}=3+4u+u^2-\frac{v^2}{2}.
\]
\item
If $v\in[1+u,2]$, then we have 
\begin{eqnarray*}
N(u,v)&=&\frac{-1-u+v}{2}\hat{l}_2+\frac{v}{2}\hat{e}_1, \\
P(u,v)&\sim_\R&\frac{5+3u-v}{2}\hat{l}_2+(5+3u-v)\hat{e}_2
+\frac{6+4u-v}{2}\hat{e}_1+(1+u)\hat{s},
\end{eqnarray*}
and 
\[
(P(u,v))^{\cdot 2}=\frac{1}{2}(1+u)(7+3u-2v).
\]
\item
If $v\in[2,5+3u]$, then we have 
\begin{eqnarray*}
N(u,v)&=&\frac{-1-u+v}{2}\hat{l}_2+\frac{2v-1}{3}\hat{e}_1+\frac{v-2}{3}\hat{s}, \\
P(u,v)&\sim_\R&\frac{5+3u-v}{2}\hat{l}_2+(5+3u-v)\hat{e}_2
+\frac{2}{3}(5+3u-v)\hat{e}_1+\frac{5+3u-v}{3}\hat{s},
\end{eqnarray*}
and 
\[
(P(u,v))^{\cdot 2}=\frac{1}{6}(5+3u-v)^2.
\]
\end{itemize}
\item
Assume that $u\in[1,2]$. 
\begin{itemize}
\item
If $v\in[0,4-2u]$, then we have 
\begin{eqnarray*}
N(u,v)&=&\frac{v}{2}\hat{e}_1, \\
P(u,v)&\sim_\R&(5-2u)\hat{l}_2+(13-5u-v)\hat{e}_2
+\frac{16-6u-v}{2}\hat{e}_1+(3-u)\hat{s},
\end{eqnarray*}
and 
\[
(P(u,v))^{\cdot 2}=21-16u+3u^2-\frac{v^2}{2}.
\]
\item
If $v\in[4-2u,3-u]$, then we have 
\begin{eqnarray*}
N(u,v)&=&\frac{-2+u+2v}{3}\hat{e}_1+\frac{-4+2u+v}{3}\hat{s}, \\
P(u,v)&\sim_\R&(5-2u)\hat{l}_2+(13-5u-v)\hat{e}_2
+\frac{2}{3}(13-5u-v)\hat{e}_1+\frac{13-5u-v}{3}\hat{s},
\end{eqnarray*}
and 
\[
(P(u,v))^{\cdot 2}=\frac{1}{3}\left(71+11u^2+2u(-28+v)-4v-v^2\right).
\]
\item
If $v\in[3-u,13-5u]$, then we have 
\begin{eqnarray*}
N(u,v)&=&\frac{-3+u+v}{2}\hat{l}_2
+\frac{-2+u+2v}{3}\hat{e}_1+\frac{-4+2u+v}{3}\hat{s}, \\
P(u,v)&\sim_\R&\frac{13-5u-v}{2}\hat{l}_2+(13-5u-v)\hat{e}_2
+\frac{2}{3}(13-5u-v)\hat{e}_1+\frac{13-5u-v}{3}\hat{s},
\end{eqnarray*}
and 
\[
(P(u,v))^{\cdot 2}=\frac{1}{6}(13-5u-v)^2.
\]
\end{itemize}
\end{itemize}
Hence we get
\begin{eqnarray*}
&&S\left(V^{E_2}_{\bullet,\bullet};\hat{e}_2\right)\\
&=&\frac{3}{28}\Biggl(
\int_0^1\biggl(\int_0^{1+u}\left(3+4u+u^2-\frac{v^2}{2}\right)dv\\
&&+\int_{1+u}^2\frac{1}{2}(1+u)(7+3u-2v)dv
+\int_2^{5+3u}\frac{1}{6}(5+3u-v)^2dv\biggr)du\\
&+&\int_1^2\biggl(2(u-1)(21-16u+3u^2)+\int_0^{4-2u}
\left(21-16u+3u^2-\frac{v^2}{2}\right)dv\\
&&+\int_{4-2u}^{3-u}\frac{1}{3}\left(71+11u^2+2u(-28+v)-4v-v^2\right)dv
+\int_{3-u}^{13-5u}\frac{1}{6}(13-5u-v)^2dv\biggr)du\Biggr)\\
&=&\frac{339}{112}.
\end{eqnarray*}
This implies that 
\[
\delta_q\left(V^{E_2}_{\bullet,\bullet}\right)\leq
\frac{A_{E_2}\left(\hat{e}_2\right)}{S\left(V^{E_2}_{\bullet,\bullet};\hat{e}_2\right)}
=\frac{112}{113}.
\]
Moreover, we have 
\begin{eqnarray*}
F_{p'_l}\left(W^{E'_2,e'_2}_{\bullet,\bullet,\bullet}\right)
&=&\frac{6}{28}\Biggl(\int_0^1\biggl(
\int_{1+u}^2\frac{1+u}{2}\cdot\frac{-1-u+v}{2}dv\\
&&+\int_2^{5+3u}\frac{5+3u-v}{6}\cdot\frac{-1-u+v}{2}dv\biggr)du\\
&&+\int_1^2\int_{3-u}^{13-5u}\frac{13-5u-v}{6}\cdot\frac{-3+u+v}{2}dvdu\Biggr)
=\frac{839}{1344}, 
\end{eqnarray*}
\begin{eqnarray*}
F_{p'_\sC}\left(W^{E'_2,e'_2}_{\bullet,\bullet,\bullet}\right)
&=&\frac{6}{28}\Biggl(\int_1^2\biggl(\int_0^{4-2u}\frac{v}{2}(u-1)dv
+\int_{4-2u}^{3-u}\frac{2-u+v}{3}(u-1)dv\\
&&+\int_{3-u}^{13-5u}\frac{13-5u-v}{6}(u-1)dv\biggr)du\Biggr)=\frac{17}{112}, 
\end{eqnarray*}
\begin{eqnarray*}
F_{p'_s}\left(W^{E'_2,e'_2}_{\bullet,\bullet,\bullet}\right)
&=&\frac{6}{28}\Biggl(\int_0^1\int_2^{5+3u}\frac{5+3u-v}{6}\cdot\frac{v-2}{6}dvdu\\
&&+\int_1^2\biggl(\int_{4-2u}^{3-u}\frac{2-u+v}{3}\cdot\frac{v+2u-4}{6}dv\\
&&+\int_{3-u}^{13-5u}
\frac{13-5u-v}{6}\cdot\frac{v+2u-4}{6}dv\biggr)du\Biggr)=\frac{269}{1344}. 
\end{eqnarray*}
Thus we get 
\begin{eqnarray*}
&&S\left(W^{E'_2, e'_2}_{\bullet,\bullet,\bullet};p'\right)\\
&=&F_{p'}\left(W^{E'_2,e'_2}_{\bullet,\bullet,\bullet}\right)
+\frac{3}{28}\Biggl(\int_0^1\biggl(
\int_0^{1+u} \left(\frac{v}{2}\right)^2dv\\
&&+\int_{1+u}^2 \left(\frac{1+u}{2}\right)^2dv
+\int_2^{5+3u}\left(\frac{5+3u-v}{6}\right)^2dv\biggr)du\\
&+&\int_1^2\biggl(\int_0^{4-2u}\left(\frac{v}{2}\right)^2dv
+\int_{4-2u}^{3-u}\left(\frac{2-u+v}{3}\right)^2dv
+\int_{3-u}^{13-5u}\left(\frac{13-5u-v}{6}\right)^2dv\biggr)du\Biggr)\\
&=&F_{p'}\left(W^{E'_2,e'_2}_{\bullet,\bullet,\bullet}\right)+\frac{361}{1344}
\begin{cases}
=\frac{15}{32} & \text{if }p'=p'_e,\\
\leq \frac{277}{366} & \text{otherwise},
\end{cases}
\end{eqnarray*}
for any $p'\in e'_2$. 
As a consequence, we get the inequality
\begin{eqnarray*}
\delta_q\left(E_2;V^{E_2}_{\bullet,\bullet}\right)\geq \min\left\{
\frac{A_{E_2}\left(\hat{e}_2\right)}{S\left(V^{E_2}_{\bullet,\bullet};
\hat{e}_2\right)},\,\,\,
\inf_{p'\in e'_2}\frac{A_{e'_2,\frac{1}{2}p'_e}(p')}
{S\left(W^{E'_2, e'_2}_{\bullet,\bullet,\bullet};p'\right)}
\right\}\geq\min\left\{\frac{112}{113},\,\,
\frac{16}{15}\right\}=\frac{112}{113}
\end{eqnarray*}
by Corollary \ref{reduction_corollary}. 
\end{proof}

\section{Local $\delta$-invariants for special points, I}\label{special-I_section}

\begin{proposition}\label{112103_proposition}
Let $X$ be as in \S \ref{311_section}. Let us take a closed point 
$p\in X\setminus\left(E_1\cup E_2\right)$ and let $Q\in|H_1|$ with $p\in Q$. 
Assume that $Q$ is singular at $p$. 
Let $l_2$, $l_3\subset Q$ be the negative curves on $Q$ with 
$\left(l_2^{\cdot 2}\right)=-1$ and $\left(l_3^{\cdot 2}\right)=-1/2$. 
We assume that $l_2\cap l_3\cap\left(E_1|_Q\right)\neq\emptyset$.
Then we have the inequality 
\[
\delta_p(X)\geq\frac{112}{103}.
\]
\end{proposition}

\begin{proof}
The following proof is divided into 10 numbers of steps since the proof is long. 

\noindent\underline{\textbf{Step 1}}\\
Set $\sC_1:=E_1|_Q$. Then we have $\sC_1\sim -K_Q$. Note that $p\in l_3$. 
Let us set $p_0:=l_2\cap l_3\cap \sC_1$. Since $E_1$ is isomorphic to 
$\sC\times\pr^1$, the curve $\sC_1\subset E_1$ is the fiber of the projection 
$\sC\times\pr^1\to\pr^1$ passing through $p_0$. Let $l_1\subset E_1$ be the 
fiber of the projection $\sC\times\pr^1\to\sC$ passing through $p_0$. 
Let 
\[
\rho_p\colon X^1\to X
\]
be the blowup at $\in X$ and let $D^1\subset X^1$ be the exceptional divisor. 
Set $Q^1:=\left(\rho_p\right)^{-1}_*Q$, $E_1^1:=\left(\rho_p\right)^{-1}_*E_1$. 
Obviously, we have $A_X(D^1)=3$. Moreover, since 
\begin{eqnarray*}
-K_X\sim 2Q+E_1,\quad \rho_p^*Q=Q^1+2D^1,\quad\rho^*_pE_1=E_1^1, 
\end{eqnarray*}
we have $\tau_X(D^1)\geq 4$. Moreover, for $0<\varepsilon\ll 1$, 
\[
\rho_p^*(-K_X)-\varepsilon D^1\sim_\R 2Q^1+E_1^1+(4-\varepsilon)D^1
\]
is ample. Therefore, for any $u\in(0,4)$ and for any irreducible curve 
$C^1\subset X^1$ with $C^1\not\subset Q^1\cup E_1^1\cup D^1$, 
we have
\[
\left(\left(\rho_p^*(-K_X)-uD^1\right)\cdot C^1\right)
=\left(\left(2Q^1+E_1^1+(4-u)D^1\right)\cdot C^1\right)>0.
\]
Note that the pair $\left(X^1, Q^1+E_1^1+D^1\right)$ is log smooth. In particular, 
the pair 
\[
\left(X^1, \frac{2}{3}Q^1+\frac{1}{3}E_1^1\right)
\]
is klt. Since 
\[
-\left(K_{X^1}+\frac{2}{3}Q^1+\frac{1}{3}E_1^1\right)\sim_\Q
\frac{2}{3}\left(\rho_p^*(-K_X)-D^1\right)
\]
is nef and big, the variety $X^1$ is a Mori dream space by 
\cite[Corollary 1.3.2]{BCHM}. 

\noindent\underline{\textbf{Step 2}}\\
Let $l_2^1$, $l_3^1$, $\sC_1^1\subset Q^1$ be the strict transforms of 
$l_2$, $l_3$, $\sC_1\subset Q$, 
respectively. Let $l_1^1\subset E_1^1$ be the strict transform of $l_1\subset E_1$. 
Moreover, let $e^1\subset Q^1$ be the exceptional curve of the morphism 
$Q^1\to Q$. Let $g^1\subset D^1(\simeq\pr^2)$ be the tangent line of the conic 
$e^1\subset D^1$ at $e^1\cap l_3^1$. 
Then we have the following intersection numbers: 
\begin{center}
\renewcommand\arraystretch{1.4}
\begin{tabular}{c|ccc|c}
$\cdot$ & $Q^1$ & $E_1^1$ & $D^1$ & $2Q^1+E_1^1+(4-u)D^1$\\ \hline
$l_2^1$ & $0$ & $1$ & $0$ & $1$\\
$l_3^1$ & $-2$ & $1$ & $1$ & $1-u$\\
$e^1$ & $4$ & $0$ & $-2$ & $2u$\\
$l_1^1$ & $1$ & $-1$ & $0$ & $1$\\
$\sC_1^1$ & $0$ & $7$ & $0$ & $7$\\
$g^1$ & $2$ & $0$ & $-1$ & $u$
\end{tabular}
\end{center}
Hence, for $u\in[0,1]$, the $\R$-divisor $\rho_p^*(-K_X)-uD^1$ is nef,
\[
\vol_{X^1}\left(\rho_p^*(-K_X)-uD^1\right)=28-u^3,
\]
and the divisor $\rho_p^*(-K_X)-D^1$ contracts the curve $l_3^1\subset X^1$. 

\noindent\underline{\textbf{Step 3}}\\
Note that 
\[
\sN_{l_3^1/X^1}\simeq\sO_{\pr^1}(-1)\oplus\sO_{\pr^1}(-2), 
\]
where $\sN_{l_3^1/X^1}$ is the normal bundle of $l_3^1\subset X^1$. Let
\[
\phi^{112}\colon X^{112}\to X^1
\]
be the blowup along $l_3^1\subset X^1$ and let $\E^{112}_1\subset X^{112}$ 
be the exceptional divisor. 
Note that $\E^{112}_1$ is isomorphic to $\pr_{\pr^1}(\sO\oplus\sO(1))$. 
Let $l^{112}_3\subset \E^{112}_1$ be the $(-1)$-curve. 
Since 
\[
\sN_{l_3^{112}/X^{112}}\simeq\sO_{\pr^1}(-1)^{\oplus 2}, 
\]
we can take Atiyah's flop
\[
X^{112}\xleftarrow{\phi^{12}}X^{12}\xrightarrow{\phi^{+12}}X^{122}.
\]
More precisely, the morphism $\phi^{12}$ is the blowup along 
$l_3^{112}\subset X^{112}$. Let $\E_2^{12}\subset X^{12}$ be the exceptional divisor. 
Since $\E_2^{12}\simeq\pr^1\times\pr^1$ with 
$-\E_2^{12}|_{\E_2^{12}}\simeq\sO(1,1)$, we can (analytically) contract 
to a complex manifold $X^{112}$, where the image $l_3^{+112}$ of $\E_2^{12}$ is 
a smooth rational curve, $\phi^{+12}$ is the blowup along 
$l_3^{+112}\subset X^{122}$, and $\phi^{12}$ and $\phi^{+12}$ are mutually different. 
Set $\E_1^{12}:=\left(\phi^{12}\right)^{-1}_*\E_1^{112}$ and 
$\E_1^{122}:=\left(\phi^{+12}\right)_*\E_1^{12}$. Then 
$\E_1^{122}\simeq\pr^2$ with $-\E_1^{122}|_{\E_1^{122}}\simeq\sO_{\pr^2}(2)$. 
By Grauert--Fujiki--Ancona--Vantan's contraction theorem 
(see \cite[Proposition 7.4]{HP}), there is the contraction 
\[
\phi^{122}\colon X^{122}\to X^2
\]
of $\E^{122}_1\subset X^{122}$ to a point. Set $\phi^1:=\phi^{112}\circ\phi^{12}$, 
$\phi^{+1}:=\phi^{122}\circ\phi^{+12}$ and 
$\chi^1:=\phi^{+1}\circ\left(\phi^1\right)^{-1}$. 
We get the commutative diagram 
\[\xymatrix{
& X^{12} \ar[dl]_{\phi^{12}} \ar[dr]^{\phi^{+12}} \ar[ddl]^{\phi^1} 
\ar[ddr]_{\phi^{+1}} &\\
X^{112} \ar[d]_{\phi^{112}} & & X^{122} \ar[d]^{\phi^{122}}\\
X^1 \ar@{-->}[rr]_{\chi^1} & & X^2.
}\]
Set $l_3^{+2}:=\left(\phi^{122}\right)_*l_3^{+122}$, and 
$l_3^{+12}:=\left(\left(\phi^1\right)^{-1}_*D^1\right)|_{\E_2^{12}}$. Then 
$l_3^{+12}\subset \E_2^{12}$ is an irreducible curve with 
$\left(\phi^{+1}\right)_*l_3^{+12}=l_3^{+2}$ and 
$\left(\phi^1\right)_*l_3^{+12}=0$.
Let us set 
\begin{eqnarray*}
Q^{12}:=\left(\phi^1\right)^{-1}_*Q^1,\quad
E_1^{12}:=\left(\phi^1\right)^{-1}_*E_1^1,\quad
D^{12}:=\left(\phi^1\right)^{-1}_*D^1,\\
Q^{2}:=\left(\phi^{+1}\right)_*Q^{12},\quad
E_1^{2}:=\left(\phi^{+1}\right)_*E_1^{12},\quad
D^{2}:=\left(\phi^{+1}\right)_*D^{12},\\
\end{eqnarray*}
and 
\begin{eqnarray*}
l_2^{12}:=\left(\phi^1\right)^{-1}_*l_2^1,\,\,
e^{12}:=\left(\phi^1\right)^{-1}_*e^1,\,\,
l_1^{12}:=\left(\phi^1\right)^{-1}_*l_1^1,\,\,
\sC_1^{12}:=\left(\phi^1\right)^{-1}_*\sC_1^1,\,\,
g^{12}:=\left(\phi^1\right)^{-1}_*g^1,\\
l_2^{2}:=\left(\phi^{+1}\right)_*l_2^{12},\,\,
e^{2}:=\left(\phi^{+1}\right)_*e^{12},\,\,
l_1^{2}:=\left(\phi^{+1}\right)_*l_1^{12},\,\,
\sC_1^{2}:=\left(\phi^{+1}\right)_*\sC_1^{12},\,\,
g^{2}:=\left(\phi^{+1}\right)_*g^{12}.
\end{eqnarray*}
Note that 
$\left(\phi^1\right)^*Q^1=Q^{12}+\E_1^{12}+2\E_2^{12}$ and 
$\left(\phi^{+1}\right)^*Q^2=Q^{12}$. Moreover, the restriction 
$\chi^1|_{Q^1}\colon Q^1\to Q^2$ is the contraction of $l_3^1$. 
In particular, we have $Q^2\simeq\pr_{\pr^1}(\sO\oplus\sO(1))$. 
Note that 
$\left(\phi^1\right)^*E_1^1=E_1^{12}$ and 
$\left(\phi^{+1}\right)^*E_1^2=E_1^{12}+(1/2)\E_1^{12}+\E_2^{12}$. 
Moreover, the restriction 
$\left(\chi^1\right)^{-1}|_{E_1^2}\colon E_1^2\to E_1^1$ is 
the weighted blowup at $p_0$ with the weights $\ord\left(\sC_1^1\right)=2$ and 
$\ord\left(l_1^1\right)=1$. 
Note that 
$\left(\phi^1\right)^*D^1=D^{12}$ and 
$\left(\phi^{+1}\right)^*D^2=D^{12}+(1/2)\E_1^{12}+\E_2^{12}$. 
Moreover, the morphism 
$\phi^{112}\colon\left(\phi^{112}\right)^{-1}_*D^1\to D^1$
is the blowup at $g^1\cap l_3^1$, the morphism 
$\phi^{12}\colon D^{12}\to \left(\phi^{112}\right)^{-1}_*D^1$ is the blowup at 
the intersection of the strict transforms of $e^1$ and $g^1$, and 
the morphism $\phi^{+1}\colon D^{12}\to D^2$ is the blowdown of the curve 
$\E_1^{12}|_{D^{12}}$. Let $h^{12}\subset D^{12}$ be the 
strict transform of the exceptional divisor of the morphism 
$\phi^{112}\colon\left(\phi^{112}\right)^{-1}_*D^1\to D^1$, i.e., $h^{12}
=\E_1^{12}|_{D^{12}}$. 

On $X^{12}$, we get the following intersection numbers: 
\begin{center}
\renewcommand\arraystretch{1.4}
\begin{tabular}{c|cc|ccc|c}
$\cdot$ & $\E_1^{12}$ & $\E_2^{12}$ & 
$\left(\phi^{+1}\right)^*Q^2$ & $\left(\phi^{+1}\right)^*E_1^2$ & 
$\left(\phi^{+1}\right)^*D^2$ & 
$\left(\phi^{+1}\right)^*\left(2Q^2+E_1^2+(4-u)D^2\right)$ \\ \hline
$l_2^{12}$ & $0$ & $1$ & $-2$ & $2$ & $1$& $2-u$\\
$l_3^{+12}$ & $1$ & $-1$ & $1$ & $-1/2$ & $-1/2$& $\frac{1}{2}(u-1)$\\
$e^{12}$ & $0$ & $1$ & $2$ & $1$ & $-1$&$1+u$\\
$l_1^{12}$ & $1$ & $0$ & $0$ & $-1/2$ & $1/2$& $\frac{1}{2}(3-u)$\\
$\sC_1^{12}$ & $0$ & $1$ & $-2$ & $8$ & $1$& $8-u$\\
$g^{12}$ & $0$ & $1$ & $0$ & $1$ & $0$& $1$\\
$h^{12}$ & $-2$ & $1$ & $0$ & $0$ & $0$& $0$
\end{tabular}
\end{center}
Thus, for $u\in(1,2)$, the $\R$-divisor $2Q^2+E_1^2+(4-u)D^2$ is ample on $X^2$. 
In particular, $X^2$ is projective and $\chi^1$ is a small $\Q$-factorial modification 
of $X^1$. In particular, for $u\in [1,2]$, we have 
\begin{eqnarray*}
&&\vol_{X^1}\left(\rho_p^*(-K_X)-uD^1\right)
=\left(\left(\phi^{+1}\right)^*
\left(2Q^2+E_1^2+(4-u)D^2\right)\right)^{\cdot 3}\\
&=&\left(\left(\phi^{1}\right)^*
\left(2Q^1+E_1^1+(4-u)D^1\right)
+\frac{1-u}{2}\left(\E_1^{12}+2\E_2^{12}\right)\right)^{\cdot 3}
=28-u^3+\frac{1}{2}(u-1)^3.
\end{eqnarray*}

\noindent\underline{\textbf{Step 4}}\\
For $u=2$, the divisor $2Q^2+E_1^2+(4-2)D^2$ gives the birational contraction 
\[
\sigma^2\colon X^2\to Y^2. 
\]
Note that the exceptional set of $\sigma^2$ is 
$Q^2\left(\simeq\pr_{\pr^1}(\sO\oplus\sO(1))\right)$ and the restriction 
$\sigma^2|_{Q^2}\colon Q^2\to B^2$ is the $\pr^1$-fibration. 
We have 
\[
\left(\sigma^2\right)^*\left(\sigma^2\right)_*
\left(2Q^2+E_1^2+(4-u)D^2\right)=\frac{6-u}{2}Q^2+E_1^2+(4-u)D^2.
\]
By Step 3, the $\R$-divisor 
$\left(\sigma^2\right)_*\left(2Q^2+E_1^2+(4-u)D^2\right)$ is ample for $u\in(2,3)$. 
Moreover, for $u=3$, the divisor 
$\left(\sigma^2\right)_*\left(2Q^2+E_1^2+(4-3)D^2\right)$ contracts the 
curve $\left(\sigma^2\right)_*l_1^2\subset Y^2$. 
Thus, for $u\in[2,3]$, we have 
\begin{eqnarray*}
\vol_{X^1}\left(\rho_p^*(-K_X)-uD^1\right)
&=&\left(\left(\frac{6-u}{2}Q^2+E_1^2+(4-u)D^2\right)\right)^{\cdot 3}\\
&=&28-u^3+\frac{1}{2}(u-1)^3+\frac{1}{2}(u-2)^2(u+7).
\end{eqnarray*}
We note that $l_1^2$ and $Q^2$ are mutually disjoint.

\noindent\underline{\textbf{Step 5}}\\
Set $X^{223}:=X^{122}$, $\phi^{223}:=\phi^{122}$ and 
$\D^{223}_1:=\E_1^{122}\simeq\pr^2$. The variety $X^{223}$ is smooth, and 
the strict transform $l_1^{223}\subset X^{223}$ of $l_1^2\subset X^2$ 
satisfies that 
$\sN_{l_1^{223}/X^{223}}\simeq\sO_{\pr^1}(-1)^{\oplus 2}$. 
Thus we can take Atiyah's flop 
\[
X^{223}\xleftarrow{\phi^{23}}X^{23}\xrightarrow{\phi^{+23}}X^{233}
\]
of $l_1^{223}\subset X^{223}$. Let $\D^{23}_2\subset X^{23}$ be the 
exceptional divisor of $\phi^{23}$, and let $l_1^{+233}\subset X^{233}$ be 
the image of $\D_2^{23}\simeq\pr^1\times\pr^1$. 
Let us set $\D_1^{23}:=\left(\phi^{23}\right)^{-1}_*\D^{223}_1$ and 
$\D_1^{233}:=\left(\phi^{+23}\right)_*\D_1^{23}$. 
Then $\D_1^{233}\simeq\pr_{\pr^1}(\sO\oplus\sO(1))$ and any fiber 
of $\D_1^{233}/\pr^1$ intersects $\D_1^{233}$ with $-1$. 
Thus we get the contraction 
\[
\phi^{233}\colon X^{233}\to X^3
\]
of $\D_1^{233}$, where $X^3$ is a complex manifold and the image 
$l_1^{+3}\subset X^3$ of $\D_1^{233}$ is a smooth rational curve such that 
$\phi^{233}$ is the blowup along $l_1^{+3}\subset X^3$. Let us set 
$l^{+23}:=\D_1^{23}|_{\D_2^{23}}$. 
Set $\phi^2:=\phi^{223}\circ\phi^{23}$, $\phi^{+2}:=\phi^{233}\circ\phi^{+23}$ and 
$\chi^2:=\phi^{+2}\circ\left(\phi^2\right)^{-1}$. 
We get the commutative diagram 
\[\xymatrix{
& X^{23} \ar[dl]_{\phi^{23}} \ar[dr]^{\phi^{+23}} \ar[ddl]^{\phi^2} 
\ar[ddr]_{\phi^{+2}} &\\
X^{223} \ar[d]_{\phi^{223}} & & X^{233} \ar[d]^{\phi^{233}}\\
X^2 \ar@{-->}[rr]_{\chi^2} & & X^3.
}\]
Let us set 
\begin{eqnarray*}
Q^{23}:=\left(\phi^2\right)^{-1}_*Q^2,\quad
E_1^{23}:=\left(\phi^2\right)^{-1}_*E_1^2,\quad
D^{23}:=\left(\phi^2\right)^{-1}_*D^2,\\
Q^{3}:=\left(\phi^{+2}\right)_*Q^{23},\quad
E_1^{3}:=\left(\phi^{+2}\right)_*E_1^{23},\quad
D^{3}:=\left(\phi^{+2}\right)_*D^{23},\\
\end{eqnarray*}
and 
\begin{eqnarray*}
l_3^{+23}:=\left(\phi^2\right)^{-1}_*l_3^{+2},\quad
l_2^{23}:=\left(\phi^2\right)^{-1}_*l_2^2,\quad
e^{23}:=\left(\phi^2\right)^{-1}_*e^2,,\\
\sC_1^{23}:=\left(\phi^2\right)^{-1}_*\sC_1^2,\quad
g^{23}:=\left(\phi^2\right)^{-1}_*g^2,\quad
h^{23}:=\left(\phi^{23}\right)^{-1}_*\left(\phi^{+12}\right)_*h^{12},\\
l_3^{+3}:=\left(\phi^{+2}\right)_*l_3^{+23},\quad
l_2^{3}:=\left(\phi^{+2}\right)_*l_2^{23},\quad
e^{3}:=\left(\phi^{+2}\right)_*e^{23},\\
\sC_1^{3}:=\left(\phi^{+2}\right)_*\sC_1^{23},\quad
g^{3}:=\left(\phi^{+2}\right)_*g^{23},\quad
h^{3}:=\left(\phi^{+2}\right)_*h^{23}.
\end{eqnarray*}
Since $\chi^2$ is an isomorphism around a neighborhood of $Q^2$, we can also get 
the contraction 
\[
\sigma^3\colon X^3\to Y^3
\]
of $Q^3$ to a curve $B^3$ as in $\sigma^2$. 
Obviously, we have 
$\left(\phi^2\right)^*Q^2=Q^{23}$ and 
$\left(\phi^{+2}\right)^*Q^3=Q^{23}$. 
Note that $\left(\phi^2\right)^*E_1^2=E_1^{23}+(1/2)\D_1^{23}+\D_2^{23}$ 
and $\left(\phi^{+2}\right)^*E_1^3=E_1^{23}$. 
Moreover, $\chi^2\colon E_1^2\to E_1^3$ is the contraction of $l_1^2\subset E_1^2$. 
Note that $\left(\phi^2\right)^*D^2=D^{23}+(1/2)\D_1^{23}$ 
and $\left(\phi^{+2}\right)^*D^3=D^{23}+\D_1^{23}+\D_2^{23}$. 
Moreover, $\phi^2\colon D^{23}\to D^2$ is isomorphic to 
$\phi^{+1}\colon D^{12}\to D^2$, and the morphism $\phi^{+2}\colon D^{23}\to D^3$ 
is an isomorphism. 

On $X^{23}$, we get the following intersection numbers: 
\begin{center}
\renewcommand\arraystretch{1.4}
\begin{tabular}{c|cc|ccc|c}
$\cdot$ & $\D_1^{23}$ & $\D_2^{23}$ & 
$\left(\phi^{+2}\right)^*Q^3$ & $\left(\phi^{+2}\right)^*E_1^3$ & 
$\left(\phi^{+2}\right)^*D^3$ & 
$\left(\phi^{+2}\right)^*\left(\frac{6-u}{2}Q^3+E_1^3+(4-u)D^3\right)$ \\ \hline
$l_2^{23}$ & $0$ & $0$ & $-2$ & $2$ & $1$& $0$\\
$l_3^{+23}$ & $1$ & $0$ & $1$ & $-1$ & $0$& $\frac{1}{2}(4-u)$\\
$e^{23}$ & $0$ & $0$ & $2$ & $1$ & $-1$&$3$\\
$l_1^{+23}$ & $0$ & $-1$ & $0$ & $1$ & $-1$& $u-3$\\
$\sC_1^{23}$ & $0$ & $0$ & $-2$ & $8$ & $1$& $6$\\
$g^{23}$ & $0$ & $0$ & $0$ & $1$ & $0$& $1$\\
$h^{23}$ & $-2$ & $0$ & $0$ & $1$ & $-1$& $u-3$
\end{tabular}
\end{center}
Thus, for $u\in(3,4)$, the $\R$-divisor 
\[
\left(\sigma^3\right)_*\left(\frac{6-u}{2}Q^3+E_1^3+(4-u)D^3\right)
\]
is ample. Thus $Y^3$ and $X^3$ are projective, and $\chi^2\circ\chi^1$ is 
a small $\Q$-factorial modification of $X^1$. In particular, for 
$u\in[3,4]$, we have
\begin{eqnarray*}
&&\vol_{X^1}\left(\rho_p^*(-K_X)-uD^1\right)
=\left(\left(\phi^{+2}\right)^*
\left(\frac{6-u}{2}Q^3+E_1^3+(4-u)D^3\right)\right)^{\cdot 3}\\
&=&\left(\left(\phi^2\right)^*
\left(\frac{6-u}{2}Q^2+E_1^2+(4-u)D^2\right)
+\frac{3-u}{2}\left(\D_1^{23}+2\D_2^{23}\right)\right)^{\cdot 3}\\
&=&\frac{1}{2}(7-u)(4-u)(2+u).
\end{eqnarray*}
Therefore, we have $\tau_X\left(D^1\right)=4$, and 
\begin{eqnarray*}
S_X\left(D^1\right)&=&\frac{1}{28}\Biggl(\int_0^1(28-u^3)du
+\int_1^2\left(28-u^3+\frac{1}{2}(u-1)^3\right)du\\
&&+\int_2^3\left(28-u^3+\frac{1}{2}(u-1)^3+\frac{1}{2}(u-2)^2(u+7)\right)du\\
&&+\int_3^4\frac{1}{2}(7-u)(4-u)(2+u)du\Biggr)=\frac{289}{112}.
\end{eqnarray*}

\noindent\underline{\textbf{Step 6}}\\
Let $\psi^{12}\colon\tilde{X}\to X^{12}$ be the blowup along $l_1^{12}\subset X^{12}$. 
Then we get the natural morphism $\psi^{23}\colon\tilde{X}\to X^{23}$. 
Set $\tilde{D}:=\left(\psi^{12}\right)^{-1}_*D^{12}$, 
$\tilde{h}:=\left(\psi^{12}\right)^{-1}_*h^{12}$, 
$\tilde{l}^+_3:=\left(\psi^{12}\right)^{-1}_*l_3^{+12}$, 
$\tilde{g}:=\left(\psi^{12}\right)^{-1}_*g^{12}$, 
$\tilde{e}:=\left(\psi^{12}\right)^{-1}_*e^{12}$, 
and $\gamma:=\left(\phi^1\circ\psi^{12}\right)|_{\tilde{D}}\colon \tilde{D}\to D^1$. 
Note that $\tilde{D}\simeq D^{12}$ and $\tilde{D}\simeq D^{23}$. Moreover, 
we have
\begin{eqnarray*}
\left(\left(\psi^{12}\right)^*\E_1^{12}\right)|_{\tilde{D}}=\tilde{h},\quad
\left(\left(\psi^{12}\right)^*\E_2^{12}\right)|_{\tilde{D}}=\tilde{l}^+_3,\quad
\left(\left(\psi^{23}\right)^*\D_1^{23}\right)|_{\tilde{D}}=\tilde{h},\quad
\left(\left(\psi^{23}\right)^*\D_2^{23}\right)|_{\tilde{D}}=0.
\end{eqnarray*}
We remark that $\tilde{e}\sim\tilde{h}+2\tilde{l}^+_3+2\tilde{g}$. 
The intersection form of $\tilde{h}$, $\tilde{l}_3^+$ and $\tilde{g}$ on $\tilde{D}$ 
is given by the symmetric matrix 
\[\begin{pmatrix}
-2&1&0\\
1&-1&1\\
0&1&-1
\end{pmatrix}.\]
For any $u\in[0,4]$, let us set 
\begin{eqnarray*}
P(u)&:=&P_\sigma\left(\tilde{X}, \left(\phi^1\circ\psi^{12}\right)^*
\left(\rho_p^*(-K_X)-uD^1\right)\right)\Big|_{\tilde{D}},\\
N(u)&:=&N_\sigma\left(\tilde{X}, \left(\phi^1\circ\psi^{12}\right)^*
\left(\rho_p^*(-K_X)-uD^1\right)\right)\Big|_{\tilde{D}}.
\end{eqnarray*}
Note that 
\[
Q^1|_{\tilde{D}}=\tilde{e}+\tilde{h}+2\tilde{l}_3^+, \quad
E_1^1|_{\tilde{D}}=0,\quad
D^1|_{\tilde{D}}\sim-\tilde{h}-2\tilde{l}_3^+-\tilde{g}. 
\]
Thus we get
\[
Q^2|_{\tilde{D}}=\tilde{e}, \quad
E_1^2|_{\tilde{D}}=\frac{1}{2}\tilde{h}+\tilde{l}_3^+,\quad
D^2|_{\tilde{D}}\sim_\Q-\frac{1}{2}\tilde{h}-\tilde{l}_3^+-\tilde{g} 
\]
and
\[
Q^3|_{\tilde{D}}=\tilde{e}, \quad
E_1^3|_{\tilde{D}}=\tilde{l}_3^+,\quad
D^3|_{\tilde{D}}\sim_\Q-\tilde{l}_3^+-\tilde{g}. 
\]
In particular, 
\begin{itemize}
\item
if $u\in[0,1]$, then 
\begin{eqnarray*}
N(u)&=&0,\\
P(u)&\sim_\R&u\tilde{h}+2u\tilde{l}^+_3+u\tilde{g},
\end{eqnarray*}
\item
if $u\in[1,2]$, then 
\begin{eqnarray*}
N(u)&=&\frac{u-1}{2}\tilde{h}+(u-1)\tilde{l}_3^+,\\
P(u)&\sim_\R&\frac{1+u}{2}\tilde{h}+(1+u)\tilde{l}^+_3+u\tilde{g},
\end{eqnarray*}
\item
if $u\in[2,3]$, then 
\begin{eqnarray*}
N(u)&=&\frac{u-1}{2}\tilde{h}+(u-1)\tilde{l}_3^++\frac{u-2}{2}\tilde{e},\\
P(u)&\sim_\R&\frac{3}{2}\tilde{h}+3\tilde{l}^+_3+2\tilde{g},
\end{eqnarray*}
\item
if $u\in[3,4]$, then 
\begin{eqnarray*}
N(u)&=&(u-2)\tilde{h}+(u-1)\tilde{l}_3^++\frac{u-2}{2}\tilde{e},\\
P(u)&\sim_\R&\frac{6-u}{2}\tilde{h}+3\tilde{l}^+_3+2\tilde{g}.
\end{eqnarray*}
\end{itemize}

\noindent\underline{\textbf{Step 7}}\\
Note that $A_{D^1}\left(g^1\right)=1$ and 
$\gamma^*g^1=\tilde{g}+\tilde{h}+2\tilde{l}_3^+$. Let us set 
$\tilde{p}_{l_3^+}:=\tilde{l}_3^+|_{\tilde{g}}$, 
$p_{l_3^+}:=\gamma\left(\tilde{p}_{l_3^+}\right)$, and 
\begin{eqnarray*}
P(u,v)&:=&P_\sigma\left(\tilde{D}, P(u)-v\tilde{g}\right), \\
N(u,v)&:=&N_\sigma\left(\tilde{D}, P(u)-v\tilde{g}\right).
\end{eqnarray*}

\begin{itemize}
\item
Assume that $u\in[0,1]$. 
\begin{itemize}
\item
If $v\in[0,u]$, then we have 
\begin{eqnarray*}
N(u,v)&=&v\tilde{h}+2v\tilde{l}_3^+, \\
P(u,v)&\sim_\R&(u-v)\tilde{h}+2(u-v)\tilde{l}_3^++(u-v)\tilde{g},
\end{eqnarray*}
and $(P(u,v))^{\cdot 2}=(u-v)^2$.
\end{itemize}
\item
Assume that $u\in[1,2]$. 
\begin{itemize}
\item
If $v\in\left[0,\frac{u-1}{2}\right]$, then we have 
\begin{eqnarray*}
N(u,v)&=&0, \\
P(u,v)&\sim_\R&\frac{1+u}{2}\tilde{h}+(1+u)\tilde{l}_3^++(u-v)\tilde{g},
\end{eqnarray*}
and \[
(P(u,v))^{\cdot 2}=\frac{1}{2}\left(-1+2u+u^2-4v-2v^2\right).
\]
\item
If $v\in\left[\frac{u-1}{2},u\right]$, then we have 
\begin{eqnarray*}
N(u,v)&=&\frac{1-u+2v}{2}\tilde{h}+(1-u+2v)\tilde{l}_3^+, \\
P(u,v)&\sim_\R&(u-v)\tilde{h}+2(u-v)\tilde{l}_3^++(u-v)\tilde{g},
\end{eqnarray*}
and $(P(u,v))^{\cdot 2}=(u-v)^2$. 
\end{itemize}
\item
Assume that $u\in[2,3]$. 
\begin{itemize}
\item
If $v\in\left[0,\frac{1}{2}\right]$, then we have 
\begin{eqnarray*}
N(u,v)&=&0, \\
P(u,v)&\sim_\R&\frac{3}{2}\tilde{h}+3\tilde{l}_3^++(2-v)\tilde{g},
\end{eqnarray*}
and \[
(P(u,v))^{\cdot 2}=\frac{7}{2}-2v-v^2.
\]
\item
If $v\in\left[\frac{1}{2},2\right]$, then we have 
\begin{eqnarray*}
N(u,v)&=&\frac{-1+2v}{2}\tilde{h}+(-1+2v)\tilde{l}_3^+, \\
P(u,v)&\sim_\R&(2-v)\tilde{h}+2(2-v)\tilde{l}_3^++(2-v)\tilde{g},
\end{eqnarray*}
and $(P(u,v))^{\cdot 2}=(2-v)^2$. 
\end{itemize}
\item
Assume that $u\in[3,4]$. 
\begin{itemize}
\item
If $v\in\left[0,\frac{4-u}{2}\right]$, then we have 
\begin{eqnarray*}
N(u,v)&=&0, \\
P(u,v)&\sim_\R&\frac{6-u}{2}\tilde{h}+3\tilde{l}_3^++(2-v)\tilde{g},
\end{eqnarray*}
and \[
(P(u,v))^{\cdot 2}=3u-\frac{u^2}{2}-(1+v)^2.
\]
\item
If $v\in\left[\frac{4-u}{2},\frac{u-2}{2}\right]$, then we have 
\begin{eqnarray*}
N(u,v)&=&\frac{-4+u+2v}{2}\tilde{l}_3^+, \\
P(u,v)&\sim_\R&\frac{6-u}{2}\tilde{h}+\frac{10-u-2v}{2}\tilde{l}_3^++(2-v)\tilde{g},
\end{eqnarray*}
and \[
(P(u,v))^{\cdot 2}=\frac{1}{4}(6-u)(2+u-4v).
\]
\item
If $v\in\left[\frac{u-2}{2},2\right]$, then we have 
\begin{eqnarray*}
N(u,v)&=&\frac{2-u+2v}{2}\tilde{h}+(-1+2v)\tilde{l}_3^+, \\
P(u,v)&\sim_\R&(2-v)\tilde{h}+2(2-v)\tilde{l}_3^++(2-v)\tilde{g},
\end{eqnarray*}
and $(P(u,v))^{\cdot 2}=(2-v)^2$. 
\end{itemize}
\end{itemize}
Hence we get
\begin{eqnarray*}
S\left(V^{D^1}_{\bullet,\bullet};g^1\right)
&=&\frac{3}{28}\Biggl(
\int_0^1\int_0^u(u-v)^2dvdu\\
&+&\int_1^2\biggl(\int_0^{\frac{u-1}{2}}\frac{1}{2}(-1+2u+u^2-4v-2v^2)dv
+\int_{\frac{u-1}{2}}^{u}(u-v)^2dv\biggr)du\\
&+&\int_2^3\biggl(\int_0^{\frac{1}{2}}\left(\frac{7}{2}-2v-v^2\right)dv
+\int_{\frac{1}{2}}^2 (2-v)^2dv\biggr)du\\
&+&\int_3^4\biggl(\int_0^{\frac{4-u}{2}}\left(3u-\frac{u^2}{2}-(1+v)^2\right)dv
+\int_{\frac{4-u}{2}}^{\frac{u-2}{2}}\frac{1}{4}(6-u)(2+u-4v)dv\\
&&+\int_{\frac{u-2}{2}}^2 (2-v)^2dv\biggr)du\Biggr)=\frac{307}{448}.
\end{eqnarray*}
Moreover, $F_{p'}\left(W^{D^1,g^1}_{\bullet,\bullet,\bullet}\right)$ is nonzero 
only if $p'=p_{l_3^+}$. Thus, for any $p'\in g^1\setminus p_{l_3^+}$, we have
\begin{eqnarray*}
S\left(W^{D^1,g^1}_{\bullet,\bullet,\bullet};p'\right)
&=&\frac{3}{28}\Biggl(\int_0^1\int_0^u(u-v)^2dvdu
+\int_1^2\biggl(\int_0^{\frac{u-1}{2}}(1+v)^2dv
+\int_{\frac{u-1}{2}}^u(u-v)^2dv\biggr)du\\
&+&\int_2^3\biggl(\int_0^{\frac{1}{2}}(1+v)^2dv+\int_{\frac{1}{2}}^2(2-v)^2dv\biggr)du\\
&+&\int_3^4\biggl(\int_0^{\frac{4-u}{2}}(1+v)^2dv
+\int_{\frac{4-u}{2}}^{\frac{u-2}{2}}\left(\frac{6-u}{2}\right)^2dv
+\int_{\frac{u-2}{2}}^2(2-v)^2dv\biggr)du\Biggr)\\
&=&\frac{227}{448}. 
\end{eqnarray*}
Therefore, we get 
\begin{eqnarray*}
\delta_{p'}\left(D^1;V^{D^1}_{\bullet,\bullet}\right)\geq
\min\left\{\frac{A_{D^1}\left(g^1\right)}{S\left(V^{D^1}_{\bullet,\bullet};g^1\right)},\,\,\,
\frac{A_{g^1}\left(p'\right)}{S\left(W^{D^1,g^1}_{\bullet,\bullet,\bullet};p'\right)}\right\}
=\min\left\{\frac{448}{307},\,\,\,\frac{448}{227}\right\}=\frac{448}{307}
\end{eqnarray*}
by Corollary \ref{reduction_corollary}.

\noindent\underline{\textbf{Step 8}}\\
Let $\gamma_0\colon D'\to D^1$ be the extraction of $\tilde{l}_3^+\subset\tilde{D}$ 
over $D^1$. Let $\gamma'\colon \tilde{D}\to D'$ be the natural morphism 
and let us set $l^{'+}_3:=\gamma'_*\tilde{l}_3^+$. 
Set $\tilde{p}_h:=\tilde{h}|_{\tilde{l}_3^+}$, $\tilde{p}_e:=\tilde{e}|_{\tilde{l}_3^+}$ 
and $\tilde{p}_g:=\tilde{g}|_{\tilde{l}_3^+}$. Then the points 
$\tilde{p}_h$, $\tilde{p}_e$ and $\tilde{p}_g$ are mutually distinct and reduced. Set 
$p'_h:=\gamma'\left(\tilde{p}_h\right)$, 
$p'_e:=\gamma'\left(\tilde{p}_e\right)$ and 
$p'_g:=\gamma'\left(\tilde{p}_g\right)$. 
We have 
\begin{eqnarray*}
\left(K_{D'}+l^{'+}_3\right)|_{l^{'+}_3}=K_{l^{'+}_3}+\frac{1}{2}p'_h,\quad
(\gamma')^* l^{'+}_3=\tilde{l}^+_3+\frac{1}{2}\tilde{h}. 
\end{eqnarray*}
Let us set 
\begin{eqnarray*}
P(u,v)&:=&P_\sigma\left(\tilde{D}, P(u)-v\tilde{l}^+_3\right), \\
N(u,v)&:=&N_\sigma\left(\tilde{D}, P(u)-v\tilde{l}^+_3\right).
\end{eqnarray*}

\begin{itemize}
\item
Assume that $u\in[0,1]$. 
\begin{itemize}
\item
If $v\in[0,u]$, then we have 
\begin{eqnarray*}
N(u,v)&=&\frac{v}{2}\tilde{h}, \\
P(u,v)&\sim_\R&\frac{2u-v}{2}\tilde{h}+(2u-v)\tilde{l}_3^++u\tilde{g},
\end{eqnarray*}
and 
\[
(P(u,v))^{\cdot 2}=u^2-\frac{v^2}{2}.
\]
\item
If $v\in[u,2u]$, then we have 
\begin{eqnarray*}
N(u,v)&=&\frac{v}{2}\tilde{h}+(v-u)\tilde{g}, \\
P(u,v)&\sim_\R&\frac{2u-v}{2}\tilde{h}+(2u-v)\tilde{l}_3^++(2u-v)\tilde{g},
\end{eqnarray*}
and 
\[
(P(u,v))^{\cdot 2}=\frac{1}{2}(2u-v)^2.
\]
\end{itemize}
\item
Assume that $u\in[1,2]$. 
\begin{itemize}
\item
If $v\in\left[0,1\right]$, then we have 
\begin{eqnarray*}
N(u,v)&=&\frac{v}{2}\tilde{h}, \\
P(u,v)&\sim_\R&\frac{1+u-v}{2}\tilde{h}+(1+u-v)\tilde{l}_3^++u\tilde{g},
\end{eqnarray*}
and \[
(P(u,v))^{\cdot 2}=\frac{1}{2}\left(-1+2u+u^2+2v-2uv-v^2\right).
\]
\item
If $v\in\left[1,1+u\right]$, then we have 
\begin{eqnarray*}
N(u,v)&=&\frac{v}{2}\tilde{h}+(v-1)\tilde{g}, \\
P(u,v)&\sim_\R&\frac{1+u-v}{2}\tilde{h}+(1+u-v)\tilde{l}_3^++(1+u-v)\tilde{g},
\end{eqnarray*}
and \[
(P(u,v))^{\cdot 2}=\frac{1}{2}(1+u-v)^2.
\] 
\end{itemize}
\item
Assume that $u\in[2,3]$. 
\begin{itemize}
\item
If $v\in\left[0,1\right]$, then we have 
\begin{eqnarray*}
N(u,v)&=&\frac{v}{2}\tilde{h}, \\
P(u,v)&\sim_\R&\frac{3-v}{2}\tilde{h}+(3-v)\tilde{l}_3^++2\tilde{g},
\end{eqnarray*}
and \[
(P(u,v))^{\cdot 2}=\frac{1}{2}(7-2v-v^2).
\]
\item
If $v\in\left[1,3\right]$, then we have 
\begin{eqnarray*}
N(u,v)&=&\frac{v}{2}\tilde{h}+(v-1)\tilde{g}, \\
P(u,v)&\sim_\R&\frac{3-v}{2}\tilde{h}+(3-v)\tilde{l}_3^++(3-v)\tilde{g},
\end{eqnarray*}
and 
\[
(P(u,v))^{\cdot 2}=\frac{1}{2}(3-v)^2.
\] 
\end{itemize}
\item
Assume that $u\in[3,4]$. 
\begin{itemize}
\item
If $v\in\left[0,u-3\right]$, then we have 
\begin{eqnarray*}
N(u,v)&=&0, \\
P(u,v)&\sim_\R&\frac{6-u}{2}\tilde{h}+(3-v)\tilde{l}_3^++2\tilde{g},
\end{eqnarray*}
and \[
(P(u,v))^{\cdot 2}=\frac{1}{2}\left(-2+6u-u^2-8v+2uv-2v^2\right).
\]
\item
If $v\in\left[u-3,1\right]$, then we have 
\begin{eqnarray*}
N(u,v)&=&\frac{-u+3+v}{2}\tilde{h}, \\
P(u,v)&\sim_\R&\frac{3-v}{2}\tilde{h}+(3-v)\tilde{l}_3^++2\tilde{g},
\end{eqnarray*}
and \[
(P(u,v))^{\cdot 2}=\frac{1}{2}(7-2v-v^2).
\]
\item
If $v\in\left[1,3\right]$, then we have 
\begin{eqnarray*}
N(u,v)&=&\frac{-u+3+v}{2}\tilde{h}+(v-1)\tilde{g}, \\
P(u,v)&\sim_\R&\frac{3-v}{2}\tilde{h}+(3-v)\tilde{l}_3^++(3-v)\tilde{g},
\end{eqnarray*}
and 
\[
(P(u,v))^{\cdot 2}=\frac{1}{2}(3-v)^2.
\] 
\end{itemize}
\end{itemize}
Hence we get
\begin{eqnarray*}
&&S\left(V^{D^1}_{\bullet,\bullet};\tilde{l}_3^+\right)\\
&=&\frac{3}{28}\Biggl(
\int_0^1\biggl(\int_0^u\left(u^2-\frac{v^2}{2}\right)dv
+\int_u^{2u}\frac{1}{2}(2u-v)^2dv\biggr)du\\
&+&\int_1^2\biggl((u-1)\frac{1}{2}(-1+2u+u^2)\\
&&+\int_0^1\frac{1}{2}(-1+2u+u^2+2v-2uv-v^2)dv
+\int_1^{1+u}\frac{1}{2}(1+u-v)^2dv\biggr)du\\
&+&\int_2^3\biggl((u-1)\frac{7}{2}+
\int_0^1\frac{1}{2}\left(7-2v-v^2\right)dv
+\int_1^3 \frac{1}{2}(3-v)^2dv\biggr)du\\
&+&\int_3^4\biggl((u-1)\frac{1}{2}(-2+6u-u^2)+
\int_0^{u-3}\frac{1}{2}\left(-2+6u-u^2-8v+2uv-2v^2\right)dv\\
&&+\int_{u-3}^1\frac{1}{2}(7-2v-v^2)dv
+\int_1^3 \frac{1}{2}(3-v)^2dv\biggr)du\Biggr)=\frac{309}{112}.
\end{eqnarray*}
Moreover, 
\begin{eqnarray*}
F_{p'_h}\left(W^{D',l^{'+}_3}_{\bullet,\bullet,\bullet}\right)
=\frac{6}{28}\Biggl(\int_3^4\int_0^{u-3}\frac{4-u+2v}{2}\cdot\frac{u-3-v}{2}dvdu\Biggr)
=\frac{3}{448}, 
\end{eqnarray*}
\begin{eqnarray*}
F_{p'_e}\left(W^{D',l^{'+}_3}_{\bullet,\bullet,\bullet}\right)
&=&\frac{6}{28}\Biggl(\int_2^3\biggl(\int_0^1\frac{1+v}{2}\cdot\frac{u-2}{2}dv
+\int_1^3\frac{3-v}{2}\cdot\frac{u-2}{2}dv\biggr)du\\
&+&\int_3^4\biggl(\int_0^{u-3}\frac{4-u+2v}{2}\cdot\frac{u-2}{2}dv
+\int_{u-3}^1\frac{1+v}{2}\cdot\frac{u-2}{2}dv\\
&&+\int_1^3\frac{3-v}{2}\cdot\frac{u-2}{2}dv\biggr)du\Biggr)
=\frac{23}{64}, 
\end{eqnarray*}
\begin{eqnarray*}
F_{p'_g}\left(W^{D',l^{'+}_3}_{\bullet,\bullet,\bullet}\right)
&=&\frac{6}{28}\Biggl(\int_0^1\int_u^{2u}\frac{2u-v}{2}(v-u)dvdu
+\int_1^2\int_1^{1+u}\frac{1+u-v}{2}(v-1)dvdu\\
&+&\int_2^4\int_1^3\frac{3-v}{2}(v-1)dvdu\Biggr)
=\frac{5}{14}.
\end{eqnarray*}
Therefore, for any $p'\in l^{'+}_3$, we have 
\begin{eqnarray*}
&&S\left(W^{D',l^{'+}_3}_{\bullet,\bullet,\bullet};p'\right)\\
&=&F_{p'}\left(W^{D',l^{'+}_3}_{\bullet,\bullet,\bullet}\right)
+\frac{3}{28}\Biggl(\int_0^1\biggl(\int_0^u\left(\frac{v}{2}\right)^2dv
+\int_u^{2u}\left(\frac{2u-v}{2}\right)^2dv\biggr)du\\
&+&\int_1^2\biggl(\int_0^1\left(\frac{-1+u+v}{2}\right)^2dv
+\int_1^{1+u}\left(\frac{1+u-v}{2}\right)^2dv\biggr)du\\
&+&\int_2^3\biggl(\int_0^1\left(\frac{1+v}{2}\right)^2dv
+\int_1^3\left(\frac{3-v}{2}\right)^2dv\biggr)du\\
&+&\int_3^4\biggl(\int_0^{u-3}\left(\frac{4-u+2v}{2}\right)^2dv
+\int_{u-3}^1\left(\frac{1+v}{2}\right)^2dv
+\int_1^3\left(\frac{3-v}{2}\right)^2dv\biggr)du\Biggr)\\
&=&F_{p'}\left(W^{D',l^{'+}_3}_{\bullet,\bullet,\bullet}\right)
+\frac{21}{64}\begin{cases}
=\frac{75}{224} & \text{if }p'=p'_h, \\
\leq\frac{11}{16} & \text{otherwise.}
\end{cases} 
\end{eqnarray*}
Therefore, we get 
\begin{eqnarray*}
\delta_{p_{l_3^+}}\left(D^1;V^{D^1}_{\bullet,\bullet}\right)
&\geq&\min\left\{\frac{A_{D^1}\left(\tilde{l}^+_3\right)}
{S\left(V^{D^1}_{\bullet,\bullet};\tilde{l}^+_3\right)},\,\,\,
\inf_{p'\in\tilde{l}^{'+}_3}
\frac{A_{l^{'+}_3,\frac{1}{2}p'_h}\left(p'\right)}
{S\left(W^{D',l^{'+}_3}_{\bullet,\bullet,\bullet};p'\right)}\right\}\\
&=&\min\left\{\frac{112}{103},\,\,\,\frac{112}{75},\,\,\,\frac{16}{11}\right\}
=\frac{112}{103}
\end{eqnarray*}
by Corollary \ref{reduction_corollary}.

\noindent\underline{\textbf{Step 9}}\\
Let $p^1\in D^1$ be any closed point with $p^1\not\in g^1$. Take the line 
$r^1\subset D^1$ passing through $p^1$ and $p_{l_3^+}$, and let us set 
$\tilde{r}:=\gamma_*^{-1}r^1$. Note that 
$\tilde{r}\sim \tilde{l}_3^++\tilde{g}$ and 
$\gamma^*r^1=\tilde{r}+\tilde{h}+\tilde{l}_3^+$. 
Set $q_e:=e^1|_{r^1}$. Then $q_e\in r^1$ is a reduced point with $q_e\neq p_{l_3^+}$. 
Let us set 
\begin{eqnarray*}
P(u,v)&:=&P_\sigma\left(\tilde{D}, P(u)-v\tilde{r}\right), \\
N(u,v)&:=&N_\sigma\left(\tilde{D}, P(u)-v\tilde{r}\right).
\end{eqnarray*}

\begin{itemize}
\item
Assume that $u\in[0,1]$. 
\begin{itemize}
\item
If $v\in[0,u]$, then we have 
\begin{eqnarray*}
N(u,v)&=&v\tilde{h}+v\tilde{l}_3^+, \\
P(u,v)&\sim_\R&(u-v)\tilde{h}+(2u-2v)\tilde{l}_3^++(u-v)\tilde{g},
\end{eqnarray*}
and 
$(P(u,v))^{\cdot 2}=(u-v)^2$. 
\end{itemize}
\item
Assume that $u\in[1,2]$. 
\begin{itemize}
\item
If $v\in\left[0,u-1\right]$, then we have 
\begin{eqnarray*}
N(u,v)&=&\frac{v}{2}\tilde{h}, \\
P(u,v)&\sim_\R&\frac{1+u-v}{2}\tilde{h}+(1+u-v)\tilde{l}_3^++(u-v)\tilde{g},
\end{eqnarray*}
and \[
(P(u,v))^{\cdot 2}=\frac{1}{2}\left(-1+2u+u^2-2v-2uv+v^2\right).
\]
\item
If $v\in\left[u-1,u\right]$, then we have 
\begin{eqnarray*}
N(u,v)&=&\frac{1-u+2v}{2}\tilde{h}+(1-u+v)\tilde{l}_3^+, \\
P(u,v)&\sim_\R&(u-v)\tilde{h}+(2u-2v)\tilde{l}_3^++(u-v)\tilde{g},
\end{eqnarray*}
and 
$(P(u,v))^{\cdot 2}=(u-v)^2$. 
\end{itemize}
\item
Assume that $u\in[2,3]$. 
\begin{itemize}
\item
If $v\in\left[0,1\right]$, then we have 
\begin{eqnarray*}
N(u,v)&=&\frac{v}{2}\tilde{h}, \\
P(u,v)&\sim_\R&\frac{3-v}{2}\tilde{h}+(3-v)\tilde{l}_3^++(2-v)\tilde{g},
\end{eqnarray*}
and \[
(P(u,v))^{\cdot 2}=\frac{1}{2}(7-6v+v^2).
\]
\item
If $v\in\left[1,2\right]$, then we have 
\begin{eqnarray*}
N(u,v)&=&\frac{-1+2v}{2}\tilde{h}+(v-1)\tilde{l}_3^+, \\
P(u,v)&\sim_\R&(2-v)\tilde{h}+(4-2v)\tilde{l}_3^++(2-v)\tilde{g},
\end{eqnarray*}
and 
$(P(u,v))^{\cdot 2}=(2-v)^2$. 
\end{itemize}
\item
Assume that $u\in[3,4]$. 
\begin{itemize}
\item
If $v\in\left[0,u-3\right]$, then we have 
\begin{eqnarray*}
N(u,v)&=&0, \\
P(u,v)&\sim_\R&\frac{6-u}{2}\tilde{h}+(3-v)\tilde{l}_3^++(2-v)\tilde{g},
\end{eqnarray*}
and \[
(P(u,v))^{\cdot 2}=\frac{1}{2}\left(-2+6u-u^2-12v+2uv\right).
\]
\item
If $v\in\left[u-3,1\right]$, then we have 
\begin{eqnarray*}
N(u,v)&=&\frac{-u+3+v}{2}\tilde{h}, \\
P(u,v)&\sim_\R&\frac{3-v}{2}\tilde{h}+(3-v)\tilde{l}_3^++(2-v)\tilde{g},
\end{eqnarray*}
and \[
(P(u,v))^{\cdot 2}=\frac{1}{2}(7-6v+v^2).
\]
\item
If $v\in\left[1,2\right]$, then we have 
\begin{eqnarray*}
N(u,v)&=&\frac{2-u+2v}{2}\tilde{h}+(v-1)\tilde{l}_3^+, \\
P(u,v)&\sim_\R&(2-v)\tilde{h}+(4-2v)\tilde{l}_3^++(2-v)\tilde{g},
\end{eqnarray*}
and 
$(P(u,v))^{\cdot 2}=(2-v)^2$.
\end{itemize}
\end{itemize}
Hence we get
\begin{eqnarray*}
&&S\left(V^{D^1}_{\bullet,\bullet};r^1\right)\\
&=&\frac{3}{28}\Biggl(
\int_0^1\int_0^u(u-v)^2dvdu\\
&+&\int_1^2\biggl(\int_0^{u-1}\frac{1}{2}(-1+2u+u^2-2v-2uv+v^2)dv
+\int_{u-1}^{u}(u-v)^2dv\biggr)du\\
&+&\int_2^3\biggl(\int_0^1\frac{1}{2}\left(7-6v+v^2\right)dv
+\int_1^2 (2-v)^2dv\biggr)du\\
&+&\int_3^4\biggl(\int_0^{u-3}\frac{1}{2}\left(-2+6u-u^2-12v+2uv\right)dv\\
&&+\int_{u-3}^1\frac{1}{2}(7-6v+v^2)dv
+\int_1^2 (2-v)^2dv\biggr)du\Biggr)=\frac{75}{112}.
\end{eqnarray*}
Moreover, 
\begin{eqnarray*}
&&F_{q_e}\left(W^{D^1,r^1}_{\bullet,\bullet,\bullet}\right)\\
&=&\frac{6}{28}\Biggl(\int_2^3\biggl(\int_0^1\frac{3-v}{2}\cdot\frac{u-2}{2}dv
+\int_1^2 (2-v)\frac{u-2}{2}dv\biggr)du\\
&+&\int_3^4\biggl(\int_0^{u-3}\frac{6-u}{2}\cdot\frac{u-2}{2}dv
+\int_{u-3}^1\frac{3-v}{2}\cdot\frac{u-2}{2}dv
+\int_1^2(2-v)\frac{u-2}{2}dv\biggr)du\Biggr)
=\frac{23}{64}. 
\end{eqnarray*}
Thus we have 
\begin{eqnarray*}
&&S\left(W^{D^1,r^1}_{\bullet,\bullet,\bullet};p^1\right)\\
&\leq&\frac{23}{64}
+\frac{3}{28}\Biggl(\int_0^1\int_0^u(u-v)^2dvdu
+\int_1^2\biggl(\int_0^{u-1}\left(\frac{u+1-v}{2}\right)^2dv
+\int_{u-1}^u\left(u-v\right)^2dv\biggr)du\\
&+&\int_2^3\biggl(\int_0^1\left(\frac{3-v}{2}\right)^2dv
+\int_1^2\left(2-v\right)^2dv\biggr)du\\
&+&\int_3^4\biggl(\int_0^{u-3}\left(\frac{6-u}{2}\right)^2dv
+\int_{u-3}^1\left(\frac{3-v}{2}\right)^2dv
+\int_1^2\left(2-v\right)^2dv\biggr)du\Biggr)\\
&=&\frac{23}{64}+\frac{227}{448}=\frac{97}{112}.
\end{eqnarray*}
Therefore, we get 
\begin{eqnarray*}
\delta_{p^1}\left(D^1;V^{D^1}_{\bullet,\bullet}\right)\geq
\min\left\{\frac{A_{D^1}\left(r^1\right)}{S\left(V^{D^1}_{\bullet,\bullet};r^1\right)},\,\,\,
\frac{A_{r^1}\left(p^1\right)}{S\left(W^{D^1,r^1}_{\bullet,\bullet,\bullet};p^1\right)}\right\}
=\min\left\{\frac{112}{75},\,\,\,\frac{112}{97}\right\}=\frac{112}{97}
\end{eqnarray*}
by Corollary \ref{reduction_corollary}.

\noindent\underline{\textbf{Step 10}}\\
By Steps 7, 8 and 9, we get 
\[
\delta\left(D^1;V^{D^1}_{\bullet,\bullet}\right)\geq
\min\left\{\frac{448}{307},\,\,\,\frac{112}{103},\,\,\,\frac{112}{97}\right\}
=\frac{112}{103}. 
\]
Thus, by Step 5, we have 
\[
\delta_p(X)\geq\min\left\{\frac{A_X\left(D^1\right)}{S_X\left(D^1\right)},\,\,\,
\delta\left(D^1;V^{D^1}_{\bullet,\bullet}\right)\right\}
=\min\left\{\frac{336}{289},\,\,\,\frac{112}{103}\right\}=\frac{112}{103}
\]
by Corollary \ref{reduction_corollary}.
\end{proof}

\begin{example}\label{112103_example}
Assume that $X$ satisfies Remark \ref{infl_remark} \eqref{infl_remark2}. 
Then, the divisor $Q\in|H_1|$ with $q\in Q$ is singular, and the singular point 
$p\in Q$ satisfies the assumptions in Proposition \ref{112103_proposition}. 
(See Claim \ref{6463_claim} in Theorem \ref{6463_thm}, Step 1.) 
\end{example}

\section{Local $\delta$-invariants for special points, II}\label{special-II_section}

\begin{thm}\label{6463_thm}
Let $X$ be as in \S \ref{311_section}. 
Assume that $X$ satisfies Remark \ref{infl_remark} \eqref{infl_remark2}. 
Then we have the equality 
\[
\delta_q(X)=\frac{64}{63}. 
\]
\end{thm}

\begin{proof}
The following proof is divided into 18 numbers of steps since the proof is very long.

\noindent\underline{\textbf{Step 1}}\\
Let $Q\in |H_1|$ be the divisor with $q\in Q$, and let $T\in |H_2|$ be the 
pull-back of the tangent line of $\sC\subset \pr^2$ at $q^\sC\in\sC$. 

\begin{claim}\label{6463_claim}
\begin{enumerate}
\renewcommand{\theenumi}{\arabic{enumi}}
\renewcommand{\labelenumi}{(\theenumi)}
\item\label{6463_claim1}
The divisor $Q$ is singular and the point $q\in Q$ is the intersection of the 
$2$ negative curves. 
One is $l\subset Q$ with $\left(l^{\cdot 2}\right)=-1/2$. 
\item\label{6463_claim2}
Let us set 
\[
T^V:=\left(\sigma_1\right)_*T\left(\simeq\pr_{\pr^1}(\sO\oplus\sO(1))\right), 
\]
and let $s_0^V\subset T^V$ be the $(-1)$-curve. The divisor $T$ has one 
$A_2$-singularity and the minimal resolution of $T$ is obtained by the composition of: 
\begin{enumerate}
\renewcommand{\theenumii}{\roman{enumii}}
\renewcommand{\labelenumii}{(\theenumii)}
\item\label{6463_claim21}
the blowup of $T^V$ at the intersection of $l^V$ and $s_0^V$, 
\item\label{6463_claim22}
the blowup at a point in the exceptional divisor of the morphism \eqref{6463_claim21} 
which does not pass through the strict transforms of $l^V$, $s_0^V$, and 
\item\label{6463_claim23}
the blowup at a point in the exceptional divisor of the morphism \eqref{6463_claim22} 
which does not pass through the strict transform of the 
exceptional divisor of the morphism \eqref{6463_claim21}.
\end{enumerate}
Moreover, the strict transform $s_0\subset X$ of $s_0^V$ is the negative curve 
in $Q$ with $\left(s_0^{\cdot 2}\right)_Q=-1$, and the exceptional divisor of 
the morphism \eqref{6463_claim23} corresponds to the $(-1)$-curve in $E_2$. 
\end{enumerate}
\end{claim}

\begin{proof}[Proof of Claim \ref{6463_claim}]
\eqref{6463_claim1}
If $Q$ is smooth, then $Q$ is isomorphic to a del Pezzo surface of degree $7$,  
$E_3|_Q$ is the disjoint union of two $(-1)$-curves on $Q$, the morphism 
$Q\to\left(\sigma_1\right)_*Q$ is an isomorphism, and the morphism 
$\pi^V\circ\sigma_1\colon Q\to \pr^2$ is birational and contracts $E_3|_Q$. 
On the other hand, $Q|_{E_2}$ is the fiber of the $\pr^1$-fibration $E_2/\pr^1$ 
passing through the point $q$. Moreover, $E_2|_Q$ is the $(-1)$-curve on $Q$ 
such that $\Supp\left(E_2|_Q\right)\not\subset\Supp\left(E_3|_Q\right)$. 
Thus, the image of $E_2|_Q$ in $\pr^2$ is a line passing through $\sC$ at 
$2$ points. This contradicts to the assumption in Remark \ref{infl_remark}
\eqref{infl_remark2}. Thus $Q$ is singular. Since $s_0:=E_2|_Q$ is a $(-1)$-curve 
in $Q$ and a fiber of $E_2/\pr^1$, and since $l\subset Q$, the remaining 
assertions are trivial. 

\eqref{6463_claim2}
Set $T^P:=\left(\sigma^V\right)_*T^V$. Then the morphism $T\to T^P$ is the 
blowup of the plane $T^P$ along the subscheme $\sC^P\cap T^P$. 
From the assumption, we have 
$\left(\sC^P\cap T^P\right)_{\operatorname{red}}=\{p^P\}$. 
For a general quadric $Q_{\gen}^P\subset P$ with $\sC^P\subset Q_{\gen}^P$, 
the scheme-theoretic intersection $Q_{\gen}^P\cap T^P$ is a smooth conic. 
(Indeed, if not, for any quadric $Q^P\subset P$ with $\sC^P\subset Q^P$, 
the intersection $Q^P\cap T^P$ is a union of two lines in $T^P$ passing through 
$p^P$. This implies that $E_3|_{Q_{\gen}}\subset T$ for a general 
$Q_{\gen}\in|H_1|$ since $E_3|_{Q_{\gen}}\subset Q_{\gen}$ is a disjoint union of 
two $(-1)$-curves. However, this implies that $E_3\subset T$, a contradiction.)
Thus the scheme $\sC^P\cap T^P\subset T^P$ is of length $4$ and is contained 
in a smooth conic $Q_{\gen}^P\cap T^P$. 
Thus the assertions follows from \cite[Lemmas 2.3 and 2.4]{NdP}. 
\end{proof}

\noindent\underline{\textbf{Step 2}}\\
Let 
\[
\rho_0\colon X_0\to X
\]
be the blowup of $X$ at $q\in X$ and let $F^0\subset X_0$ be the exceptional 
divisor. Let us set 
$Q^0:=(\rho_0)^{-1}_*Q$, 
$T^0:=(\rho_0)^{-1}_*T$, 
$E_2^0:=(\rho_0)^{-1}_*E_2$, and 
$l^0:=(\rho_0)^{-1}_*l$, 
$s_0^0:=(\rho_0)^{-1}_*s_0$, where $s_0\subset X$ is as in Claim \ref{6463_claim}. 
By Claim \ref{6463_claim}, the divisor $T^0$ is the minimal resolution of $T$. 
Let $t_1^0$, $t_2^0$, $t_3^0\subset T^0$ be the exceptional divisors of 
the morphisms \eqref{6463_claim21}, \eqref{6463_claim22}, \eqref{6463_claim23} 
in Claim \ref{6463_claim} \eqref{6463_claim2}, respectively. 
The morphism $E_2^0\to E_2$ is the blowup at $q\in E_2$. Let $e_2^0\subset E_2^0$ 
be the exceptional divisor. 
Then, 
\[
t_1^0,\,\,t_2^0,\,\,e_2^0\subset F^0\left(\simeq\pr^2\right)
\]
are mutually distinct lines. 
Moreover, since $t_1^0\subset F^0$ is the line passing through 
$s_0^0\cap F^0$ and $l^0\cap F^0$, the morphism $Q^0\to Q$ is the blowup at 
$q\in Q$ such that the exceptional divisor is nothing but the curve $t_1^0$. 
Since $l\subset E_3$ and $E_3$ tangents to $s_0$, we get 
$\left(\left(\rho_0\right)^{-1}_*E_3\right)|_{F^0}=t_1^0$. 
Note that 
\begin{eqnarray*}
\left(\rho_0\right)^*Q=Q^0+F^0,\quad
\left(\rho_0\right)^*T=T^0+2F^0,\quad
\left(\rho_0\right)^*E_2=E_2^0+F^0,\quad
-K_X\sim Q+2T+E_2.
\end{eqnarray*}
In particular, we get 
\[
\left(\rho_0\right)^*(-K_X)\sim Q^0+2T^0+E_2^0+6F^0.
\]

\noindent\underline{\textbf{Step 3}}\\
Let 
\[
\psi_{11}\colon X_{11}\to X_0
\]
be the blowup along the curve $t_1^0\subset X_0$ and let $R^{11}\subset X_{11}$ 
be the exceptional divisor. Since $\sN_{t_1^0/X_0}
\simeq\sO_{\pr^1}(1)\oplus\sO_{\pr^1}(-1)$, the divisor $R^{11}$ is 
isomorphic to $\pr_{\pr^1}(\sO\oplus\sO(2))$. 
Set 
$Q^{11}:=\left(\psi_{11}\right)^{-1}_*Q^0$, 
$T^{11}:=\left(\psi_{11}\right)^{-1}_*T^0$, 
$E_2^{11}:=\left(\psi_{11}\right)^{-1}_*E_2^0$, 
$F^{11}:=\left(\psi_{11}\right)^{-1}_*F^0$. 
Note that $A_X\left(R^{11}\right)=4$. Moreover, since 
\begin{eqnarray*}
\left(\rho_0\circ\psi_{11}\right)^*Q=Q^{11}+F^1+2R^{11}, \quad
\left(\rho_0\circ\psi_{11}\right)^*T=T^{11}+2F^1+3R^{11}, \\
\left(\rho_0\circ\psi_{11}\right)^*E_2=E_2^{11}+F^1+R^{11}, \quad
\left(\psi_{11}\right)^*F^0=F^{11}+R^{11}, 
\end{eqnarray*}
we have $\tau_X\left(R^{11}\right)\geq 9$. 

The morphisms 
\[
\psi_{11}|_{Q^{11}}\colon Q^{11}\to Q^0, \quad
\psi_{11}|_{T^{11}}\colon T^{11}\to T^0, \quad
\psi_{11}|_{F^{11}}\colon F^{11}\to F^0
\] 
are isomorphisms. Let 
\[
r_1^{11}\subset Q^{11},\quad r_2^{11}\subset T^{11},\quad s_R^{11}\subset F^{11}
\] 
be the inverse images of $t_1^0$, respectively. 
Moreover, let us set 
$l^{11}:=\left(\psi_{11}\right)^{-1}_*l^0$, 
$s_0^{11}:=\left(\psi_{11}\right)^{-1}_*s_0^0$, 
$t_2^{11}:=\left(\psi_{11}\right)^{-1}_*t_2^0$, 
$t_3^{11}:=\left(\psi_{11}\right)^{-1}_*t_3^0$, 
$e_2^{11}:=\left(\psi_{11}\right)^{-1}_*e_2^0$. 
The morphism $E_2^{11}\to E_2^0$ is the blowup of at the point $s_0^0\cap e_2^0$. 
Let $f_2^{11}\subset E_2^{11}$ be the exceptional divisor. 
On the surface $R^{11}\simeq\pr_{\pr^1}(\sO\oplus\sO(2))$, 
\begin{itemize}
\item
$r_1^{11}\subset R^{11}$ is a section of $R^{11}/t_1^0$ with 
$\left(r_1^{11}\right)_{R^{11}}^{\cdot 2}=2$, 
\item
$r_2^{11}\subset R^{11}$ is a section of $R^{11}/t_1^0$ with 
$\left(r_2^{11}\right)_{R^{11}}^{\cdot 2}=4$, 
\item
$f_2^{11}\subset R^{11}$ is a fiber of $R^{11}/t_1^0$, and 
\item
$s_R^{11}\subset R^{11}$ is the $(-2)$-curve on $R^{11}$. 
\end{itemize}
Note that $\left(r_1^{11}\cdot r_2^{11}\right)_{R^{11}}=3$ and $T|_Q=2l+s_0$. 
Thus we have: 
\begin{itemize}
\item
$r_1^{11}$ and $r_2^{11}$ on $R^{11}$ meet at the two points $l^{11}\cap R^{11}$ and 
$s_0^{11}\cap R^{11}$. Moreover, we have
\[
\length_{l^{11}\cap R^{11}}\left(r_1^{11}\cap r_2^{11}\right)=2,\quad 
\length_{s_0^{11}\cap R^{11}}\left(r_1^{11}\cap r_2^{11}\right)=1, 
\]
\item
$r_2^{11}$ and $s_R^{11}$ on $R^{11}$ meet transversely at the point 
$t_2^{11}\cap R^{11}$, and 
\item
$f_2^{11}$ on $R^{11}$ passes through the points $s_0^{11}\cap R^{11}$ 
and $e_2^{11}\cap R^{11}$. 
\end{itemize}
Let $f_R^{11}$, $f_S^{11}\subset R^{11}$ be the fibers of $R^{11}/t_1^0$ 
passing through the points $l^{11}\cap R^{11}$, $t_2^{11}\cap R^{11}$, 
respectively. 
Moreover, let $r_3^{11}\subset R^{11}$ be the unique section of $R^{11}/t_1^0$ 
with $\left(r_3^{11}\right)_{R^{11}}^{\cdot 2}=2$ such that $r_2^{11}$ and $r_3^{11}$
on $R^{11}$ meet at the point $l^{11}\cap R^{11}$ of length $3$. 

We remark that 
\begin{eqnarray*}
T^{11}|_{Q^{11}}=2l^{11}+s_0^{11},\quad
E_2^{11}|_{Q^{11}}=s_0^{11},\quad
F^{11}\cap Q^{11}=\emptyset, \\
E_2^{11}|_{T^{11}}=s_0^{11}+t_3^{11},\quad
F^{11}|_{T^{11}}=t_2^{11},\quad
F^{11}|_{E_2^{11}}=e_2^{11}.
\end{eqnarray*}
Let $Q_{\gen}\in |H_1|$ be a general member. Then 
\[
\frac{2}{3}\left(\left(\rho_0\right)^*(-K_X)-F^0\right)\sim_\Q
-\left(K_{X_0}+\frac{1}{3}\left(\rho_0\right)^*Q_{\gen}+\frac{2}{3}T^0
+\frac{1}{3}E_2^0+\frac{1}{3}F^0\right)
\]
is nef and big. Thus, the pair 
\[
\left(X_{11}, \frac{1}{3}\left(\rho_0\circ\psi_{11}\right)^*Q_{\gen}
+\frac{2}{3}T^{11}+\frac{1}{3}E_2^{11}+\frac{1}{3}F^{11}\right)
\]
is a klt pair with the anti-log canonical divisor nef and big. 
Therefore, the variety $X_{11}$ is a Mori dream space 
by \cite[Corollary 1.3.2]{BCHM}.

\noindent\underline{\textbf{Step 4}}\\
We can contract the divisor $F^{11}\subset X_{11}$ to a point and get the morphism 
$\phi_{11}\colon X_{11}\to X_1$ with $X_1$ normal projective and $\Q$-factorial. 
Let 
\[\xymatrix{
 & X_{11} \ar[dl]_{\psi_{11}} \ar[dr]^{\phi_{11}} & \\
X_0 \ar[dr]_{\rho_0} & & X_1 \ar[dl]^{\rho_1} \\
 & X & 
}\]
be the induced diagram. The morphism $\rho_1$ is a weighted blowup of the 
weights $(1,1,2)$, and $R^1:=\left(\phi_{11}\right)_*R^{11}\simeq\pr(1,1,2)$ is the 
exceptional divisor of $\rho_1$. Note that 
$\left(\phi_{11}\right)^*R^1=R^{11}+(1/2)F^{11}$. Let us set 
$Q^1:=\left(\phi_{11}\right)_*Q^{11}$, 
$T^1:=\left(\phi_{11}\right)_*T^{11}$, 
$E_2^1:=\left(\phi_{11}\right)_*E_2^{11}$. 
Then we get 
\[
\left(\rho_1\right)^*(-K_X)-uR^1\sim_\R Q^1+2T^1+E_2^1+(9-u)R^1
\]
and 
\[
\left(\phi_{11}\right)^*\left(\left(\rho_1\right)^*(-K_X)-uR^1\right)
\sim_\R Q^{11}+2T^{11}+E_2^{11}+\frac{12-u}{2}F^{11}+(9-u)R^{11}
\]
for any $u\in \R$. Note that 
$\left(\phi_{11}\right)^*Q^1=Q^{11}$, 
$\left(\phi_{11}\right)^*T^1=T^{11}+(1/2)F^{11}$, 
$\left(\phi_{11}\right)^*E_2^1=E_2^{11}+(1/2)F^{11}$. 
Therefore, for any $u\in(0,9)$ and for any irreducible curve 
$C^1\subset X_1$ with $C^1\not\subset Q^1\cup T^1\cup E_2^1\cup R^1$, we have 
$\left(\left((\rho_1)^*(-K_X)-uR^1\right)\cdot C^1\right)>0$. 

On $X_{11}$, we get the following intersection numbers: 
\begin{center}
\renewcommand\arraystretch{1.4}
\begin{tabular}{c|ccccc|c}
$\cdot$ & $(\phi_{11})^*Q^1$ & $(\phi_{11})^*T^1$ & 
$(\phi_{11})^*E_2^1$ & $(\phi_{11})^*R^1$ & 
$F^{11}$ & 
{\tiny$Q^{11}+2T^{11}+E_2^{11}+\frac{12-u}{2}F^{11}+(9-u)R^{11}$} \\ \hline
{\tiny $f_2^{11},f_R^{11},f_S^{11}$} 
& $1$ & $3/2$ & $1/2$ & $-1/2$ & $1$& $u/2$\\
{\tiny$s_R^{11},e_2^{11},t_2^{11}$} 
& $0$ & $0$ & $0$ & $0$ & $-2$ & $0$\\
{\tiny $r_1^{11},r_3^{11}$} 
& $2$ & $3$ & $1$ & $-1$ & $0$ & $u$\\
{\tiny $r_2^{11}$} 
& $3$ & $9/2$ & $3/2$ & $-3/2$ & $1$& $3u/2$\\
{\tiny $l^{11}$} 
& $-2$ & $-3$ & $0$ & $1$ & $0$& $1-u$\\
{\tiny $s_0^{11}$} 
& $-2$ & $-2$ & $-2$ & $1$ & $0$& $1-u$\\
{\tiny $t_3^{11}$} 
& $0$ & $-3/2$ & $-1/2$ & $1/2$ & $1$& $(2-u)/2$
\end{tabular}
\end{center}
Thus, for $u\in[0,1]$, the $\R$-divisor $(\rho_1)^*(-K_X)-uR^1$ is nef and 
\[
\vol_{X_1}\left((\rho_1)^*(-K_X)-uR^1\right)=28-\frac{u^3}{2}.
\]
(We note that $-R^1|_{R^1}\sim_\Q\sO_{\pr(1,1,2)}(1)$.)
Moreover, for $u=1$, the divisor $(\rho_1)^*(-K_X)-R^1$ contracts 
the strict transforms of $l^{11}$ and $s_0^{11}$ on $X_1$. 
Note that $l^{11}$, $s_0^{11}$ and $F^{11}$ are mutually disjoint, and 
\begin{eqnarray*}
\sN_{l^{11}/X_{11}}\simeq\sO_{\pr^1}(-1)\oplus\sO_{\pr^1}(-3), \quad
\sN_{s_0^{11}/X_{11}}\simeq\sO_{\pr^1}(-2)^{\oplus 2}. 
\end{eqnarray*}
Since $X_{11}$ is a Mori dream space, $X_1$ is also a Mori dream space. 
Moreover, the morphism $\rho_1$ is a plt-blowup of $X$ and 
$(K_{X_1}+R^1)|_{R^1}=K_{R^1}$
holds.

\noindent\underline{\textbf{Step 5}}\\
Let 
\[
\phi_{1112}\colon X_{1112}\to X_{11}
\]
be the blowup along $l^{11}\sqcup s_0^{11}\subset X_{11}$. Let 
$\E_{l,1}^{1112}$, $\E_s^{1112}\subset X_{1112}$ be the exceptional divisors 
over $l^{11}$, $s_0^{11}$, respectively. Note that 
$\E_{l,1}^{1112}\simeq\pr_{\pr^1}(\sO\oplus\sO(2))$ and 
$\E_s^{1112}\simeq\pr^1\times\pr^1$. 
Let $l^{1112}\subset\E_{l,1}^{1112}$ be the $(-2)$-curve. Then, we have 
$\sN_{l^{1112}/X_{1112}}\simeq\sO_{\pr^1}(-1)\oplus\sO_{\pr^1}(-2)$. 
Let 
\[
\phi_{112}\colon X_{112}\to X_{1112}
\]
be the blowup along $l^{1112}\subset X_{1112}$ and let $\E_{l,2}^{112}\subset X_{112}$ 
be the exceptional divisor. Then we have 
$\E_{l,2}^{112}\simeq\pr_{\pr^1}(\sO\oplus\sO(1))$. 
Let $l^{112}\subset \E_{l,2}^{112}$ be the $(-1)$-curve. Then we have 
$\sN_{l^{112}/X_{112}}\simeq\sO_{\pr^1}(-1)^{\oplus 2}$. Moreover, 
$l^{112}$ and the strict transform of $\E_{l,1}^{1112}$ on $X_{112}$ are disjoint. 
We can consider Atiyah's flop
\[
X_{112}\xleftarrow{\phi_{12}}X_{12}\xrightarrow{\phi_{12}^+}X_{122}
\]
from $l^{112}\subset X_{112}$. Let $\E_{l,3}^{12}\subset X_{12}$ be the 
exceptional divisor of $\phi_{12}$, and let $l^{+122}\subset X_{122}$ be the 
image of $\E_{l,3}^{12}$. Note that $X_{122}$ is a complex manifold, 
$l^{+122}\subset X_{122}$ is a smooth rational curve and $\phi^+_{12}$ is the 
blowup along $l^{+122}\subset X_{122}$. We set 
$\phi_1:=\phi_{11}\circ\phi_{1112}\circ\phi_{112}\circ\phi_{12}$. 
Let us set 
\begin{eqnarray*}
\E_s^{12}:=\left(\phi_{112}\circ\phi_{12}\right)^{-1}_*\E_s^{1112}, \quad
\E_{l,1}^{12}:=\left(\phi_{112}\circ\phi_{12}\right)^{-1}_*\E_{l,1}^{1112}, \\
\E_{l,2}^{12}:=\left(\phi_{12}\right)^{-1}_*\E_{l,2}^{112}, \quad
F^{12}:=\left(\phi_{1112}\circ\phi_{112}\circ\phi_{12}\right)^{-1}_*F^{11}.
\end{eqnarray*}
Since $\E_{l,2}^{122}:=\left(\phi_{12}^+\right)_*\E_{l,2}^{12}\simeq\pr^2$ with 
$-\E_{l,2}^{122}|_{\E_{l,2}^{122}}\simeq\sO_{\pr^2}(2)$, we can analytically 
contract $\E_{l,2}^{122}$ to a point. We denote the contraction by 
\[
\phi_{122}\colon X_{122}\to X_{1222}. 
\]
Then $\E_{l,1}^{1222}:=\left(\phi_{122}\circ\phi_{12}^+\right)_*\E_{l,1}^{12}
\simeq\pr(1,1,2)$ is a prime $\Q$-Cartier divisor in the normal analytic space 
$X_{1222}$ with $-\E_{l,1}^{1222}|_{\E_{l,1}^{1222}}\sim_\Q\sO_{\pr(1,1,2)}(3)$. 
Thus, by \cite[Proposition 7.4]{HP}, we can analytically contract 
$\E_{l,1}^{1222}$ to a point. Let 
\[
\phi_{1222}\colon X_{1222}\to X_{22}
\]
be the composition of the contraction of $\E_{l,1}^{1222}$ and the contraction of 
$\E_s^{1222}:=\left(\phi_{122}\circ\phi^+_{12}\right)_*\E_s^{12}$ to $\pr^1$
whose fibration is different from the fibration $\E_s^{1112}/s_0^{11}$. Let 
$s^{+22}\subset X_{22}$ be the image of $\E_s^{1222}$ and let 
$l^{+22}\subset X_{22}$ be the image of the curve $l^{+122}$. 
Finally, let 
\[
\phi_{22}\colon X_{22}\to X_2
\]
be the contraction of 
$F^{22}:=\left(\phi_{1222}\circ\phi_{122}\circ\phi^+_{12}\right)_*F^{12}$ 
to a point. (Note that $F^{11}$ and $l^{11}\sqcup s_0^{11}$ are disjoint.)
Set $\phi^+_1:=\phi_{22}\circ\phi_{1222}\circ\phi_{122}\circ\phi_{12}^+$ and 
$\chi_1:=\phi^+_1\circ(\phi_1)^{-1}$. We get the following commutative diagram: 
\[\xymatrix{
X_{112} \ar[d]_{\phi_{112}}
& X_{12} \ar[l]_{\phi_{12}} \ar[r]^{\phi_{12}^+} \ar[dddl]^{\phi_1} \ar[dddr]_{\phi_1^+}
& X_{122} \ar[d]^{\phi_{122}} \\
X_{1112} \ar[d]_{\phi_{1112}} & & X_{1222} \ar[d]^{\phi_{1222}} \\
X_{11} \ar[d]_{\phi_{11}} & & X_{22} \ar[d]^{\phi_{22}} \\
X_1 \ar@{-->}[rr]_{\chi_1} & & X_2. 
}\]

\noindent\underline{\textbf{Step 6}}\\
Set $Q^{1112}:=(\phi_{1112})^{-1}_*Q^{11}$. Then $Q^{1112}$ is isomorphic to 
the minimal resolution of $Q^{11}$ and the strict transform 
$\left(\phi_{1112}|_{Q^{1112}}\right)^{-1}_* l^{11}$ is equal to $l^{1112}$. 
Let $h_{1,Q}^{1112}\subset Q^{1112}$ be the exceptional divisor over $Q^{11}$. 
Then we have $\E_{l,1}^{1112}|_{Q^{1112}}=l^{1112}+h_{1,Q}^{1112}$. 
Set $Q^{112}:=(\phi_{112})^{-1}_*Q^{1112}$ and $Q^{12}:=(\phi_{12})^{-1}_*Q^{112}$. 
Then we have $Q^{112}\simeq Q^{1112}$, and the curves 
$\left(\phi_{112}|_{Q^{112}}\right)^{-1}_* l^{1112}$ and 
$\left((\phi_{112})^{-1}_*\E_{l,1}^{1112}\right)|_{\E_{l,2}^{112}}$ on $\E_{l,2}^{112}$ 
meet transversally 
at one point in $\left(\phi_{112}|_{Q^{112}}\right)^{-1}_* h_{1,Q}^{1112}$. 
Moreover, $Q^{112}$ and $l^{112}$ are mutually disjoint. In particular,  we have 
$Q^{12}\simeq Q^{112}$. Set 
\[
l_Q^{12}:=\left(\left(\phi_{112}\circ\phi_{12}\right)|_{Q^{12}}\right)^{-1}_* l^{1112}, 
\quad
s_{0,Q}^{12}:=\E_s^{12}|_{Q^{12}}, \quad
h_{1,Q}^{12}:=\left(\left(\phi_{112}\circ\phi_{12}\right)|_{Q^{12}}\right)^{-1}_* 
h_{1,Q}^{1112}.
\]
Let us set $Q^2:=(\phi_1^+)_*Q^{12}\subset X_2$ and 
$Q^{22}:=(\phi_{22})^{-1}_*Q^2\subset X_{22}$. 
Then $Q^2$ is obtained by the 
contractions of the curves $h_{1,Q}^{12}\cup l_Q^{12}$ and $s_{0,Q}^{12}$. 
Thus we have $Q^{22}\simeq Q^2\simeq\pr(1,2,3)$. Moreover, we have
\begin{eqnarray*}
(\phi_1)^*Q^1&=&Q^{12}+\E_{l,1}^{12}+2\E_{l,2}^{12}+2\E_{l,3}^{12}+\E_s^{12},\\
(\phi_1^+)^*Q^2&=&Q^{12}+\frac{1}{3}\E_{l,1}^{12}+\frac{2}{3}\E_{l,2}^{12}.
\end{eqnarray*}

Set $T^{1112}:=(\phi_{1112})^{-1}_*T^{11}$. Then $T^{1112}\simeq T^{11}$ and 
the curve $\left(\phi_{1112}|_{T^{1112}}\right)^{-1}_*l^{11}$ is equal to $l^{1112}$. 
Set $T^{112}:=(\phi_{112})^{-1}_*T^{1112}$ and $T^{12}:=(\phi_{12})^{-1}_*T^{112}$. 
Then we have $T^{12}\simeq T^{112}\simeq T^{1112}$. Moreover, the curve 
$\left(\phi_{112}|_{T^{112}}\right)^{-1}_*l^{1112}$ is equal to $l^{112}$. 
Set $l_T^{12}:=\E_{l,3}^{12}|_{T^{12}}$ and $s_{0,T}^{12}:=\E_s^{12}|_{T^{12}}$. 
The morphism 
$T^{12}\to T^{22}:=\left(\phi_{1222}\circ\phi_{122}\circ\phi_{12}^+\right)_*T^{12}$ is 
the contractions of the curves $l_T^{12}$ and $s_{0,T}^{12}$. Moreover, 
the morphism $T^{22}\to T^2:=(\phi^+_1)_*T^{12}$ is the contraction of the 
strict transform of $t_2^{11}$. We get 
\begin{eqnarray*}
(\phi_1)^*T^1&=&T^{12}+\frac{1}{2}F^{12}
+\E_{l,1}^{12}+2\E_{l,2}^{12}+3\E_{l,3}^{12}+\E_s^{12},\\
(\phi_1^+)^*T^2&=&T^{12}+\frac{1}{2}F^{12}.
\end{eqnarray*}

Set $E_2^{12}:=(\phi_1)^{-1}_*E_2^1$ and $E_2^2:=(\phi_1^+)_*E_2^{12}$. Note that 
$E_2^1$ and $l^{11}$ are disjoint. Thus we have $E_2^{12}\simeq E_2^1$. 
Set $s_{0,E_2}^{12}:=\E_s^{12}|_{E_2}$. Moreover, the morphism 
\[
E_2^{12}\to E_2^{22}:=\left(\phi_{1222}\circ\phi_{122}\circ\phi_{12}^+\right)_*E_2^{12}
\]
is the contraction of the curve $s_{0,E_2}^{12}\subset E_2^{12}$, and the morphism 
$E_2^{22}\to E_2^2$ is the contraction of the strict transform of $e_2^{11}$. 
We get 
\begin{eqnarray*}
(\phi_1)^*E_2^1&=&E_2^{12}+\frac{1}{2}F^{12}+\E_s^{12},\\
(\phi_1^+)^*E_2^2&=&E_2^{12}+\frac{1}{2}F^{12}.
\end{eqnarray*}

Set $R^{1112}:=(\phi_{1112})^{-1}_*R^{11}$. Then $R^{1112}\to R^{11}$ is the blowup 
along the (reduced) points $l^{11}\cap R^{11}$ and $s_0^{11}\cap R^{11}$. 
Let $h_1^{1112}$, $s^{+1112}\subset R^{1112}$ be the exceptional divisors 
over $l^{11}\cap R^{11}\in R^{11}$, $s_0^{11}\cap R^{11}\in R^{11}$, 
respectively. Note that $h_1^{1112}$ and $h_{1,Q}^{1112}$ are disjoint fibers 
of $\E_{l,1}^{1112}/l^{11}$. 
Moreover, for any $1\leq i\leq 3$, the curve 
$\left(\phi_{1112}|_{R^{1112}}\right)^{-1}_*r_i^{11}$ contains the point 
$l^{1112}\cap R^{1112}$. Let us set $R^{112}:=(\phi_{112})^{-1}_*R^{1112}$. Then 
$R^{112}\to R^{1112}$ is the blowup at the (reduced) point $l^{1112}\cap R^{1112}$. 
Let $h_2^{112}\subset R^{112}$ be the exceptional divisor over $R^{1112}$. 
Note that the curves
\[
\left(\left(\phi_{1112}\circ\phi_{112}\right)|_{R^{112}}\right)^{-1}_*r_2^{11}
\text{ and }
\left(\left(\phi_{1112}\circ\phi_{112}\right)|_{R^{112}}\right)^{-1}_*r_3^{11}
\]
transversally meet at the point $l^{112}\cap R^{112}$. 
Set $R^{12}:=(\phi_{12})^{-1}_*R^{112}$. Then $R^{12}\to R^{112}$ is the blowup 
at the (reduced) point $l^{112}\cap R^{112}$. Let $l^{+12}\subset R^{12}$ be the 
exceptional divisor over $R^{112}$. Let 
\[
h_1^{12}, h_2^{12}, s^{+12}, s_R^{12}, f_2^{12}, f_R^{12}, f_S^{12}, r_1^{12}, r_2^{12}, r_3^{12}
\subset R^{12}
\]
be the strict transforms of the curves
\[
h_1^{1112}, h_2^{112}, s^{+1112}, s_R^{11}, f_2^{11}, f_R^{11}, f_S^{11}, 
r_1^{11}, r_2^{11}, r_3^{11},
\]
respectively. The morphism 
\[
R^{12}\to R^{22}:=\left(\phi_{1222}\circ\phi_{122}\circ\phi_{12}^+\right)_*R^{12}
\]
is the contractions of the curves $h_1^{12}$ and $h_2^{12}$. Moreover, the morphism 
$R^{22}\to R^2:=(\phi^+_1)_*R^{12}$ is the contraction of the strict transform of 
$s_R^{12}$. We get 
\begin{eqnarray*}
(\phi_1)^*R^1&=&R^{12}+\frac{1}{2}F^{12},\\
(\phi_1^+)^*R^2&=&R^{12}+\frac{1}{2}F^{12}+
\frac{1}{3}\left(\E_{l,1}^{12}+2\E_{l,2}^{12}+3\E_{l,3}^{12}\right)+\frac{1}{2}\E_s^{12}.
\end{eqnarray*}

Let $t_2^{12}$, $t_3^{12}$, $e_2^{12}\subset X_{12}$ be the strict transforms of 
$t_2^{11}$, $t_3^{11}$, $e_2^{11}\subset X_{11}$, respectively. 
On $X_{12}$, we get the following intersection numbers: 
\begin{center}
\renewcommand\arraystretch{1.4}
\begin{tabular}{c|cccc|c|cccc}
$\cdot$ & $\E_{l,1}^{12}$ & $\E_{l,2}^{12}$ & 
$\E_{l,3}^{12}$ & $\E_s^{12}$ & $F^{12}$ & $(\phi_1^+)^*Q^2$ & 
$(\phi_1^+)^*T^2$ & $(\phi_1^+)^*E_2^2$ & $(\phi_1^+)^*R^2$\\ \hline
$f_2^{12}$ & $0$ & $0$ & $0$ & $1$ & $1$ & $0$ & $1/2$ & $-1/2$ & $0$ \\
$f_R^{12}$ & $1$ & $0$ & $0$ & $0$ & $1$ & $1/3$ & $1/2$ & $1/2$ & $-1/6$ \\
$f_S^{12}$ & $0$ & $0$ & $0$ & $0$ & $1$ & $1$ & $3/2$ & $1/2$ & $-1/2$ \\
{\tiny $s_R^{12},e_2^{12},t_2^{12}$} & 
$0$ & $0$ & $0$ & $0$ & $-2$ & $0$ & $0$ & $0$ & $0$ \\
$r_1^{12}$ & $0$ & $1$ & $0$ & $1$ & $0$ & $-1/3$ & $0$ & $0$ & $1/6$ \\
$r_2^{12}$ & $0$ & $0$ & $1$ & $1$ & $1$ & $0$ & $1/2$ & $1/2$ & $0$ \\
$r_3^{12}$ & $0$ & $0$ & $1$ & $0$ & $0$ & $0$ & $0$ & $1$ & $0$ \\
$l^{+12}$ & $0$ & $1$ & $-1$ & $0$ & $0$ & $2/3$ & $1$ & $0$ & $-1/3$ \\
$s^{+12}$ & $0$ & $0$ & $0$ & $-1$ & $0$ & $1$ & $1$ & $1$ & $-1/2$ \\
$t_3^{12}$ & $0$ & $0$ & $0$ & $0$ & $1$ & $0$ & $-3/2$ & $-1/2$ & $1/2$ 
\end{tabular}
\end{center}
Thus, for $u\in(1,2)$, the $\R$-divisor $Q^2+2T^2+E_2^2+(9-u)R^2$ is ample. 
Hence $X_2$ is projective and $\chi_1$ is a small $\Q$-factorial modification 
of $X_1$. Moreover, for $u=2$, the divisor $Q^2+2T^2+E_2^2+(9-2)R^2$ 
contracts the curve $t_3^2:=(\phi_1^+)_*t_3^{12}$. 
For $u\in[1,2]$ we get 
\begin{eqnarray*}
&&\vol_{X_1}\left((\rho_1)^*(-K_X)-uR^1\right)\\
&=&\left(Q^2+2T^2+E_2^2+(9-u)R^2\right)^{\cdot 3}\\
&=&\left((\phi_1)^*\left(Q^1+2T^1+E_2^1+(9-u)R^1\right)+
\frac{1-u}{3}\left(\E_{l,1}^{12}+2\E_{l,2}^{12}+3\E_{l,3}^{12}\right)
+\frac{1-u}{2}\E_s^{12}\right)^{\cdot 3}\\
&=&28-\frac{u^3}{2}+\frac{7}{12}(u-1)^3.
\end{eqnarray*}

\noindent\underline{\textbf{Step 7}}\\
Let $t_2^{22}$, $t_3^{22}\subset X_{22}$ be the strict transforms of 
$t_2^{12}$, $t_3^{12}\subset X_{12}$, respectively. Since 
$F^{12}\cup t_3^{12}$ and 
$\E_{l,1}^{12}\cup\E_{l,2}^{12}\cup\E_{l,3}^{12}\cup\E_s^{12}$ are mutually disjoint, 
the variety $X_{22}$ is smooth around a neighborhood of $F^{22}\cup t_3^{22}$. 
Note that $\sN_{t_3^{22}/X_{22}}\simeq\sO_{\pr^1}(-1)\oplus\sO_{\pr^1}(-2)$ 
by focusing on $T^{22}$. Let 
\[
\phi_{2223}\colon X_{2223}\to X_{22}
\]
be the blowup along $t_3^{22}\subset X_{22}$ and let $\D_{t,1}^{2223}\subset X_{2223}$ 
be the exceptional divisor. Note that 
$\D_{t,1}^{2223}\simeq\pr_{\pr^1}(\sO\oplus\sO(1))$. 
Let $t_3^{2223}\subset \D_{t,1}^{2223}$ be the $(-1)$-curve. 
Note that $T^{2223}:=(\phi_{2223})^{-1}_*T^{22}\to T^{22}$ is an isomorphism 
and $\left(\phi_{2223}|_{T^{2223}}\right)^{-1}_*t_3^{22}=t_3^{2223}$. 
Thus 
\[F^{2223}:=(\phi_{2223})^{-1}_*F^{22}
\simeq\pr_{\pr^1}(\sO\oplus\sO(1))
\]
and 
$t_2^{2223}:=(\phi_{2223})^{-1}_*t_2^{22}\subset F^{2223}$ is the fiber 
of the $\pr^1$-fibration of $F^{2223}$ with 
$t_2^{2223}\cap t_3^{2223}\neq \emptyset$. 
Since $\sN_{t_3^{2223}/X_{2223}}\simeq\sO_{\pr^1}(-1)^{\oplus 2}$, we get 
Atiyah's flop 
\[
X_{2223}\xleftarrow{\phi_{223}}X_{223}\xrightarrow{\phi_{223}^+}X_{2233}
\]
from $t_3^{2223}\subset X_{2223}$. 
Let $\D_{t,2}^{223}\subset X_{223}$ be the exceptional divisor of $\phi_{223}$ and 
let $t_3^{+2233}\subset X_{2233}$ be the image of $\D_{t,1}^{223}$ on $X_{2233}$. 
Set $\D_{t,1}^{223}:=(\phi_{223})^{-1}_*\D_{t,1}^{2223}$ and 
$\D_{t,1}^{2233}:=(\phi_{223}^+)_*\D_{t,1}^{223}$. 
Then $\D_{t,1}^{2233}\simeq\pr^2$ and 
$-\D_{t,1}^{2233}|_{\D_{t,1}^{2233}}\simeq\sO_{\pr^2}(2)$. Thus we can contract 
$\D_{t,1}^{2233}\subset X_{2233}$ to a point. We denote the contraction by 
\[
\phi_{2233}\colon X_{2233}\to X_{23}. 
\]
Set $\psi_2:=\phi_{2223}\circ\phi_{223}$ and $\psi_2^+:=\phi_{2233}\circ\phi_{223}^+$. 
We get the commutative diagram 
\[\xymatrix{
& X_{223} \ar[dl]_{\phi_{223}} \ar[dr]^{\phi_{223}^+} \ar[ddl]^{\psi_2} 
\ar[ddr]_{\psi^{+}_2} &\\
X_{2223} \ar[d]_{\phi_{2223}} & & X_{2233} \ar[d]^{\phi_{2233}}\\
X_{22}  & & X_{23}.
}\]

Set $t_3^{+223}:=\D_{t,2}^{223}|_{F^{223}}$, where $F^{223}:=(\psi_2)^{-1}_*F^{22}$. 
Note that the analytic space $X_{23}$ is smooth around a neighborhood of 
$t_2^{23}:=(\psi_2^+)_*t_2^{223}$, where $t_2^{223}\subset X_{223}$ 
is the strict transform of $t_2^{2223}\subset X_{2223}$. 
Moreover, we have $\sN_{t_2^{23}/X_{23}}\simeq\sO_{\pr^1}(-1)^{\oplus 2}$. 
Thus we can take Atiyah's flop
\[
X_{23}\xleftarrow{\phi_{233}}X_{233}\xrightarrow{\phi_{233}^+}X_{33}
\]
from $t_2^{23}\subset X_{23}$. Let $\E_t^{233}\subset X_{233}$ be the exceptional 
divisor of $\phi_{233}$ and let $t^{+33}\subset X_{33}$ be the image of $\E_t^{233}$. 
Let us set $F^{23}:=(\psi_2^+)_*F^{223}$, $F^{233}:=(\phi_{233})^{-1}_*F^{23}$ 
and $F^{33}:=(\phi_{233}^+)_*F^{233}$. 
Let $t_3^{+233}\subset X_{233}$ be the strict transform of $t_3^{+223}\subset X_{223}$. 
The divisor $F^{33}\subset X_{33}$ is 
$\Q$-Cartier in the analytic space $X_{33}$ and $F^{33}\simeq\pr(1,1,2)$ with 
$-F^{33}|_{F^{33}}\sim_\Q\sO_{\pr(1,1,2)}(3)$. 
Thus we can contract $F^{33}\subset X_{33}$ to a point and let us denote the 
morphism by 
\[
\phi_{33}\colon X_{33}\to X_{3}.
\]
We set 
\[
\chi_2:= \phi_{33}\circ\phi_{233}^+\circ(\phi_{233})^{-1}\circ\psi^+_2\circ
(\psi_2)^{-1}\circ(\phi_{22})^{-1}\colon X_2\dashrightarrow X_3.
\]

\noindent\underline{\textbf{Step 8}}\\
On $X_{22}$, recall that $Q^{22}$ and $t_3^{22}\cup t_2^{22}$ are mutually disjoint. 
Thus we have $Q^{223}:=(\psi_2)^{-1}_*Q^{22}=(\psi_2)^*Q^{22}=(\psi_2^+)^*Q^{23}$, 
where $Q^{23}:=(\psi_2^+)_*Q^{223}$, and $Q^{233}:=(\phi_{233})^{-1}_*Q^{23}
=(\phi_{233})^*Q^{23}=(\phi_{233}^+)^*Q^{33}$, where 
$Q^{33}:=(\phi_{233}^+)_*Q^{233}\simeq\pr(1,2,3)$. 

As we already observed in Step 7, we have 
$T^{223}:=(\phi_{223})^{-1}_*T^{2223}\simeq T^{22}$. Moreover, 
$T^{223}\to T^{23}:=(\psi^+_2)_*T^{223}$ is the contraction of 
$\D_{t,2}^{223}|_{T^{223}}$. We get 
\begin{eqnarray*}
(\psi_2)^*T^{22}&=&T^{223}+\D_{t,1}^{223}+2\D_{t,2}^{223},\\
(\psi_2^+)^*T^{23}&=&T^{223}.
\end{eqnarray*}
Note that $T^{233}:=(\phi_{233})^{-1}_*T^{23}\simeq T^{23}$ and 
$T^{233}\to T^{33}:=(\phi_{233}^+)_*T^{233}$ is the contraction of the curve 
$\E_t^{233}|_{T^{233}}$. Moreover, we have $T^{33}\simeq\pr(1,1,2)$. 
We get 
\begin{eqnarray*}
(\phi_{233})^*T^{23}&=&T^{233}+\E_t^{233},\\
(\phi^+_{233})^*T^{33}&=&T^{233}.
\end{eqnarray*}

Note that $E_2^{223}:=(\psi_2)^{-1}_*E_2^{22}\simeq E_2^{22}$, and the morphism 
$E_2^{223}\to E_2^{23}:=(\psi^+_2)_*E_2^{223}$ is the contraction of the curve 
$\D_{t,1}^{223}|_{E_2^{223}}$. We get 
\begin{eqnarray*}
(\psi_2)^*E_2^{22}&=&E_2^{223}+\D_{t,1}^{223}+\D_{t,2}^{223},\\
(\psi_2^+)^*E_2^{23}&=&E_2^{223}+\frac{1}{2}\D_{t,1}^{223}.
\end{eqnarray*}
Note that $E_2^{23}$ and $t_2^{23}$ are mutually disjoint. Thus 
$E_2^{233}:=(\phi_{233})^{-1}_*E_2^{23}=(\phi_{233})^*E_2^{23}
=(\phi_{233}^+)^*E_2^{33}$, where $E_2^{33}:=(\phi_{233}^+)_*E_2^{233}$. 

As we already observed in Step 7, we have 
\begin{eqnarray*}
(\psi_2)^*F^{22}&=&F^{223},\\
(\psi_2^+)^*F^{23}&=&F^{223}+\frac{1}{2}\D_{t,1}^{223}+\D_{t,2}^{223},
\end{eqnarray*}
and
\begin{eqnarray*}
(\phi_{233})^*F^{23}&=&F^{233}+\E_t^{233},\\
(\phi^+_{233})^*F^{33}&=&F^{233}.
\end{eqnarray*}

Since $R^{22}$ and $t_3^{22}$ on $X_{22}$ are mutually disjoint, we have 
$R^{223}:=(\psi_2)^{-1}_*R^{22}=(\psi_2)^*R^{22}=(\psi_2^+)^*R^{23}$, where 
$R^{23}:=(\psi_2^+)_*R^{223}$. 
Note that $R^{12}\to R^{22}(\simeq R^{23})$ is the contractions of the curves 
$h_1^{12}$, $h_2^{12}\subset R^{12}$. The morphism 
$R^{233}:=(\phi_{233})^{-1}_*R^{23}\to R^{23}$ is the blowup at the 
(reduced) point corresponds to the intersection $f_S^{12}\cap r_2^{12}\cap s_R^{12}$. 
Let $t^{+233}\subset R^{233}$ be the exceptional divisor over $R^{23}$. The curve 
$t^{+233}$ is equal to $\E_t^{233}|_{R^{233}}$. 
We note that the morphism $R^{233}\to R^{33}:=(\phi^+_{233})_*R^{233}$
is an isomorphism. We get 
\begin{eqnarray*}
(\phi_{233})^*R^{23}&=&R^{233},\\
(\phi^+_{233})^*R^{33}&=&R^{233}+\E_t^{233}.
\end{eqnarray*}

Let 
\[
f_2^{223}, f_R^{223}, f_S^{223}, s_R^{223}, e_2^{223}, r_1^{223}, r_2^{223}, r_3^{223}, 
l^{+223}, s^{+223}\subset X_{223}
\]
and 
\[
f_2^{233}, f_R^{233}, f_S^{233}, s_R^{233}, e_2^{233}, r_1^{233}, r_2^{233}, r_3^{233}, 
l^{+233}, s^{+233}\subset X_{233}
\]
be the strict transforms of
\[
f_2^{12}, f_R^{12}, f_S^{12}, s_R^{12}, e_2^{12}, r_1^{12}, r_2^{12}, r_3^{12}, 
l^{+12}, s^{+12}\subset X_{12}
\]
respectively. On $X_{223}$, we get the following intersection numbers: 
\begin{center}
\renewcommand\arraystretch{1.4}
\begin{tabular}{c|cc|ccccc}
$\cdot$ & $\D_{t,1}^{223}$ & $\D_{t,2}^{223}$ & 
$(\psi_2^+)^*Q^{23}$ & 
$(\psi_2^+)^*T^{23}$ & $(\psi_2^+)^*E_2^{23}$ & 
$(\psi_2^+)^*F^{23}$ & $(\psi_2^+)^*R^{23}$\\ \hline
$f_2^{223}$ & $0$ & $0$ & $0$ & $0$ & $-1$ & $1$ & $-1/2$ \\
$f_R^{223}$ & $0$ & $0$ & $1/3$ & $0$ & $0$ & $1$ & $-2/3$ \\
$f_S^{223}$ & $0$ & $0$ & $1$ & $1$ & $0$ & $1$ & $-1$ \\
$s_R^{223}$ & $0$ & $0$ & $0$ & $1$ & $1$ & $-2$ & $1$ \\
$e_2^{223}$ & $1$ & $0$ & $0$ & $0$ & $1/2$ & $-3/2$ & $1$ \\
$t_2^{223}$ & $0$ & $1$ & $0$ & $-1$ & $0$ & $-1$ & $1$ \\
$r_1^{223}$ & $0$ & $0$ & $-1/3$ & $0$ & $0$ & $0$ & $1/6$ \\
$r_2^{223}$ & $0$ & $0$ & $0$ & $0$ & $0$ & $1$ & $-1/2$ \\
$r_3^{223}$ & $0$ & $0$ & $0$ & $0$ & $1$ & $0$ & $0$ \\
$l^{+223}$ & $0$ & $0$ & $2/3$ & $1$ & $0$ & $0$ & $-1/3$ \\
$s^{+223}$ & $0$ & $0$ & $1$ & $1$ & $1$ & $0$ & $-1/2$ \\
$t_3^{+223}$ & $1$ & $-1$ & $0$ & $1$ & $1/2$ & $-1/2$ & $0$ 
\end{tabular}
\end{center}
On $X_{233}$, we get the following intersection numbers: 
\begin{center}
\renewcommand\arraystretch{1.4}
\begin{tabular}{c|c|ccccc}
$\cdot$ & $\E_{t}^{233}$ &  
$(\phi_{233}^+)^*Q^{33}$ & 
$(\phi_{233}^+)^*T^{33}$ & $(\phi_{233}^+)^*E_2^{33}$ & 
$(\phi_{233}^+)^*F^{33}$ & $(\phi_{233}^+)^*R^{33}$\\ \hline
$f_2^{233}$ & $0$ & $0$ & $0$ & $-1$ & $1$ & $-1/2$  \\
$f_R^{233}$ & $0$ & $1/3$ & $0$ & $0$ & $1$ & $-2/3$ \\
$f_S^{233}$ & $1$ & $1$ & $0$ & $0$ & $0$ & $0$  \\
$s_R^{233}$ & $1$ & $0$ & $0$ & $1$ & $-3$ & $2$  \\
$e_2^{233}$ & $0$ & $0$ & $0$ & $1/2$ & $-3/2$ & $1$ \\
$t^{+233}$ & $-1$ & $0$ & $1$ & $0$ & $1$ & $-1$  \\
$r_1^{233}$ & $0$ & $-1/3$ & $0$ & $0$ & $0$ & $1/6$  \\
$r_2^{233}$ & $1$ & $0$ & $-1$ & $0$ & $0$ & $1/2$  \\
$r_3^{233}$ & $0$ & $0$ & $0$ & $1$ & $0$ & $0$ \\
$l^{+233}$ & $0$ & $2/3$ & $1$ & $0$ & $0$ & $-1/3$ \\
$s^{+233}$ & $0$ & $1$ & $1$ & $1$ & $0$ & $-1/2$  \\
$t_3^{+233}$ & $1$ & $0$ & $0$ & $1/2$ & $-3/2$ & $1$ 
\end{tabular}
\end{center}
Note that $F^{33}$, $Q^{33}$ and $T^{33}$ on $X_{33}$ are mutually disjoint. Set 
$Q^3:=(\phi_{33})_*Q^{33}$, 
$T^3:=(\phi_{33})_*T^{33}$, 
$E_2^3:=(\phi_{33})_*E_2^{33}$, 
$R^3:=(\phi_{33})_*R^{33}$. 
Then we have 
$(\phi_{33})^*E_2^3=E_2^{33}+(1/3)F^{33}$ and
$(\phi_{33})^*R^3=E_2^{33}+(2/3)R^{33}$. 
We note that $Q^3$, $T^3$ and $E_2^3$ on $X_{3}$ are mutually disjoint. 
Moreover, we have 
\[
(\phi_{33})^*\left(Q^3+2T^3+E_2^3+(9-u)R^3\right)
=Q^{33}+2T^{33}+E_2^{33}+\frac{19-2u}{3}F^{33}+(9-u)R^{33}.
\]
From the above table, the $\R$-divisor $Q^3+2T^3+E_2^3+(9-u)R^3$ is ample 
for $u\in(2,5)$. In particular, $X_3$ is projective and $\chi_2\circ\chi_1$ is 
a small $\Q$-factorial modification of $X_1$. 
Moreover, for $u\in[2,5]$, we have
\begin{eqnarray*}
&&\vol_{X_1}\left((\rho_1)^*(-K_X)-uR^1\right)\\
&=&\left(\left(Q^{33}+2T^{33}+E_2^{33}+\frac{12-u}{2}F^{33}
+(9-u)R^{33}\right)+\frac{2-u}{6}F^{33}\right)^{\cdot 3}\\
&=&\left((\phi_{233})^*\left(Q^{23}+2T^{23}+E_2^{23}+\frac{12-u}{2}F^{23}
+(9-u)R^{23}\right)
+\frac{2-u}{2}\E_t^{233}\right)^{\cdot 3}\\
&&+\frac{1}{48}(2-u)^3\\
&=&\left((\psi_2)^*\left(Q^{22}+2T^{22}+E_2^{22}
+\frac{12-u}{2}F^{22}+(9-u)R^{22}\right)
+\frac{2-u}{4}\left(\D_{t,1}^{223}+2\D_{t,2}^{223}\right)\right)^{\cdot 3}\\
&&-\frac{5}{48}(2-u)^3\\
&=&28-\frac{u^3}{2}+\frac{7}{12}(u-1)^3+\frac{1}{6}(u-2)^3.
\end{eqnarray*}

\noindent\underline{\textbf{Step 9}}\\
We remark that 
\[\begin{cases}
Q^3\simeq\pr(1,2,3) & \text{with }-Q^3|_{Q^3}\sim_\Q\sO_{\pr(1,2,3)}(2), \\
T^3\simeq\pr(1,1,2) & \text{with }-T^3|_{T^3}\sim_\Q\sO_{\pr(1,1,2)}(2), \\
E_2^3\simeq\pr(1,2,3) & \text{with }-E_2^3|_{E_2^3}\sim_\Q\sO_{\pr(1,2,3)}(4).
\end{cases}\]
Thus, for $u=5$, the divisor $Q^3+2T^3+E_2^3+(9-5)R^3$ on $X_3$ contracts 
the disjoint union of $T^3$ and $E_2^3$ and get the contraction morphism 
\[
\rho_3\colon X_3\to X_4. 
\]
Note that the divisors $Q^4:=(\rho_3)_*Q^3$ and $R^4:=(\rho_3)_*R^3$ satisfy that 
$Q^3=(\rho_3)^*Q^4$ and $R^3=(\rho_3)^*R^4-(1/2)T^3-(1/4)E_2^3$. 
In particular, we get 
\[
(\rho_3)^*\left(Q^4+(9-u)R^4\right)
=Q^3+\frac{9-u}{2}T^3+\frac{9-u}{4}E_2^3+(9-u)R^3.
\]
From the table in Step 8, for $u\in (5,7)$, the $\R$-divisor 
$Q^4+(9-u)R^4$ is ample on $X_4$. In particular, for $u\in[5,7]$, we get 
\begin{eqnarray*}
&&\vol_{X_1}\left((\rho_1)^*(-K_X)-uR^1\right)\\
&=&\left(\left(Q^3+2T^3+E_2^3+(9-u)R^3\right)
+\frac{5-u}{2}T^3+\frac{5-u}{4}E_2^3\right)^{\cdot 3}\\
&=&28-\frac{u^3}{2}+\frac{7}{12}(u-1)^3+\frac{1}{6}(u-2)^3-\frac{7}{24}(u-5)^3.
\end{eqnarray*}
Moreover, for $u=7$, the divisor $Q^4+(9-7)R^4$ contracts $Q^4$ to a point. 
Let us denote the contraction by 
\[
\rho_4\colon X_4\to X_5. 
\]
Set $R^5:=(\rho_4)_*R^4$. Then we have $(\rho_4)^*R^5=R^4+(1/2)Q^4$. 
Thus we have 
\[
(\rho_4\circ\rho_3)^*\left((9-u)R^5\right)
=\frac{9-u}{2}Q^3+\frac{9-u}{2}T^3+\frac{9-u}{4}E_2^3+(9-u)R^3.
\]
Obviously, the $\R$-divisor $(9-u)R^5$ is ample for $u\in(7,9)$. 
Moreover, for $u\in[7,9]$, we get 
\begin{eqnarray*}
&&\vol_{X_1}\left((\rho_1)^*(-K_X)-uR^1\right)\\
&=&\left(\left(Q^4+(9-u)R^4\right)
+\frac{7-u}{2}Q^4\right)^{\cdot 3}\\
&=&28-\frac{u^3}{2}+\frac{7}{12}(u-1)^3+\frac{1}{6}(u-2)^3
-\frac{7}{24}(u-5)^3-\frac{1}{12}(u-7)^3\\
&=&\frac{1}{8}(9-u)^3.
\end{eqnarray*}
In particular, we get the equality $\tau_X(R^1)=9$. As a consequence, 
we get the following commutative diagram: 
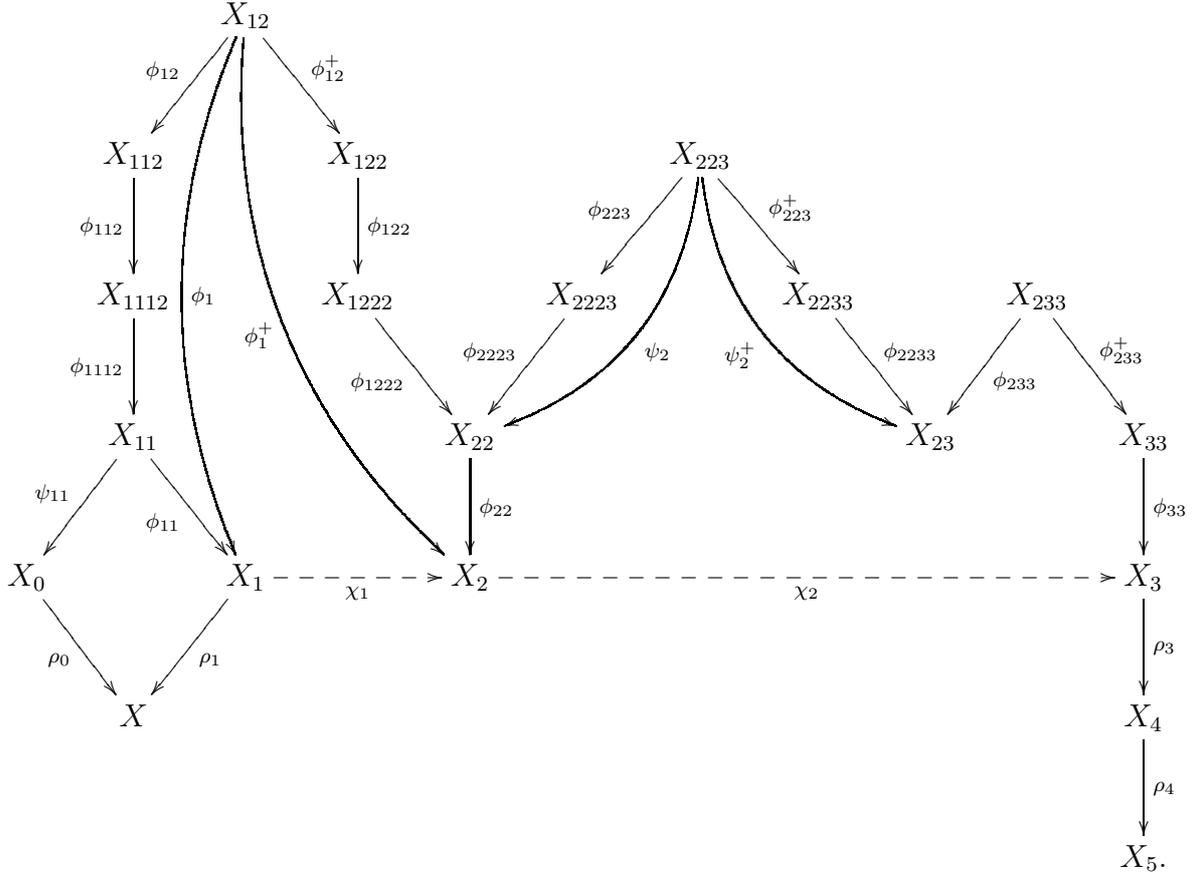
\[\xymatrix@C=12pt@R=36pt{
 & & X_{12} \ar[dl]_{\phi_{12}} \ar[dr]^{\phi_{12}^+} 
 \ar@/_24pt/[dddd]^{\phi_1} \ar@/_30pt/[ddddrr]_{\phi_1^+} & & & & & & & & \\
 & X_{112} \ar[d]_{\phi_{112}} & & X_{122} \ar[d]^{\phi_{122}} & & & 
X_{223} \ar[dl]_{\phi_{223}} \ar[dr]^{\phi_{223}^+} 
\ar@/^24pt/[ddll]^{\psi_2} \ar@/_24pt/[ddrr]_{\psi_2^+} & & & &  \\
 & X_{1112} \ar[d]_{\phi_{1112}} & & X_{1222} \ar[dr]_{\phi_{1222}} & & 
 X_{2223} \ar[dl]_{\phi_{2223}} & &  X_{2233} \ar[dr]^{\phi_{2233}} & & 
 X_{233} \ar[dl]^{\phi_{233}} \ar[dr]^{\phi_{233}^+}  & \\
 & X_{11} \ar[dl]_{\psi_{11}} \ar[dr]_{\phi_{11}} & & & 
X_{22} \ar[d]^{\phi_{22}} & & & & X_{23}  & & X_{33} \ar[d]^{\phi_{33}} \\
X_0 \ar[dr]_{\rho_0} & & X_1 \ar[dl]^{\rho_1} \ar@{-->}[rr]_{\chi_1} & & 
X_2 \ar@{-->}[rrrrrr]_{\chi_2} & & & & & & X_3 \ar[d]^{\rho_3} \\
 & X & & & & & & & & & X_4\ar[d]^{\rho_4} \\
 & & & & & & & & & & X_5.
}\]
Moreover, we get
\begin{eqnarray*}
S_X(R^1)&=&\frac{1}{28}\Biggl(\int_0^1\left(28-\frac{u^3}{2}\right)du
+\int_1^2\left(28-\frac{u^3}{2}+\frac{7}{12}(u-1)^3\right)du\\
&&+\int_2^5\left(28-\frac{u^3}{2}+\frac{7}{12}(u-1)^3+\frac{1}{6}(u-2)^3\right)du\\
&&+\int_5^7\left(28-\frac{u^3}{2}+\frac{7}{12}(u-1)^3
+\frac{1}{6}(u-2)^3-\frac{7}{24}(u-5)^3\right)du\\
&&+\int_7^9\frac{1}{8}(9-u)^3du\Biggr)=\frac{63}{16}.
\end{eqnarray*}
Therefore we get the inequality 
\[
\delta_q(X)\leq\frac{A_X(R^1)}{S_X(R^1)}=\frac{64}{63}. 
\]

\noindent\underline{\textbf{Step 10}}\\
Recall that the rational map 
\[
\psi_2^+\circ(\psi_2)^{-1}\colon X_{22}\dashrightarrow X_{23}
\]
is an isomorphism around a neighborhood of $R^{22}\subset X_{22}$. 
Take the blowup $\phi_{123}\colon X_{123}\to X_{12}$ along the curve 
$t_2^{12}\subset X_{12}$ and let $\E_t^{123}\subset X_{123}$ be the exceptional 
divisor. Set $R^{123}:=(\phi_{123})^{-1}_*R^{12}$. The divisor $\E_t^{123}$ is the 
strict transform of the divisor $\E_t^{233}\subset X_{233}$. Moreover, the curve 
$t^{+123}=\E_t^{123}|_{R^{123}}$ is the strict transform of $t^{+233}\subset X_{233}$. 
Thus there exists a common resolution $\tilde{X}$ of $X_{123}$, $X_{223}$ and 
$X_{233}$ such that the morphism $\tilde{X}\to X_{123}$ is an isomorphism around 
a neighborhood of the strict transform of $R^{123}$. Let us denote the natural 
morphisms by: 
\[\xymatrix{
 & \tilde{X} \ar[dl]_{\sigma_{123}} \ar[ddl]^{\sigma_{12}} \ar[dd]^{\sigma_{223}} 
 \ar[ddr]^{\sigma_{233}}& \\
X_{123}\ar[d]_{\phi_{123}} & & \\
X_{12} & X_{223} & X_{233}.
}\]
Set $\tilde{R}:=(\sigma_{123})^{-1}_*R^{123}$. 
Let $\tilde{t}^+\subset\tilde{R}$ be the strict transform of $t^{+123}\subset R^{123}$. 
Then the morphism $\sigma_{123}|_{\tilde{R}}\colon\tilde{R}\to R^{12}$ is the 
blowup along the (reduced) point $t^{12}_2\cap R^{12}$ with the exceptional divisor 
$\tilde{t}^+$. Let 
\[
\tilde{s}_R,\tilde{f}_S, \tilde{f}_R, \tilde{h}_1, \tilde{h}_2, \tilde{l}^+, 
\tilde{f}_2, \tilde{s}^+, \tilde{r}_1, \tilde{r}_2, \tilde{r}_3\subset\tilde{R}
\]
be the strict transforms of 
\[
s_R^{12}, f_S^{12}, f_R^{12}, h_1^{12}, h_2^{12}, l^{+12}, f_2^{12}, s^{+12}, 
r_1^{12}, r_2^{12}, r_3^{12}\subset R^{12}, 
\]
respectively. Let $\gamma\colon\tilde{R}\to R^1$ be the natural morphism. 
Moreover, let $p_v^1\in R^1$ be the vertex of the cone $R^1\simeq\pr(1,1,2)$ and 
let $p_h^1\in R^1$ be the image of $\tilde{h}_1$ (or $\tilde{h}_2$) on $R^1$. 

The following claim is trivial: 

\begin{claim}\label{R_claim}
The Kleiman--Mori cone $\overline{\NE}\left(\tilde{R}\right)$ of $\tilde{R}$ is 
spanned by the classes of the following $12$ negative curves 
\[
\tilde{s}_R, \tilde{t}^+, \tilde{f}_S, \tilde{f}_R, \tilde{h}_1, \tilde{h}_2, \tilde{l}^+, 
\tilde{f}_2, \tilde{s}^+, \tilde{r}_1, \tilde{r}_2, \tilde{r}_3.
\]
\end{claim}

\begin{proof}[Proof of Claim \ref{R_claim}]
Consider the contractions $\nu\colon \tilde{R}\to \pr^2$ of the curves 
\[
\tilde{t}^+,\quad\tilde{f}_R\cup\tilde{h}_1,\quad\tilde{l}^+,\quad\tilde{f}_2,
\quad\tilde{r}_1.
\]
Then $\nu_*\tilde{s}_R$ is the line passing through 
$\nu\left(\tilde{t}^+\right)$, 
$\nu\left(\tilde{h}_1\right)$, 
$\nu\left(\tilde{f}_2\right)$, and 
$\nu_*\tilde{h}_2$ is the line passing through 
$\nu\left(\tilde{h}_1\right)$, 
$\nu\left(\tilde{r}_1\right)$, 
$\nu\left(\tilde{l}^+\right)$.  
Note that 
\[
\nu^*\left(-\left(K_{\pr^2}+\nu_*\tilde{s}_R+\nu_*\tilde{h}_2\right)\right)
=-\left(K_{\tilde{R}}+\tilde{s}_R+\tilde{h}_2+\tilde{f}_R+\tilde{h}_2\right)
\]
is nef and big. Thus, by \cite[Proposition 3.3]{NdP}, the cone 
$\overline{\NE}\left(\tilde{R}\right)$ is spanned by the classes of 
finitely many negative curves. 
Take any irreducible curve $\tilde{C}\subset\tilde{R}$ spanning an extremal ray 
of $\overline{\NE}\left(\tilde{R}\right)$. We may assume that 
\[
\tilde{C}\neq\tilde{s}_R, \tilde{t}^+, \tilde{f}_R, \tilde{h}_1, \tilde{h}_2, 
\tilde{l}^+, \tilde{f}_2, \tilde{r}_1.
\]
In particular, $\nu_*\tilde{C}$ is not a point. Moreover, since 
\[
\left(\left(\tilde{s}_R+\tilde{h}_2+\tilde{f}_R+\tilde{h}_2\right)\cdot\tilde{C}\right)\geq 0,
\]
the extremal ray spanned by $\tilde{C}$ is $K_{\tilde{R}}$-negative. 
Thus we get $\left(-K_{\tilde{R}}\cdot \tilde{C}\right)=1$ and 
\[
\left(\left(\tilde{s}_R+\tilde{h}_2+\tilde{f}_R+\tilde{h}_2\right)\cdot\tilde{C}\right)=0.
\]
This implies that 
\[
\deg\nu_*\tilde{C}
=\left(\nu^*\left(-\left(K_{\pr^2}+\nu_*\tilde{s}_R+\nu_*\tilde{h}_2\right)\right)
\cdot\tilde{C}\right)=1.
\]
Since $\tilde{C}$ is a negative curve, we must have $\tilde{C}=\tilde{f}_2$, 
$\tilde{s}^+$, $\tilde{r}_2$ or $\tilde{r}_3$. 
\end{proof}

The intersection form of 
\[
\tilde{s}_R, \tilde{t}^+, \tilde{f}_S, \tilde{f}_R, \tilde{h}_1, \tilde{h}_2, \tilde{l}^+, 
\tilde{f}_2, \tilde{s}^+, \tilde{r}_1, \tilde{r}_2, \tilde{r}_3
\]
on $\tilde{R}$ is given by the symmetric matrix
\[\begin{pmatrix}
-3&1&0&1&0&0&0&1&0&0&0&0 \\
1&-1&1&0&0&0&0&0&0&0&1&0 \\
0&1&-1&0&0&0&0&0&0&1&0&1 \\
1&0&0&-1&1&0&0&0&0&0&0&0 \\
0&0&0&1&-2&1&0&0&0&0&0&0 \\
0&0&0&0&1&-2&1&0&0&1&0&0 \\
0&0&0&0&0&1&-1&0&0&0&1&1 \\
1&0&0&0&0&0&0&-1&1&0&0&1 \\
0&0&0&0&0&0&0&1&-1&1&1&0 \\
0&0&1&0&0&1&0&0&1&-1&0&0 \\
0&1&0&0&0&0&1&0&1&0&-1&0 \\
0&0&1&0&0&0&1&1&0&0&0&-1
\end{pmatrix}.\]
From now on, we write
\begin{eqnarray*}
&&[a_1,a_2,a_3,a_4,a_5,a_6,a_7,a_8,a_9,a_{10},a_{11},a_{12}]\\
&:=&a_1\tilde{s}_R+a_2\tilde{t}^++a_3\tilde{f}_S+a_4\tilde{f}_R
+a_5\tilde{h}_1+a_6\tilde{h_2}\\
&+&a_7\tilde{l}^++a_8\tilde{f}_2+a_9\tilde{s}^+
+a_{10}\tilde{r}_1+a_{11}\tilde{r}_2+a_{12}\tilde{r}_3
\end{eqnarray*}
for any $a_1,\dots,a_{12}\in \R$.

\noindent\underline{\textbf{Step 11}}\\
For any $u\in[0,9]$, let us set 
\begin{eqnarray*}
P(u) &:=& P_\sigma\left(\tilde{X}, (\phi_1\circ\sigma_{12})^*
\left((\rho_1)^*(-K_X)-uR^1\right)\right)\Big|_{\tilde{R}}, \\
N(u) &:=& N_\sigma\left(\tilde{X}, (\phi_1\circ\sigma_{12})^*
\left((\rho_1)^*(-K_X)-uR^1\right)\right)\Big|_{\tilde{R}}.
\end{eqnarray*}
We know that 
\begin{eqnarray*}
F^{11}|_{\tilde{R}}&=&\tilde{s}_R+\tilde{t}^+,\\
Q^{11}|_{\tilde{R}}&=&\tilde{r}_1+\tilde{h}_1+2\tilde{h}_2+2\tilde{l}^++\tilde{s}^+, \\
T^{11}|_{\tilde{R}}&=&\tilde{r}_2+\tilde{t}^+
+\tilde{h}_1+2\tilde{h}_2+3\tilde{l}^++\tilde{s}^+, \\
E_2^{11}|_{\tilde{R}}&=&\tilde{f}_2+\tilde{s}^+, \\
R^{11}|_{\tilde{R}}&\sim_\Q&-\tilde{s}_R-2\tilde{t}^+-\tilde{f}_S.
\end{eqnarray*}
Thus we get 
\begin{eqnarray*}
F^{22}|_{\tilde{R}}&=&\tilde{s}_R+\tilde{t}^+,\\
Q^{22}|_{\tilde{R}}&=&\tilde{r}_1+\frac{1}{3}\tilde{h}_1+\frac{2}{3}\tilde{h}_2, \\
T^{22}|_{\tilde{R}}&=&\tilde{r}_2+\tilde{t}^+, \\
E_2^{22}|_{\tilde{R}}&=&\tilde{f}_2, \\
R^{22}|_{\tilde{R}}&\sim_\Q&-\tilde{s}_R-2\tilde{t}^+-\tilde{f}_S
+\frac{1}{3}\tilde{h}_1+\frac{2}{3}\tilde{h}_2+\tilde{l}^++\frac{1}{2}\tilde{s}^+
\end{eqnarray*}
and 
\begin{eqnarray*}
F^{33}|_{\tilde{R}}&=&\tilde{s}_R,\\
Q^{33}|_{\tilde{R}}&=&\tilde{r}_1+\frac{1}{3}\tilde{h}_1+\frac{2}{3}\tilde{h}_2, \\
T^{33}|_{\tilde{R}}&=&\tilde{r}_2, \\
E_2^{33}|_{\tilde{R}}&=&\tilde{f}_2, \\
R^{33}|_{\tilde{R}}&\sim_\Q&-\tilde{s}_R-\tilde{t}^+-\tilde{f}_S
+\frac{1}{3}\tilde{h}_1+\frac{2}{3}\tilde{h}_2+\tilde{l}^++\frac{1}{2}\tilde{s}^+.
\end{eqnarray*}
In particular, we can determine $P(u)$ and $N(u)$. 
\begin{itemize}
\item
If $u\in [0,1]$, then 
\begin{eqnarray*}
P(u)&\sim_\R&\left[\frac{u-6}{2},\frac{3u-20}{2},u-9,0,3,6,8,1,4,1,2,0\right],\\
N(u)&=&\left[0,0,0,0,0,0,0,0,0,0,0,0\right].
\end{eqnarray*}
\item
If $u\in [1,2]$, then 
\begin{eqnarray*}
P(u)&\sim_\R&\left[\frac{u-6}{2},\frac{3u-20}{2},u-9,0,\frac{10-u}{3},\frac{20-2u}{3},
9-u,1,\frac{9-u}{2},1,2,0\right],\\
N(u)&=&\left[0,0,0,0,\frac{u-1}{3},\frac{2u-2}{3},u-1,0,\frac{u-1}{2},0,0,0\right].
\end{eqnarray*}
\item
If $u\in [2,5]$, then 
\begin{eqnarray*}
P(u)&\sim_\R&\left[\frac{u-8}{3},u-9,u-9,0,\frac{10-u}{3},\frac{20-2u}{3},
9-u,1,\frac{9-u}{2},1,2,0\right],\\
N(u)&=&\left[\frac{u-2}{6},\frac{u-2}{2},0,0,
\frac{u-1}{3},\frac{2u-2}{3},u-1,0,\frac{u-1}{2},0,0,0\right].
\end{eqnarray*}
\item
If $u\in [5,7]$, then 
\begin{eqnarray*}
P(u)&\sim_\R&\left[\frac{u-9}{4},u-9,u-9,0,\frac{10-u}{3},\frac{20-2u}{3},
9-u,\frac{9-u}{4},\frac{9-u}{2},1,\frac{9-u}{2},0\right],\\
N(u)&=&\left[\frac{u-3}{4},\frac{u-2}{2},0,0,\frac{u-1}{3},\frac{2u-2}{3},
u-1,\frac{u-5}{4},\frac{u-1}{2},0,\frac{u-5}{2},0\right].
\end{eqnarray*}
\item
If $u\in [7,9]$, then 
\begin{eqnarray*}
P(u)&\sim_\R&\left[\frac{u-9}{4},u-9,u-9,0,\frac{9-u}{2},9-u,
9-u,\frac{9-u}{4},\frac{9-u}{2},\frac{9-u}{2},\frac{9-u}{2},0\right],\\
N(u)&=&\left[\frac{u-3}{4},\frac{u-2}{2},0,0,\frac{u-3}{2},u-3,
u-1,\frac{u-5}{4},\frac{u-1}{2},\frac{u-7}{2},\frac{u-5}{2},0\right].
\end{eqnarray*}
\end{itemize}

\noindent\underline{\textbf{Step 12}}\\
Set 
\[
\varepsilon_s:=\phi_{11}|_{R^{11}}\colon R^{11}\to R^1
\]
and let $\gamma_s\colon \tilde{R}\to R^{11}$ be the natural morphism. 
Then the morphism $\varepsilon_s$ is a plt-blowup with the 
exceptional divisor $s_R^{11}:=(\gamma_s)_*\tilde{s}_R$. Note that 
$(\gamma_s)^*s_R^{11}=\tilde{s}_R+\tilde{t}^+$. Set 
$p_2^{s_R}:=\tilde{t}^+|_{\tilde{s}_R}$, 
$p_4^{s_R}:=\tilde{f}_R|_{\tilde{s}_R}$, and 
$p_8^{s_R}:=\tilde{f}_2|_{\tilde{s}_R}$. 
Then $p_2^{s_R}$, $p_4^{s_R}$, $p_8^{s_R}\in\tilde{s}_R$ 
are mutually distinct reduced points. Set 
$p_2^{11}:=\gamma_s\left(p_2^{s_R}\right)$, 
$p_4^{11}:=\gamma_s\left(p_4^{s_R}\right)$, 
$p_8^{11}:=\gamma_s\left(p_8^{s_R}\right)$, and let us set 
\begin{eqnarray*}
P(u,v) &:=& P_\sigma\left(\tilde{R}, P(u)-v\tilde{s}_R\right), \\
N(u,v) &:=& N_\sigma\left(\tilde{R}, P(u)-v\tilde{s}_R\right),
\end{eqnarray*}
where $P(u)$, $N(u)$ are as in Step 11. 

\begin{itemize}
\item
Assume that $u\in[0,1]$. 
\begin{itemize}
\item
If $v\in\left[0,\frac{u}{2}\right]$, then 
\begin{eqnarray*}
N(u,v)&=&\left[0,v,0,0,0,0,0,0,0,0,0,0\right], \\
P(u,v)&\sim_\R&\left[\frac{u-2v-6}{2},\frac{3u-2v-20}{2},u-9,0,3,6,8,1,4,1,2,0\right],
\end{eqnarray*}
and 
\[
\left(P(u,v)^{\cdot 2}\right)=\frac{1}{2}(u-2v)(u+2v).
\]
\end{itemize}
\item
Assume that $u\in[1,2]$. 
\begin{itemize}
\item
If $v\in\left[0,\frac{1}{2}\right]$, then 
\begin{eqnarray*}
N(u,v)&=&\left[0,v,0,0,0,0,0,0,0,0,0,0\right], \\
P(u,v)&\sim_\R&\left[\frac{u-2v-6}{2},\frac{3u-2v-20}{2},u-9,0,
\frac{10-u}{3},\frac{20-2u}{3},9-u,1,\frac{9-u}{2},1,2,0\right],
\end{eqnarray*}
and 
\[
\left(P(u,v)^{\cdot 2}\right)=\frac{1}{12}(-7+14u-u^2-24v^2).
\]
\item
If $v\in\left[\frac{1}{2},\frac{2+u}{6}\right]$, then 
\begin{eqnarray*}
N(u,v)&=&\left[0,v,0,0,0,0,0,\frac{2v-1}{2},0,0,0,0\right], \\
P(u,v)&\sim_\R&\biggl[\frac{u-2v-6}{2},\frac{3u-2v-20}{2},u-9,0,
\frac{10-u}{3},\frac{20-2u}{3},\\
&&\quad 9-u,\frac{3-2v}{2},\frac{9-u}{2},1,2,0\biggr],
\end{eqnarray*}
and 
\[
\left(P(u,v)^{\cdot 2}\right)=\frac{1}{12}(-4+14u-u^2-12v-12v^2).
\]
\item
If $v\in\left[\frac{2+u}{6},\frac{u}{2}\right]$, then 
\begin{eqnarray*}
N(u,v)&=&\left[0,v,0,\frac{-2-u+6v}{2},\frac{-2-u+6v}{3},\frac{-2-u+6v}{6},0,
\frac{2v-1}{2},0,0,0,0\right], \\
P(u,v)&\sim_\R&\biggl[\frac{u-2v-6}{2},\frac{3u-2v-20}{2},u-9,\frac{2+u-6v}{2},
4-2v,\frac{14-u-2v}{2},\\
&&\quad 9-u,\frac{3-2v}{2},\frac{9-u}{2},1,2,0\biggr],
\end{eqnarray*}
and 
\[
\left(P(u,v)^{\cdot 2}\right)=\frac{1}{2}(u-2v)(3-2v).
\]
\end{itemize}
\item
Assume that $u\in[2,3]$. 
\begin{itemize}
\item
If $v\in\left[0,\frac{u-2}{3}\right]$, then 
\begin{eqnarray*}
N(u,v)&=&\left[0,0,0,0,0,0,0,0,0,0,0,0\right], \\
P(u,v)&\sim_\R&\left[\frac{u-3v-8}{3},u-9,u-9,0,
\frac{10-u}{3},\frac{20-2u}{3},9-u,1,\frac{9-u}{2},1,2,0\right],
\end{eqnarray*}
and 
\[
\left(P(u,v)^{\cdot 2}\right)=\frac{1}{12}(-15+22u-3u^2-36v^2).
\]
\item
If $v\in\left[\frac{u-2}{3},\frac{5-u}{6}\right]$, then 
\begin{eqnarray*}
N(u,v)&=&\left[0,\frac{2-u+3v}{3},0,0,0,0,0,0,0,0,0,0\right], \\
P(u,v)&\sim_\R&\left[\frac{u-3v-8}{3},\frac{-29+4u-3v}{3},u-9,0,
\frac{10-u}{3},\frac{20-2u}{3},9-u,1,\frac{9-u}{2},1,2,0\right],
\end{eqnarray*}
and 
\[
\left(P(u,v)^{\cdot 2}\right)=\frac{1}{36}(-29+50u-5u^2+48v-24uv-72v^2).
\]
\item
If $v\in\left[\frac{5-u}{6},\frac{2}{3}\right]$, then 
\begin{eqnarray*}
N(u,v)&=&\left[0,\frac{2-u+3v}{3},0,0,0,0,0,\frac{-5+u+6v}{6},0,0,0,0\right], \\
P(u,v)&\sim_\R&\biggl[\frac{u-3v-8}{3},\frac{-29+4u-3v}{3},u-9,0,
\frac{10-u}{3},\frac{20-2u}{3},\\
&&\quad 9-u,\frac{11-u-6v}{6},\frac{9-u}{2},1,2,0\biggr],
\end{eqnarray*}
and 
\[
\left(P(u,v)^{\cdot 2}\right)=\frac{1}{9}(-1+10u-u^2-3v-3uv-9v^2).
\]
\item
If $v\in\left[\frac{2}{3},\frac{1+u}{3}\right]$, then 
\begin{eqnarray*}
N(u,v)&=&\left[0,\frac{2-u+3v}{3},0,3v-2,\frac{-4+6v}{3},\frac{-2+3v}{3},
0,\frac{-5+u+6v}{6},0,0,0,0\right], \\
P(u,v)&\sim_\R&\biggl[\frac{u-3v-8}{3},\frac{-29+4u-3v}{3},u-9,2-3v,
\frac{14-u-6v}{3},\frac{22-2u-3v}{3},\\
&&\quad 9-u,\frac{11-u-6v}{6},\frac{9-u}{2},1,2,0\biggr],
\end{eqnarray*}
and 
\[
\left(P(u,v)^{\cdot 2}\right)=\frac{1}{9}(1+u-3v)(11-u-6v).
\]
\end{itemize}
\item
Assume that $u\in[3,4]$. 
\begin{itemize}
\item
If $v\in\left[0,\frac{5-u}{6}\right]$, then 
\begin{eqnarray*}
N(u,v)&=&\left[0,0,0,0,0,0,0,0,0,0,0,0\right], \\
P(u,v)&\sim_\R&\left[\frac{u-3v-8}{3},u-9,u-9,0,
\frac{10-u}{3},\frac{20-2u}{3},9-u,1,\frac{9-u}{2},1,2,0\right],
\end{eqnarray*}
and 
\[
\left(P(u,v)^{\cdot 2}\right)=\frac{1}{12}(-15+22u-3u^2-36v^2).
\]
\item
If $v\in\left[\frac{5-u}{6},\frac{u-2}{3}\right]$, then 
\begin{eqnarray*}
N(u,v)&=&\left[0,0,0,0,0,0,0,\frac{-5+u+6v}{6},0,0,0,0\right], \\
P(u,v)&\sim_\R&\biggl[\frac{u-3v-8}{3},u-9,u-9,0,
\frac{10-u}{3},\frac{20-2u}{3},\\
&&\quad 9-u,\frac{11-u-6v}{6},\frac{9-u}{2},1,2,0\biggr],
\end{eqnarray*}
and 
\[
\left(P(u,v)^{\cdot 2}\right)=\frac{1}{9}(-5+14u-2u^2-15v+3uv-18v^2).
\]
\item
If $v\in\left[\frac{u-2}{3},\frac{2}{3}\right]$, then 
\begin{eqnarray*}
N(u,v)&=&\left[0,\frac{2-u+3v}{3},0,0,0,0,0,\frac{-5+u+6v}{6},0,0,0,0\right], \\
P(u,v)&\sim_\R&\biggl[\frac{u-3v-8}{3},\frac{-29+4u-3v}{3},u-9,0,
\frac{10-u}{3},\frac{20-2u}{3},\\
&&\quad 9-u,\frac{11-u-6v}{6},\frac{9-u}{2},1,2,0\biggr],
\end{eqnarray*}
and 
\[
\left(P(u,v)^{\cdot 2}\right)=\frac{1}{9}(-1+10u-u^2-3v-3uv-9v^2).
\]
\item
If $v\in\left[\frac{2}{3},\frac{11-u}{6}\right]$, then 
\begin{eqnarray*}
N(u,v)&=&\left[0,\frac{2-u+3v}{3},0,3v-2,\frac{-4+6v}{3},\frac{-2+3v}{3},
0,\frac{-5+u+6v}{6},0,0,0,0\right], \\
P(u,v)&\sim_\R&\biggl[\frac{u-3v-8}{3},\frac{-29+4u-3v}{3},u-9,2-3v,
\frac{14-u-6v}{3},\frac{22-2u-3v}{3},\\
&&\quad 9-u,\frac{11-u-6v}{6},\frac{9-u}{2},1,2,0\biggr],
\end{eqnarray*}
and 
\[
\left(P(u,v)^{\cdot 2}\right)=\frac{1}{9}(1+u-3v)(11-u-6v).
\]
\end{itemize}
\item
Assume that $u\in[4,5]$. 
\begin{itemize}
\item
If $v\in\left[0,\frac{5-u}{6}\right]$, then 
\begin{eqnarray*}
N(u,v)&=&\left[0,0,0,0,0,0,0,0,0,0,0,0\right], \\
P(u,v)&\sim_\R&\left[\frac{u-3v-8}{3},u-9,u-9,0,
\frac{10-u}{3},\frac{20-2u}{3},9-u,1,\frac{9-u}{2},1,2,0\right],
\end{eqnarray*}
and 
\[
\left(P(u,v)^{\cdot 2}\right)=\frac{1}{12}(-15+22u-3u^2-36v^2).
\]
\item
If $v\in\left[\frac{5-u}{6},\frac{2}{3}\right]$, then 
\begin{eqnarray*}
N(u,v)&=&\left[0,0,0,0,0,0,0,\frac{-5+u+6v}{6},0,0,0,0\right], \\
P(u,v)&\sim_\R&\biggl[\frac{u-3v-8}{3},u-9,u-9,0,
\frac{10-u}{3},\frac{20-2u}{3},\\
&&\quad 9-u,\frac{11-u-6v}{6},\frac{9-u}{2},1,2,0\biggr],
\end{eqnarray*}
and 
\[
\left(P(u,v)^{\cdot 2}\right)=\frac{1}{9}(-5+14u-2u^2-15v+3uv-18v^2).
\]
\item
If $v\in\left[\frac{2}{3},\frac{u-2}{3}\right]$, then 
\begin{eqnarray*}
N(u,v)&=&\left[0,0,0,3v-2,\frac{-4+6v}{3},\frac{-2+3v}{3},0,
\frac{-5+u+6v}{6},0,0,0,0\right], \\
P(u,v)&\sim_\R&\biggl[\frac{u-3v-8}{3},u-9,u-9,2-3v,
\frac{14-u-6v}{3},\frac{22-2u-3v}{3},\\
&&\quad 9-u,\frac{11-u-6v}{6},\frac{9-u}{2},1,2,0\biggr],
\end{eqnarray*}
and 
\[
\left(P(u,v)^{\cdot 2}\right)=\frac{1}{9}(7+14u-2u^2-51v+3uv+9v^2).
\]
\item
If $v\in\left[\frac{u-2}{3},\frac{11-u}{6}\right]$, then 
\begin{eqnarray*}
N(u,v)&=&\left[0,\frac{2-u+3v}{3},0,3v-2,\frac{-4+6v}{3},\frac{-2+3v}{3},
0,\frac{-5+u+6v}{6},0,0,0,0\right], \\
P(u,v)&\sim_\R&\biggl[\frac{u-3v-8}{3},\frac{-29+4u-3v}{3},u-9,2-3v,
\frac{14-u-6v}{3},\frac{22-2u-3v}{3},\\
&&\quad 9-u,\frac{11-u-6v}{6},\frac{9-u}{2},1,2,0\biggr],
\end{eqnarray*}
and 
\[
\left(P(u,v)^{\cdot 2}\right)=\frac{1}{9}(1+u-3v)(11-u-6v).
\]
\end{itemize}
\item
Assume that $u\in[5,7]$. 
\begin{itemize}
\item
If $v\in\left[0,\frac{13-u}{12}\right]$, then 
\begin{eqnarray*}
N(u,v)&=&\left[0,0,0,0,0,0,0,v,0,0,0,0\right], \\
P(u,v)&\sim_\R&\biggl[\frac{u-4v-9}{4},u-9,u-9,0,
\frac{10-u}{3},\frac{20-2u}{3},\\
&&\quad 9-u,\frac{9-u-4v}{4},\frac{9-u}{2},1,\frac{9-u}{2},0\biggr],
\end{eqnarray*}
and 
\[
\left(P(u,v)^{\cdot 2}\right)=\frac{1}{24}(145-26u+u^2-48v^2).
\]
\item
If $v\in\left[\frac{13-u}{12},\frac{9-u}{4}\right]$, then 
\begin{eqnarray*}
N(u,v)&=&\left[0,0,0,\frac{-13+u+12v}{4},\frac{-13+u+12v}{6},\frac{-13+u+12v}{12},
0,v,0,0,0,0\right], \\
P(u,v)&\sim_\R&\biggl[\frac{u-4v-9}{4},u-9,u-9,\frac{13-u-12v}{4},
\frac{11-u-4v}{2},\frac{31-3u-4v}{4},\\
&&\quad 9-u,\frac{9-u-4v}{4},
\frac{9-u}{2},1,\frac{9-u}{2},0\biggr],
\end{eqnarray*}
and 
\[
\left(P(u,v)^{\cdot 2}\right)=\frac{1}{16}(17-u-4v)(9-u-4v).
\]
\end{itemize}
\item
Assume that $u\in[7,9]$. 
\begin{itemize}
\item
If $v\in\left[0,\frac{9-u}{4}\right]$, then 
\begin{eqnarray*}
N(u,v)&=&\left[0,0,0,0,0,0,
0,v,0,0,0,0\right], \\
P(u,v)&\sim_\R&\biggl[\frac{u-4v-9}{4},u-9,u-9,0,
\frac{9-u}{2},9-u,\\
&&\quad 9-u,\frac{9-u-4v}{4},
\frac{9-u}{2},\frac{9-u}{2},\frac{9-u}{2},0\biggr],
\end{eqnarray*}
and 
\[
\left(P(u,v)^{\cdot 2}\right)=\frac{1}{8}(9-u+4v)(9-u-4v).
\]
\end{itemize}
\end{itemize}

Therefore, we get 
\begin{eqnarray*}
&&S\left(V^{\tilde{R}}_{\bullet,\bullet}; \tilde{s}_R\right)\\
&=&\frac{3}{28}\Biggl(\int_0^1\int_0^{\frac{u}{2}}\frac{1}{2}(u-2v)(u+2v)dvdu\\
&+&\int_1^2\biggl(\int_0^{\frac{1}{2}}\frac{1}{12}(-7+14u-u^2-24v^2)dv
+\int_{\frac{1}{2}}^{\frac{2+u}{6}}\frac{1}{12}(-4+14u-u^2-12v-12v^2)dv\\
&&+\int_{\frac{2+u}{6}}^{\frac{u}{2}}\frac{1}{2}(u-2v)(3-2v)dv\biggr)du\\
&+&\int_2^3\biggl(\frac{u-2}{6}\cdot\frac{1}{12}(-15+22u-3u^2)
+\int_0^{\frac{u-2}{3}}\frac{1}{12}(-15+22u-3u^2-36v^2)dv\\
&&+\int_{\frac{u-2}{3}}^{\frac{5-u}{6}}\frac{1}{36}(-29+50u-5u^2+48v-24uv-72v^2)dv\\
&&+\int_{\frac{5-u}{6}}^{\frac{2}{3}}\frac{1}{9}(-1+10u-u^2-3v-3uv-9v^2)dv
+\int_{\frac{2}{3}}^{\frac{1+u}{3}}\frac{1}{9}(1+u-3v)(11-u-6v)dv\biggr)du\\
&+&\int_3^4\biggl(\frac{u-2}{6}\cdot\frac{1}{12}(-15+22u-3u^2)
+\int_0^{\frac{5-u}{6}}\frac{1}{12}(-15+22u-3u^2-36v^2)dv\\
&&+\int_{\frac{5-u}{6}}^{\frac{u-2}{3}}\frac{1}{9}(-5+14u-2u^2-15v+3uv-18v^2)dv\\
&&+\int_{\frac{u-2}{3}}^{\frac{2}{3}}\frac{1}{9}(-1+10u-u^2-3v-3uv-9v^2)dv
+\int_{\frac{2}{3}}^{\frac{11-u}{6}}\frac{1}{9}(1+u-3v)(11-u-6v)dv\biggr)du\\
&+&\int_4^5\biggl(\frac{u-2}{6}\cdot\frac{1}{12}(-15+22u-3u^2)
+\int_0^{\frac{5-u}{6}}\frac{1}{12}(-15+22u-3u^2-36v^2)dv\\
&&+\int_{\frac{5-u}{6}}^{\frac{2}{3}}\frac{1}{9}(-5+14u-2u^2-15v+3uv-18v^2)dv\\
&&+\int_{\frac{2}{3}}^{\frac{u-2}{3}}\frac{1}{9}(7+14u-2u^2-51v+3uv+9v^2)dv
+\int_{\frac{u-2}{3}}^{\frac{11-u}{6}}\frac{1}{9}(1+u-3v)(11-u-6v)dv\biggr)du\\
&+&\int_5^7\biggl(\frac{u-3}{4}\cdot\frac{1}{24}(145-26u+u^2)
+\int_0^{\frac{13-u}{12}}\frac{1}{24}(145-26u+u^2-48v^2)dv\\
&&+\int_{\frac{13-u}{12}}^{\frac{9-u}{4}}\frac{1}{16}(17-u-4v)(9-u-4v)dv\biggr)du\\
&+&\int_7^9\biggl(\frac{u-3}{4}\cdot\frac{1}{8}(9-u)^2
+\int_0^{\frac{9-u}{4}}\frac{1}{8}(9-u+4v)(9-u-4v)dv\biggr)du\Biggr)
=\frac{207}{224}. 
\end{eqnarray*}
Moreover, we have
\begin{eqnarray*}
&&F_{p_2^{11}}\left(W^{R^{11}, s_R^{11}}_{\bullet,\bullet,\bullet}\right)\\
&=&\frac{6}{28}\Biggl(\int_2^3\int_0^{\frac{u-2}{3}}3v\frac{u-3v-2}{3}dvdu\\
&+&\int_3^4\biggl(\int_0^{\frac{5-u}{6}}3v\frac{u-3v-2}{3}dv
+\int_{\frac{5-u}{6}}^{\frac{u-2}{3}}\frac{5-u+12v}{6}\cdot\frac{u-3v-2}{3}dv\biggr)du\\
&+&\int_4^5\biggl(\int_0^{\frac{5-u}{6}}3v\frac{u-3v-2}{3}dv
+\int_{\frac{5-u}{6}}^{\frac{2}{3}}\frac{5-u+12v}{6}\cdot\frac{u-3v-2}{3}dv\\
&&+\int_{\frac{2}{3}}^{\frac{u-2}{3}}\frac{17-u-6v}{6}\cdot\frac{u-3v-2}{3}dv\biggr)du\\
&+&\int_5^7\biggl(\int_0^{\frac{13-u}{12}}2v\frac{u-4v-1}{4}dv
+\int_{\frac{13-u}{12}}^{\frac{9-u}{4}}\frac{13-u-4v}{4}
\cdot\frac{u-4v-1}{4}dv\biggr)du\\
&+&\int_7^9\int_0^{\frac{9-u}{4}}2v\frac{u-4v-1}{4}dvdu\Biggr)
=\frac{15}{56},
\end{eqnarray*}
\begin{eqnarray*}
&&F_{p_4^{11}}\left(W^{R^{11}, s_R^{11}}_{\bullet,\bullet,\bullet}\right)\\
&=&\frac{6}{28}\Biggl(\int_1^2\int_{\frac{2+u}{6}}^{\frac{u}{2}}
\frac{3+u-4v}{2}\cdot\frac{-2-u+6v}{2}dvdu\\
&+&\int_2^3\int_{\frac{2}{3}}^{\frac{1+u}{3}}\frac{13+u-12v}{6}(3v-2)dvdu
+\int_3^4\int_{\frac{2}{3}}^{\frac{11-u}{6}}\frac{13+u-12v}{6}(3v-2)dvdu\\
&+&\int_4^5\biggl(\int_{\frac{2}{3}}^{\frac{u-2}{3}}\frac{17-u-6v}{6}(3v-2)dv
+\int_{\frac{u-2}{3}}^{\frac{11-u}{6}}\frac{13+u-12v}{6}(3v-2)dv\biggr)du\\
&+&\int_5^7\int_{\frac{13-u}{12}}^{\frac{9-u}{4}}\frac{13-u-4v}{4}
\cdot\frac{-13+u+12v}{4}dvdu\Biggr)=\frac{23}{112},
\end{eqnarray*}
and
\begin{eqnarray*}
&&F_{p_8^{11}}\left(W^{R^{11}, s_R^{11}}_{\bullet,\bullet,\bullet}\right)\\
&=&\frac{6}{28}\Biggl(\int_1^2\biggl(\int_{\frac{1}{2}}^{\frac{2+u}{6}}
\frac{1+2v}{2}\cdot\frac{2v-1}{2}dv+\int_{\frac{2+u}{6}}^{\frac{u}{2}}
\frac{3+u-4v}{2}\cdot\frac{2v-1}{2}dv\biggr)du\\
&+&\int_2^3\biggl(\int_{\frac{5-u}{6}}^{\frac{2}{3}}
\frac{1+u+6v}{6}\cdot\frac{-5+u+6v}{6}dv+\int_{\frac{2}{3}}^{\frac{1+u}{3}}
\frac{13+u-12v}{6}\cdot\frac{-5+u+6v}{6}dv\biggr)du\\
&+&\int_3^4\biggl(\int_{\frac{5-u}{6}}^{\frac{u-2}{3}}
\frac{5-u+12v}{6}\cdot\frac{-5+u+6v}{6}dv+\int_{\frac{u-2}{3}}^{\frac{2}{3}}
\frac{1+u+6v}{6}\cdot\frac{-5+u+6v}{6}dv\\
&&+\int_{\frac{2}{3}}^{\frac{11-u}{6}}\frac{13+u-12v}{6}
\cdot\frac{-5+u+6v}{6}dv\biggr)du\\
&+&\int_4^5\biggl(\int_{\frac{5-u}{6}}^{\frac{2}{3}}
\frac{5-u+12v}{6}\cdot\frac{-5+u+6v}{6}dv+\int_{\frac{2}{3}}^{\frac{u-2}{3}}
\frac{17-u-6v}{6}\cdot\frac{-5+u+6v}{6}dv\\
&&+\int_{\frac{u-2}{3}}^{\frac{11-u}{6}}
\frac{13+u-12v}{6}\cdot\frac{-5+u+6v}{6}dv\biggr)du\\
&+&\int_5^7\biggl(\int_0^{\frac{13-u}{12}}2v\frac{u+4v-5}{4}dv
+\int_{\frac{13-u}{12}}^{\frac{9-u}{4}}
\frac{13-u-4v}{4}\cdot\frac{u+4v-5}{4}dv\biggr)du\\
&+&\int_7^9\int_0^{\frac{9-u}{4}}2v\frac{u+4v-5}{4}dvdu\Biggr)=\frac{25}{56}. 
\end{eqnarray*}
Therefore, for any closed point $p^{11}\in s_R^{11}\subset R^{11}$, we get 
the following inequality: 
\begin{eqnarray*}
&&S\left(W^{R^{11},s_R^{11}}_{\bullet,\bullet,\bullet}; p^{11}\right)\\
&\leq&\frac{25}{56}+\frac{3}{28}\Biggl(\int_0^1\int_0^{\frac{u}{2}}(2v)^2dvdu\\
&+&\int_1^2\biggl(\int_0^{\frac{1}{2}}(2v)^2dv+\int_{\frac{1}{2}}^{\frac{2+u}{6}}
\left(\frac{1+2v}{2}\right)^2dv+\int_{\frac{2+u}{6}}^{\frac{u}{2}}
\left(\frac{3+u-4v}{2}\right)^2dv\biggr)du\\
&+&\int_2^3\biggl(\int_0^{\frac{u-2}{3}}(3v)^2dv+\int_{\frac{u-2}{3}}^{\frac{5-u}{6}}
\left(\frac{-2+u+6v}{3}\right)^2dv\\
&&+\int_{\frac{5-u}{6}}^{\frac{2}{3}}\left(\frac{1+u+6v}{6}\right)^2dv
+\int_{\frac{2}{3}}^{\frac{1+u}{3}}\left(\frac{13+u-12v}{6}\right)^2dv\biggr)du\\
&+&\int_3^4\biggl(\int_0^{\frac{5-u}{6}}(3v)^2dv+\int_{\frac{5-u}{6}}^{\frac{u-2}{3}}
\left(\frac{5-u+12v}{6}\right)^2dv\\
&&+\int_{\frac{u-2}{3}}^{\frac{2}{3}}\left(\frac{1+u+6v}{6}\right)^2dv
+\int_{\frac{2}{3}}^{\frac{11-u}{6}}\left(\frac{13+u-12v}{6}\right)^2dv\biggr)du\\
&+&\int_4^5\biggl(\int_0^{\frac{5-u}{6}}(3v)^2dv+\int_{\frac{5-u}{6}}^{\frac{2}{3}}
\left(\frac{5-u+12v}{6}\right)^2dv\\
&&+\int_{\frac{2}{3}}^{\frac{u-2}{3}}\left(\frac{17-u-6v}{6}\right)^2dv
+\int_{\frac{u-2}{3}}^{\frac{11-u}{6}}\left(\frac{13+u-12v}{6}\right)^2dv\biggr)du\\
&+&\int_5^7\biggl(\int_0^{\frac{13-u}{12}}(2v)^2dv
+\int_{\frac{13-u}{12}}^{\frac{9-u}{4}}\left(\frac{13-u-4v}{4}\right)^2dv\biggr)du\\
&+&\int_7^9\int_0^{\frac{9-u}{4}}(2v)^2dvdu\Biggr)=\frac{25}{56}+\frac{13}{28}
=\frac{51}{56}. 
\end{eqnarray*}
In particular, we get the inequality
\begin{eqnarray*}
\delta_{p_v^1}\left(R^1;V^{\tilde{R}}_{\bullet,\bullet}\right)\geq
\min\left\{\frac{A_{R^1}\left(\tilde{s}_R\right)}
{S\left(V^{\tilde{R}}_{\bullet,\bullet};\tilde{s}_R\right)},\,\,\,
\inf_{p^{11}\in s_R^{11}}\frac{A_{s_R^{11}}\left(p^{11}\right)}
{S\left(W^{R^{11},s_R^{11}}_{\bullet,\bullet,\bullet}; p^{11}\right)}\right\}
=\min\left\{\frac{224}{207},\,\,\,\frac{56}{51}\right\}=\frac{224}{207}
\end{eqnarray*}
by Corollary \ref{reduction_corollary}.

\noindent\underline{\textbf{Step 13}}\\
Let us set $f_S^1:=\gamma_*\tilde{f}_S\subset R^1$. Moreover, set 
$\tilde{p}_{r_1}:=\tilde{r}_1|_{\tilde{f}_S}$, 
$\tilde{p}_{r_3}:=\tilde{r}_3|_{\tilde{f}_S}$, 
$p_{r_1}^1:=\gamma\left(\tilde{p}_{r_1}\right)$, 
$p_{r_3}^1:=\gamma\left(\tilde{p}_{r_3}\right)$. 
Note that $\tilde{p}_{r_1}$ and $\tilde{p}_{r_3}$ are mutually distinct 
reduced points. Moreover, the pair $\left(R^1, f_S^1\right)$ is plt with 
\[
\left(K_{R^1}+f_S^1\right)|_{f_S^1}=K_{f_S^1}+\frac{1}{2}p_v^1,\quad
\gamma^*f_S^1=\frac{1}{2}\tilde{s}_R+\frac{3}{2}\tilde{t}^++\tilde{f}_S. 
\]
Let us set 
\begin{eqnarray*}
P(u,v) &:=& P_\sigma\left(\tilde{R}, P(u)-v\tilde{f}_S\right), \\
N(u,v) &:=& N_\sigma\left(\tilde{R}, P(u)-v\tilde{f}_S\right),
\end{eqnarray*}
where $P(u)$, $N(u)$ are as in Step 11. 
\begin{itemize}
\item
Assume that $u\in[0,1]$. 
\begin{itemize}
\item
If $v\in\left[0,u\right]$, then 
\begin{eqnarray*}
N(u,v)&=&\left[\frac{v}{2},\frac{3v}{2},0,0,0,0,0,0,0,0,0,0\right], \\
P(u,v)&\sim_\R&\left[\frac{u-v-6}{2},\frac{3u-3v-20}{2},u-v-9,0,3,6,8,1,4,1,2,0\right],
\end{eqnarray*}
and 
\[
\left(P(u,v)^{\cdot 2}\right)=\frac{1}{2}(u-v)^2.
\]
\end{itemize}
\item
Assume that $u\in[1,2]$. 
\begin{itemize}
\item
If $v\in\left[0,\frac{7-u}{6}\right]$, then 
\begin{eqnarray*}
N(u,v)&=&\left[\frac{v}{2},\frac{3v}{2},0,0,0,0,0,0,0,0,0,0\right], \\
P(u,v)&\sim_\R&\biggl[\frac{u-v-6}{2},\frac{3u-3v-20}{2},u-v-9,0,
\frac{10-u}{3},\frac{20-2u}{3},\\
&&\quad 9-u,1,\frac{9-u}{2},1,2,0\biggr],
\end{eqnarray*}
and 
\[
\left(P(u,v)^{\cdot 2}\right)=\frac{1}{12}(-7+14u-u^2-12uv+6v^2).
\]
\item
If $v\in\left[\frac{7-u}{6},1\right]$, then 
\begin{eqnarray*}
N(u,v)&=&\left[\frac{v}{2},\frac{3v}{2},0,0,\frac{-7+u+6v}{6},\frac{-7+u+6v}{3},
0,0,0,\frac{-7+u+6v}{2},0,0\right], \\
P(u,v)&\sim_\R&\biggl[\frac{u-v-6}{2},\frac{3u-3v-20}{2},u-v-9,0,
\frac{9-u-2v}{2},9-u-2v,\\
&&\quad 9-u,1,\frac{9-u}{2},\frac{9-u-6v}{2},2,0\biggr],
\end{eqnarray*}
and 
\[
\left(P(u,v)^{\cdot 2}\right)=\frac{7}{2}(1-v)^2.
\]
\end{itemize}
\item
Assume that $u\in\left[2,\frac{11}{3}\right]$. 
\begin{itemize}
\item
If $v\in\left[0,\frac{u-2}{3}\right]$, then 
\begin{eqnarray*}
N(u,v)&=&\left[0,0,0,0,0,0,0,0,0,0,0,0\right], \\
P(u,v)&\sim_\R&\left[\frac{u-8}{3},u-9,u-v-9,0,
\frac{10-u}{3},\frac{20-2u}{3},9-u,1,\frac{9-u}{2},1,2,0\right],
\end{eqnarray*}
and 
\[
\left(P(u,v)^{\cdot 2}\right)=\frac{1}{12}(-15+22u-3u^2-24v-12v^2).
\]
\item
If $v\in\left[\frac{u-2}{3},\frac{7-u}{6}\right]$, then 
\begin{eqnarray*}
N(u,v)&=&\left[\frac{-u+3v+2}{6},\frac{-u+3v+2}{2},0,0,0,0,0,0,0,0,0,0\right], \\
P(u,v)&\sim_\R&\biggl[\frac{u-v-6}{2},\frac{3u-3v-20}{2},u-v-9,0,
\frac{10-u}{3},\frac{20-2u}{3},\\
&&\quad 9-u,1,\frac{9-u}{2},1,2,0\biggr],
\end{eqnarray*}
and 
\[
\left(P(u,v)^{\cdot 2}\right)=\frac{1}{12}(-7+14u-u^2-12uv+6v^2).
\]
\item
If $v\in\left[\frac{7-u}{6},1\right]$, then 
\begin{eqnarray*}
N(u,v)&=&\biggl[\frac{-u+3v+2}{6},\frac{-u+3v+2}{2},0,0,\frac{-7+u+6v}{6},
\frac{-7+u+6v}{3},\\
&&\quad 0,0,0,\frac{-7+u+6v}{2},0,0\biggr], \\
P(u,v)&\sim_\R&\biggl[\frac{u-v-6}{2},\frac{3u-3v-20}{2},u-v-9,0,
\frac{9-u-2v}{2},9-u-2v,\\
&&\quad 9-u,1,\frac{9-u}{2},\frac{9-u-6v}{2},2,0\biggr],
\end{eqnarray*}
and 
\[
\left(P(u,v)^{\cdot 2}\right)=\frac{7}{2}(1-v)^2.
\]
\end{itemize}
\item
Assume that $u\in\left[\frac{11}{3},5\right]$. 
\begin{itemize}
\item
If $v\in\left[0,\frac{7-u}{6}\right]$, then 
\begin{eqnarray*}
N(u,v)&=&\left[0,0,0,0,0,0,0,0,0,0,0,0\right], \\
P(u,v)&\sim_\R&\left[\frac{u-8}{3},u-9,u-v-9,0,
\frac{10-u}{3},\frac{20-2u}{3},9-u,1,\frac{9-u}{2},1,2,0\right],
\end{eqnarray*}
and 
\[
\left(P(u,v)^{\cdot 2}\right)=\frac{1}{12}(-15+22u-3u^2-24v-12v^2).
\]
\item
If $v\in\left[\frac{7-u}{6},\frac{u-2}{3}\right]$, then 
\begin{eqnarray*}
N(u,v)&=&\left[0,0,0,0,\frac{-7+u+6v}{6},\frac{-7+u+6v}{3},0,0,0,
\frac{-7+u+6v}{2},0,0\right], \\
P(u,v)&\sim_\R&\biggl[\frac{u-8}{3},u-9,u-v-9,0,
\frac{9-u-2v}{2},9-u-2v,\\
&&\quad 9-u,1,\frac{9-u}{2},\frac{9-u-6v}{2},2,0\biggr],
\end{eqnarray*}
and 
\[
\left(P(u,v)^{\cdot 2}\right)=\frac{1}{6}(17+4u-u^2-54v+6uv+12v^2).
\]
\item
If $v\in\left[\frac{u-2}{3},1\right]$, then 
\begin{eqnarray*}
N(u,v)&=&\biggl[\frac{-u+3v+2}{6},\frac{-u+3v+2}{2},0,0,\frac{-7+u+6v}{6},
\frac{-7+u+6v}{3},\\
&&\quad 0,0,0,\frac{-7+u+6v}{2},0,0\biggr], \\
P(u,v)&\sim_\R&\biggl[\frac{u-v-6}{2},\frac{3u-3v-20}{2},u-v-9,0,
\frac{9-u-2v}{2},9-u-2v,\\
&&\quad 9-u,1,\frac{9-u}{2},\frac{9-u-6v}{2},2,0\biggr],
\end{eqnarray*}
and 
\[
\left(P(u,v)^{\cdot 2}\right)=\frac{7}{2}(1-v)^2.
\]
\end{itemize}
\item
Assume that $u\in[5,7]$. 
\begin{itemize}
\item
If $v\in\left[0,\frac{7-u}{6}\right]$, then 
\begin{eqnarray*}
N(u,v)&=&\left[0,0,0,0,0,0,0,0,0,0,0,0\right], \\
P(u,v)&\sim_\R&\biggl[\frac{u-9}{4},u-9,u-v-9,0,
\frac{10-u}{3},\frac{20-2u}{3},\\
&&\quad 9-u,\frac{9-u}{4},\frac{9-u}{2},1,\frac{9-u}{2},0\biggr],
\end{eqnarray*}
and 
\[
\left(P(u,v)^{\cdot 2}\right)=\frac{1}{24}(145-26u+u^2-48v-24v^2).
\]
\item
If $v\in\left[\frac{7-u}{6},\frac{9-u}{4}\right]$, then 
\begin{eqnarray*}
N(u,v)&=&\left[0,0,0,0,\frac{-7+u+6v}{6},\frac{-7+u+6v}{3},
0,0,0,\frac{-7+u+6v}{2},0,0\right], \\
P(u,v)&\sim_\R&\biggl[\frac{u-9}{4},u-9,u-v-9,0,
\frac{9-u-2v}{2},9-u-2v,\\
&&\quad 9-u,\frac{9-u}{4},
\frac{9-u}{2},\frac{9-u-6v}{2},\frac{9-u}{2},0\biggr],
\end{eqnarray*}
and 
\[
\left(P(u,v)^{\cdot 2}\right)=\frac{1}{8}(9-u-4v)^2.
\]
\end{itemize}
\item
Assume that $u\in[7,9]$. 
\begin{itemize}
\item
If $v\in\left[0,\frac{9-u}{4}\right]$, then 
\begin{eqnarray*}
N(u,v)&=&\left[0,0,0,0,v,2v,
0,0,0,3v,0,0\right], \\
P(u,v)&\sim_\R&\biggl[\frac{u-9}{4},u-9,u-v-9,0,
\frac{9-u-2v}{2},9-u-2v,\\
&&\quad 9-u,\frac{9-u}{4},
\frac{9-u}{2},\frac{9-u-6v}{2},\frac{9-u}{2},0\biggr],
\end{eqnarray*}
and 
\[
\left(P(u,v)^{\cdot 2}\right)=\frac{1}{8}(9-u-4v)^2.
\]
\end{itemize}
\end{itemize}
Therefore we get 
\begin{eqnarray*}
&&S\left(V^{\tilde{R}}_{\bullet,\bullet}; \tilde{f}_S\right)\\
&=&\frac{3}{28}\Biggl(\int_0^1\int_0^u\frac{1}{2}(u-v)^2dvdu\\
&+&\int_1^2\biggl(\int_0^{\frac{7-u}{6}}\frac{1}{12}(-7+14u-u^2-12uv+6v^2)dv
+\int_{\frac{7-u}{6}}^1\frac{7}{2}(1-v)^2dv\biggr)du\\
&+&\int_2^{\frac{11}{3}}\biggl(\int_0^{\frac{u-2}{3}}
\frac{1}{12}(-15+22u-3u^2-24v-12v^2)dv\\
&&+\int_{\frac{u-2}{3}}^{\frac{7-u}{6}}\frac{1}{12}(-7+14u-u^2-12uv+6v^2)dv
+\int_{\frac{7-u}{6}}^1\frac{7}{2}(1-v)^2dv\biggr)du\\
&+&\int_{\frac{11}{3}}^5\biggl(\int_0^{\frac{7-u}{6}}
\frac{1}{12}(-15+22u-3u^2-24v-12v^2)dv\\
&&+\int_{\frac{7-u}{6}}^{\frac{u-2}{3}}\frac{1}{6}(17+4u-u^2-54v+6uv+12v^2)dv
+\int_{\frac{u-2}{3}}^1\frac{7}{2}(1-v)^2dv\biggr)du\\
&+&\int_5^7\biggl(\int_0^{\frac{7-u}{6}}\frac{1}{24}(145-26u+u^2-48v-24v^2)dv
+\int_{\frac{7-u}{6}}^{\frac{9-u}{4}}\frac{1}{8}(9-u-4v)^2dv\biggr)du\\
&+&\int_7^9\int_0^{\frac{9-u}{4}}\frac{1}{8}(9-u-4v)^2dvdu\Biggr)
=\frac{3}{8}. 
\end{eqnarray*}
Moreover, we have
\begin{eqnarray*}
&&F_{p_{r_1}^1}\left(W^{R^1, f_S^1}_{\bullet,\bullet,\bullet}\right)\\
&=&\frac{6}{28}\Biggl(\int_1^{\frac{11}{3}}\int_{\frac{7-u}{6}}^1
\frac{7}{2}(1-v)\frac{-7+u+6v}{2}dvdu\\
&+&\int_{\frac{11}{3}}^5\biggl(\int_{\frac{7-u}{6}}^{\frac{u-2}{3}}
\frac{9-u-4v}{2}\cdot\frac{-7+u+6v}{2}dv+\int_{\frac{u-2}{3}}^1
\frac{7}{2}(1-v)\frac{-7+u+6v}{2}dv\biggr)du\\
&+&\int_5^7\int_{\frac{7-u}{6}}^{\frac{9-u}{4}}
\frac{9-u-4v}{2}\cdot\frac{-7+u+6v}{2}dvdu
+\int_7^9\int_0^{\frac{9-u}{4}}\frac{9-u-4v}{2}\cdot\frac{-7+u+6v}{2}dvdu\Biggr)\\
&=&\frac{103}{504}
\end{eqnarray*}
and $F_{p_{r_3}^1}\left(W^{R^1, f_S^1}_{\bullet,\bullet,\bullet}\right)=0$. 
Thus, for any closed point $p^1\in f_S^1\setminus\left\{p_v^1\right\}$, we have 
\begin{eqnarray*}
&&S\left(W^{{R^1},f_S^1}_{\bullet,\bullet,\bullet}; p^1\right)\\
&\leq&\frac{103}{504}+\frac{3}{28}\Biggl(\int_0^1\int_0^u
\left(\frac{u-v}{2}\right)^2dvdu\\
&+&\int_1^2\biggl(\int_0^{\frac{7-u}{6}}\left(\frac{u-v}{2}\right)^2dv
+\int_{\frac{7-u}{6}}^1\left(\frac{7}{2}(1-v)\right)^2dv\biggr)du\\
&+&\int_2^{\frac{11}{3}}\biggl(\int_0^{\frac{u-2}{3}}(1+v)^2dv
+\int_{\frac{u-2}{3}}^{\frac{7-u}{6}}\left(\frac{u-v}{2}\right)^2dv
+\int_{\frac{7-u}{6}}^1\left(\frac{7}{2}(1-v)\right)^2dv\biggr)du\\
&+&\int_{\frac{11}{3}}^5\biggl(\int_0^{\frac{7-u}{6}}(1+v)^2dv
+\int_{\frac{7-u}{6}}^{\frac{u-2}{3}}\left(\frac{9-u-4v}{2}\right)^2dv
+\int_{\frac{u-2}{3}}^1\left(\frac{7}{2}(1-v)\right)^2dv\biggr)du\\
&+&\int_5^7\biggl(\int_0^{\frac{7-u}{6}}(1+v)^2dv
+\int_{\frac{7-u}{6}}^{\frac{9-u}{4}}\left(\frac{9-u-4v}{2}\right)^2dv\biggr)du\\
&+&\int_7^9\int_0^{\frac{9-u}{4}}\left(\frac{9-u-4v}{2}\right)^2dvdu\Biggr)
=\frac{103}{504}+\frac{487}{1008}=\frac{11}{16}.
\end{eqnarray*}
In particular, we get the inequality
\begin{eqnarray*}
\delta_{p^1}\left(R^1; V^{\tilde{R}}_{\bullet,\bullet}\right)\geq
\min\left\{\frac{A_{R^1}\left(f_S^1\right)}
{S\left(V^{\tilde{R}}_{\bullet,\bullet};\tilde{f}_S\right)},\,\,\,
\frac{A_{f_S^1,\frac{1}{2}p_t^1}\left(p^1\right)}
{S\left(W^{R^1,f_S^1}_{\bullet,\bullet,\bullet}; p^1\right)}\right\}
\geq\min\left\{\frac{8}{3},\,\,\,\frac{16}{11}\right\}=\frac{16}{11}
\end{eqnarray*}
by Corollary \ref{reduction_corollary}.

\noindent\underline{\textbf{Step 14}}\\
Let us set $f_R^1:=\gamma_*\tilde{f}_R$. Note that the pair 
$(R^1,f_R^1)$ is plt and $\left(K_{R^1}+f_R^1\right)|_{f_R^1}=K_{f_R^1}+(1/2)p_v^1$. 
Let us set 
\begin{eqnarray*}
P(u,v) &:=& P_\sigma\left(\tilde{R}, P(u)-v\tilde{f}_R\right), \\
N(u,v) &:=& N_\sigma\left(\tilde{R}, P(u)-v\tilde{f}_R\right),
\end{eqnarray*}
where $P(u)$, $N(u)$ are as in Step 11. 
\begin{itemize}
\item
Assume that $u\in[0,1]$. 
\begin{itemize}
\item
If $v\in\left[0,u\right]$, then 
\begin{eqnarray*}
N(u,v)&=&\left[\frac{v}{2},\frac{v}{2},0,0,v,v,v,0,0,0,0,0\right], \\
P(u,v)&\sim_\R&\left[\frac{u-v-6}{2},\frac{3u-v-20}{2},u-9,-v,3-v,6-v,
8-v,1,4,1,2,0\right],
\end{eqnarray*}
and 
\[
\left(P(u,v)^{\cdot 2}\right)=\frac{1}{2}(u-v)^2.
\]
\end{itemize}
\item
Assume that $u\in[1,2]$. 
\begin{itemize}
\item
If $v\in\left[0,u-1\right]$, then 
\begin{eqnarray*}
N(u,v)&=&\left[\frac{v}{2},\frac{v}{2},0,0,\frac{2v}{3},\frac{v}{3},0,0,0,0,0,0\right], \\
P(u,v)&\sim_\R&\biggl[\frac{u-v-6}{2},\frac{3u-v-20}{2},u-9,-v,
\frac{10-u-2v}{3},\frac{20-2u-v}{3},\\
&&\quad 9-u,1,\frac{9-u}{2},1,2,0\biggr],
\end{eqnarray*}
and 
\[
\left(P(u,v)^{\cdot 2}\right)=\frac{1}{12}(-7+14u-u^2-8v-4uv+2v^2).
\]
\item
If $v\in\left[u-1,1\right]$, then 
\begin{eqnarray*}
N(u,v)&=&\left[\frac{v}{2},\frac{v}{2},0,0,\frac{1-u+3v}{3},\frac{2-2u+3v}{3},
1-u+v,0,0,0,0,0\right], \\
P(u,v)&\sim_\R&\biggl[\frac{u-v-6}{2},\frac{3u-v-20}{2},u-9,-v,
3-v,6-v,\\
&&\quad 8-v,1,\frac{9-u}{2},1,2,0\biggr],
\end{eqnarray*}
and 
\[
\left(P(u,v)^{\cdot 2}\right)=\frac{1}{4}(-1+2u+u^2-4uv+2v^2).
\]
\item
If $v\in\left[1,\frac{1+u}{2}\right]$, then 
\begin{eqnarray*}
N(u,v)&=&\left[\frac{2v-1}{2},\frac{2v-1}{2},0,0,\frac{1-u+3v}{3},\frac{2-2u+3v}{3},
1-u+v,v-1,0,0,0,0\right], \\
P(u,v)&\sim_\R&\biggl[\frac{u-v-6}{2},\frac{3u-v-20}{2},u-9,-v,
3-v,6-v,\\
&&\quad 8-v,2-v,\frac{9-u}{2},1,2,0\biggr],
\end{eqnarray*}
and 
\[
\left(P(u,v)^{\cdot 2}\right)=\frac{1}{4}(1+u-2v)^2.
\]
\end{itemize}
\item
Assume that $u\in\left[2,3\right]$. 
\begin{itemize}
\item
If $v\in\left[0,u-2\right]$, then 
\begin{eqnarray*}
N(u,v)&=&\left[\frac{v}{3},0,0,0,\frac{2v}{3},\frac{v}{3},0,0,0,0,0,0\right], \\
P(u,v)&\sim_\R&\left[\frac{u-v-8}{3},u-9,u-9,-v,
\frac{10-u-2v}{3},\frac{20-2u-v}{3},9-u,1,\frac{9-u}{2},1,2,0\right],
\end{eqnarray*}
and 
\[
\left(P(u,v)^{\cdot 2}\right)=\frac{1}{12}(-15+22u-3u^2-16v).
\]
\item
If $v\in\left[u-2,1\right]$, then 
\begin{eqnarray*}
N(u,v)&=&\left[\frac{-u+3v+2}{6},\frac{-u+v+2}{2},0,0,\frac{2v}{3},\frac{v}{3},
0,0,0,0,0,0\right], \\
P(u,v)&\sim_\R&\biggl[\frac{u-v-6}{2},\frac{3u-v-20}{2},u-9,-v,
\frac{10-u-2v}{3},\frac{20-2u-v}{3},\\
&&\quad 9-u,1,\frac{9-u}{2},1,2,0\biggr],
\end{eqnarray*}
and 
\[
\left(P(u,v)^{\cdot 2}\right)=\frac{1}{12}(-7+14u-u^2-8v-4uv+2v^2).
\]
\item
If $v\in\left[1,u-1\right]$, then 
\begin{eqnarray*}
N(u,v)&=&\biggl[\frac{-u+6v-1}{6},\frac{-u+2v+1}{2},0,0,\frac{2v}{3},
\frac{v}{3},\\
&&\quad 0,v-1,0,0,0,0\biggr], \\
P(u,v)&\sim_\R&\biggl[\frac{u-2v-5}{2},\frac{3u-2v-19}{2},u-9,-v,
\frac{10-u-2v}{3},\frac{20-2u-v}{3},\\
&&\quad 9-u,2-v,\frac{9-u}{2},1,2,0\biggr],
\end{eqnarray*}
and 
\[
\left(P(u,v)^{\cdot 2}\right)=\frac{1}{12}(-1+14u-u^2-20v-4uv+8v^2).
\]
\item
If $v\in\left[u-1,\frac{1+u}{2}\right]$, then 
\begin{eqnarray*}
N(u,v)&=&\biggl[\frac{-u+6v-1}{6},\frac{-u+2v+1}{2},0,0,\frac{1-u+3v}{3},
\frac{2-2u+3v}{3},\\
&&\quad 1-u+v,v-1,0,0,0,0\biggr], \\
P(u,v)&\sim_\R&\biggl[\frac{u-2v-5}{2},\frac{3u-2v-19}{2},u-9,-v,
3-v,6-v,\\
&&\quad 8-v,2-v,\frac{9-u}{2},1,2,0\biggr],
\end{eqnarray*}
and 
\[
\left(P(u,v)^{\cdot 2}\right)=\frac{1}{4}(1+u-2v)^2.
\]
\end{itemize}
\item
Assume that $u\in\left[3,4\right]$. 
\begin{itemize}
\item
If $v\in\left[0,\frac{5-u}{2}\right]$, then 
\begin{eqnarray*}
N(u,v)&=&\left[\frac{v}{3},0,0,0,\frac{2v}{3},\frac{v}{3},0,0,0,0,0,0\right], \\
P(u,v)&\sim_\R&\left[\frac{u-v-8}{3},u-9,u-9,-v,
\frac{10-u-2v}{3},\frac{20-2u-v}{3},9-u,1,\frac{9-u}{2},1,2,0\right],
\end{eqnarray*}
and 
\[
\left(P(u,v)^{\cdot 2}\right)=\frac{1}{12}(-15+22u-3u^2-16v).
\]
\item
If $v\in\left[\frac{5-u}{2},\frac{u-1}{2}\right]$, then 
\begin{eqnarray*}
N(u,v)&=&\left[\frac{u-5+6v}{12},0,0,0,\frac{2v}{3},\frac{v}{3},
0,\frac{-5+u+2v}{4},0,0,0,0\right], \\
P(u,v)&\sim_\R&\biggl[\frac{u-2v-9}{4},u-9,u-9,-v,
\frac{10-u-2v}{3},\frac{20-2u-v}{3},\\
&&\quad 9-u,\frac{9-u-2v}{4},\frac{9-u}{2},1,2,0\biggr],
\end{eqnarray*}
and 
\[
\left(P(u,v)^{\cdot 2}\right)=\frac{1}{24}(-5+34u-5u^2-52v+4uv+4v^2).
\]
\item
If $v\in\left[\frac{u-1}{2},\frac{7-u}{2}\right]$, then 
\begin{eqnarray*}
N(u,v)&=&\biggl[\frac{-u+6v-1}{6},\frac{-u+2v+1}{2},0,0,\frac{2v}{3},
\frac{v}{3},\\
&&\quad 0,v-1,0,0,0,0\biggr], \\
P(u,v)&\sim_\R&\biggl[\frac{u-2v-5}{2},\frac{3u-2v-19}{2},u-9,-v,
\frac{10-u-2v}{3},\frac{20-2u-v}{3},\\
&&\quad 9-u,2-v,\frac{9-u}{2},1,2,0\biggr],
\end{eqnarray*}
and 
\[
\left(P(u,v)^{\cdot 2}\right)=\frac{1}{12}(-1+14u-u^2-20v-4uv+8v^2).
\]
\item
If $v\in\left[\frac{7-u}{2},2\right]$, then 
\begin{eqnarray*}
N(u,v)&=&\biggl[\frac{-u+6v-1}{6},\frac{-u+2v+1}{2},0,0,\frac{u+6v-7}{6},
\frac{u+3v-7}{3},\\
&&\quad 0,v-1,0,\frac{-7+u+2v}{2},0,0\biggr], \\
P(u,v)&\sim_\R&\biggl[\frac{u-2v-5}{2},\frac{3u-2v-19}{2},u-9,-v,
\frac{9-u-2v}{2},9-u-v,\\
&&\quad 9-u,2-v,\frac{9-u}{2},\frac{9-u-2v}{2},2,0\biggr],
\end{eqnarray*}
and 
\[
\left(P(u,v)^{\cdot 2}\right)=(2-v)^2.
\]
\end{itemize}
\item
Assume that $u\in\left[4,5\right]$. 
\begin{itemize}
\item
If $v\in\left[0,\frac{5-u}{2}\right]$, then 
\begin{eqnarray*}
N(u,v)&=&\left[\frac{v}{3},0,0,0,\frac{2v}{3},\frac{v}{3},0,0,0,0,0,0\right], \\
P(u,v)&\sim_\R&\left[\frac{u-v-8}{3},u-9,u-9,-v,
\frac{10-u-2v}{3},\frac{20-2u-v}{3},9-u,1,\frac{9-u}{2},1,2,0\right],
\end{eqnarray*}
and 
\[
\left(P(u,v)^{\cdot 2}\right)=\frac{1}{12}(-15+22u-3u^2-16v).
\]
\item
If $v\in\left[\frac{5-u}{2},\frac{7-u}{2}\right]$, then 
\begin{eqnarray*}
N(u,v)&=&\left[\frac{u-5+6v}{12},0,0,0,\frac{2v}{3},\frac{v}{3},
0,\frac{-5+u+2v}{4},0,0,0,0\right], \\
P(u,v)&\sim_\R&\biggl[\frac{u-2v-9}{4},u-9,u-9,-v,
\frac{10-u-2v}{3},\frac{20-2u-v}{3},\\
&&\quad 9-u,\frac{9-u-2v}{4},\frac{9-u}{2},1,2,0\biggr],
\end{eqnarray*}
and 
\[
\left(P(u,v)^{\cdot 2}\right)=\frac{1}{24}(-5+34u-5u^2-52v+4uv+4v^2).
\]
\item
If $v\in\left[\frac{7-u}{2},\frac{u-1}{2}\right]$, then 
\begin{eqnarray*}
N(u,v)&=&\biggl[\frac{u-5+6v}{12},0,0,0,\frac{u+6v-7}{6},
\frac{u+3v-7}{3},\\
&&\quad 0,\frac{-5+u+2v}{4},0,\frac{-7+u+2v}{2},0,0\biggr], \\
P(u,v)&\sim_\R&\biggl[\frac{u-2v-9}{4},u-9,u-9,-v,
\frac{9-u-2v}{2},9-u-v,\\
&&\quad 9-u,\frac{9-u-2v}{4},\frac{9-u}{2},\frac{9-u-2v}{2},2,0\biggr],
\end{eqnarray*}
and 
\[
\left(P(u,v)^{\cdot 2}\right)=\frac{1}{8}(31+2u-u^2-36v+4uv+4v^2).
\]
\item
If $v\in\left[\frac{u-1}{2},2\right]$, then 
\begin{eqnarray*}
N(u,v)&=&\biggl[\frac{-u+6v-1}{6},\frac{-u+2v+1}{2},0,0,\frac{u+6v-7}{6},
\frac{u+3v-7}{3},\\
&&\quad 0,v-1,0,\frac{-7+u+2v}{2},0,0\biggr], \\
P(u,v)&\sim_\R&\biggl[\frac{u-2v-5}{2},\frac{3u-2v-19}{2},u-9,-v,
\frac{9-u-2v}{2},9-u-v,\\
&&\quad 9-u,2-v,\frac{9-u}{2},\frac{9-u-2v}{2},2,0\biggr],
\end{eqnarray*}
and 
\[
\left(P(u,v)^{\cdot 2}\right)=(2-v)^2.
\]
\end{itemize}
\item
Assume that $u\in[5,7]$. 
\begin{itemize}
\item
If $v\in\left[0,\frac{7-u}{2}\right]$, then 
\begin{eqnarray*}
N(u,v)&=&\left[\frac{v}{2},0,0,0,\frac{2v}{3},\frac{v}{3},0,\frac{v}{2},0,0,0,0\right], \\
P(u,v)&\sim_\R&\biggl[\frac{u-2v-9}{4},u-9,u-9,-v,
\frac{10-u-2v}{3},\frac{20-2u-v}{3},\\
&&\quad 9-u,\frac{9-u-2v}{4},\frac{9-u}{2},1,\frac{9-u}{2},0\biggr],
\end{eqnarray*}
and 
\[
\left(P(u,v)^{\cdot 2}\right)=\frac{1}{24}(145-26u+u^2-52v+4uv+4v^2).
\]
\item
If $v\in\left[\frac{7-u}{2},\frac{9-u}{2}\right]$, then 
\begin{eqnarray*}
N(u,v)&=&\left[\frac{v}{2},0,0,0,\frac{u+6v-7}{6},\frac{u+3v-7}{3},
0,\frac{v}{2},0,\frac{-7+u+2v}{2},0,0\right], \\
P(u,v)&\sim_\R&\biggl[\frac{u-2v-9}{4},u-9,u-9,-v,
\frac{9-u-2v}{2},9-u-v,\\
&&\quad 9-u,\frac{9-u-2v}{4},\frac{9-u}{2},\frac{9-u-2v}{2},\frac{9-u}{2},0\biggr],
\end{eqnarray*}
and 
\[
\left(P(u,v)^{\cdot 2}\right)=\frac{1}{8}(9-u-2v)^2.
\]
\end{itemize}
\item
Assume that $u\in[7,9]$. 
\begin{itemize}
\item
If $v\in\left[0,\frac{9-u}{2}\right]$, then 
\begin{eqnarray*}
N(u,v)&=&\left[\frac{v}{2},0,0,0,v,v,
0,\frac{v}{2},0,v,0,0\right], \\
P(u,v)&\sim_\R&\biggl[\frac{u-2v-9}{4},u-9,u-9,-v,
\frac{9-u-2v}{2},9-u-v,\\
&&\quad 9-u,\frac{9-u-2v}{4},
\frac{9-u}{2},\frac{9-u-2v}{2},\frac{9-u}{2},0\biggr],
\end{eqnarray*}
and 
\[
\left(P(u,v)^{\cdot 2}\right)=\frac{1}{8}(9-u-2v)^2.
\]
\end{itemize}
\end{itemize}
Therefore, we get 
\begin{eqnarray*}
&&S\left(V^{\tilde{R}}_{\bullet,\bullet}; \tilde{f}_R\right)\\
&=&\frac{3}{28}\Biggl(\int_0^1\int_0^u\frac{1}{2}(u-v)^2dvdu\\
&+&\int_1^2\biggl(\int_0^{u-1}\frac{1}{12}(-7+14u-u^2-8v-4uv+2v^2)dv\\
&&+\int_{u-1}^1\frac{1}{4}(-1+2u+u^2-4uv+2v^2)dv
+\int_1^{\frac{1+u}{2}}\frac{1}{4}(1+u-2v)^2dv\biggr)du\\
&+&\int_2^3\biggl(\int_0^{u-1}\frac{1}{12}(-15+22u-3u^2-16v)dv\\
&&+\int_{u-2}^1\frac{1}{12}(-7+14u-u^2-8v-4uv+2v^2)dv\\
&&+\int_1^{u-1}\frac{1}{12}(-1+14u-u^2-20v-4uv+8v^2)dv
+\int_{u-1}^{\frac{1+u}{2}}\frac{1}{4}(1+u-2v)^2dv\biggr)du\\
&+&\int_3^4\biggl(\int_0^{\frac{5-u}{2}}\frac{1}{12}(-15+22u-3u^2-16v)dv\\
&&+\int_{\frac{5-u}{2}}^{\frac{u-1}{2}}\frac{1}{24}(-5+34u-5u^2-52v+4uv+4v^2)dv\\
&&+\int_{\frac{u-1}{2}}^{\frac{7-u}{2}}\frac{1}{12}(-1+14u-u^2-20v-4uv+8v^2)dv
+\int_{\frac{7-u}{2}}^2(2-v)^2dv\biggr)du\\
&+&\int_4^5\biggl(\int_0^{\frac{5-u}{2}}\frac{1}{12}(-15+22u-3u^2-16v)dv\\
&&+\int_{\frac{5-u}{2}}^{\frac{7-u}{2}}\frac{1}{24}(-5+34u-5u^2-52v+4uv+4v^2)dv\\
&&+\int_{\frac{7-u}{2}}^{\frac{u-1}{2}}\frac{1}{8}(31+2u-u^2-36v+4uv+4v^2)dv
+\int_{\frac{u-1}{2}}^2(2-v)^2dv\biggr)du\\
&+&\int_5^7\biggl(\int_0^{\frac{7-u}{2}}\frac{1}{24}(145-26u+u^2-52v+4uv+4v^2)dv
+\int_{\frac{7-u}{2}}^{\frac{9-u}{2}}\frac{1}{8}(9-u-2v)^2dv\biggr)du\\
&+&\int_7^9\int_0^{\frac{9-u}{2}}\frac{1}{8}(9-u-2v)^2dvdu\Biggr)
=\frac{75}{112}. 
\end{eqnarray*}
On the other hand, for any closed point 
$p^1\in f_R^1\setminus\left\{p_v^1, p_h^1\right\}$, we have
\begin{eqnarray*}
&&S\left(W^{{R^1},f_R^1}_{\bullet,\bullet,\bullet}; p^1\right)\\
&\leq&\frac{3}{28}\Biggl(\int_0^1\int_0^u\left(\frac{u-v}{2}\right)^2dvdu\\
&+&\int_1^2\biggl(\int_0^{u-1}\left(\frac{2+u-v}{6}\right)^2dv
+\int_{u-1}^1\left(\frac{u-v}{2}\right)^2dv
+\int_1^{\frac{1+u}{2}}\left(\frac{1+u-2v}{2}\right)^2dv\biggr)du\\
&+&\int_2^3\biggl(\int_0^{u-2}\left(\frac{2}{3}\right)^2dv
+\int_{u-2}^1\left(\frac{2+u-v}{6}\right)^2dv\\
&&+\int_1^{u-1}\left(\frac{5+u-4v}{6}\right)^2dv
+\int_{u-1}^{\frac{1+u}{2}}\left(\frac{1+u-2v}{2}\right)^2dv\biggr)du\\
&+&\int_3^4\biggl(\int_0^{\frac{5-u}{2}}\left(\frac{2}{3}\right)^2dv
+\int_{\frac{5-u}{2}}^{\frac{u-1}{2}}\left(\frac{13-u-2v}{12}\right)^2dv\\
&&+\int_{\frac{u-1}{2}}^{\frac{7-u}{2}}\left(\frac{5+u-4v}{6}\right)^2dv
+\int_{\frac{7-u}{2}}^2(2-v)^2dv\biggr)du\\
&+&\int_4^5\biggl(\int_0^{\frac{5-u}{2}}\left(\frac{2}{3}\right)^2dv
+\int_{\frac{5-u}{2}}^{\frac{7-u}{2}}\left(\frac{13-u-2v}{12}\right)^2dv\\
&&+\int_{\frac{7-u}{2}}^{\frac{u-1}{2}}\left(\frac{9-u-2v}{4}\right)^2dv
+\int_{\frac{u-1}{2}}^2(2-v)^2dv\biggr)du\\
&+&\int_5^7\biggl(\int_0^{\frac{7-u}{2}}\left(\frac{13-u-2v}{12}\right)^2dv
+\int_{\frac{7-u}{2}}^{\frac{9-u}{2}}\left(\frac{9-u-2v}{4}\right)^2dv\biggr)du\\
&+&\int_7^9\int_0^{\frac{9-u}{2}}\left(\frac{9-u-2v}{4}\right)^2dvdu\Biggr)
=\frac{29}{112}. 
\end{eqnarray*}
Therefore, we get the inequality
\begin{eqnarray*}
\delta_{p^1}\left(R^1; V^{\tilde{R}}_{\bullet,\bullet}\right)\geq
\min\left\{\frac{A_{R^1}\left(f_R^1\right)}
{S\left(V^{\tilde{R}}_{\bullet,\bullet};\tilde{f}_R\right)},\,\,\,
\frac{A_{f_R^1,\frac{1}{2}p_v^1}\left(p^1\right)}
{S\left(W^{R^1,f_R^1}_{\bullet,\bullet,\bullet}; p^1\right)}\right\}
=\min\left\{\frac{112}{75},\,\,\,\frac{112}{29}\right\}=\frac{112}{75}
\end{eqnarray*}
by Corollary \ref{reduction_corollary}.

\noindent\underline{\textbf{Step 15}}\\
Let $\theta\colon\hat{R}\to R^1$ be the extraction of the divisor 
$\tilde{h}_2\subset\tilde{R}$, let $\gamma_h\colon\tilde{R}\to \hat{R}$ be the 
natural morphism, and let us set $\hat{h}_2:=(\gamma_h)_*\tilde{h}_2$. 
Let us set 
$\tilde{p}_{h,5}:=\tilde{h}_1|_{\tilde{h}_2}$, 
$\tilde{p}_{h,7}:=\tilde{l}^+|_{\tilde{h}_2}$, 
$\tilde{p}_{h,10}:=\tilde{r}_1|_{\tilde{h}_2}$. 
Moreover, we set 
$\hat{p}_{h,5}:=\gamma_h\left(\tilde{p}_{h,5}\right)$, 
$\hat{p}_{h,7}:=\gamma_h\left(\tilde{p}_{h,7}\right)$, 
$\hat{p}_{h,10}:=\gamma_h\left(\tilde{p}_{h,10}\right)$. 
We know that the morphism $\theta$ is a plt-blowup with 
\[
A_{R^1}\left(\tilde{h}_2\right)=3,\quad 
\left(K_{\hat{R}}+\hat{h}_2\right)|_{\hat{h}_2}=K_{\hat{h}_2}+\frac{1}{2}\hat{p}_{h,5}, \quad
(\gamma_h)^*\hat{h}_2=\frac{1}{2}\tilde{h}_1+\tilde{h}_2+\tilde{l}^+. 
\]
Let us set 
\begin{eqnarray*}
P(u,v) &:=& P_\sigma\left(\tilde{R}, P(u)-v\tilde{h}_2\right), \\
N(u,v) &:=& N_\sigma\left(\tilde{R}, P(u)-v\tilde{h}_2\right),
\end{eqnarray*}
where $P(u)$, $N(u)$ are as in Step 11. 
\begin{itemize}
\item
Assume that $u\in[0,1]$. 
\begin{itemize}
\item
If $v\in\left[0,u\right]$, then 
\begin{eqnarray*}
N(u,v)&=&\left[0,0,0,0,\frac{v}{2},0,v,0,0,0,0,0\right], \\
P(u,v)&\sim_\R&\left[\frac{u-6}{2},\frac{3u-20}{2},u-9,0,\frac{6-v}{2},6-v,
8-v,1,4,1,2,0\right],
\end{eqnarray*}
and 
\[
\left(P(u,v)^{\cdot 2}\right)=\frac{1}{2}(u-v)(u+v).
\]
\end{itemize}
\item
Assume that $u\in[1,2]$. 
\begin{itemize}
\item
If $v\in\left[0,\frac{u-1}{3}\right]$, then 
\begin{eqnarray*}
N(u,v)&=&\left[0,0,0,0,\frac{v}{2},0,0,0,0,0,0,0\right], \\
P(u,v)&\sim_\R&\biggl[\frac{u-6}{2},\frac{3u-20}{2},u-9,0,
\frac{20-2u-3v}{6},\frac{20-2u-3v}{3},\\
&&\quad 9-u,1,\frac{9-u}{2},1,2,0\biggr],
\end{eqnarray*}
and 
\[
\left(P(u,v)^{\cdot 2}\right)=\frac{1}{12}(-7+14u-u^2-18v^2).
\]
\item
If $v\in\left[\frac{u-1}{3},\frac{7-u}{6}\right]$, then 
\begin{eqnarray*}
N(u,v)&=&\left[0,0,0,0,\frac{v}{2},0,
\frac{-u+1+3v}{3},0,0,0,0,0\right], \\
P(u,v)&\sim_\R&\biggl[\frac{u-6}{2},\frac{3u-20}{2},u-9,0,
\frac{20-2u-3v}{6},\frac{20-2u-3v}{3},\\
&&\quad \frac{26-2u-3v}{3},1,\frac{9-u}{2},1,2,0\biggr],
\end{eqnarray*}
and 
\[
\left(P(u,v)^{\cdot 2}\right)=\frac{1}{36}(-17+34u+u^2+24v-24uv-18v^2).
\]
\item
If $v\in\left[\frac{7-u}{6},\frac{2+u}{3}\right]$, then 
\begin{eqnarray*}
N(u,v)&=&\left[0,0,0,0,\frac{v}{2},0,
\frac{-u+1+3v}{3},0,0,\frac{-7+u+6v}{6},0,0\right], \\
P(u,v)&\sim_\R&\biggl[\frac{u-6}{2},\frac{3u-20}{2},u-9,0,
\frac{20-2u-3v}{6},\frac{20-2u-3v}{3},\\
&&\quad \frac{26-2u-3v}{3},1,\frac{9-u}{2},\frac{13-u-6v}{6},2,0\biggr],
\end{eqnarray*}
and 
\[
\left(P(u,v)^{\cdot 2}\right)=\frac{1}{18}(2+u-3v)(8+u-3v).
\]
\end{itemize}
\item
Assume that $u\in\left[2,3\right]$. 
\begin{itemize}
\item
If $v\in\left[0,\frac{u-1}{3}\right]$, then 
\begin{eqnarray*}
N(u,v)&=&\left[0,0,0,0,\frac{v}{2},0,0,0,0,0,0,0\right], \\
P(u,v)&\sim_\R&\biggl[\frac{u-8}{3},u-9,u-9,0,
\frac{20-2u-3v}{6},\frac{20-2u-3v}{3},9-u,1,\frac{9-u}{2},1,2,0\biggr],
\end{eqnarray*}
and 
\[
\left(P(u,v)^{\cdot 2}\right)=\frac{1}{12}(-15+22u-3u^2-18v^2).
\]
\item
If $v\in\left[\frac{u-1}{3},\frac{7-u}{6}\right]$, then 
\begin{eqnarray*}
N(u,v)&=&\left[0,0,0,0,\frac{v}{2},0,\frac{-u+1+3v}{3},0,0,0,0,0\right], \\
P(u,v)&\sim_\R&\biggl[\frac{u-8}{3},u-9,u-9,0,
\frac{20-2u-3v}{6},\frac{20-2u-3v}{3},\\
&&\quad \frac{26-2u-3v}{3},1,\frac{9-u}{2},1,2,0\biggr],
\end{eqnarray*}
and 
\[
\left(P(u,v)^{\cdot 2}\right)=\frac{1}{36}(-41+58u-5u^2+24v-24uv-18v^2).
\]
\item
If $v\in\left[\frac{7-u}{6},\frac{4}{3}\right]$, then 
\begin{eqnarray*}
N(u,v)&=&\left[0,0,0,0,\frac{v}{2},0,\frac{-u+1+3v}{3},0,0,\frac{-7+u+6v}{6},0,0\right], \\
P(u,v)&\sim_\R&\biggl[\frac{u-8}{3},u-9,u-9,0,
\frac{20-2u-3v}{6},\frac{20-2u-3v}{3},\\
&&\quad \frac{26-2u-3v}{3},1,\frac{9-u}{2},
\frac{13-u-6v}{6},2,0\biggr],
\end{eqnarray*}
and 
\[
\left(P(u,v)^{\cdot 2}\right)=\frac{1}{18}(4+22u-2u^2-30v-6uv+9v^2).
\]
\item
If $v\in\left[\frac{4}{3},\frac{u+2}{3}\right]$, then 
\begin{eqnarray*}
N(u,v)&=&\left[\frac{3v-4}{3},0,0,3v-4,2v-2,0,
\frac{-u+1+3v}{3},0,0,\frac{-7+u+6v}{6},0,0\right], \\
P(u,v)&\sim_\R&\biggl[\frac{u-3v-4}{3},u-9,u-9,4-3v,
\frac{16-u-6v}{3},\frac{20-2u-3v}{3},\\
&&\quad \frac{26-2u-3v}{3},1,\frac{9-u}{2},
\frac{13-u-6v}{6},2,0\biggr],
\end{eqnarray*}
and 
\[
\left(P(u,v)^{\cdot 2}\right)=\frac{1}{9}(2+u-3v)(13-u-6v).
\]
\end{itemize}
\item
Assume that $u\in\left[3,5\right]$. 
\begin{itemize}
\item
If $v\in\left[0,\frac{7-u}{6}\right]$, then 
\begin{eqnarray*}
N(u,v)&=&\left[0,0,0,0,\frac{v}{2},0,0,0,0,0,0,0\right], \\
P(u,v)&\sim_\R&\biggl[\frac{u-8}{3},u-9,u-9,0,
\frac{20-2u-3v}{6},\frac{20-2u-3v}{3},9-u,1,\frac{9-u}{2},1,2,0\biggr],
\end{eqnarray*}
and 
\[
\left(P(u,v)^{\cdot 2}\right)=\frac{1}{12}(-15+22u-3u^2-18v^2).
\]
\item
If $v\in\left[\frac{7-u}{6},\frac{u-1}{3}\right]$, then 
\begin{eqnarray*}
N(u,v)&=&\left[0,0,0,0,\frac{v}{2},0,0,0,0,\frac{-7+u+6v}{6},0,0\right], \\
P(u,v)&\sim_\R&\biggl[\frac{u-8}{3},u-9,u-9,0,
\frac{20-2u-3v}{6},\frac{20-2u-3v}{3},\\
&&\quad 9-u,1,\frac{9-u}{2},\frac{13-u-6v}{6},2,0\biggr],
\end{eqnarray*}
and 
\[
\left(P(u,v)^{\cdot 2}\right)=\frac{1}{18}(2+26u-4u^2-42v+6uv-9v^2).
\]
\item
If $v\in\left[\frac{u-1}{3},\frac{4}{3}\right]$, then 
\begin{eqnarray*}
N(u,v)&=&\left[0,0,0,0,\frac{v}{2},0,\frac{-u+1+3v}{3},0,0,\frac{-7+u+6v}{6},0,0\right], \\
P(u,v)&\sim_\R&\biggl[\frac{u-8}{3},u-9,u-9,0,
\frac{20-2u-3v}{6},\frac{20-2u-3v}{3},\\
&&\quad \frac{26-2u-3v}{3},1,\frac{9-u}{2},
\frac{13-u-6v}{6},2,0\biggr],
\end{eqnarray*}
and 
\[
\left(P(u,v)^{\cdot 2}\right)=\frac{1}{18}(4+22u-2u^2-30v-6uv+9v^2).
\]
\item
If $v\in\left[\frac{4}{3},\frac{13-u}{6}\right]$, then 
\begin{eqnarray*}
N(u,v)&=&\left[\frac{3v-4}{3},0,0,3v-4,2v-2,0,
\frac{-u+1+3v}{3},0,0,\frac{-7+u+6v}{6},0,0\right], \\
P(u,v)&\sim_\R&\biggl[\frac{u-3v-4}{3},u-9,u-9,4-3v,
\frac{16-u-6v}{3},\frac{20-2u-3v}{3},\\
&&\quad \frac{26-2u-3v}{3},1,\frac{9-u}{2},
\frac{13-u-6v}{6},2,0\biggr],
\end{eqnarray*}
and 
\[
\left(P(u,v)^{\cdot 2}\right)=\frac{1}{9}(2+u-3v)(13-u-6v).
\]
\end{itemize}
\item
Assume that $u\in[5,7]$. 
\begin{itemize}
\item
If $v\in\left[0,\frac{7-u}{6}\right]$, then 
\begin{eqnarray*}
N(u,v)&=&\left[0,0,0,0,\frac{v}{2},0,0,0,0,0,0,0\right], \\
P(u,v)&\sim_\R&\biggl[\frac{u-9}{4},u-9,u-9,0,
\frac{20-2u-3v}{6},\frac{20-2u-3v}{3},\\
&&\quad 9-u,\frac{9-u}{4},\frac{9-u}{2},1,\frac{9-u}{2},0\biggr],
\end{eqnarray*}
and 
\[
\left(P(u,v)^{\cdot 2}\right)=\frac{1}{24}(145-26u+u^2-36v^2).
\]
\item
If $v\in\left[\frac{7-u}{6},\frac{13-u}{6}\right]$, then 
\begin{eqnarray*}
N(u,v)&=&\left[0,0,0,0,\frac{v}{2},0,0,0,0,\frac{-7+u+6v}{6},0,0\right], \\
P(u,v)&\sim_\R&\biggl[\frac{u-9}{4},u-9,u-9,0,
\frac{20-2u-3v}{6},\frac{20-2u-3v}{3},\\
&&\quad 9-u,\frac{9-u}{4},\frac{9-u}{2},\frac{13-u-6v}{6},\frac{9-u}{2},0\biggr],
\end{eqnarray*}
and 
\[
\left(P(u,v)^{\cdot 2}\right)=\frac{1}{72}(41-5u+6v)(13-u-6v).
\]
\end{itemize}
\item
Assume that $u\in[7,9]$. 
\begin{itemize}
\item
If $v\in\left[0,\frac{9-u}{2}\right]$, then 
\begin{eqnarray*}
N(u,v)&=&\left[0,0,0,0,\frac{v}{2},0,
0,0,0,v,0,0\right], \\
P(u,v)&\sim_\R&\biggl[\frac{u-9}{4},u-9,u-9,0,
\frac{9-u-v}{2},9-u-v,\\
&&\quad 9-u,\frac{9-u}{4},
\frac{9-u}{2},\frac{9-u-2v}{2},\frac{9-u}{2},0\biggr],
\end{eqnarray*}
and 
\[
\left(P(u,v)^{\cdot 2}\right)=\frac{1}{8}(9-u+2v)(9-u-2v).
\]
\end{itemize}
\end{itemize}
Therefore, we get 
\begin{eqnarray*}
&&S\left(V^{\tilde{R}}_{\bullet,\bullet}; \tilde{h}_2\right)\\
&=&\frac{3}{28}\Biggl(\int_0^1\int_0^u\frac{1}{2}(u-v)(u+v)dvdu\\
&+&\int_1^2\biggl(\frac{2u-2}{3}\cdot\frac{1}{12}(-7+14u-u^2)
+\int_0^{\frac{u-1}{3}}\frac{1}{12}(-7+14u-u^2-18v^2)dv\\
&&+\int_{\frac{u-1}{3}}^{\frac{7-u}{6}}\frac{1}{36}(-17+34u+u^2+24v-24uv-18v^2)dv\\
&&+\int_{\frac{7-u}{6}}^{\frac{2+u}{3}}\frac{1}{18}(2+u-3v)(8+u-3v)\biggr)du\\
&+&\int_2^3\biggl(\frac{2u-2}{3}\cdot\frac{1}{12}(-15+22u-3u^2)
+\int_0^{\frac{u-1}{3}}\frac{1}{12}(-15+22u-3u^2-18v^2)dv\\
&&+\int_{\frac{u-1}{3}}^{\frac{7-u}{6}}\frac{1}{36}(-41+58u-5u^2+24v-24uv-18v^2)dv\\
&&+\int_{\frac{7-u}{6}}^{\frac{4}{3}}\frac{1}{18}(4+22u-2u^2-30v-6uv+9v^2)dv\\
&&+\int_{\frac{4}{3}}^{\frac{u+2}{3}}\frac{1}{9}(2+u-3v)(13-u-6v)dv\biggr)du\\
&+&\int_3^5\biggl(\frac{2u-2}{3}\cdot\frac{1}{12}(-15+22u-3u^2)
+\int_0^{\frac{7-u}{6}}\frac{1}{12}(-15+22u-3u^2-18v^2)dv\\
&&+\int_{\frac{7-u}{6}}^{\frac{u-1}{3}}\frac{1}{18}(2+26u-4u^2-42v+6uv-9v^2)dv\\
&&+\int_{\frac{u-1}{3}}^{\frac{4}{3}}\frac{1}{18}(4+22u-2u^2-30v-6uv+9v^2)dv\\
&&+\int_{\frac{4}{3}}^{\frac{13-u}{6}}\frac{1}{9}(2+u-3v)(13-u-6v)dv\biggr)du\\
&+&\int_5^7\biggl(\frac{2u-2}{3}\cdot\frac{1}{24}(145-26u+u^2)
+\int_0^{\frac{7-u}{6}}\frac{1}{24}(145-26u+u^2-36v^2)dv\\
&&+\int_{\frac{7-u}{6}}^{\frac{13-u}{6}}\frac{1}{72}(41-5u+6v)(13-u-6v)dv\biggr)du\\
&+&\int_7^9\biggl((u-3)\frac{1}{8}(9-u)^2+\int_0^{\frac{9-u}{2}}
\frac{1}{8}(9-u+2v)(9-u-2v)dv\biggr)du\Biggr)=\frac{309}{112}. 
\end{eqnarray*}
Moreover, we have
\begin{eqnarray*}
&&F_{\hat{p}_{h,5}}\left(W^{\hat{R}, \hat{h}_2}_{\bullet,\bullet,\bullet}\right)\\
&=&\frac{6}{28}\Biggl(\int_2^3\int_{\frac{4}{3}}^{\frac{u+2}{3}}
\frac{17+u-12v}{6}\cdot\frac{3v-4}{2}dvdu
+\int_3^5\int_{\frac{4}{3}}^{\frac{13-u}{6}}\frac{17+u-12v}{6}
\cdot\frac{3v-4}{2}dvdu\Biggr)\\
&=&\frac{3}{448},
\end{eqnarray*}
\begin{eqnarray*}
&&F_{\hat{p}_{h,7}}\left(W^{\hat{R}, \hat{h}_2}_{\bullet,\bullet,\bullet}\right)\\
&=&\frac{6}{28}\Biggr(\int_1^3\int_0^{\frac{u-1}{3}}\frac{3v}{2}
\cdot\frac{u-1-3v}{3}dvdu\\
&+&\int_3^5\biggl(\int_0^{\frac{7-u}{6}}\frac{3v}{2}\cdot\frac{u-1-3v}{3}dv
+\int_{\frac{7-u}{6}}^{\frac{u-1}{3}}\frac{7-u+3v}{6}\cdot\frac{u-1-3v}{3}dv\biggr)du\\
&+&\int_5^7\biggl(\int_0^{\frac{7-u}{6}}\frac{3v}{2}\cdot\frac{u-1-3v}{3}dv
+\int_{\frac{7-u}{6}}^{\frac{13-u}{6}}\frac{7-u+3v}{6}\cdot\frac{u-1-3v}{3}dv\biggr)du\\
&+&\int_7^9\int_0^{\frac{9-u}{2}}\frac{v}{2}(2-v)dvdu\Biggr)=\frac{5}{14}, 
\end{eqnarray*}
\begin{eqnarray*}
&&F_{\hat{p}_{h,10}}\left(W^{\hat{R}, \hat{h}_2}_{\bullet,\bullet,\bullet}\right)\\
&=&\frac{6}{28}\Biggr(\int_1^2\int_{\frac{7-u}{6}}^{\frac{2+u}{3}}
\frac{5+u-3v}{6}\cdot\frac{-7+u+6v}{6}dvdu\\
&+&\int_2^3\biggl(\int_{\frac{7-u}{6}}^{\frac{4}{3}}
\frac{5+u-3v}{6}\cdot\frac{-7+u+6v}{6}dv+\int_{\frac{4}{3}}^{\frac{u+2}{3}}
\frac{17+u-12v}{6}\cdot\frac{-7+u+6v}{6}dv\biggr)du\\
&+&\int_3^5\biggl(\int_{\frac{7-u}{6}}^{\frac{u-1}{3}}
\frac{7-u+3v}{6}\cdot\frac{-7+u+6v}{6}dv+\int_{\frac{u-1}{3}}^{\frac{4}{3}}
\frac{5+u-3v}{6}\cdot\frac{-7+u+6v}{6}dv\\
&&+\int_{\frac{4}{3}}^{\frac{13-u}{6}}\frac{17+u-12v}{6}\cdot\frac{-7+u+6v}{6}dv
\biggr)du\\
&+&\int_5^7\int_{\frac{7-u}{6}}^{\frac{13-u}{6}}
\frac{7-u+3v}{6}\cdot\frac{-7+u+6v}{6}dvdu
+\int_7^9\int_0^{\frac{9-u}{2}}\frac{v}{2}\cdot\frac{u+2v-7}{2}dvdu\Biggr)=\frac{23}{64}. 
\end{eqnarray*}
Thus, for any closed point $\hat{p}\in\hat{h}_2$, we have 
\begin{eqnarray*}
&&S\left(W^{\hat{R},\hat{h}_2}_{\bullet,\bullet,\bullet}; \hat{p}\right)\\
&=&F_{\hat{p}}\left(W^{\hat{R}, \hat{h}_2}_{\bullet,\bullet,\bullet}\right)
+\frac{3}{28}\Biggl(\int_0^1\int_0^u\left(\frac{v}{2}\right)^2dvdu\\
&+&\int_1^2\biggl(\int_0^{\frac{u-1}{3}}\left(\frac{3v}{2}\right)^2dv
+\int_{\frac{u-1}{3}}^{\frac{7-u}{6}}\left(\frac{2u-2+3v}{6}\right)^2dv
+\int_{\frac{7-u}{6}}^{\frac{2+u}{3}}\left(\frac{5+u-3v}{6}\right)^2dv\biggr)du\\
&+&\int_2^3\biggl(\int_0^{\frac{u-1}{3}}\left(\frac{3v}{2}\right)^2dv
+\int_{\frac{u-1}{3}}^{\frac{7-u}{6}}\left(\frac{2u-2+3v}{6}\right)^2dv\\
&&+\int_{\frac{7-u}{6}}^{\frac{4}{3}}\left(\frac{5+u-3v}{6}\right)^2dv
+\int_{\frac{4}{3}}^{\frac{2+u}{3}}\left(\frac{17+u-12v}{6}\right)^2dv\biggr)du\\
&+&\int_3^5\biggl(\int_0^{\frac{7-u}{6}}\left(\frac{3v}{2}\right)^2dv
+\int_{\frac{7-u}{6}}^{\frac{u-1}{3}}\left(\frac{7-u+3v}{6}\right)^2dv\\
&&+\int_{\frac{u-1}{3}}^{\frac{4}{3}}\left(\frac{5+u-3v}{6}\right)^2dv
+\int_{\frac{4}{3}}^{\frac{13-u}{6}}\left(\frac{17+u-12v}{6}\right)^2dv\biggr)du\\
&+&\int_5^7\biggl(\int_0^{\frac{7-u}{6}}\left(\frac{3v}{2}\right)^2dv
+\int_{\frac{7-u}{6}}^{\frac{13-u}{6}}\left(\frac{7-u+3v}{6}\right)^2dv\biggr)du
+\int_7^9\int_0^{\frac{9-u}{2}}\left(\frac{v}{2}\right)^2dvdu\Biggr)\\
&=&F_{\hat{p}}\left(W^{\hat{R}, \hat{h}_2}_{\bullet,\bullet,\bullet}\right)
+\frac{21}{64}
\begin{cases}
=\frac{75}{224} & \text{if }\hat{p}=\hat{p}_{h,5}, \\
\leq\frac{11}{16} & \text{otherwise}.
\end{cases}
\end{eqnarray*}
In particular, we get the inequality
\begin{eqnarray*}
\delta_{p_h^1}\left(R^1; V^{\tilde{R}}_{\bullet,\bullet}\right)\geq
\min\left\{\frac{A_{R^1}\left(\tilde{h}_2\right)}
{S\left(V^{\tilde{R}}_{\bullet,\bullet};\tilde{h}_2\right)},\,\,\,
\inf_{\hat{p}\in\hat{h}_2}\frac{A_{\hat{h}_2,\frac{1}{2}\hat{p}_{h,5}}\left(\hat{p}\right)}
{S\left(W^{\hat{R},\hat{h}_2}_{\bullet,\bullet,\bullet}; \hat{p}\right)}\right\}
=\min\left\{\frac{112}{109},\,\,\,\frac{112}{75},\,\,\,\frac{16}{11}\right\}
=\frac{112}{109}
\end{eqnarray*}
by Corollary \ref{reduction_corollary}.

\noindent\underline{\textbf{Step 16}}\\
Let us set $f_2^1:=\gamma_*\tilde{f}_2$, 
$\tilde{p}_9:=\tilde{s}^+|_{\tilde{f}_2}$, 
$\tilde{p}_{12}:=\tilde{r}_3|_{\tilde{f}_2}$, and 
$p_9^1:=\gamma\left(\tilde{p}_9\right)$, 
$p_{12}^1:=\gamma\left(\tilde{p}_{12}\right)$. 
Note that the pair $(R^1, f_2^1)$ is plt and 
$\gamma^*f_2^1=(1/2)\tilde{s}_R+(1/2)\tilde{t}^++\tilde{f}_2+\tilde{s}^+$. 
Let us set 
\begin{eqnarray*}
P(u,v) &:=& P_\sigma\left(\tilde{R}, P(u)-v\tilde{f}_2\right), \\
N(u,v) &:=& N_\sigma\left(\tilde{R}, P(u)-v\tilde{f}_2\right),
\end{eqnarray*}
where $P(u)$, $N(u)$ are as in Step 11. 
\begin{itemize}
\item
Assume that $u\in[0,1]$. 
\begin{itemize}
\item
If $v\in\left[0,u\right]$, then 
\begin{eqnarray*}
N(u,v)&=&\left[\frac{v}{2},\frac{v}{2},0,0,0,0,0,0,v,0,0,0\right], \\
P(u,v)&\sim_\R&\left[\frac{u-v-6}{2},\frac{3u-v-20}{2},u-9,0,3,6,
8,1-v,4-v,1,2,0\right],
\end{eqnarray*}
and 
\[
\left(P(u,v)^{\cdot 2}\right)=\frac{1}{2}(u-v)^2.
\]
\end{itemize}
\item
Assume that $u\in[1,2]$. 
\begin{itemize}
\item
If $v\in\left[0,\frac{u-1}{2}\right]$, then 
\begin{eqnarray*}
N(u,v)&=&\left[\frac{v}{2},\frac{v}{2},0,0,0,0,0,0,0,0,0,0\right], \\
P(u,v)&\sim_\R&\biggl[\frac{u-v-6}{2},\frac{3u-v-20}{2},u-9,0,
\frac{10-u}{3},\frac{20-2u}{3},
 9-u,1-v,\frac{9-u}{2},1,2,0\biggr],
\end{eqnarray*}
and 
\[
\left(P(u,v)^{\cdot 2}\right)=\frac{1}{12}(-7+14u-u^2-12v-6v^2).
\]
\item
If $v\in\left[\frac{u-1}{2},1\right]$, then 
\begin{eqnarray*}
N(u,v)&=&\left[\frac{v}{2},\frac{v}{2},0,0,0,0,
0,0,\frac{1-u+2v}{2},0,0,0\right], \\
P(u,v)&\sim_\R&\biggl[\frac{u-v-6}{2},\frac{3u-v-20}{2},u-9,0,
\frac{10-u}{3},\frac{20-2u}{3},
 9-u,1-v,4-v,1,2,0\biggr],
\end{eqnarray*}
and 
\[
\left(P(u,v)^{\cdot 2}\right)=\frac{1}{6}(-2+4u+u^2-6uv+3v^2).
\]
\item
If $v\in\left[1,\frac{2+u}{3}\right]$, then 
\begin{eqnarray*}
N(u,v)&=&\left[\frac{v}{2},\frac{v}{2},0,0,0,0,
0,0,\frac{1-u+2v}{2},0,0,v-1\right], \\
P(u,v)&\sim_\R&\biggl[\frac{u-v-6}{2},\frac{3u-v-20}{2},u-9,0,
\frac{10-u}{3},\frac{20-2u}{3},\\
&&\quad 9-u,1-v,4-v,1,2,1-v\biggr],
\end{eqnarray*}
and 
\[
\left(P(u,v)^{\cdot 2}\right)=\frac{1}{6}(2+u-3v)^2.
\]
\end{itemize}
\item
Assume that $u\in\left[2,3\right]$. 
\begin{itemize}
\item
If $v\in\left[0,u-2\right]$, then 
\begin{eqnarray*}
N(u,v)&=&\left[\frac{v}{3},0,0,0,0,0,0,0,0,0,0,0\right], \\
P(u,v)&\sim_\R&\left[\frac{u-v-8}{3},u-9,u-9,0,
\frac{10-u}{3},\frac{20-2u}{3},9-u,1-v,\frac{9-u}{2},1,2,0\right],
\end{eqnarray*}
and 
\[
\left(P(u,v)^{\cdot 2}\right)=\frac{1}{12}(-15+22u-3u^2-20v+4uv-8v^2).
\]
\item
If $v\in\left[u-2,\frac{u-1}{2}\right]$, then 
\begin{eqnarray*}
N(u,v)&=&\left[\frac{-u+3v+2}{6},\frac{-u+v+2}{2},0,0,0,0,
0,0,0,0,0,0\right], \\
P(u,v)&\sim_\R&\biggl[\frac{u-v-6}{2},\frac{3u-v-20}{2},u-9,0,
\frac{10-u}{3},\frac{20-2u}{3},
 9-u,1-v,\frac{9-u}{2},1,2,0\biggr],
\end{eqnarray*}
and 
\[
\left(P(u,v)^{\cdot 2}\right)=\frac{1}{12}(-7+14u-u^2-12v-6v^2).
\]
\item
If $v\in\left[\frac{u-1}{2},1\right]$, then 
\begin{eqnarray*}
N(u,v)&=&\left[\frac{-u+3v+2}{6},\frac{-u+v+2}{2},0,0,0,0,
0,0,\frac{1-u+2v}{2},0,0,0\right], \\
P(u,v)&\sim_\R&\biggl[\frac{u-v-6}{2},\frac{3u-v-20}{2},u-9,0,
\frac{10-u}{3},\frac{20-2u}{3},
 9-u,1-v,4-v,1,2,0\biggr],
\end{eqnarray*}
and 
\[
\left(P(u,v)^{\cdot 2}\right)=\frac{1}{6}(-2+4u+u^2-6uv+3v^2).
\]
\item
If $v\in\left[1,\frac{2+u}{3}\right]$, then 
\begin{eqnarray*}
N(u,v)&=&\left[\frac{-u+3v+2}{6},\frac{-u+v+2}{2},0,0,0,0,
0,0,\frac{1-u+2v}{2},0,0,v-1\right], \\
P(u,v)&\sim_\R&\biggl[\frac{u-v-6}{2},\frac{3u-v-20}{2},u-9,0,
\frac{10-u}{3},\frac{20-2u}{3},\\
&&\quad 9-u,1-v,4-v,1,2,1-v\biggr],
\end{eqnarray*}
and 
\[
\left(P(u,v)^{\cdot 2}\right)=\frac{1}{6}(2+u-3v)^2.
\]
\end{itemize}
\item
Assume that $u\in\left[3,4\right]$. 
\begin{itemize}
\item
If $v\in\left[0,1\right]$, then 
\begin{eqnarray*}
N(u,v)&=&\left[\frac{v}{3},0,0,0,0,0,0,0,0,0,0,0\right], \\
P(u,v)&\sim_\R&\left[\frac{u-v-8}{3},u-9,u-9,0,
\frac{10-u}{3},\frac{20-2u}{3},9-u,1-v,\frac{9-u}{2},1,2,0\right],
\end{eqnarray*}
and 
\[
\left(P(u,v)^{\cdot 2}\right)=\frac{1}{12}(-15+22u-3u^2-20v+4uv-8v^2).
\]
\item
If $v\in\left[1,\frac{u-1}{2}\right]$, then 
\begin{eqnarray*}
N(u,v)&=&\left[\frac{v}{3},0,0,0,0,0,
0,0,0,0,0,v-1\right], \\
P(u,v)&\sim_\R&\biggl[\frac{u-v-8}{3},u-9,u-9,0,
\frac{10-u}{3},\frac{20-2u}{3},
 9-u,1-v,\frac{9-u}{2},1,2,1-v\biggr],
\end{eqnarray*}
and 
\[
\left(P(u,v)^{\cdot 2}\right)=\frac{1}{12}(-3+22u-3u^2-44v+4uv+4v^2).
\]
\item
If $v\in\left[\frac{u-1}{2},u-2\right]$, then 
\begin{eqnarray*}
N(u,v)&=&\left[\frac{v}{3},0,0,0,0,0,
0,0,\frac{1-u+2v}{2},0,0,v-1\right], \\
P(u,v)&\sim_\R&\biggl[\frac{u-v-8}{3},u-9,u-9,0,
\frac{10-u}{3},\frac{20-2u}{3},
 9-u,1-v,4-v,1,2,1-v\biggr],
\end{eqnarray*}
and 
\[
\left(P(u,v)^{\cdot 2}\right)=\frac{2}{3}(u-2v)(2-v).
\]
\item
If $v\in\left[u-2,\frac{2+u}{3}\right]$, then 
\begin{eqnarray*}
N(u,v)&=&\left[\frac{-u+3v+2}{6},\frac{-u+v+2}{2},0,0,0,0,
0,0,\frac{1-u+2v}{2},0,0,v-1\right], \\
P(u,v)&\sim_\R&\biggl[\frac{u-v-6}{2},\frac{3u-v-20}{2},u-9,0,
\frac{10-u}{3},\frac{20-2u}{3},\\
&&\quad 9-u,1-v,4-v,1,2,1-v\biggr],
\end{eqnarray*}
and 
\[
\left(P(u,v)^{\cdot 2}\right)=\frac{1}{6}(2+u-3v)^2.
\]
\end{itemize}
\item
Assume that $u\in\left[4,5\right]$. 
\begin{itemize}
\item
If $v\in\left[0,1\right]$, then 
\begin{eqnarray*}
N(u,v)&=&\left[\frac{v}{3},0,0,0,0,0,0,0,0,0,0,0\right], \\
P(u,v)&\sim_\R&\left[\frac{u-v-8}{3},u-9,u-9,0,
\frac{10-u}{3},\frac{20-2u}{3},9-u,1-v,\frac{9-u}{2},1,2,0\right],
\end{eqnarray*}
and 
\[
\left(P(u,v)^{\cdot 2}\right)=\frac{1}{12}(-15+22u-3u^2-20v+4uv-8v^2).
\]
\item
If $v\in\left[1,\frac{u-1}{2}\right]$, then 
\begin{eqnarray*}
N(u,v)&=&\left[\frac{v}{3},0,0,0,0,0,
0,0,0,0,0,v-1\right], \\
P(u,v)&\sim_\R&\biggl[\frac{u-v-8}{3},u-9,u-9,0,
\frac{10-u}{3},\frac{20-2u}{3},
 9-u,1-v,\frac{9-u}{2},1,2,1-v\biggr],
\end{eqnarray*}
and 
\[
\left(P(u,v)^{\cdot 2}\right)=\frac{1}{12}(-3+22u-3u^2-44v+4uv+4v^2).
\]
\item
If $v\in\left[\frac{u-1}{2},2\right]$, then 
\begin{eqnarray*}
N(u,v)&=&\left[\frac{v}{3},0,0,0,0,0,
0,0,\frac{1-u+2v}{2},0,0,v-1\right], \\
P(u,v)&\sim_\R&\biggl[\frac{u-v-8}{3},u-9,u-9,0,
\frac{10-u}{3},\frac{20-2u}{3},
 9-u,1-v,4-v,1,2,1-v\biggr],
\end{eqnarray*}
and 
\[
\left(P(u,v)^{\cdot 2}\right)=\frac{2}{3}(u-2v)(2-v).
\]
\end{itemize}
\item
Assume that $u\in[5,7]$. 
\begin{itemize}
\item
If $v\in\left[0,\frac{9-u}{4}\right]$, then 
\begin{eqnarray*}
N(u,v)&=&\left[\frac{v}{3},0,0,0,0,0,0,0,0,0,0,0\right], \\
P(u,v)&\sim_\R&\biggl[\frac{3u-4v-27}{12},u-9,u-9,0,
\frac{10-u}{3},\frac{20-2u}{3},\\
&&\quad 9-u,\frac{9-u-4v}{4},\frac{9-u}{2},1,\frac{9-u}{2},0\biggr],
\end{eqnarray*}
and 
\[
\left(P(u,v)^{\cdot 2}\right)=\frac{1}{24}(145-26u+u^2-16v^2).
\]
\item
If $v\in\left[\frac{9-u}{4},\frac{13-u}{4}\right]$, then 
\begin{eqnarray*}
N(u,v)&=&\left[\frac{v}{3},0,0,0,0,0,
0,0,0,0,0,\frac{-9+u+4v}{4}\right], \\
P(u,v)&\sim_\R&\biggl[\frac{3u-4v-27}{12},u-9,u-9,0,
\frac{10-u}{3},\frac{20-2u}{3},\\
&&\quad 9-u,\frac{9-u-4v}{4},\frac{9-u}{2},1,\frac{9-u}{2},
\frac{9-u-4v}{4}\biggr],
\end{eqnarray*}
and 
\[
\left(P(u,v)^{\cdot 2}\right)=\frac{1}{48}(13-u-4v)(41-5u-4v).
\]
\end{itemize}
\item
Assume that $u\in[7,9]$. 
\begin{itemize}
\item
If $v\in\left[0,\frac{9-u}{4}\right]$, then 
\begin{eqnarray*}
N(u,v)&=&\left[\frac{v}{3},0,0,0,0,0,
0,0,0,0,0,0\right], \\
P(u,v)&\sim_\R&\biggl[\frac{3u-4v-27}{12},u-9,u-9,0,
\frac{9-u}{2},9-u,\\
&&\quad 9-u,\frac{9-u-4v}{4},
\frac{9-u}{2},\frac{9-u}{2},\frac{9-u}{2},0\biggr],
\end{eqnarray*}
and 
\[
\left(P(u,v)^{\cdot 2}\right)=\frac{1}{24}(243-54u+3u^2-16v^2).
\]
\item
If $v\in\left[\frac{9-u}{4},\frac{27-3u}{4}\right]$, then 
\begin{eqnarray*}
N(u,v)&=&\left[\frac{v}{3},0,0,0,0,0,
0,0,0,0,0,\frac{-9+u+4v}{4}\right], \\
P(u,v)&\sim_\R&\biggl[\frac{3u-4v-27}{12},u-9,u-9,0,
\frac{9-u}{2},9-u,\\
&&\quad 9-u,\frac{9-u-4v}{4},
\frac{9-u}{2},\frac{9-u}{2},\frac{9-u}{2},\frac{9-u-4v}{4}\biggr],
\end{eqnarray*}
and 
\[
\left(P(u,v)^{\cdot 2}\right)=\frac{1}{48}(27-3u-4v)^2.
\]
\end{itemize}
\end{itemize}
Therefore, we get 
\begin{eqnarray*}
&&S\left(V^{\tilde{R}}_{\bullet,\bullet}; \tilde{f}_2\right)\\
&=&\frac{3}{28}\Biggl(\int_0^1\int_0^u\frac{1}{2}(u-v)^2dvdu\\
&+&\int_1^2\biggl(\int_0^{\frac{u-1}{2}}\frac{1}{12}(-7+14u-u^2-12v-6v^2)dv\\
&&+\int_{\frac{u-1}{2}}^1\frac{1}{6}(-2+4u+u^2-6uv+3v^2)dv
+\int_1^{\frac{2+u}{3}}\frac{1}{6}(2+u-3v)^2dv\biggr)du\\
&+&\int_2^3\biggl(\int_0^{u-2}\frac{1}{12}(-15+22u-3u^2-20v+4uv-8v^2)dv\\
&&+\int_{u-2}^{\frac{u-1}{2}}\frac{1}{12}(-7+14u-u^2-12v-6v^2)dv\\
&+&\int_{\frac{u-1}{2}}^1\frac{1}{6}(-2+4u+u^2-6uv+3v^2)dv
+\int_1^{\frac{2+u}{3}}\frac{1}{6}(2+u-3v)^2dv\biggr)du\\
&+&\int_3^4\biggl(\int_0^1\frac{1}{12}(-15+22u-3u^2-20v+4uv-8v^2)dv\\
&&+\int_1^{\frac{u-1}{2}}\frac{1}{12}(-3+22u-3u^2-44v+4uv+4v^2)dv\\
&+&\int_{\frac{u-1}{2}}^{u-2}\frac{2}{3}(u-2v)(2-v)dv
+\int_{u-2}^{\frac{2+u}{3}}\frac{1}{6}(2+u-3v)^2dv\biggr)du\\
&+&\int_4^5\biggl(\int_0^1\frac{1}{12}(-15+22u-3u^2-20v+4uv-8v^2)dv\\
&&+\int_1^{\frac{u-1}{2}}\frac{1}{12}(-3+22u-3u^2-44v+4uv+4v^2)dv
+\int_{\frac{u-1}{2}}^2\frac{2}{3}(u-2v)(2-v)dv\biggr)du\\
&+&\int_5^7\biggl(\frac{u-5}{4}\cdot\frac{1}{24}(145-26u+u^2)
+\int_0^{\frac{9-u}{4}}\frac{1}{24}(145-26u+u^2-16v^2)dv\\
&&+\int_{\frac{9-u}{4}}^{\frac{13-u}{4}}\frac{1}{48}(13-u-4v)(41-5u-4v)dv\biggr)du\\
&+&\int_7^9\biggl(\frac{u-5}{4}\cdot\frac{1}{24}(243-54u+3u^2)
+\int_0^{\frac{9-u}{4}}\frac{1}{24}(243-54u+3u^2-16v^2)dv\\
&&+\int_{\frac{9-u}{4}}^{\frac{27-3u}{4}}\frac{1}{48}(27-3u-4v)^2dv\biggr)du\Biggr)
=\frac{51}{56}.
\end{eqnarray*}
Moreover, we have
\begin{eqnarray*}
&&F_{p_9^1}\left(W^{R^1, f_2^1}_{\bullet,\bullet,\bullet}\right)\\
&=&\frac{6}{28}\Biggl(\int_1^2\int_0^{\frac{u-1}{2}}
\frac{1+v}{2}\cdot\frac{u-1-2v}{2}dvdu\\
&+&\int_2^3\biggl(\int_0^{u-2}\frac{5-u+4v}{6}\cdot\frac{u-1-2v}{2}dv
+\int_{u-2}^{\frac{u-1}{2}}\frac{1+v}{2}\cdot\frac{u-1-2v}{2}dv\biggr)du\\
&+&\int_3^5\biggl(\int_0^1\frac{5-u+4v}{6}\cdot\frac{u-1-2v}{2}dv
+\int_1^{\frac{u-1}{2}}\frac{11-u-2v}{6}\cdot\frac{u-1-2v}{2}dv\biggr)du\\
&+&\int_5^7\biggl(\int_0^{\frac{9-u}{4}}\frac{2v}{3}\cdot\frac{u+3-4v}{4}dv
+\int_{\frac{9-u}{4}}^{\frac{13-u}{4}}\frac{27-3u-4v}{12}\cdot\frac{u+3-4v}{4}dv
\biggr)du\\
&+&\int_7^9\biggl(\int_0^{\frac{9-u}{4}}\frac{2v}{3}\cdot\frac{u+3-4v}{4}dv
+\int_{\frac{9-u}{4}}^{\frac{27-3u}{4}}\frac{27-3u-4v}{12}\cdot\frac{u+3-4v}{4}dv
\biggr)du\Biggr)
=\frac{839}{1344},
\end{eqnarray*}
\begin{eqnarray*}
&&F_{p_{12}^1}\left(W^{R^1, f_2^1}_{\bullet,\bullet,\bullet}\right)\\
&=&\frac{6}{28}\Biggl(\int_1^3\int_1^{\frac{2+u}{3}}\frac{2+u-3v}{2}(v-1)dvdu\\
&+&\int_3^4\biggl(\int_1^{\frac{u-1}{2}}\frac{11-u-2v}{6}(v-1)dv
+\int_{\frac{u-1}{2}}^{u-2}\frac{4+u-4v}{3}(v-1)dv\\
&&+\int_{u-2}^{\frac{2+u}{3}}\frac{2+u-3v}{2}(v-1)dv\biggr)du\\
&+&\int_4^5\biggl(\int_1^{\frac{u-1}{2}}\frac{11-u-2v}{6}(v-1)dv
+\int_{\frac{u-1}{2}}^2\frac{4+u-4v}{3}(v-1)dv\biggr)du\\
&+&\int_5^7\int_{\frac{9-u}{4}}^{\frac{13-u}{4}}
\frac{27-3u-4v}{12}\cdot\frac{-9+u+4v}{4}dvdu\\
&+&\int_7^9\int_{\frac{9-u}{4}}^{\frac{27-3u}{4}}
\frac{27-3u-4v}{12}\cdot\frac{-9+u+4v}{4}dvdu\Biggr)=\frac{17}{112}.
\end{eqnarray*}
Thus, for any closed point $p^1\in f_2^1\setminus\left\{p_v^1 \right\}$, 
we get
\begin{eqnarray*}
&&S\left(W^{{R^1},f_2^1}_{\bullet,\bullet,\bullet}; p^1\right)\\
&\leq&\frac{839}{1344}+\frac{3}{28}\Biggl(\int_0^1\int_0^u
\left(\frac{u-v}{2}\right)^2dvdu\\
&+&\int_1^2\biggl(\int_0^{\frac{u-1}{2}}\left(\frac{1+v}{2}\right)^2dv
+\int_{\frac{u-1}{2}}^1\left(\frac{u-v}{2}\right)^2dv
+\int_1^{\frac{2+u}{3}}\left(\frac{2+u-3v}{2}\right)^2dv\biggr)du\\
&+&\int_2^3\biggl(\int_0^{u-2}\left(\frac{5-u+4v}{6}\right)^2dv
+\int_{u-2}^{\frac{u-1}{2}}\left(\frac{1+v}{2}\right)^2dv\\
&&+\int_{\frac{u-1}{2}}^1\left(\frac{u-v}{2}\right)^2dv
+\int_1^{\frac{2+u}{3}}\left(\frac{2+u-3v}{2}\right)^2dv\biggr)du\\
&+&\int_3^4\biggl(\int_0^1\left(\frac{5-u+4v}{6}\right)^2dv
+\int_1^{\frac{u-1}{2}}\left(\frac{11-u-2v}{6}\right)^2dv\\
&&+\int_{\frac{u-1}{2}}^{u-2}\left(\frac{4+u-4v}{3}\right)^2dv
+\int_{u-2}^{\frac{2+u}{3}}\left(\frac{2+u-3v}{2}\right)^2dv\biggr)du\\
&+&\int_4^5\biggl(\int_0^1\left(\frac{5-u+4v}{6}\right)^2dv
+\int_1^{\frac{u-1}{2}}\left(\frac{11-u-2v}{6}\right)^2dv
+\int_{\frac{u-1}{2}}^2\left(\frac{4+u-4v}{3}\right)^2dv\biggr)du\\
&+&\int_5^7\biggl(\int_0^{\frac{9-u}{4}}\left(\frac{2v}{3}\right)^2dv
+\int_{\frac{9-u}{4}}^{\frac{13-u}{4}}\left(\frac{27-3u-4v}{12}\right)^2dv\biggr)du\\
&+&\int_7^9\biggl(\int_0^{\frac{9-u}{4}}\left(\frac{2v}{3}\right)^2dv
+\int_{\frac{9-u}{4}}^{\frac{27-3u}{4}}\left(\frac{27-3u-4v}{12}\right)^2dv\biggr)
du\Biggr)=\frac{839}{1344}+\frac{361}{1344}=\frac{25}{28}. 
\end{eqnarray*}
In particular, we get the inequality
\begin{eqnarray*}
\delta_{p^1}\left(R^1; V^{\tilde{R}}_{\bullet,\bullet}\right)\geq
\min\left\{\frac{A_{R^1}\left(\tilde{f}_2\right)}
{S\left(V^{\tilde{R}}_{\bullet,\bullet};\tilde{f}_2\right)},\,\,\,
\frac{A_{f_2^1,\frac{1}{2}p_v^1}\left(p^1\right)}
{S\left(W^{R^1,f_2^1}_{\bullet,\bullet,\bullet}; p^1\right)}\right\}
\geq\min\left\{\frac{56}{51},\,\,\,\frac{28}{25}\right\}
=\frac{56}{51}
\end{eqnarray*}
by Corollary \ref{reduction_corollary}.

\noindent\underline{\textbf{Step 17}}\\
Take any closed point $p^1\in R^1\setminus\left(f_S^1\cup f_R^1\cup f_2^1\right)$. 
The morphism $\gamma\colon\tilde{R}\to R^1$ is an isomorphism over $p^1$. 
Take the line $f^1\subset R^1$ 
(i.e., $f^1\sim_\Q\sO_{\pr(1,1,2)}(1)$ on $R^1\simeq\pr(1,1,2)$) passing through
$p_v^1$ and $p^1$. Then the strict transform $\tilde{f}:=\gamma^{-1}_*f^1$ 
is linearly equivalent to $\tilde{t}^++\tilde{f}_S$. Let us set 
\[
q_{r_1}:=\gamma\left(\tilde{r}_1|_{\tilde{f}}\right), \quad
q_{r_2}:=\gamma\left(\tilde{r}_2|_{\tilde{f}}\right), \quad
q_{r_3}:=\gamma\left(\tilde{r}_3|_{\tilde{f}}\right).
\]
Then the points $q_{r_1}$, $q_{r_2}$, $q_{r_3}\in f^1$ are mutually distinct. 
Let us set 
\begin{eqnarray*}
P(u,v) &:=& P_\sigma\left(\tilde{R}, P(u)-v\tilde{f}\right), \\
N(u,v) &:=& N_\sigma\left(\tilde{R}, P(u)-v\tilde{f}\right),
\end{eqnarray*}
where $P(u)$, $N(u)$ are as in Step 11. 
\begin{itemize}
\item
Assume that $u\in[0,1]$. 
\begin{itemize}
\item
If $v\in\left[0,u\right]$, then 
\begin{eqnarray*}
N(u,v)&=&\left[\frac{v}{2},\frac{v}{2},0,0,0,0,0,0,0,0,0,0\right], \\
P(u,v)&\sim_\R&\left[\frac{u-v-6}{2},\frac{3u-3v-20}{2},u-v-9,0,3,6,
8,1,4,1,2,0\right],
\end{eqnarray*}
and 
\[
\left(P(u,v)^{\cdot 2}\right)=\frac{1}{2}(u-v)^2.
\]
\end{itemize}
\item
Assume that $u\in[1,2]$. 
\begin{itemize}
\item
If $v\in\left[0,\frac{7-u}{6}\right]$, then 
\begin{eqnarray*}
N(u,v)&=&\left[\frac{v}{2},\frac{v}{2},0,0,0,0,0,0,0,0,0,0\right], \\
P(u,v)&\sim_\R&\biggl[\frac{u-v-6}{2},\frac{3u-3v-20}{2},u-v-9,0,
\frac{10-u}{3},\frac{20-2u}{3},
 9-u,1,\frac{9-u}{2},1,2,0\biggr],
\end{eqnarray*}
and 
\[
\left(P(u,v)^{\cdot 2}\right)=\frac{1}{12}(-7+14u-u^2-12uv+6v^2).
\]
\item
If $v\in\left[\frac{7-u}{6},1\right]$, then 
\begin{eqnarray*}
N(u,v)&=&\left[\frac{v}{2},\frac{v}{2},0,0,\frac{-7+u+6v}{6},\frac{-7+u+6v}{3},
0,0,0,\frac{-7+u+6v}{2},0,0\right], \\
P(u,v)&\sim_\R&\biggl[\frac{u-v-6}{2},\frac{3u-3v-20}{2},u-v-9,0,
\frac{9-u-2v}{2},9-u-2v,\\
&&\quad 9-u,1,\frac{9-u}{2},\frac{9-u-6v}{2},2,0\biggr],
\end{eqnarray*}
and 
\[
\left(P(u,v)^{\cdot 2}\right)=\frac{7}{2}(1-v)^2.
\]
\end{itemize}
\item
Assume that $u\in\left[2,\frac{19}{7}\right]$. 
\begin{itemize}
\item
If $v\in\left[0,u-2\right]$, then 
\begin{eqnarray*}
N(u,v)&=&\left[\frac{v}{3},0,0,0,0,0,0,0,0,0,0,0\right], \\
P(u,v)&\sim_\R&\left[\frac{u-v-8}{3},u-v-9,u-v-9,0,
\frac{10-u}{3},\frac{20-2u}{3},9-u,1,\frac{9-u}{2},1,2,0\right],
\end{eqnarray*}
and 
\[
\left(P(u,v)^{\cdot 2}\right)=\frac{1}{12}(-15+22u-3u^2-8v-8uv+4v^2).
\]
\item
If $v\in\left[u-2,\frac{7-u}{6}\right]$, then 
\begin{eqnarray*}
N(u,v)&=&\left[\frac{-u+3v+2}{6},\frac{-u+v+2}{2},0,0,0,0,
0,0,0,0,0,0\right], \\
P(u,v)&\sim_\R&\biggl[\frac{u-v-6}{2},\frac{3u-3v-20}{2},u-v-9,0,
\frac{10-u}{3},\frac{20-2u}{3},\\
&&\quad 9-u,1,\frac{9-u}{2},1,2,0\biggr],
\end{eqnarray*}
and 
\[
\left(P(u,v)^{\cdot 2}\right)=\frac{1}{12}(-7+14u-u^2-12uv+6v^2).
\]
\item
If $v\in\left[\frac{7-u}{6},1\right]$, then 
\begin{eqnarray*}
N(u,v)&=&\biggl[\frac{-u+3v+2}{6},\frac{-u+v+2}{2},0,0,\frac{-7+u+6v}{6},
\frac{-7+u+6v}{3},\\
&&\quad 0,0,0,\frac{-7+u+6v}{2},0,0\biggr], \\
P(u,v)&\sim_\R&\biggl[\frac{u-v-6}{2},\frac{3u-3v-20}{2},u-v-9,0,
\frac{9-u-2v}{2},9-u-2v,\\
&&\quad 9-u,1,\frac{9-u}{2},\frac{9-u-6v}{2},2,0\biggr],
\end{eqnarray*}
and 
\[
\left(P(u,v)^{\cdot 2}\right)=\frac{7}{2}(1-v)^2.
\]
\end{itemize}
\item
Assume that $u\in\left[\frac{19}{7},3\right]$. 
\begin{itemize}
\item
If $v\in\left[0,\frac{7-u}{6}\right]$, then 
\begin{eqnarray*}
N(u,v)&=&\left[\frac{v}{3},0,0,0,0,0,0,0,0,0,0,0\right], \\
P(u,v)&\sim_\R&\left[\frac{u-v-8}{3},u-v-9,u-v-9,0,
\frac{10-u}{3},\frac{20-2u}{3},9-u,1,\frac{9-u}{2},1,2,0\right],
\end{eqnarray*}
and 
\[
\left(P(u,v)^{\cdot 2}\right)=\frac{1}{12}(-15+22u-3u^2-8v-8uv+4v^2).
\]
\item
If $v\in\left[\frac{7-u}{6},u-2\right]$, then 
\begin{eqnarray*}
N(u,v)&=&\left[\frac{v}{3},0,0,0,\frac{-7+u+6v}{6},\frac{-7+u+6v}{3},
0,0,0,\frac{-7+u+6v}{2},0,0\right], \\
P(u,v)&\sim_\R&\biggl[\frac{u-v-8}{3},u-v-9,u-v-9,0,
\frac{9-u-2v}{2},9-u-2v,\\
&&\quad 9-u,1,\frac{9-u}{2},\frac{9-u-6v}{2},2,0\biggr],
\end{eqnarray*}
and 
\[
\left(P(u,v)^{\cdot 2}\right)=\frac{1}{6}(17+4u-u^2-46v+2uv+20v^2).
\]
\item
If $v\in\left[u-2,1\right]$, then 
\begin{eqnarray*}
N(u,v)&=&\biggl[\frac{-u+3v+2}{6},\frac{-u+v+2}{2},0,0,\frac{-7+u+6v}{6},
\frac{-7+u+6v}{3},\\
&&\quad 0,0,0,\frac{-7+u+6v}{2},0,0\biggr], \\
P(u,v)&\sim_\R&\biggl[\frac{u-v-6}{2},\frac{3u-3v-20}{2},u-v-9,0,
\frac{9-u-2v}{2},9-u-2v,\\
&&\quad 9-u,1,\frac{9-u}{2},\frac{9-u-6v}{2},2,0\biggr],
\end{eqnarray*}
and 
\[
\left(P(u,v)^{\cdot 2}\right)=\frac{7}{2}(1-v)^2.
\]
\end{itemize}
\item
Assume that $u\in\left[3,4\right]$. 
\begin{itemize}
\item
If $v\in\left[0,\frac{7-u}{6}\right]$, then 
\begin{eqnarray*}
N(u,v)&=&\left[\frac{v}{3},0,0,0,0,0,0,0,0,0,0,0\right], \\
P(u,v)&\sim_\R&\left[\frac{u-v-8}{3},u-v-9,u-v-9,0,
\frac{10-u}{3},\frac{20-2u}{3},9-u,1,\frac{9-u}{2},1,2,0\right],
\end{eqnarray*}
and 
\[
\left(P(u,v)^{\cdot 2}\right)=\frac{1}{12}(-15+22u-3u^2-8v-8uv+4v^2).
\]
\item
If $v\in\left[\frac{7-u}{6},\frac{5-u}{2}\right]$, then 
\begin{eqnarray*}
N(u,v)&=&\left[\frac{v}{3},0,0,0,\frac{-7+u+6v}{6},\frac{-7+u+6v}{3},
0,0,0,\frac{-7+u+6v}{2},0,0\right], \\
P(u,v)&\sim_\R&\biggl[\frac{u-v-8}{3},u-v-9,u-v-9,0,
\frac{9-u-2v}{2},9-u-2v,\\
&&\quad 9-u,1,\frac{9-u}{2},\frac{9-u-6v}{2},2,0\biggr],
\end{eqnarray*}
and 
\[
\left(P(u,v)^{\cdot 2}\right)=\frac{1}{6}(17+4u-u^2-46v+2uv+20v^2).
\]
\item
If $v\in\left[\frac{5-u}{2},\frac{9-u}{6}\right]$, then 
\begin{eqnarray*}
N(u,v)&=&\biggl[\frac{u+6v-5}{12},0,0,0,\frac{-7+u+6v}{6},\frac{-7+u+6v}{3},\\
&&\quad 0,\frac{-5+u+2v}{4},0,\frac{-7+u+6v}{2},\frac{-5+u+2v}{2},0\biggr], \\
P(u,v)&\sim_\R&\biggl[\frac{u-2v-9}{4},u-v-9,u-v-9,0,
\frac{9-u-2v}{2},9-u-2v,\\
&&\quad 9-u,\frac{9-u-2v}{4},\frac{9-u}{2},\frac{9-u-6v}{2},\frac{9-u-2v}{2},0\biggr],
\end{eqnarray*}
and 
\[
\left(P(u,v)^{\cdot 2}\right)=\frac{1}{8}(9-u-6v)^2.
\]
\end{itemize}
\item
Assume that $u\in\left[4,5\right]$. 
\begin{itemize}
\item
If $v\in\left[0,\frac{5-u}{2}\right]$, then 
\begin{eqnarray*}
N(u,v)&=&\left[\frac{v}{3},0,0,0,0,0,0,0,0,0,0,0\right], \\
P(u,v)&\sim_\R&\left[\frac{u-v-8}{3},u-v-9,u-v-9,0,
\frac{10-u}{3},\frac{20-2u}{3},9-u,1,\frac{9-u}{2},1,2,0\right],
\end{eqnarray*}
and 
\[
\left(P(u,v)^{\cdot 2}\right)=\frac{1}{12}(-15+22u-3u^2-8v-8uv+4v^2).
\]
\item
If $v\in\left[\frac{5-u}{2},\frac{7-u}{6}\right]$, then 
\begin{eqnarray*}
N(u,v)&=&\left[\frac{u+6v-5}{12},0,0,0,0,0,
0,\frac{-5+u+2v}{4},0,0,\frac{-5+u+2v}{2},0\right], \\
P(u,v)&\sim_\R&\biggl[\frac{u-2v-9}{4},u-v-9,u-v-9,0,
\frac{10-u}{3},\frac{20-2u}{3},\\
&&\quad 9-u,\frac{9-u-2v}{4},\frac{9-u}{2},1,\frac{9-u-2v}{2},0\biggr],
\end{eqnarray*}
and 
\[
\left(P(u,v)^{\cdot 2}\right)=\frac{1}{24}(145-26u+u^2-156v+12uv+36v^2).
\]
\item
If $v\in\left[\frac{7-u}{6},\frac{9-u}{6}\right]$, then 
\begin{eqnarray*}
N(u,v)&=&\biggl[\frac{u+6v-5}{12},0,0,0,\frac{-7+u+6v}{6},\frac{-7+u+6v}{3},\\
&&\quad 0,\frac{-5+u+2v}{4},0,\frac{-7+u+6v}{2},\frac{-5+u+2v}{2},0\biggr], \\
P(u,v)&\sim_\R&\biggl[\frac{u-2v-9}{4},u-v-9,u-v-9,0,
\frac{9-u-2v}{2},9-u-2v,\\
&&\quad 9-u,\frac{9-u-2v}{4},\frac{9-u}{2},\frac{9-u-6v}{2},\frac{9-u-2v}{2},0\biggr],
\end{eqnarray*}
and 
\[
\left(P(u,v)^{\cdot 2}\right)=\frac{1}{8}(9-u-6v)^2.
\]
\end{itemize}
\item
Assume that $u\in[5,7]$. 
\begin{itemize}
\item
If $v\in\left[0,\frac{7-u}{6}\right]$, then 
\begin{eqnarray*}
N(u,v)&=&\left[\frac{v}{2},0,0,0,0,0,0,\frac{v}{2},0,0,v,0\right], \\
P(u,v)&\sim_\R&\biggl[\frac{u-2v-9}{4},u-v-9,u-v-9,0,
\frac{10-u}{3},\frac{20-2u}{3},\\
&&\quad 9-u,\frac{9-u-2v}{4},\frac{9-u}{2},1,\frac{9-u-2v}{2},0\biggr],
\end{eqnarray*}
and 
\[
\left(P(u,v)^{\cdot 2}\right)=\frac{1}{24}(145-26u+u^2-156v+12uv+36v^2).
\]
\item
If $v\in\left[\frac{7-u}{6},\frac{9-u}{6}\right]$, then 
\begin{eqnarray*}
N(u,v)&=&\left[\frac{v}{2},0,0,0,\frac{-7+u+6v}{6},\frac{-7+u+6v}{3},
0,\frac{v}{2},0,\frac{-7+u+6v}{2},v,0\right], \\
P(u,v)&\sim_\R&\biggl[\frac{u-2v-9}{4},u-v-9,u-v-9,0,
\frac{9-u-2v}{2},9-u-2v,\\
&&\quad 9-u,\frac{9-u-2v}{4},\frac{9-u}{2},\frac{9-u-6v}{2},\frac{9-u-2v}{2},0\biggr],
\end{eqnarray*}
and 
\[
\left(P(u,v)^{\cdot 2}\right)=\frac{1}{8}(9-u-6v)^2.
\]
\end{itemize}
\item
Assume that $u\in[7,9]$. 
\begin{itemize}
\item
If $v\in\left[0,\frac{9-u}{6}\right]$, then 
\begin{eqnarray*}
N(u,v)&=&\left[\frac{v}{2},0,0,0,v,2v,
0,\frac{v}{2},0,3v,v,0\right], \\
P(u,v)&\sim_\R&\biggl[\frac{u-2v-9}{4},u-v-9,u-v-9,0,
\frac{9-u-2v}{2},9-u-2v,\\
&&\quad 9-u,\frac{9-u-2v}{4},\frac{9-u}{2},\frac{9-u-6v}{2},\frac{9-u-2v}{2},0\biggr],
\end{eqnarray*}
and 
\[
\left(P(u,v)^{\cdot 2}\right)=\frac{1}{8}(9-u-6v)^2.
\]
\end{itemize}
\end{itemize}
Therefore, we get 
\begin{eqnarray*}
&&S\left(V^{\tilde{R}}_{\bullet,\bullet}; \tilde{f}\right)\\
&=&\frac{3}{28}\Biggl(\int_0^1\int_0^u\frac{1}{2}(u-v)^2dvdu\\
&+&\int_1^2\biggl(\int_0^{\frac{7-u}{6}}\frac{1}{12}(-7+14u-u^2-12uv+6v^2)dv
+\int_{\frac{7-u}{6}}^2\frac{7}{2}(1-v)^2dv\biggr)du\\
&+&\int_2^{\frac{19}{7}}\biggl(\int_0^{u-2}
\frac{1}{12}(-15+22u-3u^2-8v-8uv+4v^2)dv\\
&&+\int_{u-2}^{\frac{7-u}{6}}\frac{1}{12}(-7+14u-u^2-12uv+6v^2)dv
+\int_{\frac{7-u}{6}}^2\frac{7}{2}(1-v)^2dv\biggr)du\\
&+&\int_{\frac{19}{7}}^3\biggl(\int_0^{\frac{7-u}{6}}
\frac{1}{12}(-15+22u-3u^2-8v-8uv+4v^2)dv\\
&&+\int_{\frac{7-u}{6}}^{u-2}\frac{1}{6}(17+4u-u^2-46v+2uv+20v^2)dv
+\int_{u-2}^1\frac{7}{2}(1-v)^2dv\biggr)du\\
&+&\int_3^4\biggl(\int_0^{\frac{7-u}{6}}
\frac{1}{12}(-15+22u-3u^2-8v-8uv+4v^2)dv\\
&&+\int_{\frac{7-u}{6}}^{\frac{5-u}{2}}\frac{1}{6}(17+4u-u^2-46v+2uv+20v^2)dv
+\int_{\frac{5-u}{2}}^{\frac{9-u}{6}}\frac{1}{8}(9-u-6v)^2dv\biggr)du\\
&+&\int_4^5\biggl(\int_0^{\frac{5-u}{2}}
\frac{1}{12}(-15+22u-3u^2-8v-8uv+4v^2)dv\\
&&+\int_{\frac{5-u}{2}}^{\frac{7-u}{6}}
\frac{1}{24}(145-26u+u^2-156v+12uv+36v^2)dv
+\int_{\frac{7-u}{6}}^{\frac{9-u}{6}}\frac{1}{8}(9-u-6v)^2dv\biggr)du\\
&+&\int_5^7\biggl(\int_0^{\frac{7-u}{6}}
\frac{1}{24}(145-26u+u^2-156v+12uv+36v^2)dv
+\int_{\frac{7-u}{6}}^{\frac{9-u}{6}}\frac{1}{8}(9-u-6v)^2dv\biggr)du\\
&+&\int_7^9\int_0^{\frac{9-u}{6}}\frac{1}{8}(9-u-6v)^2dvdu\Biggr)
=\frac{5}{16}. 
\end{eqnarray*}
Moreover, we have
\begin{eqnarray*}
&&F_{q_{r_1}}\left(W^{R^1, f^1}_{\bullet,\bullet,\bullet}\right)\\
&=&\frac{6}{28}\Biggl(\int_1^{\frac{19}{7}}\int_{\frac{7-u}{6}}^1
\frac{7}{2}(1-v)\frac{-7+u+6v}{2}dvdu\\
&+&\int_{\frac{19}{7}}^3\biggl(\int_{\frac{7-u}{6}}^{u-2}
\frac{23-u-20v}{6}\cdot\frac{-7+u+6v}{2}dv+\int_{u-2}^1
\frac{7}{2}(1-v)\frac{-7+u+6v}{2}dv\biggr)du\\
&+&\int_3^4\biggl(\int_{\frac{7-u}{6}}^{\frac{5-u}{2}}
\frac{23-u-20v}{6}\cdot\frac{-7+u+6v}{2}dv+\int_{\frac{5-u}{2}}^{\frac{9-u}{6}}
\frac{3}{4}(9-u-6v)\frac{-7+u+6v}{2}dv\biggr)du\\
&+&\int_4^7\int_{\frac{7-u}{6}}^{\frac{9-u}{6}}
\frac{3}{4}(9-u-6v)\frac{-7+u+6v}{2}dvdu
+\int_7^9\int_0^{\frac{9-u}{6}}\frac{3}{4}(9-u-6v)\frac{-7+u+6v}{2}dvdu\Biggr)\\
&=&\frac{149}{1568}, 
\end{eqnarray*}
\begin{eqnarray*}
&&F_{q_{r_2}}\left(W^{R^1, f^1}_{\bullet,\bullet,\bullet}\right)\\
&=&\frac{6}{28}\Biggl(\int_3^4\int_{\frac{5-u}{2}}^{\frac{9-u}{6}}
\frac{3}{4}(9-u-6v)\frac{-5+u+2v}{2}dvdu\\
&+&\int_4^5\biggl(\int_{\frac{5-u}{2}}^{\frac{7-u}{6}}
\frac{13-u-6v}{4}\cdot\frac{-5+u+2v}{2}dv+\int_{\frac{7-u}{6}}^{\frac{9-u}{6}}
\frac{3}{4}(9-u-6v)\frac{-5+u+2v}{2}dv\biggr)du\\
&+&\int_5^7\biggl(\int_0^{\frac{7-u}{6}}
\frac{13-u-6v}{4}\cdot\frac{-5+u+2v}{2}dv+\int_{\frac{7-u}{6}}^{\frac{9-u}{6}}
\frac{3}{4}(9-u-6v)\frac{-5+u+2v}{2}dv\biggr)du\\
&+&\int_7^9\int_0^{\frac{9-u}{6}}\frac{3}{4}(9-u-6v)\frac{-5+u+2v}{2}dvdu\Biggr)
=\frac{23}{112}, 
\end{eqnarray*}
and $F_{q_{r_3}}\left(W^{R^1, f^1}_{\bullet,\bullet,\bullet}\right)=0$. 
Therefore, we get 
\begin{eqnarray*}
&&S\left(W^{{R^1},f^1}_{\bullet,\bullet,\bullet}; p^1\right)\\
&\leq&\frac{23}{112}+\frac{3}{28}\Biggl(\int_0^1\int_0^u
\left(\frac{u-v}{2}\right)^2dvdu\\
&+&\int_1^2\biggl(\int_0^{\frac{7-u}{6}}\left(\frac{u-v}{2}\right)^2dv
+\int_{\frac{7-u}{6}}^1\left(\frac{7}{2}(1-v)\right)^2dv\biggr)du\\
&+&\int_2^{\frac{19}{7}}\biggl(\int_0^{u-2}\left(\frac{u+1-v}{3}\right)^2dv
+\int_{u-2}^{\frac{7-u}{6}}\left(\frac{u-v}{2}\right)^2dv
+\int_{\frac{7-u}{6}}^1\left(\frac{7}{2}(1-v)\right)^2dv\biggr)du\\
&+&\int_{\frac{19}{7}}^3\biggl(\int_0^{\frac{7-u}{6}}\left(\frac{u+1-v}{3}\right)^2dv
+\int_{\frac{7-u}{6}}^{u-2}\left(\frac{23-u-20v}{6}\right)^2dv\\
&&+\int_{u-2}^1\left(\frac{7}{2}(1-v)\right)^2dv\biggr)du\\
&+&\int_3^4\biggl(\int_0^{\frac{7-u}{6}}\left(\frac{u+1-v}{3}\right)^2dv
+\int_{\frac{7-u}{6}}^{\frac{5-u}{2}}\left(\frac{23-u-20v}{6}\right)^2dv\\
&&+\int_{\frac{5-u}{2}}^{\frac{9-u}{6}}\left(\frac{3}{4}(9-u-6v)\right)^2dv\biggr)du\\
&+&\int_4^5\biggl(\int_0^{\frac{5-u}{2}}\left(\frac{u+1-v}{3}\right)^2dv
+\int_{\frac{5-u}{2}}^{\frac{7-u}{6}}\left(\frac{13-u-6v}{4}\right)^2dv\\
&&+\int_{\frac{7-u}{6}}^{\frac{9-u}{6}}\left(\frac{3}{4}(9-u-6v)\right)^2dv\biggr)du\\
&+&\int_5^7\biggl(\int_0^{\frac{7-u}{6}}\left(\frac{13-u-6v}{4}\right)^2dv
+\int_{\frac{7-u}{6}}^{\frac{9-u}{6}}\left(\frac{3}{4}(9-u-6v)\right)^2dv\biggr)du\\
&+&\int_7^9\int_0^{\frac{9-u}{6}}\left(\frac{3}{4}(9-u-6v)\right)^2dvdu\Biggr)
=\frac{23}{112}+\frac{929}{1568}=\frac{1251}{1568}. 
\end{eqnarray*}
In particular, we get the inequality
\begin{eqnarray*}
\delta_{p^1}\left(R^1; V^{\tilde{R}}_{\bullet,\bullet}\right)\geq
\min\left\{\frac{A_{R^1}\left(\tilde{f}\right)}
{S\left(V^{\tilde{R}}_{\bullet,\bullet};\tilde{f}\right)},\,\,\,
\frac{A_{f^1,\frac{1}{2}p_v^1}\left(p^1\right)}
{S\left(W^{R^1,f^1}_{\bullet,\bullet,\bullet}; p^1\right)}\right\}
\geq\min\left\{\frac{16}{5},\,\,\,\frac{1568}{1251}\right\}
=\frac{1568}{1251}
\end{eqnarray*}
by Corollary \ref{reduction_corollary}.

\noindent\underline{\textbf{Step 18}}\\
By Steps 12--17, we get 
\[
\delta\left(R^1; V^{\tilde{R}}_{\bullet,\bullet}\right)\geq\min
\left\{\frac{224}{207},\,\,\frac{16}{11},\,\,\frac{112}{75},\,\,
\frac{112}{109},\,\,\frac{56}{51},\,\,\frac{1568}{1251}\right\}=\frac{112}{109}. 
\]
Moreover,  by Step 15, we get the equality 
\[
\delta\left(R^1; V^{\tilde{R}}_{\bullet,\bullet}\right)=\frac{112}{109} 
\]
by looking at the divisor $\tilde{h}_2\subset\tilde{R}$. Therefore, together with 
Step 9, we get the inequality 
\[
\delta_q(X)\geq \min\left\{\frac{A_X(R^1)}{S_X(R^1)},\,\,\,
\delta\left(R^1; V^{\tilde{R}}_{\bullet,\bullet}\right)\right\}
=\min\left\{\frac{64}{63},\,\,\frac{112}{109}\right\}=\frac{64}{63}
\]
by Corollary \ref{reduction_corollary}.

As a consequence, we have completed the proof of Theorem \ref{6463_thm}.
\end{proof}

\begin{remark}\label{6463_remark}
On might ask to evaluate $\delta\left(F^0; V^{F^0}_{\bullet,\bullet}\right)$ 
in order to evaluate $\delta_q(X)$. However, one can check that 
\[
\frac{A_X(F^0)}{S_X(F^0)}=\frac{28}{27} \quad \text{but}\quad
\delta\left(F^0; V^{F^0}_{\bullet,\bullet}\right)\leq\frac{112}{117}.
\]
That is why we consider the divisor $R^1$ over $X$. 
\end{remark}

\section{Main theorem}\label{main_section}

In this section, we prove the following: 

\begin{thm}\label{mainthm}
The Fano threefold given in Example \ref{CS_example} \eqref{CS_example2} is K-stable. 
\end{thm}

\begin{proof}
Note that $\Aut^0(X)=\{1\}$ by \cite{PCS}. 
Take any $G$-invariant dreamy prime divisor $E$ over $X$, 
where $G=\boldsymbol{\mu}_2 \times\boldsymbol{\mu}_3$ is as in Example 
\ref{CS_example} \eqref{CS_example2}. 
By Theorem \ref{equiv-K_thm}, 
it is enough to show the inequality 
\[
\frac{A_X(E)}{S_X(E)}>1.
\]
Let $Z:=c_X(E)\subset X$ be the center of $E$ on $X$. 
Note that the variety $Z$ is $G$-invariant. 
If $Z\cap\left(E_2\cup l\right)\neq \emptyset$, then, 
by Corollary \ref{112107_corollary}, Propositions \ref{112109_proposition}, 
\ref{112103_proposition} 
and Theorem \ref{6463_thm}, we have $A_X(E)>S_X(E)$. 
Thus we may assume that $Z\cap\left(E_2\cup l\right)=\emptyset$. 
Note that $E_2\cup l$ is the inverse image of $l^P\subset P$. 

If $Z$ is a divisor (i.e., if $E$ is a prime divisor on $X$), then we have
$A_X(E)>S_X(E)$ by \cite[\S 10]{div-stability}. If $Z$ contains a $G$-invariant point, 
then the point must be one of $p_x$, $p_y$ or $p_t$ in Remark \ref{G-inv_remark}. 
By Corollaries \ref{5651_corollary}, \ref{112107_corollary} and Proposition 
\ref{112103_proposition}, we have $A_X(E)>S_X(E)$. 

Thus we may further assume that $Z$ is a $G$-invariant curve such that 
$Z$ does not contain any $G$-invariant point on $X$. In particular, 
$Z$ must be a non-rational curve, since any action $G\curvearrowright\pr^1$ 
must have a fixed point. We remark that 
$Z^P:=\left(\sigma^V\circ\sigma_1\right)_*Z\subset P$ is also a $G$-invariant 
non-rational curve with $Z^P\cap l^P=\emptyset$. 

Let $\eta_Z\in Z$ be the generic point of $Z$. 
We assume that 
\[
\alpha_{G, \eta_Z}(X)<\frac{3}{4}.
\]
Then there exists a positive rational number $\alpha\in(0,3/4)\cap\Q$ 
and an effective $G$-invariant $\Q$-divisor $D\sim_\Q -K_X$ 
such that the pair $(X, \alpha D)$ is lc but not klt at $\eta_Z\in X$. 
Let $\Nklt(X, \alpha D)$ be the non-klt locus of the pair $(X, \alpha D)$ 
(see \cite[2.3.11]{fujino}). From the construction, $\Nklt(X, \alpha D)$ contains $Z$ 
and is $G$-invariant. 
If $\Nklt(X, \alpha D)$ is one-dimensional around a neighborhood of $\eta_Z\in X$, 
then $Z$ must be a rational curve by \cite[Corollary 4.2]{V22KE}. Thus 
there exists a $G$-irreducible effective $\Z$-divisor $D_0$ with 
$D_0\subset \Nklt(X, \alpha D)$ and $Z\subset D_0$. 
The divisor $D_0$ satisfies that $D_0\leq \alpha D$. In particular, 
\[
-K_X-\frac{4}{3}D_0
\]
is big. From the structure of $\Eff(X)$ in \S \ref{311_section} 
(see also \cite[\S 10]{div-stability}), we have one of 
$D_0\sim H_1$, $H_2$ or $H_3$. 
(Note that $D_0\neq E_2$ since $Z\cap E_2=\emptyset$.)

Assume that $D_0\sim H_1$. By Example \ref{CS-H1_example}, 
we have $D_0=Q_0$ or $Q_\infty$. 
Since the morphism 
\[
Q_0\setminus(l\cup E_2)\to Q_0^P\setminus l^P
\]
is an isomorphism and $Q_0^P\setminus l^P$ is affine, we have $D_0\neq Q_0$. 
Assume that $D_0=Q_\infty$. Note that $Q_\infty$ is smooth. 
Then we have either $A_X(E)>S_X(E)$ or 
$Z$ must be contained in the union of the $3$ negative curves in $Q_\infty$ 
by Proposition \ref{5651_proposition}. However, since $Z$ is a non-rational curve, 
we must have $A_X(E)>S_X(E)$. 

Assume that $D_0\sim H_2$. Then $D_0=H_x$, $H_y$ or $H_t$ by Remark 
\ref{G-inv_remark}. Since $l\subset H_y$, $H_t$ and $H_y^P\setminus l^P$ 
and $H_t^P\setminus l^P$ are affine, we must have $D_0=H_x$. 
Assume that $A_X(E)\leq S_X(E)$.
By Example \ref{CS-H1_example} and Proposition \ref{5651_proposition}, 
the curve $Z$ is contained in 
\[
H_x\cap\left(E_3\cup Q_0\cup Q_1\cup Q_\omega\cup Q_{\omega^2}\right).
\]
Thus, under the natural isomorphism 
$H_x^P\simeq\pr^2_{yzt}$, the curve $Z^P$ is contained in the locus
\begin{eqnarray*}
&&(t^3-y^3=0)\cup\left(yz+t^2-(y^2+zt)=0\right)\\
&\cup&
\left(yz+t^2-\omega(y^2+zt)=0\right)
\cup
\left(yz+t^2-\omega^2(y^2+zt)=0\right).
\end{eqnarray*}
The locus is a union of rational curves. This leads to a contradiction. 
Thus we have $A_X(E)>S_X(E)$. 

Assume that $D_0\sim H_3$. Then $D_0=H_z$ by Remark \ref{G-inv_remark}. 
Assume that $A_X(E)\leq S_X(E)$.
By Example \ref{CS-H1_example} and Proposition \ref{5651_proposition}, 
$Z$ is contained in 
\[
H_z\cap\left(E_3\cup Q_0\cup Q_1\cup Q_\omega\cup Q_{\omega^2}\right).
\]
Since $Z^P$ is a non-rational curve, $Z^P$ must be equal to 
\[
H_z^P\cap E_3=(t^3-x^2y-y^3=0)
\]
under the natural isomorphism $H_z^P\simeq\pr^2_{xyt}$. 
However, in this case, we have $p_x\in Z$. This leads to a contradiction since 
$Z\cap l=\emptyset$. Thus we have $A_X(E)>S_X(E)$. 

Therefore, we may assume that 
\[
\alpha_{G, \eta_Z}(X)\geq \frac{3}{4}.
\]
In this case, we have $A_X(E)>S_X(E)$ 
by Proposition \ref{alpha_proposition} \eqref{alpha_proposition2}. 

As a consequence, we have completed the proof of Theorem \ref{mainthm}. 
\end{proof}

\section{Appendix}\label{appendix_section}

In this section, we see several basic properties of local $\delta$-invariants. 

\subsection{Positivity of local $\delta$-invariants}\label{positivity_subsection}

We show that the local $\delta$-invariant for a graded linear series 
under some mild conditions is always positive. 

\begin{proposition}[{cf.\ \cite[Theorem A]{BJ}}]\label{positivity_proposition}
Let $X$ be a projective variety, let $\Delta$ be an effective $\Q$-Weil 
divisor on $X$, 
and let $V_{\vec{\bullet}}$ be the Veronese equivalence class of a 
graded linear series on $X$ associated to $L_1,\dots,L_r\in\CaCl(X)\otimes_\Z\Q$ 
which has bounded support and contains an ample series. 
\begin{enumerate}
\renewcommand{\theenumi}{\arabic{enumi}}
\renewcommand{\labelenumi}{(\theenumi)}
\item\label{positivity_proposition1}
Take a scheme-theoretic point $\eta\in X$ such that $(X, \Delta)$ is klt at $\eta$. 
Then 
we have $\alpha_\eta\left(X,\Delta;V_{\vec{\bullet}}\right)>0$ and 
$\delta_\eta\left(X,\Delta;V_{\vec{\bullet}}\right)>0$. 
\item\label{positivity_proposition2}
Assume that $(X,\Delta)$ is a klt pair. Then we have 
$\alpha\left(X,\Delta;V_{\vec{\bullet}}\right)>0$ and 
$\delta\left(X,\Delta;V_{\vec{\bullet}}\right)>0$. 
\end{enumerate}
\end{proposition}

\begin{proof}
By Definition \ref{delta-AZ_definition}, it is enough to show the positivity of 
$\alpha$-invariants. We may assume that 
$V_{\vec{\bullet}}$ is a $\Z_{\geq 0}^r$-graded linear series on $X$ associated to 
Cartier divisors $L_1,\dots,L_r$. Since $V_{\vec{\bullet}}$ has bounded support, 
there exists $M>0$ such that $V_{\vec{a}}=0$ for any 
$\vec{a}=(a_1,\dots,a_r)\in\Z_{\geq 0}^r$ with $a_i>M a_1$ for some $2\leq i\leq r$. 
Take a very ample Cartier divisor $H$ on $X$ such that 
\[
\left|H-\sum_{i=1}^r k_i L_i\right|\neq\emptyset\quad
\text{for any }k_i\in\{0,1,\dots,M\}.
\]
For any $\vec{a}=(a_1,\dots,a_r)\in \Z_{\geq 0}^r$ with $V_{\vec{a}}\neq 0$, 
since 
\begin{eqnarray*}
a_1 r H-\vec{a}\cdot\vec{L}
&=&a_1\left(H-\sum_{i=1}^r\lfloor a_i/a_1\rfloor L_i\right)\\
&&+\sum_{i=2}^r a_1\{a_i/a_1\}(H-L_i)+a_1\left((r-1)-\sum_{i=2}^r\{a_i/a_1\}\right)H
\end{eqnarray*}
is effective, there is an inclusion $V_{\vec{a}}\subset H^0(X,a_1 r H)$. 
Thus we get the inequality 
\begin{eqnarray*}
\alpha_\eta(X,\Delta;V_{\vec{\bullet}})&\geq&\alpha_\eta(X,\Delta; r H), \\
\alpha(X,\Delta;V_{\vec{\bullet}})&\geq&\alpha(X,\Delta; r H), 
\end{eqnarray*}
where we identify $r H$ and the complete linear series of $r H$. 
Thus we may assume that $V_{\vec{\bullet}}$ is the complete linear series of $r H$. 
By Example \ref{filter_example} \eqref{filter_example2}, 
by taking the normalization of $X$, we may further assume that $X$ is normal. 
Moreover, we may assume that $K_X+\Delta$ is $\Q$-Cartier after replacing 
$\Delta$ suitably outside $\eta\in X$ for \eqref{positivity_proposition1}. 
Let $\sigma\colon\tilde{X}\to X$ be a log resolution of $(X,\Delta)$, and let us set 
$K_{\tilde{X}}+\tilde{\Delta}:=\sigma^*(K_X+\Delta)$. 
Let $\tilde{\Delta}=\sum_id_i\Delta_i$ be the irreducible decomposition and let us set 
$\Delta':=\sum_{d_i\in (0,1)}d_i\Delta_i$. 

\eqref{positivity_proposition1}
For any prime divisor $E$ over $X$ with $\eta\in c_X(E)$, we have 
$A_{X,\Delta}(E)\geq A_{\tilde{X},\Delta'}(E)$. 

\eqref{positivity_proposition2}
For any prime divisor $E$ over $X$, we have 
$A_{X,\Delta}(E)\geq A_{\tilde{X},\Delta'}(E)$. 

Thus, it is enough to show the inequality 
\[
\alpha\left(\tilde{X}, \Delta'; \sigma^*(r H)\right)>0. 
\]
The inequality is well-known. See \cite[Theorem A]{BJ} for example. 
\end{proof}

\subsection{A generalization of an adjunction-type theorem}\label{adj_subsection}

In \cite[Theorem 3.3]{AZ} (see Theorem \ref{AZ_thm}), the authors consider the 
refinements of graded linear series by either \emph{Cartier divisors} or 
\emph{plt-type prime divisors} over klt $(X,\Delta)$. It seems to be 
important to consider the refinements by more singular prime divisors 
for the future studies for K-stability of Fano varieties. 
For example, in \cite{FANO}, in order to consider the Fano threefold $X$ 
in No.2.20 with $\Aut^0(X)=\G_m$, they apply \cite[Theorem 3.3]{AZ} for 
non-plt type (but Cartier) prime divisor $Y$ on $X$. 
In \S \ref{adj_subsection}, we give a slight generalization of Theorem \ref{AZ_thm} 
from another approach by using the notion of \emph{subbasis type divisors}. 
Note that, in order to prove Theorem \ref{mainthm}, the formulation in 
Theorem \ref{AZ_thm} is enough for us.

\begin{definition}\label{subbasis_definition}
Let $\sF$ be a filtration on a finite dimensional complex vector space $W$
(see Definition \ref{filter-vs_definition}). 
\begin{enumerate}
\renewcommand{\theenumi}{\arabic{enumi}}
\renewcommand{\labelenumi}{(\theenumi)}
\item\label{subbasis_definition1}
We set $\sF^{>\lambda}W:=\bigcup_{\lambda'>\lambda}\sF^{\lambda'}W$ and 
$\Gr^\lambda_{\sF}W:=\sF^\lambda W/\sF^{>\lambda}W$ for any $\lambda\in\R$. 
Moreover, for any $s\in W\setminus\{0\}$, we set 
\[
v_{\sF}(s):=\max\{\lambda\in \R_{\geq 0}\,\,|\,\,s\in\sF^\lambda W\}
\]
as in \cite[Definition 2.19]{AZ}. We can naturally get an element 
\[
\bar{s}\in\Gr_\sF^{v_\sF(s)}W\setminus\{0\}
\]
from $s$ and $\sF$. 
\item\label{subbasis_definition2}
Let $\Lambda\subset\R_{\geq 0}$ be any subset. 
A subset $\{s_1,\dots,s_M\}\subset W$ is said to be an 
\emph{$(\sF,\Lambda)$-subbasis of $W$} if there is a decomposition 
\[
\{s_1,\dots,s_M\}=\bigsqcup_{\lambda\in\Lambda}\{s_1^\lambda,\dots,
s_{M_\lambda}^\lambda\}
\]
such that 
\begin{itemize}
\item
for any $\lambda\in\Lambda$ and for any $1\leq i\leq M_\lambda$, 
we have $v_\sF(s_i^\lambda)=\lambda$, and 
\item
the naturally-induced subset 
\[
\{\bar{s}_1^\lambda,\dots,\bar{s}_{M_\lambda}^\lambda\}\subset\Gr_\sF^\lambda W
\]
forms a basis of $\Gr_\sF^\lambda W$ for any $\lambda\in\Lambda$. 
\end{itemize}
Obviously, $s_1,\dots,s_M\in W$ are linearly independent. 
An $(\sF,\R_{\geq 0})$-subbasis of $W$ is said to be a \emph{basis of $W$ 
compatible with $\sF$}. 
\end{enumerate}
We have $\Gr_\sF^\lambda W=0$ for all but finite $\lambda\in\R_{\geq 0}$, since 
we have the equation 
\[
\dim W=\sum_{\lambda\in\R_{\geq 0}}\dim\Gr_\sF^\lambda W.
\]
\end{definition}

\begin{definition}[{cf.\ \cite[Lemma 3.1]{AZ}}]\label{filter-filter_definition}
Let $\sF$ and $\sG$ be filtrations on $W$. 
For any $\mu\in\R_{\geq 0}$, we can naturally take the filtration $\bar{\sF}$ on 
$\Gr_\sG^\mu W$ by 
\[
\bar{\sF}^\lambda\left(\Gr_\sG^\mu W\right):=
\left(\left(\sF^\lambda W+\sG^{>\mu} W\right)\cap \sG^\mu W\right)/\sG^{>\mu}W.
\]
By \cite[Lemma 3.1]{AZ}, for any $\lambda$, $\mu\in\R_{\geq 0}$, 
we have the natural commutative diagram 
\[\xymatrix{
 & \bar{\sF}^\lambda\left(\Gr_\sG^\mu W\right) \ar@{->>}[r]
 & \Gr_{\bar{\sF}}^\lambda\left(\Gr_\sG^\mu W\right) \ar[dd]^\simeq \\
\sF^\lambda W\cap\sG^\mu W \ar[rd] \ar[ru] & & \\
 & \bar{\sG}^\mu\left(\Gr_\sF^\lambda W\right) \ar@{->>}[r]
 & \Gr_{\bar{\sG}}^\mu\left(\Gr_\sF^\lambda W\right),
}\]
and the kernel of the surjection 
$\sF^\lambda W\cap\sG^\mu W\to
\Gr_{\bar{\sF}}^\lambda\left(\Gr_\sG^\mu W\right)$ is equal to 
\[
\left(\sF^{>\lambda}W\cap\sG^\mu W\right)+\left(\sF^\lambda W\cap
\sG^{>\mu}W\right). 
\]
For arbitrary subsets $\Lambda$, $\Xi\subset\R_{\geq 0}$, a subset 
$\{s_1,\dots,s_M\}\subset W$ is said to be an 
\emph{$\left((\sF,\Lambda), (\sG,\Xi)\right)$-subbasis of $W$} if 
there is a decomposition 
\[
\{s_1,\dots,s_M\}=\bigsqcup_{\mu\in\Xi}
\bigsqcup_{\lambda\in\Lambda}\{s_1^{\lambda,\mu},\dots,
s_{N_{\lambda,\mu}}^{\lambda,\mu}\}
\]
such that, for any $\lambda\in\Lambda$ and $\mu\in\Xi$, we have 
\begin{itemize}
\item
$\left\{s_1^{\lambda,\mu},\dots,
s_{N_{\lambda,\mu}}^{\lambda,\mu}\right\}\subset\sF^\lambda W\cap\sG^\mu W$, 
and 
\item
the image $\left\{\tilde{s}_1^{\lambda,\mu},\dots,
\tilde{s}_{N_{\lambda,\mu}}^{\lambda,\mu}\right\}$
of $\left\{s_1^{\lambda,\mu},\dots,
s_{N_{\lambda,\mu}}^{\lambda,\mu}\right\}$
under the surjection $\sF^\lambda W\cap\sG^\mu W\to
\Gr_{\bar{\sF}}^\lambda\left(\Gr_\sG^\mu W\right)$
forms a basis of $\Gr_{\bar{\sF}}^\lambda\left(\Gr_\sG^\mu W\right)$. 
\end{itemize}
Moreover, if $\Lambda=\R_{\geq 0}$ (resp., if $\Lambda=\R_{\geq 0}$ and 
$\Xi=\R_{\geq 0}$), then we call it \emph{a $(\sG,\Xi)$-subbasis of $W$ 
compatible with $\sF$} (resp., \emph{a basis of $W$ compatible with 
$\sF$ and $\sG$}). 
\end{definition}

\begin{lemma}[{cf.\ \cite[Lemma 3.1]{AZ}}]\label{filter-filter_lemma}
Let $\sF$ and $\sG$ be filtrations on $W$. 
\begin{enumerate}
\renewcommand{\theenumi}{\arabic{enumi}}
\renewcommand{\labelenumi}{(\theenumi)}
\item\label{filter-filter_lemma1}
Take any $s\in W\setminus\{0\}$ and we set $\mu:=v_\sG(s)$ and let 
$\bar{s}\in\Gr_\sG^\mu W\setminus\{0\}$ be the element induced by $s$ and $\sG$. 
Then we have the inequality 
$v_{\bar{\sF}}(\bar{s})\geq v_\sF(s)$. 
\item\label{filter-filter_lemma2}
Take arbitrary subsets $\Lambda$, $\Xi\subset\R_{\geq 0}$ and let us take 
any $\left((\sF,\Lambda),(\sG,\Xi)\right)$-subbasis $\{s_1,\dots,s_M\}\subset W$
of $W$. Then, for any $1\leq i\leq M$, if we set 
$\lambda:=v_\sF(s_i)$ and $\mu:=v_\sG(s_i)$, then the naturally-induced elements 
$\bar{s}_i^{\sF}\in\Gr_\sF^\lambda W\setminus\{0\}$ and 
$\bar{s}_i^{\sG}\in\Gr_\sG^\mu W\setminus\{0\}$ satisfy that 
$v_{\bar{\sG}}(\bar{s}_i^{\sF})=\mu$ and $v_{\bar{\sF}}(\bar{s}_i^{\sG})=\lambda$. 
\item\label{filter-filter_lemma3}
Take an arbitrary subset $\Xi\subset\R_{\geq 0}$. Then, any $(\sG,\Xi)$-subbasis 
of $W$ compatible with $\sF$ is a $(\sG,\Xi)$-subbasis of $W$. 
\end{enumerate}
\end{lemma}

\begin{proof}
\eqref{filter-filter_lemma1} and \eqref{filter-filter_lemma2} are trivial from the 
construction. 
For \eqref{filter-filter_lemma3}, for all $\mu\in\Xi$, the images 
\[
\bigsqcup_{\lambda\in\R_{\geq 0}}\{\bar{s}_1^{\lambda,\mu},\dots,
\bar{s}_{N_{\lambda,\mu}}^{\lambda,mu}\}\subset\Gr_{\sG}^\mu W
\]
give bases of $\Gr_{\sG}^\mu W$ compatible with $\bar{\sF}$. 
\end{proof}

\begin{corollary}[{cf.\ \cite[Lemma 3.5]{BJ}}]\label{BJ-type_corollary}
Let $\sF$ and $\sG$ be filtrations on $W$ and let $\Xi\subset\R_{\geq 0}$ 
be any subset. 
\begin{enumerate}
\renewcommand{\theenumi}{\arabic{enumi}}
\renewcommand{\labelenumi}{(\theenumi)}
\item\label{BJ-type_corollary1}
For any $(\sG,\Xi)$-subbasis $\{s_1,\dots,s_M\}\subset W$
of $W$, we have 
\[
\sum_{i=1}^Mv_\sF(s_i)\leq\sum_{\mu\in\Xi}\int_0^\infty
\dim\bar{\sF}^\lambda\left(\Gr_\sG^\mu W\right)d\lambda.
\]
\item\label{BJ-type_corollary2}
For any $(\sG,\Xi)$-subbasis $\{s_1,\dots,s_M\}\subset W$
of $W$ compatible with $\sF$, we have 
\[
\sum_{i=1}^Mv_\sF(s_i)=\sum_{\mu\in\Xi}\int_0^\infty
\dim\bar{\sF}^\lambda\left(\Gr_\sG^\mu W\right)d\lambda.
\]
\end{enumerate}
\end{corollary}

\begin{proof}
Fix $\mu\in \Xi$ and let us set $N_\mu:=\dim\Gr_\sG^\mu W$. 
For any $1\leq i\leq N_\mu$, let us set 
\[
e_i:=\max\left\{\lambda\in\R_{\geq 0}\,\,|\,\,\dim\bar{\sF}^\lambda
\left(\Gr_\sG^\mu W\right)\geq N_\mu+1-i\right\}.
\]
Then we have the equality 
\[
\int_0^\infty\dim\bar{\sF}^\lambda\left(\Gr_\sG^\mu W\right)d\lambda
=\sum_{i=1}^{N_\mu}e_i.
\]
Take any basis $\left\{\bar{s}_1,\dots,\bar{s}_{N_\mu}\right\}\subset\Gr_\sG^\mu W$
with $v_{\bar{\sF}}(\bar{s}_1)\leq\cdots\leq v_{\bar{\sF}}(\bar{s}_{N_\mu})$. 
Then, as in the proof of \cite[Lemma 3.5]{BJ}, we have 
$v_{\bar{\sF}}(\bar{s}_i)\leq e_i$. Moreover, if the basis is compatible with 
$\bar{\sF}$, we have $v_{\bar{\sF}}(\bar{s}_i)=e_i$. 
Thus we get the assertion by Lemma \ref{filter-filter_lemma}. 
\end{proof}

From now on, unless otherwise stated, we fix: 
\begin{itemize}
\item
an $n$-dimensional normal projective variety $X$, 
\item
a $\Z_{\geq 0}^r$-graded linear series $W_{\vec{\bullet}}$ on $X$ associated to 
Cartier divisors $L_1,\dots,L_r$ which has bounded support and 
contains an ample series, 
\item
a projective birational morphism $\sigma\colon\tilde{X}\to X$ with $\tilde{X}$ 
normal, 
\item
a prime divisor $Y\subset\tilde{X}$ such that $e Y$ is Cartier for some 
$e\in\Z_{>0}$, 
\item
an admissible flag $Y_\bullet$ on $\tilde{X}$ with $Y_1=Y$ (and let $Y'_\bullet$ be 
the admissible flag on $Y$ induced by $Y_\bullet$), 
\item
the linear transform 
\begin{eqnarray*}
\bar{h}\colon\R^{r-1+n}&\to&\R^{r-1+n}\\
(x_1,\dots,x_{r-1+n})&\mapsto&(x_1,\dots,x_{r-1},e x_r,x_{r+1},\dots,x_{r-1+n}), 
\end{eqnarray*}
\item
$\sG:=\sF_Y$ on $W_{\vec{\bullet}}$, 
and 
\item
the graded linear series $W_{\vec{\bullet}}^{(Y,e)}:=
\sigma^*W_{\vec{\bullet}}^{(Y,e)}$ on $Y$ as in Lemma \ref{refinement_lemma}. 
\end{itemize}
We note that, for any $\vec{a}\in\Z_{\geq 0}^r$ and for any $j\in\Z_{\geq 0}$, 
we have $W_{\vec{a},j}^{(Y,e)}=\Gr_\sG^{j e}W_{\vec{a}}$.

\begin{lemma}[{cf.\ \cite[Lemma 2.9]{AZ21}}]\label{AZ21_lemma}
Let $\sF$ be a linearly bounded filtration on $W_{\vec{\bullet}}$. 
As in \cite[Definition 2.8]{AZ21}, we can naturally get the linearly bounded 
filtration $\bar{\sF}$ on $W_{\vec{\bullet}}^{(Y,e)}$ from $\sF$. 
We have the equality 
\[
T\left(W_{\vec{\bullet}};\sF\right)=T\left(W_{\vec{\bullet}}^{(Y,e)};\bar{\sF}\right). 
\]
Moreover, for any $t\in\left[0,T\left(W_{\vec{\bullet}};\sF\right)\right)$, we have 
\[
\bar{h}\left(\Delta_{Y'_\bullet}\left(W_{\vec{\bullet}}^{(Y,e),\bar{\sF},t}\right)\right)
=\Delta_{Y_\bullet}\left(\sigma^*W_{\vec{\bullet}}^{\sF,t}\right).
\]
In particular, we have $G_{\bar{\sF}}=G_\sF\circ\bar{h}$ and 
\[
S\left(W_{\vec{\bullet}};\sF\right)=S\left(W_{\vec{\bullet}}^{(Y,e)};\bar{\sF}\right).
\]
\end{lemma}

\begin{proof}
Since $\sF^\lambda W_{m,\vec{a}}=0$ trivially implies 
$\bar{\sF}^\lambda W_{m,\vec{a},j}^{(Y,e)}=0$, we get the inequality 
$T_m\left(W_{\vec{\bullet}};\sF\right)\geq 
T_m\left(W_{\vec{\bullet}}^{(Y,e)};\bar{\sF}\right)$. 
For any $s\in\sF^\lambda W_{m,\vec{a}}\setminus\{0\}$, if we set 
$j:=\ord_Y(s)$, then we have 
\[
s^e\in\sF^{e\lambda}W_{e m, e\vec{a},j}\cap\sG^{j e}W_{e m,e\vec{a}}\setminus\{0\}
\]
and the element $s^e$ induces 
\[
\bar{s^e}\in\bar{\sF}^{e\lambda} W_{e m,e\vec{a},j}^{(Y,e)}\setminus\{0\}. 
\]
This implies the inequality 
$e T_m\left(W_{\vec{\bullet}};\sF\right)\leq 
T_{e m}\left(W_{\vec{\bullet}}^{(Y,e)};\bar{\sF}\right)$. 
Thus we get the equality 
$T\left(W_{\vec{\bullet}};\sF\right)=T\left(W_{\vec{\bullet}}^{(Y,e)};\bar{\sF}\right)$. 

As in the proof of Lemma \ref{refinement_lemma}, for any 
$t\in\left[0,T\left(W_{\vec{\bullet}};\sF\right)\right)$, we have 
\[
h\left(\Sigma_{Y'_\bullet}\left(W_{\vec{\bullet}}^{(Y,e),\bar{\sF},t}\right)\right)
=\Sigma_{Y_\bullet}\left(\sigma^*W_{\vec{\bullet}}^{\sF,t}\right),
\]
where $h\colon\R^{r+n}\to\R^{r+n}$ be as in Lemma \ref{refinement_lemma}. 
Thus we get the assertion. 
\end{proof}

\begin{example}\label{AZ21_example}
If $\sF=\sG$, then we have 
\[
\bar{\sG}^\lambda W_{\vec{a},j}^{(Y,e)}=\begin{cases}
W_{\vec{a},j}^{(Y,e)} & \text{if }\lambda\leq j e, \\
0 & \text{if }\lambda>j e.
\end{cases}\]
\end{example}

\begin{definition}\label{subbasis-divisor_definition}
Let $\sF$ be a linearly bounded filtration on $W_{\vec{\bullet}}$. 
\begin{enumerate}
\renewcommand{\theenumi}{\arabic{enumi}}
\renewcommand{\labelenumi}{(\theenumi)}
\item\label{subbasis-divisor_definition1}
Take any $\vec{a}\in\Z_{\geq 0}^r$. An effective Cartier divisor $D$ on $X$ is 
said to be a \emph{$(Y,e)$-subbasis type divisor of $W_{\vec{a}}$}
(resp., a \emph{$(Y,e)$-subbasis type divisor of $W_{\vec{a}}$ compatible with $\sF$}) 
if there is a $(\sG,e\Z_{\geq 0})$-subbasis 
$\{s_1,\dots,s_M\}\subset W_{\vec{a}}$ of $W_{\vec{a}}$ 
(resp., compatible with $\sF$) such that $D$ is of the form 
\[
D=\sum_{i=1}^M\{s_i=0\}. 
\]
From the construction, we have 
\[
M=\sum_{j\in\Z_{\geq 0}}\dim W_{\vec{a},j}^{(Y,e)}, \quad
D\sim M\vec{a}\cdot\vec{L}.
\]
\item\label{subbasis-divisor_definition2}
Take any $m\in\Z_{>0}$ with $M_m:=h^0\left(W_{m,\vec{\bullet}}^{(Y,e)}\right)>0$. 
An effective $\Q$-Cartier $\Q$-divisor $D$ on $X$ is said to be 
an \emph{$m$-$(Y,e)$-subbasis type $\Q$-divisor of $W_{\vec{\bullet}}$} 
(resp., an \emph{$m$-$(Y,e)$-subbasis type $\Q$-divisor of $W_{\vec{\bullet}}$ 
compatible with $\sF$}) if $D$ is of the form 
\[
D=\frac{1}{m M_m}\sum_{\substack{\vec{a}\in\Z_{\geq 0}^{r-1};\\ 
(m,\vec{a})\in\sS(W_{\vec{\bullet}})}}D_{\vec{a}}, 
\]
where each $D_{\vec{a}}$ is a $(Y,e)$-subbasis type divisor of $W_{m,\vec{a}}$ 
(resp., compatible with $\sF$). 
\end{enumerate}
\end{definition}

\begin{proposition}[{cf.\ \cite[\S 3.1]{AZ}}]\label{subbasis-divisor_proposition}
Let $E$ be a prime divisor over $X$ and let 
$D$ be an $m$-$(Y,e)$-subbasis type $\Q$-divisor of $W_{\vec{\bullet}}$. 
\begin{enumerate}
\renewcommand{\theenumi}{\arabic{enumi}}
\renewcommand{\labelenumi}{(\theenumi)}
\item\label{subbasis-divisor_proposition1}
We have 
$\ord_E(D)\leq S_m\left(W_{\vec{\bullet}}^{(Y,e)};\bar{\sF}_E\right)$ and 
$\ord_Y(D)= S_m\left(W_{\vec{\bullet}}^{(Y,e)};\bar{\sG}\right)$. 
\item\label{subbasis-divisor_proposition2}
If $D$ is compatible with $\sF_E$, then we have 
$\ord_E(D)= S_m\left(W_{\vec{\bullet}}^{(Y,e)};\bar{\sF}_E\right)$. 
\item\label{subbasis-divisor_proposition3}
Let us set 
\begin{eqnarray*}
\sigma^*D&=:&S_m\left(W_{\vec{\bullet}}^{(Y,e)};\bar{\sG}\right)\cdot Y+\tilde{D}, \\
D_Y&:=&\tilde{D}|_Y.
\end{eqnarray*}
The $\Q$-Cartier $\Q$-divisor $D_Y$ on $Y$ is an 
\emph{$m$-basis type $\Q$-divisor of $W_{\vec{\bullet}}^{(Y,e)}$} in the sense of 
\cite[Definition 2.18]{AZ}. If moreover $D$ is compatible with $\sF_E$, then 
the $D_Y$ is compatible with $\bar{\sF}_E$ in the sense of \cite[Definition 2.18]{AZ}. 
\end{enumerate}
\end{proposition}

\begin{proof}
Let us write 
\[
D=\frac{1}{m M_m}\sum_{\substack{\vec{a}\in\Z_{\geq 0}^{r-1};\\ 
(m,\vec{a})\in\sS(W_{\vec{\bullet}})}}D_{\vec{a}}, 
\]
with 
\[
D_{\vec{a}}=\sum_{j\in\Z_{\geq 0}}\sum_{i=1}^{M_{\vec{a},j}}\left\{
s_i^{\vec{a},j}=0\right\},
\]
where each 
\[
\bigsqcup_{j\in\Z_{\geq 0}}\left\{s_1^{\vec{a},j},\dots,s_{M_{\vec{a},j}}^{\vec{a},j}\right\}
\]
is a $(\sG,e\Z_{\geq 0})$-subbasis of $W_{m,\vec{a}}$ (resp., compatible with 
$\bar{\sF}_E$) with $\ord_Y\left(s_i^{\vec{a},j}\right)=j e$, that is, the image 
\[
\left\{\bar{s}_1^{\vec{a},j},\dots,\bar{s}_{M_{\vec{a},j}}^{\vec{a},j}
\right\}\subset W_{m,\vec{a},j}^{(Y,e)}
\]
is a basis of $W_{m,\vec{a},j}^{(Y,e)}$ (resp., compatible with $\bar{\sF}_E$). 

\eqref{subbasis-divisor_proposition1}
By Corollary \ref{BJ-type_corollary}, we have 
\begin{eqnarray*}
\ord_E\left(D_{\vec{a}}\right)=\sum_{j\in\Z_{\geq 0}}\sum_{i=1}^{M_{\vec{a},j}}
v_{\sF_E}\left(s_i^{\vec{a},j}\right)
\leq\sum_{j\in\Z_{\geq 0}}\int_0^\infty
\dim\bar{\sF}_E^\lambda W_{m,\vec{a},j}^{(Y,e)}d\lambda.
\end{eqnarray*}
Thus we get
\[
\ord_E\left(D\right)\leq\frac{1}{m M_m}\sum_{\vec{a}, j}\int_0^\infty
\dim\bar{\sF}_E^\lambda W_{m,\vec{a},j}^{(Y,e)}d\lambda
=S_m\left(W_{\vec{\bullet}}^{(Y,e)};\bar{\sF}_E\right). 
\]

\eqref{subbasis-divisor_proposition2}
If $D$ is compatible with $\sF_E$, then the above inequalities are equal. 
Note that, for any $m$-$(Y,e)$-subbasis type $\Q$-divisor $D$ of $W_{\vec{\bullet}}$, 
$D$ is compatible with $\sG$ (see Example \ref{AZ21_example}). 

\eqref{subbasis-divisor_proposition3}
Since $D_Y$ is of the form
\[
D_Y=\frac{1}{m M_m}\sum_{\vec{a},j}\sum_{i=1}^{M_{\vec{a},j}}
\left\{\bar{s}_i^{\vec{a},j}=0\right\}, 
\]
the assertion is trivial. 
\end{proof}

From now on, we fix an effective $\Q$-Weil divisor $\Delta$ on $X$ and 
a scheme-theoretic point $\eta\in X$ such that $(X,\Delta)$ is klt at $\eta$. 
Recall that, in \cite[Definition 2.19]{AZ}, for any $m\in\Z_{>0}$ with 
$h^0\left(W_{m,\vec{\bullet}}\right)>0$, they set 
\[
\delta_{\eta,m}\left(X,\Delta;W_{\vec{\bullet}}\right)
:=\inf_{\substack{D:\text{ $m$-basis type }\\\text{$\Q$-divisor of }W_{\vec{\bullet}}}}
\lct_\eta(X,\Delta; D), 
\]
and showed in \cite[Lemma 2.21]{AZ} that 
\[
\delta_\eta\left(X,\Delta;W_{\vec{\bullet}}\right)
=\lim_{m\to\infty}\delta_{\eta,m}\left(X,\Delta;W_{\vec{\bullet}}\right).
\]
Let us consider its analogue.

\begin{definition}\label{delta-m_definition}
Take any $m\in\Z_{>0}$ with $h^0\left(W_{m,\vec{\bullet}}^{(Y,e)}\right)>0$. 
\begin{enumerate}
\renewcommand{\theenumi}{\arabic{enumi}}
\renewcommand{\labelenumi}{(\theenumi)}
\item\label{delta-m_definition1}
Set 
\[
\delta_{\eta,m}^{(Y,e)}\left(X,\Delta;W_{\vec{\bullet}}\right)
:=\inf_D\lct_\eta(X,\Delta; D), 
\]
where $D$ runs through all $m$-$(Y,e)$-subbasis type $\Q$-divisors 
of $W_{\vec{\bullet}}$. 
\item\label{delta-m_definition2}
Assume that $(X,\Delta)$ is a klt pair. Set 
\[
\delta_{m}^{(Y,e)}\left(X,\Delta;W_{\vec{\bullet}}\right):=\inf_D\lct(X,\Delta; D), 
\]
where $D$ runs through all $m$-$(Y,e)$-subbasis type $\Q$-divisors 
of $W_{\vec{\bullet}}$. 
\end{enumerate}
\end{definition}

\begin{proposition}[{see \cite[Proposition 4.3]{BJ}}]\label{delta-m_proposition}
\begin{enumerate}
\renewcommand{\theenumi}{\arabic{enumi}}
\renewcommand{\labelenumi}{(\theenumi)}
\item\label{delta-m_proposition1}
We have 
\[
\delta_{\eta,m}^{(Y,e)}\left(X,\Delta;W_{\vec{\bullet}}\right)
=\inf_{\substack{E:\text{ prime divisor}\\ \text{over }X\\ \text{with }\eta\in c_X(E)}}
\frac{A_{X,\Delta}(E)}{S_m\left(W_{\vec{\bullet}}^{(Y,e)};\bar{\sF}_E\right)}. 
\]
\item\label{delta-m_proposition2}
Assume that $(X,\Delta)$ is a klt pair. Then we have 
\[
\delta_{m}^{(Y,e)}\left(X,\Delta;W_{\vec{\bullet}}\right)
=\inf_{\substack{E:\text{ prime divisor} \\ \text{over }X}}
\frac{A_{X,\Delta}(E)}{S_m\left(W_{\vec{\bullet}}^{(Y,e)};\bar{\sF}_E\right)}. 
\]
\end{enumerate}
\end{proposition}

\begin{proof}
We only see \eqref{delta-m_proposition1}. By the definition of the log canonical 
threshold, we have 
\[
\delta_{\eta,m}^{(Y,e)}\left(X,\Delta;W_{\vec{\bullet}}\right)
=\inf_D\inf_{\substack{E/X;\\ \eta\in c_X(E)}}\frac{A_{X,\Delta}(E)}{\ord_E(D)}. 
\]
On the other hand, by Proposition \ref{subbasis-divisor_proposition}, we have 
$\sup_D\ord_E(D)=S_m\left(W_{\vec{\bullet}}^{(Y,e)};\bar{\sF}_E\right)$.
\end{proof}

The following lemma is well-known and essentially same as \cite[Corollary 2.10]{BJ}. 
We omit the proof. See also the proof of \cite[Lemma 2.21]{AZ}. 

\begin{lemma}[{\cite[Corollary 2.10]{BJ}}]\label{delta-concave_lemma}
For any $\varepsilon\in\R_{>0}$, there exists $m_0\in\Z_{>0}$ such that, 
for any linearly bounded filtration $\sH$ on $W_{\vec{\bullet}}^{(Y,e)}$ and for any 
$m\geq m_0$, we have the inequality 
\[
S_m\left(W_{\vec{\bullet}}^{(Y,e)};\sH\right)\leq(1+\varepsilon)\cdot
S\left(W_{\vec{\bullet}}^{(Y,e)};\sH\right).
\]
\end{lemma}

Thanks to Lemma \ref{delta-concave_lemma}, we can get the following: 

\begin{proposition}[{cf.\ \cite[Theorem 4.4]{BJ} and 
\cite[Lemma 2.21]{AZ}}]\label{delta-concave_proposition}
\begin{enumerate}
\renewcommand{\theenumi}{\arabic{enumi}}
\renewcommand{\labelenumi}{(\theenumi)}
\item\label{delta-concave_proposition1}
We have 
\[
\delta_\eta\left(X,\Delta;W_{\vec{\bullet}}\right)=\lim_{m\to\infty}
\delta_{\eta,m}^{(Y,e)}\left(X,\Delta;W_{\vec{\bullet}}\right). 
\]
\item\label{delta-concave_proposition2}
Assume that $(X,\Delta)$ is a klt pair. Then we have 
\[
\delta\left(X,\Delta;W_{\vec{\bullet}}\right)=\lim_{m\to\infty}
\delta_{m}^{(Y,e)}\left(X,\Delta;W_{\vec{\bullet}}\right). 
\]
\end{enumerate}
\end{proposition}

\begin{proof}
The proof is same as the proof of \cite[Theorem 4.4]{BJ}. 
We give the proof of \eqref{delta-concave_proposition1} 
just for the readers' convenience. 
By Proposition \ref{delta-m_proposition} and Lemma \ref{AZ21_lemma}, 
we have 
\begin{eqnarray*}
&&\limsup_{m\to\infty}\delta_{\eta,m}^{(Y,e)}\left(X,\Delta;W_{\vec{\bullet}}\right)
=\limsup_{m\to\infty}\inf_{\substack{E/X;\\ \eta\in c_X(E)}}
\frac{A_{X,\Delta}(E)}{S_m\left(W_{\vec{\bullet}}^{(Y,e)};\bar{\sF}_E\right)}\\
&\leq&\inf_{\substack{E/X; \\ \eta\in c_X(E)}}
\frac{A_{X,\Delta}(E)}{S\left(W_{\vec{\bullet}}^{(Y,e)};\bar{\sF}_E\right)}
=\delta_\eta\left(X,\Delta;W_{\vec{\bullet}}\right). 
\end{eqnarray*}
On the other hand, by Lemma \ref{delta-concave_lemma}, for any 
$\varepsilon\in\R_{>0}$, we get 
\begin{eqnarray*}
\liminf_{m\to\infty}\delta_{\eta,m}^{(Y,e)}\left(X,\Delta;W_{\vec{\bullet}}\right)
\geq\inf_{\substack{E/X;\\ \eta\in c_X(E)}}\frac{1}{1+\varepsilon}\cdot
\frac{A_{X,\Delta}(E)}{S\left(W_{\vec{\bullet}}^{(Y,e)};\bar{\sF}_E\right)}
=\frac{1}{1+\varepsilon}\cdot\delta_\eta\left(X,\Delta;W_{\vec{\bullet}}\right). 
\end{eqnarray*}
Thus we get the assertion. 
\end{proof}

We are ready to generalize Theorem \ref{AZ_thm}. 
Note that, \cite[Theorem 3.3]{AZ} treats the equality case much more. 
We omit to discuss the case since we do not use it in order to prove 
Theorem \ref{mainthm}.

\begin{thm}[{cf.\ \cite[Theorem 3.3]{AZ}}]\label{AZ-R_thm}
Let $\sigma\colon\tilde{X}\to X$ be a birational morphism between 
normal projective varieties and let $Z_0\subset Z\subset X$ be closed subvarieties 
on $X$. Let $\eta_0$, $\eta\in X$ be the generic points of $Z_0$, $Z$, respectively. 
Let $F\subset\tilde{X}$ be a prime $\Q$-Cartier divisor on $\tilde{X}$ with 
$\eta_0\in c_X(F)$ and let $\Delta$ be an effective $\Q$-Weil divisor on $X$. 
Assume that there is an open subset $\eta_0\in U\subset X$ such that 
the pair $(U,\Delta|_U)$ is klt, the prime divisor $F$ is a plt-type prime divisor 
over $(U,\Delta|_U)$, and the morphism $\sigma$ is the plt-blowup of $F$ 
over $(U,\Delta|_U)$. Let us take the Veronese equivalence class $V_{\vec{\bullet}}$ 
of a graded linear series on $X$ associated to 
$L_1,\dots,L_r\in\CaCl(X)\otimes_\Z\Q$ which has bounded support and contains 
an ample series, and let $W_{\vec{\bullet}}$ be the refinement of 
$\sigma^*V_{\vec{\bullet}}$ by $F\subset\tilde{X}$. Take an effective 
$\Q$-Weil divisor $\tilde{\Delta}$ on $\tilde{X}$ and an effective $\Q$-Weil 
divisor $\Delta_F$ on $F$ such that 
\begin{eqnarray*}
K_{\tilde{X}}+\tilde{\Delta}+\left(1-A_{X,\Delta}(F)\right)F&=&\sigma^*(K_X+\Delta),\\
K_F+\Delta_F&=&\left(K_{\tilde{X}}+\tilde{\Delta}+F\right)|_F
\end{eqnarray*}
hold over $U$. 
\begin{enumerate}
\renewcommand{\theenumi}{\arabic{enumi}}
\renewcommand{\labelenumi}{(\theenumi)}
\item\label{AZ-R_thm1}
If $Z\subset c_X(F)$, then we have 
\[
\delta_\eta\left(X,\Delta;V_{\vec{\bullet}}\right)\geq
\min\left\{\frac{A_{X,\Delta}(F)}{S\left(V_{\vec{\bullet}};F\right)},\quad
\inf_{\eta'\in\tilde{X}; \sigma(\eta')=\eta_0}
\delta_{\eta'}\left(F,\Delta_F;W_{\vec{\bullet}}\right)
\right\}.
\]
\item\label{AZ-R_thm2}
If $Z\not\subset c_X(F)$, then we have 
\[
\delta_\eta\left(X,\Delta;V_{\vec{\bullet}}\right)\geq
\inf_{\eta'\in\tilde{X}; \sigma(\eta')=\eta_0}
\delta_{\eta'}\left(F,\Delta_F;W_{\vec{\bullet}}\right).
\]
\end{enumerate}
If the inequality in \eqref{AZ-R_thm1} holds and there exists a prime divisor $E_0$ 
over $X$ with $Z\subset c_X(E_0)$, $c_{\tilde{X}}(E_0)\subset F$ and 
\[
\delta_\eta\left(X,\Delta;V_{\vec{\bullet}}\right)
=\frac{A_{X,\Delta}(E_0)}{S\left(V_{\vec{\bullet}};E_0\right)},
\]
then we must have 
\[
\delta_\eta\left(X,\Delta;V_{\vec{\bullet}}\right)
=\frac{A_{X,\Delta}(F)}{S\left(V_{\vec{\bullet}};F\right)}.
\]
\end{thm}

\begin{proof}
The core of the proof is essentially same as the proof of \cite[Theorem 3.3]{AZ}. 
Let us fix $e\in\Z_{>0}$ with $e F\subset\tilde{X}$ Cartier. 
By Lemma \ref{filter_lemma}, we may assume that $V_{\vec{\bullet}}$ is 
a $\Z_{\geq 0}^r$-graded linear series on $X$ associated to Cartier divisors, and 
$W_{\vec{\bullet}}=\sigma^*V_{\vec{\bullet}}^{(F,e)}$ as in 
Lemma \ref{refinement_lemma}. Take any $m\in\Z_{>0}$ with 
$h^0\left(W_{m,\vec{\bullet}}\right)>0$. Set 
\begin{eqnarray*}
\lambda'_m &:=&\inf_{\eta'\in\tilde{X}; \sigma(\eta')=\eta_0}
\delta_{\eta',m}\left(F,\Delta_F;W_{\vec{\bullet}}\right), \\
\lambda_m &:=&\min\left\{
\frac{A_{X,\Delta}(F)}{S_m\left(W_{\vec{\bullet}};\bar{\sF}_F\right)}, 
\quad \lambda'_m\right\}.
\end{eqnarray*}
Take any $m$-$(F,e)$-subbasis type $\Q$-divisor $D$ of $V_{\vec{\bullet}}$. 
As in Proposition \ref{subbasis-divisor_proposition}, let us set
\begin{eqnarray*}
\sigma^*D&=:&S_m\left(W_{\vec{\bullet}};\bar{\sF}_F\right)\cdot F+\tilde{D}, \\
D_F&:=&\tilde{D}|_F.
\end{eqnarray*}
We know that the $\Q$-divisor $D_F$ is an $m$-basis type $\Q$-divisor 
of $W_{\vec{\bullet}}$. By the definition of $\lambda'_m$, 
the pair $\left(F,\Delta_F+\lambda'_m D_F\right)$ 
is log canonical around a neighborhood of $\sigma^{-1}(\eta_0)$. 
By inversion of adjunction \cite{kawakita}, the pair 
$\left(\tilde{X},\tilde{\Delta}+F+\lambda'_m\tilde{D}\right)$ is log canonical 
around a neighborhood of $\sigma^{-1}(\eta_0)$. 

\eqref{AZ-R_thm2}
For any prime divisor $E$ over $X$ with $\eta\in c_X(E)$, we have 
$c_{\tilde{X}}(E)\not\subset F$ since $Z\not\subset c_X(F)$. 
Thus, if we set 
\[
a'_m:=1-A_{X,\Delta}(F)+\lambda'_m\cdot S_m\left(W_{\vec{\bullet}};\bar{\sF}_F\right), 
\]
then we have 
\[
0\leq A_{\tilde{X},\tilde{\Delta}+F+\lambda'_m\tilde{D}}(E)
=A_{\tilde{X},\tilde{\Delta}+a'_m F+\lambda'_m\tilde{D}}(E)
=A_{X,\Delta+\lambda'_m D}(E).
\]
Thus the pair $(X,\Delta+\lambda'_m D)$ is log canonical at $\eta$. This implies that 
$\delta_{\eta,m}^{(F,e)}\left(X,\Delta;V_{\vec{\bullet}}\right)\geq \lambda'_m$. 
Thus we get the assertion by Proposition \ref{delta-concave_proposition}. 

\eqref{AZ-R_thm1}
Set 
\[
a_m:=1-A_{X,\Delta}(F)+\lambda_m\cdot S_m\left(W_{\vec{\bullet}};\bar{\sF}_F\right)
\leq 1. 
\]
Since the pair $\left(\tilde{X},\tilde{\Delta}+a_m F+\lambda_m\tilde{D}\right)$ 
is log canonical around a neighborhood of $\sigma^{-1}(\eta_0)$, the pair 
$(X,\Delta+\lambda_m D)$ is log canonical at $\eta$. This implies that 
$\delta_{\eta,m}^{(F,e)}\left(X,\Delta;V_{\vec{\bullet}}\right)\geq \lambda_m$. 

Let us consider the equality case. 
We have 
\[
0\leq A_{\tilde{X},\tilde{\Delta}+F+\lambda_m\tilde{D}}(E_0)
=A_{X,\Delta+\lambda_m D}(E_0)
+\left(-A_{X,\Delta}(F)+\lambda_m\cdot 
S_m\left(W_{\vec{\bullet}};\bar{\sF}_F\right)\right)\cdot \ord_{E_0}(F). 
\]
If $D$ is compatible with $\sF_{E_0}$, we get 
\[
A_{X,\Delta}(E_0)-\lambda_m\cdot S_m\left(W_{\vec{\bullet}};\bar{\sF}_{E_0}\right)
\geq \left(A_{X,\Delta}(F)-\lambda_m\cdot 
S_m\left(W_{\vec{\bullet}};\bar{\sF}_F\right)\right)\cdot\ord_{E_0}(F).
\]
Set $\lambda:=\lim_{m\to\infty}\lambda_m$. Then we get 
\[
0=A_{X,\Delta}(E_0)-\lambda\cdot S\left(V_{\vec{\bullet}};E_0\right)
\geq \left(A_{X,\Delta}(F)-\lambda\cdot 
S\left(V_{\vec{\bullet}};F\right)\right)\cdot\ord_{E_0}(F)\geq 0.
\]
Since $\ord_{E_0}(F)$ is positive, we get the assertion. 
\end{proof}

\subsection{The barycenters of Okounkov bodies}\label{barycenter_subsection}

We see a relationship between local $\delta$-invariants and Okounkov bodies. 

\begin{thm}\label{pl_thm}
Let $X$ be an $n$-dimensional normal projective variety, let $Y_\bullet$ be an 
admissible flag on $X$, and let $V_{\vec{\bullet}}$ be the Veronese equivalence class 
of a graded linear series on $X$ associated to 
$L_1,\dots,L_r\in \CaCl(X)\otimes_\Z\Q$ which has bounded support and 
contains an ample series. Let us consider the Okounkov body 
$\Delta_{Y_\bullet}\left(V_{\vec{\bullet}}\right)\subset\R_{\geq 0}^{r-1+n}$ of 
$V_{\vec{\bullet}}$ associated to $Y_\bullet$. 
Let $\eta_j\in X$ be the generic point of $Y_j$ for $1\leq j\leq n$. 
Let 
$\vec{b}=(b_1,\dots,b_{r-1+n})\in\R_{\geq 0}^{r-1+n}$ be the barycenter of 
$\Delta_{Y_\bullet}\left(V_{\vec{\bullet}}\right)$. Then we have the inequalities 
\[
\min\left\{\frac{1}{b_r},\dots,\frac{1}{b_{r-1+j}}\right\}\leq
\delta_{\eta_j}\left(X; V_{\vec{\bullet}}\right)\leq \frac{1}{b_r}
\]
for any $1\leq j\leq n$. In particular, we have the inequalities
\[
\min\left\{\frac{1}{b_r},\dots,\frac{1}{b_{r-1+n}}\right\}\leq
\delta_{Y_n}\left(X; V_{\vec{\bullet}}\right)\leq \frac{1}{b_r}.
\]
\end{thm}

\begin{proof}
Since $\eta_j\in X$ is a smooth point, the value 
$\delta_{\eta_j}\left(X; V_{\vec{\bullet}}\right)$ makes sense. 
Take a resolution $\sigma\colon\tilde{X}\to X$ of singularities such that 
$\sigma$ is an isomorphism over $Y_n$. Let $\tilde{Y}_\bullet$ be the admissible 
flag on $\tilde{X}$ defined by $\tilde{Y}_i:=\sigma^{-1}_*Y_i$. 
Let $\tilde{\eta}_j\in\tilde{X}$ be the generic point of $\tilde{Y}_j$. 
Obviously, we have $\delta_{\eta_j}\left(X; V_{\vec{\bullet}}\right)
=\delta_{\tilde{\eta}_j}\left(\tilde{X}; \sigma^*V_{\vec{\bullet}}\right)$. Moreover, 
from the definition of the function $\nu_{Y_\bullet}$ in 
Definition \ref{okounkov_definition}, we have 
$\Delta_{Y_\bullet}\left(V_{\vec{\bullet}}\right)
=\Delta_{\tilde{Y}_\bullet}\left(\sigma^*V_{\vec{\bullet}}\right)$. 
Therefore, we may assume that $X$ and $Y_i$ are smooth. 

Let $V_{\vec{\bullet}}^i$ ($0\leq i\leq n-1$) be the Veronese equivalence class of 
graded linear series on $Y_i$ defined inductively as follows: 
\begin{itemize}
\item
When $i=0$, then $V_{\vec{\bullet}}^0:=V_{\vec{\bullet}}$.
\item
When $i\geq 1$, then $V_{\vec{\bullet}}^i$ is the refinement of 
$V_{\vec{\bullet}}^{i-1}$ by $Y_i$.
\end{itemize}
As we have already seen in Definition \ref{refinement_definition}, we have 
$\Delta_{Y_\bullet}\left(V_{\vec{\bullet}}\right)
=\Delta_{Y^i_\bullet}\left(V^i_{\vec{\bullet}}\right)\subset\R^{r-1+n}$ 
for any $0\leq i\leq n-1$. Moreover, by Proposition \ref{barycenter_proposition}, 
we have $b_{r+i}=S(V^i_{\vec{\bullet}};Y_{i+1})$ for any $0\leq i\leq n-1$. 
Since $A_{Y_i}(Y_{i+1})=1$, we get the assertion by applying 
Theorem \ref{AZ_thm} (or Theorem \ref{AZ-R_thm}) $j$ times. (The inequality 
$\delta_{\eta_j}\left(X; V_{\vec{\bullet}}\right)\leq 
1/b_r$ is trivial since $1/b_r=A_{X}(Y_1)/S\left(V_{\vec{\bullet}};Y_1\right)$ holds.)
\end{proof}

\begin{corollary}\label{pl_corollary}
Let $X$ be an $n$-dimensional normal projective variety, let $Y_\bullet$ be an 
admissible flag on $X$, let $L$ be a big $\Q$-Cartier $\Q$-divisor on $X$, and 
let $\Delta_{Y_\bullet}(L)\subset\R_{\geq 0}^n$ be the Okounkov body of $L$ 
associated to $Y_\bullet$. Let $\eta_j\in X$ be the generic point of $Y_j$ 
for $1\leq j\leq n$. 
\begin{enumerate}
\renewcommand{\theenumi}{\arabic{enumi}}
\renewcommand{\labelenumi}{(\theenumi)}
\item\label{pl_corollary1}
Let 
$\vec{b}=(b_1,\dots,b_{n})\in\R_{\geq 0}^{n}$ be the barycenter of 
$\Delta_{Y_\bullet}\left(L\right)$. Then we have the inequalities 
\[
\min\left\{\frac{1}{b_1},\dots,\frac{1}{b_{j}}\right\}\leq
\delta_{\eta_j}\left(X; L\right)\leq \frac{1}{b_1}
\]
for any $1\leq j\leq n$. In particular, we have 
\[
\min\left\{\frac{1}{b_1},\dots,\frac{1}{b_{n}}\right\}\leq
\delta_{Y_n}\left(X; L\right)\leq \frac{1}{b_1}.
\]
\item\label{pl_corollary2}
Assume moreover that 
$Y_n\not\in\B_-(L)$ $($e.g., L is nef$)$. Let $T_i\in\R_{> 0}$ be the maximum 
of the $i$-th coordinate of the Okounkov body $\Delta_{Y_\bullet}(L)
\subset\R_{\geq 0}^n$ for $1\leq i\leq n$. 
Then we have the inequality
\[
\delta_{\eta_j}(X; L)\geq \min\left\{\frac{n+1}{n T_1},\dots,\frac{n+1}{n T_j}\right\}.
\]
for any $1\leq j\leq n$. 
\end{enumerate}
\end{corollary}

\begin{proof}
\eqref{pl_corollary1} is a direct consequence of Theorem \ref{pl_thm}. 
For \eqref{pl_corollary2}, 
as we have already seen in \cite[Theorem 4.2]{CHPW}, the Okounkov body 
$\Delta_{Y_\bullet}(L)$ contains the origin. Therefore, the value $U_i$ in 
Corollary \ref{barycenter_corollary} is equal to $0$ for any $1\leq i\leq n$. 
Thus we get the assertion from Corollary \ref{barycenter_corollary} 
and \eqref{pl_corollary1}. 
\end{proof}

\end{document}